\documentclass[11pt]{amsart}
\usepackage[utf8]{inputenc}
\usepackage[T1]{fontenc}
\usepackage{mathtools}
\usepackage{amsmath,amsthm,amsfonts,amssymb,mathrsfs,pb-diagram,amscd}
\usepackage{amsmath}
\usepackage{amsxtra}
\usepackage{amscd}
\usepackage{amsthm}
\usepackage{amsfonts}
\usepackage{amssymb}
\usepackage{eucal}
\usepackage{xcolor}
\usepackage[enableskew,vcentermath]{youngtab}
\usepackage{bm}
\usepackage{leftindex}
\usepackage{graphics}
\usepackage{mathrsfs}
\usepackage{enumerate}
\usepackage{tikz}

\usepackage{pstricks}
\usepackage{graphicx}

\usepackage[pagebackref=false]{hyperref}

\DeclareMathOperator\Std{Std}
\DeclareMathOperator\res{res}
\DeclareMathOperator\sgn{sgn}
\DeclareMathOperator\rad{rad}
\allowdisplaybreaks[4]

	\newcommand{\hs}{\hspace*}
	\newcommand{\vs}{\vspace*}
\newtheorem{thm}{Theorem}[section]
\theoremstyle{plain}

\newtheorem{lem}[thm]{Lemma}

\newtheorem{prop}[thm]{Proposition}
\newtheorem{cor}[thm]{Corollary}

\theoremstyle{definition}
\newtheorem{defn}[thm]{Definition}
\newtheorem{example}[thm]{Example}

\theoremstyle{remark}
\newtheorem{rem}[thm]{Remark}

\definecolor{A}{rgb}{.75,1,.75}

\numberwithin{equation}{section}

\newcommand{\Z}{\mathbb Z}
\newcommand{\N}{\mathbb N}

\newcommand{\mHcn}{\mathcal{H}_{\rm R}} % affine Hecke-Clifford superalgebra
\newcommand{\mt}{\mathfrak{t}}
\newcommand{\ms}{\mathfrak{s}}
\newcommand{\mfku}{\mathfrak{u}}
\newcommand{\mfkv}{\mathfrak{v}}

\newcommand{\End}{\text{End}}

\newcommand{\supp}{\text{supp}}

\newcommand{\undla}{\underline{\lambda}}
\newcommand{\undmu}{\underline{\mu}}
\newcommand{\undQ}{\underline{Q}}

\newcommand{\Add}{{\rm Add}}
\newcommand{\Rem}{{\rm Rem}}
\newcommand{\Hom}{{\rm Hom}}

\def\O{\mathscr{O}}
\def\K{\mathscr{K}}
\def\hO{{\hat{\O}}}
\def\hK{{\hat{\K}}}
\def\rd{{\rm{d}}}
\def\Sym{\mathfrak{S}}
\def\pr{{\rm pr}}
\def\tQ{\widetilde{Q}}
\let\<=\langle
\let\>=\rangle
\def\({\big(}
\def\){\big)}

\newcommand{\fourlinerightarrow}{%
  \mathrel{%
    \vcenter{%
      \offinterlineskip
      \halign{##\cr
        $\rule{0.6em}{0.3pt}\kern-0.2em\rule{0.3em}{0.3pt}$\cr
        \noalign{\kern0.9pt}
        $\rule{0.6em}{0.3pt}\kern-0.2em\rule{0.3em}{0.3pt}$\cr
        \noalign{\kern0.9pt}
        $\rule{0.6em}{0.3pt}\kern-0.2em\rule{0.3em}{0.3pt}$\cr
        \noalign{\kern0.9pt}
        $\rule{0.6em}{0.3pt}\kern-0.2em\rule{0.3em}{0.3pt}$\cr
      }%
    }%
    \kern-0.2em\raisebox{0.0ex}{$\succ$}%
  }%
}

\newcommand{\fourlineleftarrow}{%
  \mathrel{%
    \raisebox{-0.2ex}{$\prec$}% 把箭头升高到和右箭头一样的高度
    \kern-0.2em
    \vbox{%
      \offinterlineskip
      \halign{##\cr
        $\rule{0.6em}{0.3pt}\kern-0.2em\rule{0.3em}{0.3pt}$\cr
        \noalign{\kern0.9pt}
        $\rule{0.6em}{0.3pt}\kern-0.2em\rule{0.3em}{0.3pt}$\cr
        \noalign{\kern0.9pt}
        $\rule{0.6em}{0.3pt}\kern-0.2em\rule{0.3em}{0.3pt}$\cr
        \noalign{\kern0.9pt}
        $\rule{0.6em}{0.3pt}\kern-0.2em\rule{0.3em}{0.3pt}$\cr
      }%
    }%
  }%
}

{\vskip-\lastskip\medskip
	\noindent
	{\em #1.}\enspace
}%
{\qed\par\medskip
}

\usepackage{enumitem}

\newlist{caselist}{enumerate}{1}
\setlist[caselist,1]{
	label=\textbf{Case \arabic{section}.\arabic{subsection}.\arabic*:},
	ref=\arabic{section}.\arabic{subsection}.\arabic*,
	before=\setcounter{caselisti}{0}
}

\newcounter{case}

\newenvironment{symbols}{
    \section*{Index of notation}
    \begin{description}
}{
    \end{description}
}

\newcommand{\symitem}[3]{\item[#1] #2 \hfill\rlap{\pageref{#3}}}

\begin{document}

	\title[cyclotomic quiver Hecke-Clifford superalgebras]{On the generalized graded cellular bases for cyclotomic quiver Hecke-Clifford superalgebras}
	
	\author{Shuo Li}\address{School of Mathematics and Statistics\\
		Beijing Institute of Technology\\
		Beijing, 100081, P.R. China}
	\email{shuoli1203@163.com}

    \author{Lei Shi}\address{Academy of Mathematics and Systems Science\\
    	Chinese Academy of Sciences, Beijing 100190\\
    	P.R.China}
    	\address{Max-Planck-Institut f\"ur Mathematik\\
	    Vivatsgasse 7, 53111 Bonn\\
	    Germany}
    \email{leishi202406@163.com}

	\begin{abstract}
In this paper, we construct semisimple deformations for cyclotomic quiver Hecke-Clifford superalgebras of types $A^{(1)}_{s-1}$, $C^{(1)}_{s}$, $A^{(2)}_{2s}$, $D^{(2)}_{s}$. We derive a unified dimension formula for the bi-weight spaces for cyclotomic quiver Hecke-Clifford superalgebras of types $A^{(1)}_{s-1}$, $C^{(1)}_{s}$, $A^{(2)}_{2s}$, $D^{(2)}_{s}$. We introduce the notion of generalized graded cellular superalgebra. We prove a large class of cyclotomic quiver Hecke-Clifford superalgebras of types $A^{(1)}_{s-1}$, $C^{(1)}_{s}$, $A^{(2)}_{2s}$, $D^{(2)}_{s}$ is generalized graded cellular. By taking idempotent truncation, this recovers the known graded cellualr results for cyclotomic quiver Hecke algebras of types $A^{(1)}_{s-1}$, $C^{(1)}_{s}$.
	\end{abstract}
	\maketitle
	
	\setcounter{tocdepth}{1}
	\tableofcontents
	
	\section{Introduction}
	The quiver Hecke algebras (or, KLR algebras) and their cyclotomic quotients were introduced in the work of Khovanov-Lauda (\cite{KL1}, \cite{KL2}) and of Rouquier (\cite{Rou1}). They play an important role in the categorification of quantum groups and their integrable highest weight modules (\cite{KK}). In the past decade, there have been many remarkable applications of these algebras in the modular representation theory of symmetric groups and Hecke algebras, low-dimensional topology and other areas, see \cite{Bow},  \cite{BK:GradedKL}, \cite{DVV}, \cite{Ev}, \cite{EK},\cite{HM1}, \cite{K2}, \cite{Rou2}, \cite{SVV}, \cite{VV}, \cite{Web} and the references therein.

	Kang, Kashiwara and Tshchioka \cite{KKT} generalized above KLR construction to the super case. They introduced several new families of algebras including the quiver Hecke superalgebras and quiver Hecke-Clifford superalgebras in  \cite{KKT}. To define these superalgebras, one has to decompose the index set of a generalised Cartan matrix $A$ (\cite{Kac}) as $I=I_{\rm{even}}\sqcup I_{\rm{odd}}$ subject to some natural conditions. When $I_{\rm{odd}}=\emptyset$, the construction of quiver Hecke superalgebras in \cite{KKT} reduces to the original KLR construction. Both of quiver Hecke superalgebras and quiver Hecke-Clifford superalgebras are $(\Z\times\Z_2)$-graded algebras. They also have some natural finite dimensional quotients, which are called cyclotomic quiver Hecke superalgebras and cyclotomic quiver Hecke-Clifford superalgebras. Kang, Kashiwara and Oh \cite{KKO1} used cyclotomic quiver Hecke superalgebras to categorify quantum Kac-Moody algebras, generalizeing \cite{KK}. In \cite{BKM}, quantum Kac-Moody superalgebras were introduced and well-studied. Hill and Wang \cite{HW} first observed that cyclotomic quiver Hecke superalgebras categorify quantum Kac-Moody superalgebras. Kang, Kashiwara and Oh \cite{KKO2} then defined several families of quantum superalgebras and used cyclotomic quiver Hecke superalgebras to study the supercategorification of quantum superalgebras. Recently, the quiver Hecke superalgebras have remarkable applications in the study of spin symmetric groups and the double cover of symmetric groups \cite{FKM,K3,KleL}.
	
	The cyclotomic quiver Hecke algebras are well understood for the quiver of types $A_{s-1}^{(1)}$ and $A_\infty$. In this case, Brundan and Kleshchev constructed in \cite{BK:GradedKL} an explicit algebra isomorphism between cyclotomic quiver Hecke algebra and the block algebra of the cyclotomic Hecke algebra. Based on this isomorphism, graded cellular bases, Specht modules and categorification theorem have been extensively studied in the literature \cite{BKgraded,BKW,HM1}. For the quiver of types $C_s^{(1)}$ and $C_\infty$,  Ariki, Park and Speyer \cite{APS} studied Specht modules for cyclotomic quiver Hecke algebras. Influenced by the combinatorics in \cite{APS}, Mathas and Tubbenhauer \cite{MT1} constructed graded cellular bases in affine type $C$ using the weighted KLRW algebras. In a remarkable paper \cite{EM}, Evseev and Mathas introduced a new notion called graded content system. They used graded content system to give a graded semisimple deformation for the cyclotomic quiver Hecke algebra and then constructed graded cellular structure for both cyclotomic quiver Hecke algebra of affine type $A$ and affine type $C$, following a similar idea as in \cite{HM1,HM2}. We emphasize that in \cite{EM,HM1,HM2}, the semisimple deformation and semisimple representatrion theory play key roles in approaching the graded cellular bases theory. In general, Hu and the second author of this paper \cite{HS2} gave a $(\Z\times\Z_2)$-graded
	dimension formula for the bi-weight spaces of the cyclotomic quiver Hecke (super)algebras for arbitrary symmetrisable Cartan superdatum and studied monomial bases for some bi-weight spaces, which generalized \cite{HS}. Unfortunately, a ``cellular structure'' for the general cyclotomic quiver Hecke superalgebra is still missing. This is the motivation of our work.
	
	\label{pag:N} Let $\N:=\{1,2,\ldots\},$ $n\in \N$ and $\mathbb{K}$ be an algebraically closed field of characteristic different from $2$. We use $R^\Lambda_\nu,\,RC^\Lambda_\nu$ to denote the cyclotomic quiver Hecke superalgebra and cyclotomic quiver Hecke-Clifford superalgebra over the field $\mathbb{K}$ associated to the Cartan superdatum $\bigl({\rm{A}}=(a_{ij})_{i,j\in I},P,\Pi,\Pi^\vee\bigr)$, $\nu\in Q_n^+$ and $\Lambda\in P^+$ as defined in \cite{KKT}. It was shown in \cite{KKT} that $R^\Lambda_\nu$ and $RC^\Lambda_\nu$ are weakly Morita superequivalent. On the other hand, Kang, Kashiwara and Tshchioka \cite{KKT} gave an isomorphism between $RC^\Lambda_\nu$ of affine types $A^{(1)}_{s-1}$, $C^{(1)}_{s}$, $A^{(2)}_{2s}$, $D^{(2)}_{s}$ and some ``blocks'' of the cyclotomic Hecke-Clifford superalgebra $\mathcal{H}^{f}_{\mathbb{K}}=\mathcal{H}^{f}_{\mathbb{K}}(n)$, which can be viewed as a super analogue of the Brundan-Kleshchev isomorphism. To be presice, for each given defining polynomial $f$ of $\mathcal{H}^{f}_{\mathbb{K}}$, we can associate $f$ with a Cartan superdatumn $I_f$ and a dominant weight $\Lambda_f \in P^+$. Then Kang-Kashiwara-Tsuchioka proved that there is a non-trivial isomorphism between $\mathcal{H}^{f}_{\mathbb{K}}$ and the corresponding cyclotomic quiver Hecke-Clifford superalgebra $RC^{\Lambda_f}_n=\bigoplus\limits_{\nu\in Q_n^+} RC^{\Lambda_f}_\nu$. The following Theorem is the first main result of this paper, where we refer the readers to Sections \ref{basic-Non-dege}, \ref{Idempotents and seminormal forms} for unexplained notations used here.
	
		\begin{thm}\label{Main1}
		Let ${\bf i},{\bf j} \in(J_f)^n$. We have
		\begin{align*}
			\dim_{\mathbb{K}}e({\bf i}) RC^{\Lambda_f}_n e({\bf j})=\sum_{\substack{\undla\in\mathscr{P}^{\bullet,m}_{n}}}2^{d_{\undla}}\sharp {\rm Tri}(\undla,{\bf i}) \sharp {\rm Tri}(\undla,{\bf j}).
		\end{align*}
		\end{thm}
		
	In contrast to \cite[Theorem 1.2]{HS2}, the terms appearing in the above equality are always non-negative.
%Under Kang-Kashiwara-Tsuchioka's isomorphism,
Theorem \ref{Main1} gives a unified dimension formula for the bi-weight space of quiver Hecke-Clifford superalgebra of affine types $A^{(1)}_{s-1}$, $C^{(1)}_{s}$, $A^{(2)}_{2s}$, $D^{(2)}_{s}$.
	By taking idempotent truncation, this further yields a dimension formula for the bi-weight spaces of the corresponding quiver Hecke superalgebras. In affine types $A^{(1)}_{s-1}$, $C^{(1)}_{s}$, this recovers the ungraded version of \cite[Theorem 4.20]{BKgraded} and \cite[Theorem 2.5]{APS}, while in affine types $A^{(2)}_{2s}$, $D^{(2)}_{s}$, this is \cite[Theorem B]{AP14}, \cite[Corollary 3.3]{AP16d} in the case when $\Lambda=\Lambda_0$. In other cases, our dimension formula seems to be new. Note that all of the proofs in \cite[Theorem 4.20]{BKgraded}, \cite[Theorem 2.5]{APS}, \cite[Theorem B]{AP14} and \cite[Corollary 3.3]{AP16d} rely on the Fock space realization with respect to certain dominant weight $\Lambda\in P^+$. It is natural to ask whether there is a Fock space model underlying Theorem \ref{Main1}.
	
	To prove Theorem \ref{Main1} we introduce a certain semisimple deformation of $\mathcal{H}^{f}_{\mathbb{K}}$. In fact, we construct two algebras $\mathcal{H}^{f'}_{\hK}, \mathcal{H}^{f'}_{\hO}$, where $\hO$ is a certain complete valuation ring and $\hK$ is the fraction field of $\hO$ satisfying  $$
	\mathcal{H}^{f'}_{\hK}\cong \hK\otimes_{\hO} \mathcal{H}^{f'}_{\hO},\qquad\mathcal{H}^{f}_{\mathbb{K}}\cong \mathbb{K}\otimes_{\hO}\mathcal{H}^{f'}_{\hO}
	$$  and $\mathcal{H}^{f'}_{\hK}$ is semisimple over $\hK$. The above semisimple deformation is obtained using \cite{SW}, where Wan and the second author of this paper introduced a separate condition for cyclotomic Hecke-Clifford superalgebra and proved that the cyclotomic Hecke-Clifford superalgebra is split semisimple if the separate condition holds. In \cite{LS2}, we further constructed a complete set of primitive idempotents and seminormal bases of $\mathcal{H}^{f'}_{\hK}$  (see also \cite{KMS} for the Sergeev superalgebra). This enables us to lift each idempotent $e({\bf i})\in \mathcal{H}^{f}_{\mathbb{K}}$ to $e({\bf i})^\hO\in\mathcal{H}^{f'}_{\hO}$ as a sum of some primitive idempotents in $\mathcal{H}^{f'}_{\hK}$. Then Theorem \ref{Main1} follows from seminormal bases theory. As a byproduct, we also obtain an upper bound of nilpotent index of polynomial generators $y_ke({\bf i})$ in quiver Hecke-Clifford superalgebra of affine types $A^{(1)}_{s-1}$, $C^{(1)}_{s}$, $A^{(2)}_{2s}$, $D^{(2)}_{s}$, which generalizes \cite[In the end of \S 4]{EM} and \cite[Corollary 4.31]{HM2}. By taking idempotent truncation, our construction gives a new semisimple deformation for quiver Hecke algebra of affine type $A^{(1)}$ or $C^{(1)}$. It would be interesting to study the relationship between our new semisimple deformation for cyclotomic quiver Hecke algebras of affine type $A$ and type $C$ with the content system in \cite[Definition 3A.1]{EM}.
		
	With the semisimple deformation and seminormal bases theory in hand, we are able to mimic the construction in \cite{EM} and \cite{HM1} to give some nice bases for cyclotomic quiver Hecke-Clifford superalgebra  $RC^\Lambda_\nu$. To explain our result, we introduce some notations. Let $q^2\neq \pm 1$, $\undQ=(Q_1,\cdots,Q_m)\in({\mathbb{K}}^*)^m$ and $f=f^{\mathsf{(0)}}_{\underline{Q}}=\prod_{i=1}^m \left(X_1+X^{-1}_1-\mathtt{q}(Q_i)\right)$, where $\mathtt{q}(x):=2\frac{x+x^{-1}}{q+q^{-1}}$ for any $x\in \mathbb{K}^*$. Recall that we have identified  the cyclotomic Hecke-Clifford superalgebra $\mathcal{H}^{f}_{\mathbb{K}}$ with the certain corresponding cyclotomic quiver Hecke-Clifford superalgebra $RC^{\Lambda_f}_n$ under Kang-Kashiwara-Tsuchioka's isomorphism. For any $\nu\in Q_n^+$ , we have a central idempotent $e^J_\nu \in\mathcal{H}^{f}_{\mathbb{K}}$ and $e^J_\nu \mathcal{H}^f_{\mathbb{K}}\cong RC^{\Lambda_f}_\nu$. We need an extra condition on $\nu$, namely, $\nu$ is $\undQ$-unremovable (see Definition \ref{tech condition}). Then we have the following Theorem, which is the second main result of this paper.

	\begin{thm}\label{Main2}
	Suppose $\nu\in Q_n^+$ is $\undQ$-unremovable.
	Then the algebra $RC^{\Lambda_f}_\nu$ is a generalized graded cellular superalgebra. Moreover, it is a graded supersymmetric superalgebra with a homogeneous supersymmetrizing form $t_{\nu}$ of degree $-2{\rm def}(\nu)$.
	\end{thm}

In \cite{LS3}, we introduced a supersymmetrizing form $t_{2m,n}$ on $\mathcal{H}^f_{\mathbb{K}}$ and computed the corresponding Schur elements. These are crucial in the proof of Theorem \ref{Main2}. As in \cite{EM} and \cite{HM1}, for any $\nu\in Q_n^+$, we first construct two sets $\Psi^{\hO,\lhd}_{\nu}, \Psi^{\hO,\rhd}_{\nu}\subset \mathcal{H}^{f'}_{\hO}$ and study the relations of elements in $\Psi^{\hO,\lhd}_{\nu}$ and $\Psi^{\hO,\rhd}_{\nu}$ with seminormal bases, which is quite more complicated than \cite{EM} and \cite{HM1}. It's not difficult to deduce $\Psi^{\hO,\lhd}_{\nu}$ and $\Psi^{\hO,\rhd}_{\nu}$ form two $\hK$-bases of $\mathcal{H}^{f'}_{\hK}$ by Theorem \ref{Main1}.  To prove $\Psi^{\hO,\lhd}_{\nu}$ and $\Psi^{\hO,\rhd}_{\nu}$ form two $\O$-bases of $\mathcal{H}^{f'}_{\hO}$, we need the condition that $\nu$ is $\undQ$-unremovable. Under this condition, we are able to prove that the Gram matrix of $\Psi^{\hO,\lhd}_{\nu}$ and $\Psi^{\hO,\rhd}_{\nu}$ with respect to the supersymmetrizing form $t^\hO_{2m,n}$ is invertible in $\hO$. Hence we obtain two homogeneous bases by specializing $\Psi^{\hO,\lhd}_{\nu}, \Psi^{\hO,\rhd}_{\nu}$ to $e^J_\nu \mathcal{H}^f_{\mathbb{K}}\cong RC^{\Lambda_f}_\nu$.

In proving Theorem \ref{Main2}, we also systematically study the degrees of standard tableaux with respect to different cyclotomic polynomials of cyclotomic Hecke-Clifford superalgebra (Definition \ref{deg of std tableaux}). By Kang-Kashiwara-Tsuchioka's isomorphism, this gives a unified definition for the degrees of standard tableaux in affine types $A^{(1)}_{s-1}$, $C^{(1)}_{s}$, $A^{(2)}_{2s}$, $D^{(2)}_{s}$, generalizing  \cite[(3.5), (3.6)]{BKW} and \cite[Definition 4D.3]{EM}. Our homogeneous supersymmetrizing form $t_{\nu}$ in Theorem \ref{Main2} is obtained by taking homogeneous truncation of $t_{2m,n}$, which is similar as in \cite{HM1}.

The generalized graded cellular superalgebra proposed here is a natural generalization of $\Z$-graded cellular algebra in \cite{HM1} to the $\Z\times \Z_2$-graded algebra. For example, we can similarly define specht modules and study the simple modules and decomposition matrix. We remark that our generalized graded cellular superalgebra is a special case of a more general definition given by Mori \cite{Mo}. Therefore, we can use Mori's general result in our setting.

By taking idempotent truncation, we have the following Corollary.

	\begin{cor}\label{CellandSym}
	Let $p={\rm Char}\,\mathbb{K}\neq 2$ and $s\geq 1$.
	\begin{enumerate}
		\item Suppose $p\nmid s$. Let $I$ be the Cartan datum corresponds to Dynkin quiver of type $A^{(1)}_{s-1}$ ($s\geq 2$) or $C^{(1)}_{s}$. Then for any $\nu\in Q^+_n$ and any $\Lambda\in P^+ $, the cyclotomic quiver Hecke algebra $R^\Lambda_\nu$ is a graded cellular algebra with a homogeneous symmetrizing form of degree $-2{\rm def}(\nu)$.
		\item Suppose $p\nmid 2s+1$. Let $I$ be the Cartan datum corresponds to Dynkin quiver of type $A^{(2)}_{2s}$, $\nu=\sum_{i\in I}m_i \nu_i\in Q^+_n$  and $\Lambda=\sum_{i\in I}k_i \Lambda_i\in P^+ $. Suppose $m_i\leq 1$ and $k_i\in 2\Z$ for any $i\in I_{{\rm odd}}$. Then the superalgebra $R^\Lambda_\nu\otimes \mathcal{C}_{m(\nu)}$ is a generalized graded cellular superalgebra with a homogeneous supersymmetrizing form of degree $-2{\rm def}(\nu)$.
		\item Suppose $p\nmid s$. Let $I$ be the Cartan datum corresponds to Dynkin quiver of type $D^{(2)}_{s}$, $\nu=\sum_{i\in I}m_i \nu_i\in Q^+_n$ and $\Lambda=\sum_{i\in I}k_i \Lambda_i\in P^+ $.  Suppose $m_i\leq 1$ and $k_i\in 2\Z$ for any $i\in I_{{\rm odd}}$. Then the superalgebra $R^\Lambda_\nu\otimes \mathcal{C}_{m(\nu)}$ is a generalized graded cellular superalgebra with a homogeneous supersymmetrizing form of degree $-2{\rm def}(\nu)$.
	\end{enumerate}
\end{cor}

Corollary \ref{CellandSym} (1) recovers the main result in \cite{EM,HM1}. We remark that for cyclotomic quiver Hecke algebra $R^\Lambda_\nu$ of affine type $A$, \cite[Corollary 6.18]{HM1} also gave a homogeneous symmetrizing form $\tau^{\text{HM}}_\nu$ of degree $-2{\rm def}(\nu)$. For both cyclotomic quiver Hecke algebra $R^\Lambda_\nu$ of affine type $A$ and type $C$, new homogeneous symmetrizing forms $\tau^{\text{EM}}_\nu$ of degree $-2{\rm def}(\nu)$ were obtained in \cite[Corollary 4F.8]{EM}. In general, Shan, Varagnolo and Vasserot \cite[Proposition 3.10]{SVV} have shown that the algebra $R^\Lambda_\nu$  is a $\Z$-graded symmetric algebra which is equipped with a homogeneous symmetrizing form $\tau^{\text{SVV}}_\nu$  of degree $-2{\rm def}(\nu)$. It's interesting to compare above-mentioned symmetrizing forms with $t_\nu$ in our paper.

We remark that our construction above should also work in degenerate case, i.e. cyclotomic Sergeev algebra.

Here is the layout of this paper. In Section \ref{preli}, we first recall some basics on general superalgebras and $\Z\times\Z_2$-graded algebras.
In Section \ref{Generalized graded cellular superalgebra}, we define generalized graded cellular superalgebra and study the representation theory of generalized graded cellular superalgebra.
%with some nice assumptions.
In Section \ref{Quiver Hecke superalgebra and Quiver Hecke-Clifford superalgebra}, we recall the definition of quiver Hecke superalgebras and quiver Hecke-Clifford superalgebras as well as their cyclotomic quotients.
In Section \ref{basic-Non-dege}, we recall the notion of affine Hecke-Clifford superalgebra $\mathcal{H}_{\rm R}$, cyclotomic Hecke-Clifford superalgebra $\mathcal{H}^f_{\rm R}$ over integral domain ${\rm R}$, as well as the associated combinatorics and the Separate Conditions. We explain how to relate $\mathcal{H}^f_{\mathbb{K}}$ with a Dynkin quiver and then recall Kang-Kashiwara-Tsuchioka's isomorphism in subsections \ref{DynkinDiagrams}, \ref{KKT's isomorphism}. We also define and study the degrees of standard tableaux in subsection \ref{Degrees of standard tableaux}.
In Section \ref{Idempotents and seminormal forms}, we recall the separate condition and seminormal bases theory for $\mathcal{H}^f_{\mathbb{K}}$. We construct a semisimple deformation in subsection \ref{Lifting idempotents} and prove the Theorem \ref{Main1}.
Section \ref{Generalized graded super cellular bases for cyclotomic quiver Hecke-Clifford superalgebra} is the core of this paper. We define some integral elements inside the deformed algebra $\mathcal{H}^{f'}_{\hO}$ and study the relations of these elements with seminormal bases. We define $\undQ$-unremovable element and then prove a graded bases result for $e^J_\nu \mathcal{H}^f_{\mathbb{K}}$  in subsection \ref{Seminormal bases and integral bases}. The proof of our Theorem \ref{Main2} is completed in subsections \ref{Generalized graded super cellular datum} and \ref{Graded supersymmetrizing form}. We then prove Corollary \ref{CellandSym} in subsection \ref{Idempotent truncation}.

\bigskip
\centerline{\bf Acknowledgements}
\bigskip
The research is supported by the National Natural Science Foundation of China (No. 12431002).
%and the Natural Science Foundation of Beijing Municipality (No. 1232017).
The second author is partially supported by the Postdoctoral Fellowship Program of CPSF under Grant Number GZB20250717. Both the authors thank Weiqiang Wang for his comments.
\bigskip

	\section{Preliminary}\label{preli}
{\bf Throughout this paper, ${\rm R}$ \label{pag:integral domain R} is an integral domain of characteristic different from $2$ and $\mathbb{K}$ is an algebraically closed field of characteristic different from $2$.}

	\subsection{Some basics about superalgebra}
	We recall some basic notions of superalgebras. We refer the reader to ~\cite[\S 2-b]{BK}. Let us denote by
	${\rm p}(v)\in\mathbb{Z}_2$ \label{pag:||} the parity of a homogeneous vector $v$ of a
	${\rm R}$-vector superspace. By a superalgebra, we mean a
	$\mathbb{Z}_2$-graded associative ${\rm R}$-algebra. Let $\mathcal{A}$ be a
	${\rm R}$-superalgebra. By an $\mathcal{A}$-module, we mean a $\mathbb{Z}_2$-graded
	left $\mathcal{A}$-module.	A homomorphism $f:V\rightarrow W$ of
	$\mathcal{A}$-modules $V$ and $W$ means a linear map such that $
	f(av)=(-1)^{{\rm p}(f){\rm p}(a)}af(v).$  Note that this and other such
	expressions only make sense for homogeneous $a, f$ and the meaning
	for arbitrary elements is to be obtained by extending linearly from
	the homogeneous case. A non-zero element $e\in\mathcal{A}$ is called a super primitive idempotent if
	it is an idempotent with ${\rm p}(e)=\bar{0}$ and it cannot be decomposed as the sum of
	two nonzero orthogonal idempotents with parity $\bar{0}.$ Let $V$ be an	$\mathcal{A}$-module. Let $\Pi
	V$ \label{pag:parity shift} be the same underlying vector space but with the opposite
	$\mathbb{Z}_2$-grading. The new action of $a\in\mathcal{A}$ on $v\in\Pi
	V$ is defined by $a\cdot
	v:=(-1)^{{\rm p}(a)}av$. Note that the identity map on $V$ defines
	an isomorphism from $V$ to $\Pi V.$
    More generally, the homomorphism $f:V\rightarrow W$ of
	$\mathcal{A}$-modules $V$ and $W$ is odd (resp., even) if and only if the same map
    $f:\Pi V\rightarrow W$ or $f:V\rightarrow \Pi W$ is even (resp., odd).

	A superalgebra analog of Schur's Lemma states that the endomorphism
	algebra of a finite dimensional irreducible module over a
	$\mathbb{K}$-superalgebra is either one dimensional or two dimensional. In the
	former case, we call the module of {\em type }\texttt{M} while in
	the latter case the module is called of {\em type }\texttt{Q}.

	\begin{example}\label{simple algebra}
		1). Let $V$ be a superspace with superdimension $(m,n)$ over $\mathbb{K}$, then $\mathcal{M}_{m,n}:={\text{End}}_{\mathbb{K}}(V)$ is a simple superalgebra with simple module $V$ of {\em type }\texttt{M}. Then the set of super primitive idempotents of $\mathcal{M}_{m,n}$ is
		$\{E_{ii} \mid i=1,\ldots,m+n\}.$ One can see that there is an evenly $\mathcal{M}_{m,n}$-supermodule isomorphism $V \cong \mathcal{M}_{m,n}E_{ii}$ if $i\in \{1,\ldots,m\},$ and there is an evenly $\mathcal{M}_{m,n}$-supermodule isomorphism $\Pi V \cong \mathcal{M}_{m,n}E_{ii}$ if $i\in \{m+1,\ldots,m+n\}.$\\
		2). Let $V$ be a superspace with superdimension $(n,n)$ over field $\mathbb{K}$. We define $\mathcal{Q}_n:=
\Biggl\{\biggl(\begin{matrix} A & B\\
                             -B & A
		\end{matrix}\biggr) \biggm| A,B\in  M_n\Biggr\}\subset \mathcal{M}_{n,n}$. Then the set of super primitive idempotents of $\mathcal{Q}_n$ is
		$\Biggl\{\biggl(\begin{matrix}
            E_{ii} & 0\\
			0& E_{ii}
		\end{matrix}\biggr) \biggm| i\in  \{1,\ldots,n\} \Biggr\}$
		and there is an evenly $\mathcal{Q}_n$-supermodule isomorphism $V \cong \mathcal{Q}_n \biggl(\begin{matrix}
            E_{ii} & 0\\
			0& E_{ii}
		\end{matrix}\biggr)$ for each $i=1,\ldots,n.$
		
	\end{example}

	Recall that $\mathbb{K}$ is an algebraically closed field of characteristic different from $2$. For any $a\in \mathbb{K}$, we fix a solution of the equation $x^2=a$ and denote it by $\sqrt{a}$.
	
	Let $A$ be any algebra and $a_1,a_2,\ldots,a_p\in A$, we define the ordered product \label{pag:ordered product} as
	$$\overrightarrow{\prod_{i=1,2,\ldots, p}}a_i:=a_1 a_2 \ldots a_p.$$
	
	\begin{example}\cite[Lemma 2.4]{LS2}\label{lem:clifford rep}
			Let $\mathcal{C}_n$ \label{pag:Clifford algebra} be the Clifford superalgebra over $\mathbb{K}$ generated by odd generators $C_1,\ldots,C_n,$ subject to the following relations
		$$C_i^2=1,C_iC_j=-C_jC_i, \quad 1\leq i\neq j\leq n.$$
We define
		$$I_n:=\begin{cases}
			\{1\}, &\text{if $n=1$;}\\
			\Biggl\{2^{-\lfloor n/2 \rfloor}\cdot\overrightarrow{\prod}_{k=1,\cdots,{\lfloor n/2 \rfloor}}(1+(-1)^{a_k} \sqrt{-1}C_{2k-1}C_{2k})\Biggm|a_k\in\Z_2,\,1\leq k\leq {\lfloor n/2 \rfloor} \Biggr\}, &\text{if $n>1$,}
		\end{cases}$$
		where $\lfloor n/2 \rfloor$ \label{pag:round down} denotes the greatest integer less than or equal to $n/2.$
		Then $\mathcal{C}_n$ is a simple superalgebra with the unique simple (super)module of type $\texttt{Q}$ if $n$ is odd, of type $\texttt{M}$ if $n$ is even. The set $I_n$ forms a complete set of super primitive idempotents for $\mathcal{C}_n.$
	
	\end{example}
	
	{\bf In the rest of this subsection, we assume ${\rm R}=\mathbb{K}$.}
	
	Given two superalgebras $\mathcal{A}$ and $\mathcal{B}$, we view
	the tensor product of superspaces $\mathcal{A}\otimes\mathcal{B}$
	as a superalgebra with multiplication defined by
	$$
	(a\otimes b)(a'\otimes b')=(-1)^{{\rm p}(b){\rm p}(a')}(aa')\otimes (bb'),
	\qquad a,a'\in\mathcal{A}, b,b'\in\mathcal{B}.
	$$
	Suppose $V$ is an $\mathcal{A}$-module and $W$ is a
	$\mathcal{B}$-module. Then $V\otimes W$ affords $\mathcal{A}\otimes\mathcal{B}$-module
	denoted by $V\boxtimes W$ via
	$$
	(a\otimes b)(v\otimes w)=(-1)^{{\rm p}(b){\rm p}(v)}av\otimes bw,~a\in A,
	b\in B, v\in V, w\in W.
	$$
	
	If $V$ is an irreducible $\mathcal{A}$-module and $W$ is an
	irreducible $\mathcal{B}$-module, $V\boxtimes W$ may not be
	irreducible. Indeed, we have the following standard lemma (cf.
	\cite[Lemma 12.2.13]{K1}).
	\begin{lem}\label{tensorsmod}
		Let $V$ be an irreducible $\mathcal{A}$-module and $W$ be an
		irreducible $\mathcal{B}$-module.
		\begin{enumerate}
			\item If both $V$ and $W$ are of type $\texttt{M}$, then
			$V\boxtimes W$ is an irreducible
			$\mathcal{A}\otimes\mathcal{B}$-module of type $\texttt{M}$.
			
			\item If one of $V$ or $W$ is of type $\texttt{M}$ and the other
			is of type $\texttt{Q}$, then $V\boxtimes W$ is an irreducible
			$\mathcal{A}\otimes\mathcal{B}$-module of type $\texttt{Q}$.
			
			\item If both $V$ and $W$ are of type $\texttt{Q}$, then
			$V\boxtimes W\cong X\oplus \Pi X$ for a type $\texttt{M}$
			irreducible $\mathcal{A}\otimes\mathcal{B}$-module $X$.
		\end{enumerate}
		Moreover, all irreducible $\mathcal{A}\otimes\mathcal{B}$-modules
		arise as constituents of $V\boxtimes W$ for some choice of
		irreducibles $V,W$.
	\end{lem}
	
	If $V$ is an irreducible $\mathcal{A}$-module and $W$ is an
	irreducible $\mathcal{B}$-module, denote by $V\circledast W$ \label{pag:irrtensor} an
	irreducible component of $V\boxtimes W$. Thus,
	$$
	V\boxtimes W=\left\{
	\begin{array}{ll}
		V\circledast W\oplus \Pi (V\circledast W), & \text{ if both } V \text{ and } W
		\text{ are of type }\texttt{Q}, \\
		V\circledast W, &\text{ otherwise}.
	\end{array}
	\right.
	$$
%	\subsection{Clifford algebra $\mathcal{C}_n$}
%	In this subsection, we shall recall the representation theory of Clifford superalgebra $\mathcal{C}_n$ which will be used in later sections.

 \subsection{Generality on $\Z\times\Z_2$ graded algebra}
 A $\Z\times\Z_2$-graded ${\rm R}$-module (or graded ${\rm R}$-supermodule, or shortly, graded) is an ${\rm R}$-module $M$ which has a direct sum
 decomposition $$M=\bigoplus_{(d,a)\in\Z\times\Z_2}M_{d,a},$$ such that each $M_{d,a}$ is an ${\rm R}$-module, for any $(d,a)\in\Z\times\Z_2$.

Let $M$ be a $\Z\times\Z_2$-graded ${\rm R}$-module. We set $M_d=\bigoplus_{a\in\Z_2}M_{d,a}$ for any $d\in \Z,$ and $M_{a}=\bigoplus_{d\in\Z}M_{d,a},$ for any $a\in\Z_2$.
Let $(d,a)\in\Z\times\Z_2$ and $m\in M_{d,a}$. We say $m$ is ($\Z\times\Z_2$)-homogeneous of bidegree $(d,a)$ and use notations $\deg m=d,$ ${\rm p}(m)=a$. We use $\underline{M}$ to denote the
 ungraded ${\rm R}$-module obtained from $M$ by forgetting the $\Z\times\Z_2$-grading on $M$. For $l\in\Z$, let $M\langle l\rangle$ \label{pag:degree shift} be the $\Z\times\Z_2$-graded
 ${\rm R}$-module obtained by shifting the $\Z$-grading on $M$ up by $l,$ that is,
 $M\langle l\rangle_{d,a}=M_{d-l,a}$ for $d\in\Z$. Furthermore, for $b\in\Z_2$, the $\Z\times\Z_2$-graded ${\rm R}$-module
 $\Pi^b M\langle l\rangle$ is obtained by setting $(\Pi^b M\langle l\rangle)_{d,a}=M_{d-l,a+b}$ for $(d,a)\in\Z\times\Z_2.$

 A $\Z\times\Z_2$-graded ${\rm R}$-algebra is a unital associative ${\rm R}$-algebra
 $\mathcal{A}=\bigoplus_{(d,a)\in\Z\times\Z_2}\mathcal{A}_{d,a}$ which is a $\Z\times\Z_2$-graded ${\rm R}$-module such that
 $\mathcal{A}_{d,a}\mathcal{A}_{e,b}\subseteq \mathcal{A}_{d+e,a+b},$ for all $d,e\in\Z,$ $a,b\in\Z_2.$
It follows from definition that $1\in\mathcal{A}_{0,\bar{0}}.$
 A graded (left) $\mathcal{A}$-module is a $\Z\times\Z_2$-graded ${\rm R}$-module $M$ such that $\underline{M}$ is
 an $\underline{\mathcal{A}}$-module and $\mathcal{A}_{d,a} M_{e,b}\subseteq M_{d+e,a+b}$, for all
 $d,e\in\Z,$ $a,b\in\Z_2.$ Then the notions of  $\Z\times\Z_2$-graded submodules, $\Z\times\Z_2$-graded quotient modules, and $\Z\times\Z_2$-graded right $\mathcal{A}$-modules are defined in the obvious way.

 Let $\mathcal{A}$ be a $\Z\times\Z_2$-graded ${\rm R}$-algebra. We define $\mathcal{A}\text{-}{\rm Mod}$ to be the category of all
 finitely generated $\Z\times\Z_2$-graded left $\mathcal{A}$-modules together with bidegree preserving
 homomorphisms, that is,
 $${\rm hom}_{\mathcal{A}}(M,N)=\{f\in{\rm Hom}_{\underline{\mathcal{A}}}(\underline{M},\underline{N})\mid
 f(M_{d,a})\subseteq N_{d,a}\text{ for all }(d,a)\in\Z\times\Z_2 \},$$
 for all $M,N\in \mathcal{A}\text{-}{\rm Mod}.$ We define
 $$
 {\rm Hom}_{\mathcal{A}}(M,N):=\bigoplus_{(d,a)\in\Z\times\Z_2}{\rm hom}_{\mathcal{A}}(\Pi^a M\langle d\rangle,N)
 $$
 for $M,N\in \mathcal{A}\text{-}{\rm Mod}.$ Then ${\rm Hom}_{\mathcal{A}}(M,N)$ is a $\Z\times\Z_2$-graded ${\rm R}$-module with ${\rm Hom}_{\mathcal{A}}(M,N)_{d,a}:={\rm hom}_{\mathcal{A}}(\Pi^a M\langle {d}\rangle,N)$. Therefore, any $f\in {\rm hom}_{\mathcal{A}}(\Pi^a M\langle d\rangle,N)$ is a homogeneous map from $M$ to $N$ of bidegree $(d,a)\in\Z\times\Z_2,$, i.e., $\deg f=d,\,{\rm p}(f)=a.$ In particular, the elements of ${\rm hom}_{\mathcal{A}}(M,N)$ are homogeneous
 maps of bidegree $(0,\bar{0}).$

\section{Generalized graded cellular superalgebra}\label{Generalized graded cellular superalgebra}
In this section, we introduce the notion of generalized graded cellular superalgebras
and establish their representation theory. This generalises Graham-Lehrer's \cite{GL} theory of cellular algebras and
Hu-Mathas's \cite{HM1} theory of $\Z$-graded cellular algebras.

\subsection{Generalized graded cellular superalgebra}
Let ${\rm K}$ be a field of characteristic different from $2$.
\label{pag:GGCdatum}
\begin{defn}
Suppose	$\mathcal{A}$ is a finite dimensional $\Z$-graded ${\rm K}$-superalgebra, and ${\rm K}$ is concentrated on $\Z$ degree $0$ and $\Z_2$ degree $\bar{0}$. A {\bf generalized graded super cell datum} for $\mathcal{A}$ is an ordered hextuple $(\mathscr{P}, \mathscr{T},\mathscr{B},\mathscr{C},\deg,{\rm p})$, where
\begin{enumerate}
\item[(1)] $(\mathscr{P},\lhd)$ is a finite poset;
\item[(2)] for any $\lambda\in\mathscr{P}$, there is a finite set $\mathscr{T}(\lambda)$;
\item[(3)] for any $\lambda\in\mathscr{P}$, there is a (finite dimensional) {\bf semisimple} superalgebra $\mathcal{B}_\lambda$ with a homogeneous ${\rm K}$-basis $\mathscr{B}_\lambda$, which is concentrated on $\Z$-degree $0$;
\item[(4)] $\mathscr{C}:\bigsqcup_{\lambda\in\mathscr{P}} \mathscr{T}(\lambda)\times {\mathcal{B}_\lambda}\times  \mathscr{T}(\lambda)\rightarrow \mathcal{A}; (i,u,j)\mapsto c^\lambda_{i,u,j},\,\deg: \bigsqcup_{\lambda\in\mathscr{P}}   \mathscr{T}(\lambda)\rightarrow \Z,\, {\rm p}:\bigsqcup_{\lambda\in\mathscr{P}} \mathscr{T}(\lambda)\rightarrow \Z_2$ are three functions such that $\mathscr{C}$ is injective.
\end{enumerate}
Moreover, we have the following conditions.
\begin{enumerate}
	\item[(GCd)] Each element $c^\lambda_{i,u,j}$ is homogeneous of $\Z$-degree $\deg(i)+\deg(j)$ and $\Z_2$-degree ${\rm p}(i)+{\rm p}(j)+{\rm p}(u)$, where $i,j\in  \mathscr{T}(\lambda), u\in\mathscr{B}_\lambda,\lambda\in \mathscr{P}$.
	\item[(GC1)] $\{c^\lambda_{i,u,j}\mid i,j\in  \mathscr{T}(\lambda), u\in\mathscr{B}_\lambda,\lambda\in \mathscr{P}\}$ forms a homogeneous ${\rm K}$-basis of $\mathcal{A}$ for $i,j\in  \mathscr{T}(\lambda), u\in\mathscr{B}_\lambda,\lambda\in \mathscr{P}$.
 \item[(GC2)] The function $\mathscr{C}$ is ${\rm K}$-linear in ${\mathcal{B}_\lambda},$ that means, we have $rc^\lambda_{i,u,j}+r'c^\lambda_{i,u',j}=c^\lambda_{i,ru+r'u',j}$ for $i,j\in  \mathscr{T}(\lambda), u,u'\in {\mathcal{B}_\lambda},\lambda\in \mathscr{P}$ and $r,r'\in{\rm K}$.
	\item [(GC3)] For any $i,j,i',j'\in  \mathscr{T}(\lambda), u,u',u''\in\mathscr{B}_\lambda,\lambda\in \mathscr{P}$, we have a function $r^{i',u'}_{i,u}: \mathcal{A} \rightarrow {\rm K}: a\mapsto r^{i',u'}_{i,u}(a)$ such that for any $a\in \mathcal{A}$ and $c^\lambda_{i,u,j}$ where $i,j\in  \mathscr{T}(\lambda), u\in\mathscr{B}_\lambda,\lambda\in \mathscr{P}$, we have
\begin{align}\label{eq2.GC2}
a c^\lambda_{i,uu'',j}=\sum_{\substack{i'\in  \mathscr{T}(\lambda)\\u'\in \mathscr{B}_\lambda}}r^{i',u'}_{i,u}(a)c^\lambda_{i',u'u'',j}\pmod{\mathcal{A}^{\lhd \lambda}},
\end{align}
where $$\mathcal{A}^{\lhd\lambda}:=\sum_{\substack{(i,u,j)\in \mathscr{T}(\mu)\times \mathcal{B}_\mu\times \mathscr{T}(\mu)\\ \mu\lhd\lambda}}{\rm K}c^\mu_{i,u,j}.$$
	\item[(GC4)]  For each $\lambda\in \mathscr{P}$, there is an ${\rm K}$-algebraic anti-involution $\omega_\lambda$ on ${\mathcal{B}_\lambda}$ and the ${\rm K}$-linear map $*:\mathcal{A}\rightarrow \mathcal{A}$ determined by $(c^\lambda_{i,u,j})^*=c^\lambda_{j,\omega_\lambda(u),i}$ where $i,j\in  \mathscr{T}(\lambda),u\in \mathscr{B}_\lambda,\lambda\in\mathscr{P}$ is an anti-isomorphism of $\mathcal{A}$.
		\end{enumerate}

	A {\bf generalized graded cellular superalgebra} is a $\Z$-graded superalgebra which has a generalized graded super cellular datum. The
	basis $\{c^\lambda_{i,u,j} \mid i,j\in  \mathscr{T}(\lambda), u\in\mathscr{B}_\lambda,\lambda\in \mathscr{P}\}$ is a {\bf generalized graded super cellular basis} of $\mathcal{A}$.
	\end{defn}

\begin{rem}
For any $i,j,i',j'\in  \mathscr{T}(\lambda), u,u'\in\mathscr{B}_\lambda,\lambda\in \mathscr{P}$, by \eqref{eq2.GC2}, we have in particular
		 \begin{align}\label{eq1.GC2}
			a c^\lambda_{i,u,j}=\sum_{\substack{i'\in  \mathscr{T}(\lambda)\\u'\in \mathscr{B}_\lambda}}r^{i',u'}_{i,u}(a)c^\lambda_{i',u',j}\pmod{\mathcal{A}^{\lhd \lambda}}.
		\end{align}
	
	\end{rem}
	\begin{example}
		\begin{enumerate}
\item If we forget the $\Z_2$ grading, and ${\mathcal{B}_\lambda}={\rm K},\,\omega_\lambda={\rm id}_{\rm K}$ for all $\lambda\in\mathscr{P}$, then we recover the definition of $\Z$-graded cellular algebra \cite{HM1}. If we further forget the $\Z$-grading, then we recover the original definition of cellular algebra \cite{GL}.

\item  Let's consider the semisimple superalgebra $\mathcal{M}_{1,1}$. Then $\mathcal{M}_{1,1}$ is a generalized graded cellular superalgebra with $\mathscr{P}=\{\star\}$ being the set consisting of a single element,
    $\mathscr{T}(\star)=\{1,2\},\,\mathcal{B}_\star={\rm K},\,\mathscr{B}(\star)=\{1\}$, and $$c^\star_{1,1}=E_{12},\,c^\star_{1,2}=E_{11},\,c^\star_{2,1}=E_{22},\,\,c^\star_{2,2}=E_{21},$$
where $\deg(1)=-1,\,\deg(2)=1,{\rm p}(1)={\rm p}(2)=\bar{0}$ and $\omega_\star$ being the identity map.

\item Let's consider the semisimple superalgebra $\mathcal{Q}_2$. We have $\mathcal{Q}_2\cong \mathcal{M}_2(\mathcal{C}_1)$. Then $\mathcal{Q}_2$ is a generalized graded cellular superalgebra with $\mathscr{P}=\{\star\}$ being the set consisting of a single element, $ \mathscr{T}(\star)=\{1,2\},\,\mathcal{B}_\star=\mathcal{C}_1,\,\mathscr{B}_\star=\{1,C_1\},$ and $$c^\star_{1,u,1}=E_{12}(u),\,c^\star_{1,u,2}=E_{11}(u),\,c^\star_{2,u,1}=E_{22}(u),\,\,c^\star_{2,u,2}=E_{21}(u),$$   where $u\in\mathscr{B}(\star)$, $\deg(1)=-1,\,\deg(2)=1,{\rm p}(1)={\rm p}(2)=\bar{0}$ for $u\in \mathscr{B}(\star),$ and $\omega_\star$ being the identity map.
\end{enumerate}
\end{example}

{\bf Throughout this section, we shall assume $\mathcal{A}$ is a generalized graded cellular superalgebra over ${\rm K}$ with generalized graded super cellular datum $(\mathscr{P},\mathscr{T},\mathscr{B},\mathscr{C},\deg,{\rm p})$.}
\begin{lem}\label{multiply}
	For any $i,j,i',j'\in  \mathscr{T}(\lambda), u,u',u''\in\mathscr{B}_\lambda,\lambda\in \mathscr{P}$ and $a\in \mathcal{A}$, we have$$c^\lambda_{i,\omega_\lambda(u),j}a=\sum_{\substack{j'\in  \mathscr{T}(\lambda)\\u'\in\mathscr{B}_\lambda}}r^{j',u'}_{j,u}(a^*)c^\lambda_{i,\omega_\lambda(u'),j'}\pmod{\mathcal{A}^{\lhd \lambda}}
	$$  and $$c^\lambda_{i,\omega_\lambda(u'')\omega_\lambda(u),j}a=\sum_{\substack{j'\in  \mathscr{T}(\lambda)\\u'\in\mathscr{B}_\lambda}}r^{j',u'}_{j,u}(a^*)c^\lambda_{i,\omega_\lambda(u'')\omega_\lambda(u'),j'}\pmod{\mathcal{A}^{\lhd \lambda}}.
	$$
	\end{lem}
	
	\begin{proof}
		By \eqref{eq1.GC2} and (GC4), we have \begin{align*}
			c^\lambda_{i,\omega_\lambda(u),j} a&=(a^*c^\lambda_{j,u,i})^*\\
			&=\sum_{\substack{j'\in  \mathscr{T}(\lambda)\\u'\in\mathscr{B}_\lambda}}(r^{j',u'}_{j,u}(a^*)c^\lambda_{j',u',i})^*\\
			&=\sum_{\substack{j'\in  \mathscr{T}(\lambda)\\u'\in\mathscr{B}_\lambda}}r^{j',u'}_{j,u}(a^*)c^\lambda_{i,\omega_\lambda(u'),j'}\pmod{\mathcal{A}^{\lhd \lambda}}.
			\end{align*}  This proves the first equation. By (GC2)  and (GC4), we deduce that\begin{align}\label{dual 0}(c^\lambda_{i,u,j})^*=c^\lambda_{j,\omega_\lambda(u),i}\end{align} for  $i,j\in  \mathscr{T}(\lambda),u\in {\mathcal{B}_\lambda},\lambda\in\mathscr{P}$. We can compute \begin{align*}
c^\lambda_{i,\omega_\lambda(u'')\omega_\lambda(u),j}a&=(a^*c^\lambda_{i,uu'',j})^*\\
&=\left(\sum_{\substack{i'\in  \mathscr{T}(\lambda)\\u'\in \mathscr{B}_\lambda}}r^{i',u'}_{i,u}(a^*)c^\lambda_{i',u'u'',j}\right)^*\\
&=\sum_{\substack{i'\in  \mathscr{T}(\lambda)\\u'\in \mathscr{B}_\lambda}}r^{i',u'}_{i,u}(a^*)c^\lambda_{i',\omega_\lambda(u'')\omega_\lambda(u'),j}\pmod{\mathcal{A}^{\lhd \lambda}},
			\end{align*} where in the first and last equation, we have used \eqref{dual 0}. This completes the proof.
		\end{proof}

	For $\lambda\in \mathscr{P}$, let $$\mathcal{A}^{\unlhd\lambda}:=\sum_{\substack{(i,u,j)\in T(\mu)\times \mathcal{B}_\mu\times T(\mu)\\ \mu\unlhd\lambda}}{\rm K}c^\lambda_{i,u,j},$$ then by (GC2), (GC3) and Lemma \ref{multiply}, we deduce that $\mathcal{A}^{\lhd\lambda}$  and $\mathcal{A}^{\unlhd\lambda}$ is a two-sided ideal of $\mathcal{A}$.
	
%\begin{cor}
%	Suppose $\mathcal{A}$ is a generalized (graded) cellular superalgebra over ${\rm K}$ with generalized super cell datum $(\mathscr{P},T,\mathscr{B},C)$. Suppose for $j,j'\in  \mathscr{T}(\lambda), u,u'\in\mathscr{B}_\lambda,\lambda\in \mathscr{P}$, $r^{j',u'}_{j,u}\neq 0$, then we have $t^\lambda_{j',u',i}t^\lambda_{i,u,j}\in R^\times$ is independent of $i$ for $i\in  \mathscr{T}(\lambda)$, which we denote by $k^{j',u'}_{j,u}\in R^\times$.
%	\end{cor}		
%	\begin{proof}
%		This follows from Lemma \ref{multiply} and (GC2).
%		\end{proof}
%		
%		Henceforth, for each $\lambda\in\mathscr{P}$, we fix $i(\lambda)\in  \mathscr{T}(\lambda)$ and define $k^{j',u'}_{j,u}:=t^\lambda_{j',u',i(\lambda)}t^\lambda_{i(\lambda),u,j}$ where $j,j'\in  \mathscr{T}(\lambda), u,u'\in\mathscr{B}_\lambda$. Now the second equation of Lemma \ref{multiply} can be written as \begin{equation}\label{right action}c^\lambda_{i,u,j}a=\sum_{\substack{j'\in  \mathscr{T}(\lambda)\\u'\in\mathscr{B}_\lambda}}r^{j',u'}_{j,u}(a^*)k^{j',u'}_{j,u}c^\lambda_{i,u',j'}\pmod{A^{\rhd \lambda}}.
%	\end{equation}

\begin{defn}
	For $\lambda\in \mathscr{P}$, we define a $(\mathcal{A},{\mathcal{B}_\lambda})$-bimodule $M_\lambda$ as a finitely generated ${\mathcal{B}_\lambda}$-module with right homogeneous ${\mathcal{B}_\lambda}$-basis $\{a^\lambda_i \mid i\in  \mathscr{T}(\lambda)\}$, where $\deg(a^\lambda_i)=\deg(i)$ and ${\rm p}(a^\lambda_i)={\rm p}(i)$, for $i\in  \mathscr{T}(\lambda)$, and the $(A,{\mathcal{B}_\lambda})$-bimodule structure on $M_\lambda$ is given by
$$a \cdot\bigl(a^\lambda_iu\bigr)=\sum_{\substack{i'\in  \mathscr{T}(\lambda)\\u'\in \mathscr{B}_\lambda}}r^{i',u'}_{i,u}(a)a^\lambda_{i'}u',\qquad (a^\lambda_iu)\cdot v:=a^\lambda_i(uv),$$
	for $a\in \mathcal{A},u,v\in \mathscr{B}_\lambda,i\in  \mathscr{T}(\lambda)$.
%	Note that the $(A,{\mathcal{B}_\lambda})$-bimodule structure is well-defined due to both \eqref{eq1.GC2} and \eqref{eq2.GC2} in (GC3).

	Similarly, we define a $({\mathcal{B}_\lambda},\mathcal{A})$-bimodule $N_\lambda$ as a  finitely generated ${\mathcal{B}_\lambda}$-module with a left homogeneous ${\mathcal{B}_\lambda}$-basis $\{b^\lambda_i \mid i\in  \mathscr{T}(\lambda)\}$ where $\deg(b^\lambda_i)=\deg(i)$ and ${\rm p}(b^\lambda_i)={\rm p}(i)$, for $i\in  \mathscr{T}(\lambda)$, and the $({\mathcal{B}_\lambda},\mathcal{A})$-bimodule structure on $N_\lambda$ is given by $$v\cdot(ub^\lambda_i):=(vu)b^\lambda_i,\qquad\bigl(\omega_\lambda(u)b^\lambda_j\bigr)\cdot a=\sum_{\substack{j'\in  \mathscr{T}(\lambda)\\u'\in \mathscr{B}_\lambda}}r^{j',u'}_{j,u}(a^*)\omega_\lambda(u')b^\lambda_{j'},
	$$ for $a\in A,u,v\in \mathscr{B}_\lambda,i\in  \mathscr{T}(\lambda)$.
	\end{defn}
By \eqref{eq2.GC2}, \eqref{eq1.GC2} and Lemma \ref{multiply}, the $(\mathcal{A},{\mathcal{B}_\lambda})$-bimodule structure on $M_\lambda$ and the $({\mathcal{B}_\lambda},\mathcal{A})$-bimodule structure on $N_\lambda$ are both well-defined. Moreover, we have an $(\mathcal{A},\mathcal{A})$-bimodule isomorphism $$h_\lambda: M_\lambda\otimes_{{\mathcal{B}_\lambda}}N_\lambda\cong \mathcal{A}^{\unlhd\lambda}/\mathcal{A}^{\lhd\lambda}; a^\lambda_iu\otimes b^\lambda_j\mapsto c^\lambda_{i,u,j}+\mathcal{A}^{\lhd\lambda}.$$
%		
%		\begin{cor}\label{action cell}
%		For $\lambda\in \mathscr{P}$, we have an $(A,A)$-bimodule isomorphism $$h_\lambda: M_\lambda\otimes_{{\mathcal{B}_\lambda}}N_\lambda\cong A^{\unrhd\lambda}/A^{\rhd\lambda}; m^\lambda_iu\otimes n^\lambda_j\mapsto c^\lambda_{i,u,j}+A^{\rhd\lambda}.$$
%%			Then the $\mathcal{A}$-module structure of $M_\lambda$ is given as follows $$a \cdot\bigl(m^\lambda_iu\bigr)=\sum_{\substack{i'\in  \mathscr{T}(\lambda)\\u'\in \mathscr{B}_\lambda}}r^{i',u'}_{i,u}(a)m^\lambda_{i'}u',
%%			$$ for $i\in  \mathscr{T}(\lambda), u\in \mathscr{B}_\lambda, a\in A, v\in {\mathcal{B}_\lambda}$. Similarly,  the $\mathcal{A}$-bimodule structure of $N_\lambda$ is given as follows $$\bigl(\omega_\lambda(u)n^\lambda_j\bigr)\cdot a=\sum_{\substack{j'\in  \mathscr{T}(\lambda)\\u'\in \mathscr{B}_\lambda}}r^{j',u'}_{j,u}(a^*)\omega_\lambda(u')n^\lambda_{j'},
%%			$$ for $j\in  \mathscr{T}(\lambda), u\in \mathscr{B}_\lambda, a\in A, v\in {\mathcal{B}_\lambda}$.
%			\end{cor}
%			\begin{proof}
%				This follows from Lemma \ref{multiply}, (GC3).
%				\end{proof}
		\begin{cor}\label{dual 1}
				For $\lambda\in \mathscr{P}$, we have the $(\mathcal{A},{\mathcal{B}_\lambda})$-bimodule isomorphism $M_\lambda\cong \prescript{\omega_\lambda}{}{(N_\lambda)}^*$, where the left $\mathcal{A}$-module structure of $\prescript{\omega_\lambda}{}{(N_\lambda)}^*$ is induced by the anti-involution $*$ and the right ${\mathcal{B}_\lambda}$-module structure is induced by the anti-involution $\omega_\lambda$ .
			\end{cor}
			
			\begin{proof}
			This follows from Lemma \ref{multiply} and (GC3).
				\end{proof}
			 By Lemma \ref{multiply} and (GC3), we have $$c^\lambda_{i',u,j}c^\lambda_{i,v,j'}=c^\lambda_{i',w,j'} \pmod{\mathcal{A}^{\lhd \lambda}}
			$$ for $i,j,i',j'\in  \mathscr{T}(\lambda), u,u'\in {\mathcal{B}_\lambda}, w\in {\mathcal{B}_\lambda}, \lambda\in \mathscr{P}.$
			
			\begin{defn}\label{pairing}
			We define the ${\rm K}$-linear map $$ f^\lambda: N_\lambda\otimes_{\mathcal{A}} M_\lambda\rightarrow {\mathcal{B}_\lambda}; ub^\lambda_j\otimes a^\lambda_i v\mapsto  f^\lambda(ub^\lambda_j\otimes a^\lambda_i v)
				$$ such that $$c^\lambda_{i',u,j}c^\lambda_{i,v,j'}=c^\lambda_{i',f^\lambda(ub^\lambda_j\otimes a^\lambda_i v),j'}\pmod{\mathcal{A}^{\lhd \lambda}}
				$$ for $i,j,i',j'\in  \mathscr{T}(\lambda), u,u'\in\mathscr{B}_\lambda, \lambda\in \mathscr{P}.$
				\end{defn}
By (GCd), it's easy to see that the map $f^\lambda$ is even and of $\Z$-degree $0.$
				
			\begin{lem}\label{bilinear}
	We have $f^\lambda$ is a $({\mathcal{B}_\lambda},{\mathcal{B}_\lambda})$-bilinear homomorphism for $\lambda\in\mathscr{P}$.
			\end{lem}
			
			\begin{proof}
				 Let $i,j,i',j'\in  \mathscr{T}(\lambda), u,v\in {\mathcal{B}_\lambda},\lambda\in \mathscr{P}.$
			 By Corollary \ref{dual 1}, we deduce that $$h^\lambda(a^\lambda_{i'}\otimes ub^\lambda_j)c^\lambda_{i,v,j'}=h^\lambda(a^\lambda_{i'}\otimes ub^\lambda_jc^\lambda_{i,v,j'}).
				$$ By definition, $$c^\lambda_{i',u,j}c^\lambda_{i,v,j'}=c^\lambda_{i',f^\lambda(ub^\lambda_j\otimes a^\lambda_i v),j'}\pmod{\mathcal{A}^{\lhd \lambda}}.
				$$ It follows that
				 $$ub^\lambda_jc^\lambda_{i,v,j'}=f^\lambda(ub^\lambda_j\otimes a^\lambda_i v)b^\lambda_{j'}.
				$$
				This implies that $f^\lambda$  is left ${\mathcal{B}_\lambda}$-linear. Similarly, we can prove $f^\lambda$ is right ${\mathcal{B}_\lambda}$-linear.
				\end{proof}
				
			{\bf Recall that ${\mathcal{B}_\lambda}$ is semisimple for any $\lambda\in\mathscr{P}$.}
Moreover, for $\lambda\in\mathscr{P}$, we assume ${\mathcal{B}_\lambda}$ has $m_\lambda$ non-isomorphic simple modules and $$
			\{{\mathcal{B}_\lambda} e^\lambda_k \mid 1\leq k\leq m_\lambda\}
			$$ forms a complete set of non-isomorphic simple modules, where $e^\lambda_k$ are primitive idempotents of ${\mathcal{B}_\lambda}$.
			
\label{pag:cell module}
			\begin{defn}\label{cell}
			We define $$\Delta(\lambda,k):={\rm K}\text{-span}\{a^\lambda_i u \mid i\in  \mathscr{T}(\lambda),u\in {\mathcal{B}_\lambda} e^\lambda_k\}\subset M_\lambda,$$ and $$\Delta(k,\lambda):={\rm K}\text{-span}\{ub^\lambda_i  \mid i\in  \mathscr{T}(\lambda),u\in  \omega_\lambda(e^\lambda_k){\mathcal{B}_\lambda}\}\subset N_\lambda
			$$ for $\lambda\in \mathscr{P}$ and $1\leq k\leq m_\lambda$.
				\end{defn}

			\begin{lem}	\label{cell module}
				\begin{enumerate}
					\item Suppose  ${\mathcal{B}_\lambda} e^\lambda_k$ is of type  $\texttt{M}$, then $\Delta(\lambda,k)$ is a left $\mathcal{A}$-module.
					\item Suppose  ${\mathcal{B}_\lambda} e^\lambda_k$ is of type  $\texttt{Q}$, then $\Delta(\lambda,k)$ is a $(A,\mathcal{C}_1)$-bimodule. Moreover, it is free as $\mathcal{C}_1$-module with ${\rm rank}_{\mathcal{C}_1}{\Delta(\lambda,k)}
=(\sharp \mathscr{T}(\lambda)\cdot \dim_{\rm K} {\mathcal{B}_\lambda} e^\lambda_k)/2$.
					\end{enumerate}
				\end{lem}
				
				\begin{proof}
					The left $\mathcal{A}$-module structure in both cases is clear. We only need to explain the right $\mathcal{C}_1$-action in the second case. Since ${\mathcal{B}_\lambda} e^\lambda_k$ is of type  $\texttt{Q}$, then $\End_{{\mathcal{B}_\lambda}}({\mathcal{B}_\lambda} e^\lambda_k)\cong \mathcal{C}_1=\langle 1,C_1\rangle$, where $C_1$ is the odd involution. Then the action of $C_1$ on $\Delta(\lambda,k)$ is given as follows: $$
			(a^\lambda_i u )\cdot C_1=a^\lambda_i C_1(u),\qquad \forall i\in  \mathscr{T}(\lambda),\,u\in {\mathcal{B}_\lambda e^\lambda_k}.
					$$ Using (GC3) and the fact that $c$ is a left ${\mathcal{B}_\lambda}$-module isomorphism, it's easy to check that the action of $C_1$ commutes with action of $\mathcal{A}$ on $\Delta(\lambda,k).$
In fact, let $u\in\mathscr{B}_\lambda$ and $u''=C_1(e_k^\lambda),$ for $a\in\mathcal{A},$ we have
\begin{align*}
a\cdot (a^\lambda_i ue_k^\lambda\cdot C_1)
=a\cdot a^\lambda_i u u''
&=\sum_{\substack{i'\in  \mathscr{T}(\lambda)\\u'\in \mathscr{B}_\lambda}}r^{i',u'}_{i,u}(a)a^\lambda_{i'}u'u''\\
&=\sum_{\substack{i'\in  \mathscr{T}(\lambda)\\u'\in \mathscr{B}_\lambda}}r^{i',u'}_{i,u}(a)a^\lambda_{i'}C_1(u'e_k^\lambda)
=(a\cdot a^\lambda_i ue_k^\lambda)\cdot C_1.
\end{align*}

On the other hand, $C_1$ induces an involution $({\mathcal{B}_\lambda} e^\lambda_k)_{\bar{0}}\rightarrow ({\mathcal{B}_\lambda} e^\lambda_k)_{\bar{1}}$.
%hence $({\mathcal{B}_\lambda} e^\lambda_k)_{\bar{0}}$ is the basis of  ${\mathcal{B}_\lambda} e^\lambda_k$ as $\mathcal{C}_1$-module.
This completes the proof of Lemma.
 					\end{proof}
 					It's an easy exercise to check that $M_\lambda$ (resp. $N_\lambda$) can be decomposed as direct sum of copies of $\Delta(\lambda,k)$ (resp. $\Delta(k,\lambda)$) for $1\leq k\leq m_\lambda$, and \begin{align}\label{dual 2}\Delta(\lambda,k)\cong \Delta(k,\lambda)^*\end{align} as left $\mathcal{A}$-modules by Corollary \ref{dual 1}.

               \begin{defn}\label{radical}
               For $\lambda\in \mathscr{P}$ and $1\leq k\leq m_\lambda$, let
$$\rad\Delta(\lambda,k):=\{x\in\Delta(\lambda,k) \mid f^\lambda(u\otimes x)=0,\forall u\in N_\lambda\}.$$
                \end{defn}
				\begin{lem}\label{jacobson radical}
			    Let $\lambda\in \mathscr{P}$ and $1\leq k\leq m_\lambda$.
%$$\rad\Delta(\lambda,k):=\{x\in\Delta(\lambda,k) \mid f^\lambda(u\otimes x)=0,\forall u\in N_\lambda\}
%				$$

(1) If ${\mathcal{B}_\lambda} e^\lambda_k$ is of type $\texttt{M}$ as ${\mathcal{B}_\lambda}$-module, then $\rad\Delta(\lambda,k)$ is a $\Z\times\Z_2$-graded $\mathcal{A}$-submodule of $\Delta(\lambda,k).$ Further, if $\Delta(\lambda,k)\neq \rad\Delta(\lambda,k),$ then
 $\rad\Delta(\lambda,k)$ is the unique maximal $\Z$-graded $\mathcal{A}$-submodule of $\Delta(\lambda,k).$

(2) If ${\mathcal{B}_\lambda} e^\lambda_k$ is of type $\texttt{Q}$ as ${\mathcal{B}_\lambda}$-module, then $\rad\Delta(\lambda,k)$
is a $\Z\times\Z_2$-graded $(\mathcal{A},\mathcal{C}_1)$-subbimodule of $\Delta(\lambda,k).$
Further, if $\Delta(\lambda,k)\neq \rad\Delta(\lambda,k),$ then $\rad\Delta(\lambda,k)$ is the unique maximal $\Z\times\Z_2$-graded $\mathcal{A}$-submodule of $\Delta(\lambda,k).$
%Moreover, $\rad\Delta(\lambda,k)$ is the (usual) Jacobson radical of $\Delta(\lambda,k)$ if $\Delta(\lambda,k)\neq \rad\Delta(\lambda,k)$.
					\end{lem}
					
				\begin{proof}
					It is clear that $\rad\Delta(\lambda,k)$ is a submodule of $\Delta(\lambda,k)$ in both cases. It is graded since $f^\lambda$ is even and is of $\Z$-degree $0$. Next we just prove (2). Since $f^\lambda$ is $({\mathcal{B}_\lambda},{\mathcal{B}_\lambda})$-bilinear, the right $\mathcal{C}_1$-module structure on $\Delta(\lambda, k)$ naturally induces one on $\rad \Delta(\lambda, k).$
Now we assume $\Delta(\lambda,k)\neq \rad\Delta(\lambda,k).$ We claim that
%if ${\mathcal{B}_\lambda} e^\lambda_k$ is of type $\texttt{M}$ (resp. type $\texttt{Q}$), then, for any $\Z$-homogeneous (resp. $\Z \times \Z_2$-homogeneous)
if the $\Z \times \Z_2$-homogeneous element $x\in \Delta(\lambda,k)\setminus \rad\Delta(\lambda,k),$ then $x$ generates $\Delta(\lambda,k)$. Actually, Lemma \ref{bilinear} implies that $\{f^\lambda(u\otimes x) \mid u\in N_\lambda\}$ is a left ${\mathcal{B}_\lambda}$-submodule of ${\mathcal{B}_\lambda} e^\lambda_k$. Hence this is equal to ${\mathcal{B}_\lambda} e^\lambda_k$. Fix $u\in {\mathcal{B}_\lambda} e^\lambda_k$, suppose $\sum_{j\in  \mathscr{T}(\lambda),u_j\in {\mathcal{B}_\lambda}}f^\lambda(u_jb^\lambda_j\otimes x)=u$ and $x=\sum_{i\in  \mathscr{T}(\lambda),v_i\in {\mathcal{B}_\lambda}}a^\lambda_i v_i$. Then for any $i',j'\in  \mathscr{T}(\lambda)$, we have $$
					\sum_{\substack{j\in  \mathscr{T}(\lambda),u_j\in {\mathcal{B}_\lambda}\\i\in  \mathscr{T}(\lambda),v_i\in {\mathcal{B}_\lambda}}}c^\lambda_{i',u_j,j}c^\lambda_{i,v_i,j'}
=\sum_{\substack{j\in  \mathscr{T}(\lambda),u_j\in {\mathcal{B}_\lambda}\\i\in  \mathscr{T}(\lambda),v_i\in {\mathcal{B}_\lambda}}}c^\lambda_{i',f^\lambda (u_jb^\lambda_j\otimes a^\lambda_i v_i),j'}
=c^\lambda_{i',u,j'} \pmod{\mathcal{A}^{\lhd\lambda}},
					$$ i.e.,
$$\sum_{\substack{j\in  \mathscr{T}(\lambda),u_j\in {\mathcal{B}_\lambda}}}c^\lambda_{i',u_j,j} x=a^\lambda_{i'}u$$
holds in $\Delta(\lambda,k).$
This proves that $x$ generates $\Delta(\lambda,k)$. Hence, $\rad\Delta(\lambda,k)$ is the unique maximal $\Z\times\Z_2$-graded submodule of $\Delta(\lambda,k).$
					\end{proof}

				\label{pag:simple module}	
				\begin{defn}
					Suppose that $\lambda\in\mathscr{P}$. Let $D(\lambda,k):=\Delta(\lambda,k)/\rad \Delta(\lambda,k)$ for $1\leq k\leq m_\lambda$.
					\end{defn}
					Let $\mathscr{P}_0:=\{(\lambda,k) \mid \lambda\in\mathscr{P},1\leq k\leq m_\lambda, D(\lambda,k)\neq 0\}.$
				
					Similarly, we can define $\rad\Delta(k,\lambda),D(k,\lambda)$ and define $\mathscr{P}'_0:=\{(\lambda,k) \mid \lambda\in\mathscr{P},1\leq k\leq m_\lambda, D(k,\lambda)\neq 0\}$. By \eqref{dual 2}, we deduce that  \begin{align}\label{dual 3}D(\lambda,k)\cong D(k,\lambda)^*
				\end{align}
				as left $\mathcal{A}$-modules. Hence $\mathscr{P}'_0=\mathscr{P}_0.$
										
					\begin{prop}\label{dual 4}
						Let $(\lambda,k)\in\mathscr{P}_0$ with $\omega_\lambda(e_j^\lambda)$ and $e_k^\lambda$ belongs to the same block of ${\mathcal{B}_\lambda}$, then we have $(\lambda,j)\in\mathscr{P}_0$ and $$D(\lambda,k)\cong \begin{cases}{\rm Hom}_{\rm K}(D(j,\lambda),\rm { K})\,(\text{as left $\mathcal{A}$-module}),&\qquad\text{ if  ${\mathcal{B}_\lambda} e^\lambda_k$ is of type  $\texttt{M}$,}\\
							{\rm Hom}_{\mathcal{C}_1}(D(j,\lambda),\mathcal{C}_1)\,(\text{as $(\mathcal{A},\mathcal{C}_1)$-bimodule}),&\qquad\text{ if  ${\mathcal{B}_\lambda} e^\lambda_k$ is of type  $\texttt{Q}$.}\\
						\end{cases}$$
					\end{prop}
					
					\begin{proof}
						By Lemma \ref{bilinear} and Definition \ref{cell module}, we have that $f^\lambda$ restricts to the ${\rm K}$-linear map $$\Delta(j,\lambda)\otimes_{\mathcal{A}}\Delta(\lambda,k)\rightarrow \omega_{\lambda}(e_j^\lambda){\mathcal{B}_\lambda} e_k^\lambda\cong \begin{cases}{\rm K},&\qquad\text{ if  ${\mathcal{B}_\lambda} e^\lambda_k$ is of type  $\texttt{M}$,}\\
							\mathcal{C}_1,&\qquad\text{ if  ${\mathcal{B}_\lambda} e^\lambda_k$ is of type  $\texttt{Q}$.}\\
							\end{cases}$$By Lemma \ref{jacobson radical}, this gives rise to a non-degenerate pairing  $$D(j,\lambda)\otimes_{\mathcal{A}} D(\lambda,k)\rightarrow \omega_{\lambda}(e_j^\lambda){\mathcal{B}_\lambda} e_k^\lambda\cong \begin{cases}{\rm K},&\qquad\text{ if  ${\mathcal{B}_\lambda} e^\lambda_k$ is of type  $\texttt{M}$,}\\
							\mathcal{C}_1,&\qquad\text{ if  ${\mathcal{B}_\lambda} e^\lambda_k$ is of type  $\texttt{Q}$.}\\
							\end{cases}$$This, combining with Lemma \ref{cell module}, implies $$D(\lambda,k)\cong \begin{cases}{\rm Hom}_{\rm K}(D(j,\lambda),{\rm K}),&\qquad\text{ if  ${\mathcal{B}_\lambda} e^\lambda_k$ is of type  $\texttt{M}$,}\\
								{\rm Hom}_{\mathcal{C}_1}(D(j,\lambda),\mathcal{C}_1)&\qquad\text{ if  ${\mathcal{B}_\lambda} e^\lambda_k$ is of type  $\texttt{Q}$,}\\
							\end{cases}$$ and $D(j,\lambda)\neq 0$. Now $(\lambda,j)\in\mathscr{P}_0$ follows from \eqref{dual 3}.
					\end{proof}

					\begin{cor}
						Let $(\lambda,k)\in\mathscr{P}_0$ with $\omega_\lambda(e_j^\lambda)$ and $e_k^\lambda$ belongs to the same block of ${\mathcal{B}_\lambda}$, then we have $$D(j,\lambda)\cong \begin{cases}{\rm Hom}_{\rm K}(D(\lambda, k),{\rm K})^*\,\text{ (as right $\mathcal{A}$-module)},&\qquad\text{ if  ${\mathcal{B}_\lambda} e^\lambda_k$ is of type  $\texttt{M}$,}\\
							{\rm Hom}_{\mathcal{C}_1}(D(\lambda, k),\mathcal{C}_1)^*\,\text{ (as $(\mathcal{C}_1,A)$-bimodule)},&\qquad\text{ if  ${\mathcal{B}_\lambda} e^\lambda_k$ is of type  $\texttt{Q}$.}\\
						\end{cases}$$   In particular, if ${\mathcal{B}_\lambda}$ is simple, we always have the isomorphism.
					\end{cor}				
					
					\begin{proof}
						This follows from \eqref{dual 3} and Proposition \ref{dual 4}.
					\end{proof}

			\begin{thm}\label{simple mdoule}Suppose that ${\rm K}$ is a field,  $\mathcal{A}$ is a generalized graded cellular superalgebra over ${\rm K}$ with generalized graded super cell datum $(\mathscr{P},\mathscr{T},\mathscr{B},\mathscr{C},\deg,{\rm p})$ and ${\mathcal{B}_\lambda}$ is semisimple for any $\lambda\in\mathscr{P}$.
				\begin{enumerate}
					\item[(a)] If $(\lambda,k)\in\mathscr{P}_0$ and ${\mathcal{B}_\lambda}$ is split, then $D(\lambda,k)$ is an absolutely irreducible graded $\mathcal{A}$-module.
					\item[(b)]  If $(\lambda,k)\in\mathscr{P}_0$, then the simple $\mathcal{A}$-module $D(\lambda,k)$ has the same type with the simple $\mathcal{B}_\lambda$-module ${\mathcal{B}_\lambda} e^\lambda_k$.
					\item[(c)] $\{D(\lambda,k) \mid (\lambda,k)\in\mathscr{P}_0\}$ forms a complete set of pairwise non-isomorphic simple graded $\mathcal{A}$-modules.
					\end{enumerate}
				\end{thm}		
					
					\begin{proof}
				(a) Let ${\rm K}\subset {\rm K}'$ be a field extension. Since ${\mathcal{B}_\lambda}$ is split semisimple,  $$
				\{{\rm K}'\otimes_{\rm K} {\mathcal{B}_\lambda} e^\lambda_k \mid 1\leq k\leq m_\lambda\}
				$$ still forms a complete set of non-isomorphic simple ${\rm K}'\otimes_{\rm K} {\mathcal{B}_\lambda}$-modules after field extension. By the definition of $\rad$, it is easy to see that ${\rm K}'\otimes_{\rm K} \rad \Delta(\lambda,k)=\rad \left({\rm K}'\otimes_{\rm K} \Delta(\lambda,k)\right).$ Hence, if $D(\lambda,k)\neq 0$, then ${\rm K}'\otimes_{\rm K} D(\lambda,k)\neq 0$ and is still irreducible by Lemma \ref{jacobson radical} over ${\rm K}'$. This shows that $D(\lambda,k)$ is an absolutely irreducible graded $\mathcal{A}$-module.
				
			(b) Suppose ${\mathcal{B}_\lambda} e^\lambda_k$ is a simple module of type $\texttt{M}$, then by Lemma \ref{jacobson radical}, $D(\lambda,k)$ remains irreducible after forgetting super structure, hence it is still of type $\texttt{M}$. If ${\mathcal{B}_\lambda} e^\lambda_k$ is a simple module of type $\texttt{Q}$, then after forgetting super structure, we have $$D(\lambda,k)\cong D(\lambda,k)\cdot {\frac{1+C_1}{2}}\oplus D(\lambda,k)\cdot {\frac{1-C_1}{2}}$$  as left $\mathcal{A}$-module.  Suppose $D(\lambda,k)\cdot {\frac{1+C_1}{2}}=0$, then $D(\lambda,k)= D(\lambda,k)\cdot {\frac{1-C_1}{2}}$, i.e. for any homogeneous element $x\in D(\lambda,k)$, we have $x\cdot \frac{1-C_1}{2}=x$, comparing parity, we have $x=x/2$, hence $x=0$. This implies $D(\lambda,k)=0$, which is a contradiction. Hence, $D(\lambda,k)\cdot {\frac{1+C_1}{2}}\neq0$. Similarly, we can prove $D(\lambda,k)\cdot {\frac{1-C_1}{2}}=0$. This implies that $D(\lambda,k)$ is of type $\texttt{Q}$.
				
				(c) This can be proved as in \cite{GL} and \cite{HM1}. However, we can directly use \cite[Theorem 4.7]{Mo} to obtain (c).
				\end{proof}

			\subsection{Decomposition matrix}
			{\bf  In this subsection, ${\rm K}$ is a field,  $\mathcal{A}$ is a generalized graded cellular superalgebra over ${\rm K}$ with generalized graded super cell datum $(\mathscr{P},\mathscr{T},\mathscr{B},\mathscr{C},\deg,{\rm p})$ and ${\mathcal{B}_\lambda}$ is split semisimple for each $\lambda\in\mathscr{P}$. }
			
			 If $M$ is a graded $\mathcal{A}$-module and $D$ is a
			graded simple module, for $(l,a)\in \mathbb{Z}\times \mathbb{Z}_2,$ let $[M:\Pi^a D\langle l \rangle]$ be the multiplicity of the simple module $\Pi^a D\langle l \rangle$ as a composition factor of $M$.
			We set $\mathscr{P}_1:=\{(\lambda,k) \mid \lambda\in\mathscr{P},1\leq k\leq m_\lambda\}.$

             Let $x,t$ be two indeterminates over $\Z$. Consider the quotient ring $\Z[x]/\<x^4-1\>$. We define $$
            \pi:=x^2+\<x^4-1\>,\quad\, \sqrt{\pi}:=x+\<x^4-1\> .
            $$
            Then $\Z[x]/\<x^4-1\>=\Z[\sqrt{\pi}]$. For any ring $R$, we set $R^\pi:=R\otimes_{\Z}\Z[\pi]$.
            \label{pag:graded decomposition matrix}
			\begin{defn}
			The graded decomposition matrix of $\mathcal{A}$ is the matrix $\mathbf{D}_\mathcal{A}(t,\pi)=(d_{(\lambda,k_1),(\nu,k_2)}(t,\pi))$, where $$d_{(\lambda,k_1),(\nu,k_2)}(t,\pi):=\sum_{l\in\Z, a\in\Z_2}
[\Delta(\lambda,k_1): \Pi^a D(\nu,k_2)\langle l \rangle]t^l \pi^a\in \Z[t^\pm]^\pi
				$$ for $(\lambda,k_1)\in \mathscr{P}_1$ and $(\nu,k_2)\in\mathscr{P}_0$.
				\end{defn}
				Using the same proof as in \cite{GL}, or alternatively, applying \cite[Lemma 4.8]{Mo}, we have the following.
				\begin{lem}
					Suppose $(\lambda,k_1)\in \mathscr{P}_1$ and $(\nu,k_2)\in\mathscr{P}_0$. Then \begin{enumerate}
						\item[(a)]  $d_{(\lambda,k_1),(\nu,k_2)}(t,\pi)\neq0$ only if $\lambda\unlhd\nu$;
						\item[(b)]  if $\lambda=\nu$, then $d_{(\lambda,k_1),(\nu,k_2)}(t,\pi)=\delta_{k_1,k_2}$.
						\end{enumerate}
					\end{lem}
		Next we study the projective $\mathcal{A}$-modules with the aim of describing the composition factors of these modules using the decomposition matrix.
		An $\mathcal{A}$-module $M$ has a cell module filtration if there exists a filtration			$$
		0=M_0\subset M_1\subset M_2\subset \cdots\subset M_l=M
		$$such that each $M_i$ is a submodule of $M$ and if $1\leq i\leq l$ then $M_i/M_{i-1}\cong \Delta(\lambda,j)$ for some $(\lambda,j)\in \mathscr{P}_1$.
	
	\begin{prop}	Let $P$ be a projective $\mathcal{A}$-module. Then $P$ has a cell module filtration.
		\end{prop}
		
		\begin{proof}
			Following the same proof as in \cite{GL} or \cite{HM1}, $\mathcal{A}$ has a cell module filtration with each subquotient isomorphic to $\mathcal{A}^{\unrhd\lambda}/\mathcal{A}^{\rhd\lambda}$ for $\lambda\in\mathscr{P}$. Note that for any idempotent $e\in \mathcal{A}$, the ${\mathcal{B}_\lambda}$-module $N_\lambda e$ is semisimple and hence the left $\mathcal{A}$-module $$
			\mathcal{A}^{\unrhd\lambda}/\mathcal{A}^{\rhd\lambda}\otimes_{\mathcal{A}} \mathcal{A}e\cong M_\lambda\otimes_{{\mathcal{B}_\lambda}} N_\lambda e
			$$ is isomorphic to some direct sum of some $\Delta(\lambda,i)$ for $1\leq i\leq m_\lambda$. This completes the proof of the Proposition.
			\end{proof}

\label{pag:Cartan matrix}
			\begin{defn}
				The Cartan matrix of $\mathcal{A}$ is the matrix $\mathbf{C}_\mathcal{A}(t,\pi)=\left(c_{(\lambda,k_1),(\nu,k_2)}(t,\pi)\right)$, where $$c_{(\lambda,k_1),(\nu,k_2)}(t,\pi):=\sum_{l\in\Z, a\in\Z_2} [P(\lambda,k_1):\Pi^a D(\nu,k_2)\langle l \rangle]t^l \pi^a
				$$ for $(\lambda,k_1),(\nu,k_2)\in\mathscr{P}_0$.
			\end{defn}
			
			\begin{thm}(Brauer-Humphreys reciprocity)
				Suppose ${\rm K}$ is a field,  $\mathcal{A}$ is a generalized graded cellular superalgebra over ${\rm K}$ with generalized graded super cell datum $(\mathscr{P},\mathscr{T},\mathscr{B},\mathscr{C},\deg,{\rm p})$, ${\mathcal{B}_\lambda}$ is split semisimple for $\lambda\in\mathscr{P}$ with $\omega_\lambda(e_k^\lambda)$ and $e_k^\lambda$ belong to the same block of ${\mathcal{B}_\lambda}$, for $k=1,\cdots,m_\lambda$.  Then $\mathbf{C}_{\mathcal{A}}(t,\pi)=\mathbf{D}_{\mathcal{A}}(t,\pi)^{\rm tr}\mathbf{D}_{\mathcal{A}}(t,\pi)$.
			\end{thm}
			\begin{proof}
				This can be proved as in \cite{GL} or \cite{HM1}, Alternatively, one can apply \cite[Theorem 4.15]{Mo} and \eqref{dual 2}, \eqref{dual 3} to obtain the result directly.
			\end{proof}

\section{Quiver Hecke superalgebra and Quiver Hecke-Clifford superalgebra}\label{Quiver Hecke superalgebra and Quiver Hecke-Clifford superalgebra}
In this section, we shall recall the definition of quiver Hecke superalgebras and quiver Hecke-Clifford superalgebras. They are two remarkable classes of $\Z\times \Z_2$-graded algebras, which were first introduced by Kang, Kashiwara and Tsuchioka in \cite{KKT}.
\subsection{Cartan superdatum}
Let $I$ be an index set. An integral matrix $(a_{ij})_{i,j\in I}$ is called a Cartan matrix if it satisfies: i) $a_{ii}=2$, ii) $a_{ij}\leq 0$ for $i\neq j$, iii) $a_{ij}=0$ if and only if $a_{ji}=0$. We say ${\rm{A}}$ is symmetrizable if there is a diagonal matrix ${\rm{D}}={\rm{diag}}(\rd_i\in\Z_{>0}|i\in I)$ such that ${\rm{DA}}$ is symmetric.
\label{pag:Cartan superdatum}
Let $\bigl({\rm{A}}=(a_{ij})_{i,j\in I},P,\Pi,\Pi^\vee\bigr)$ be a Cartan superdatum in the sense of \cite[\S4.1]{KKO2}. That means, \begin{enumerate}
	\item[CS1)] ${\rm{A}}$ is a symmetrizable Cartan matrix;
	\item[CS2)] $P$ is a free abelian group, which is called the weight lattice;
	\item[CS3)] $\Pi=\{\nu_i\in P|i\in I\}$, called the set of simple roots, is $\Z$-linearly independent;
	\item[CS4)] $\Pi^\vee=\{h_i\in P|i\in I\}\subset P^\vee=\Hom_\Z(P,\Z)$, called the set of simple coroots, satisfies that $\<h_i,\nu_j\>=a_{ij}$ for all $i,j\in I$;
	\item[CS5)] there is a decomposition $I=I_{\rm{even}}\sqcup I_{\rm{odd}}$ such that \begin{equation}\label{evenodd}
		a_{ij}\in 2\Z, \quad \text{for all $i\in I_{\rm{odd}}$ and $j\in I$.}
	\end{equation}
\end{enumerate}

The diagonal matrix ${\rm{D}}$ gives rise to a symmetric bilinear form $(-|-)$ on $P$ which satisfies: $$
(\nu_i|\lambda)=\rd_i\<h_i,\lambda\>\quad \text{for all $\lambda\in P$.}
$$
In particular, we have $(\nu_i|\nu_j)=\rd_i a_{ij}$ and hence $\rd_i=(\nu_i|\nu_i)/2$ for each $i\in I$.

We define the root lattice $Q:=\oplus_{i\in I}\Z\nu_i$
and the positive root lattice $Q^+:=\oplus_{i\in I}\Z_{\geq 0}\nu_i$.
For any $\nu=\sum_{i\in I}k_i\nu_i\in Q^+$, we define ${\rm ht}(\nu):=\sum_{i\in I}k_i$.
For any $n\in\Z_{\geq 0}$, we define $Q_n^+:=\{\nu\in Q^+\mid {\rm ht}(\nu)=n\}$.
\label{pag:dominant integral weights}
Let $P^+:=\{\Lambda\in P|\text{$\<h_i,\Lambda\>\in\Z_{\geq 0}$ for all $i\in I$}\}$.
Any element $\Lambda\in P^+$ is called a dominant integral weight.
Let $\Lambda_i$, $i\in I$ be the fundamental dominant integral weights, which satisfy $\<h_i,\Lambda_j\>=\delta_{i,j}$, $i,j\in I$.
Then $P^+=\oplus_{i\in I}\Z_{\geq 0}\Lambda_i$.

\label{pag:parity function on I}
For a Cartan superdatum $({\rm{A}},P,\Pi,\Pi^\vee)$,
we define the parity function ${\rm p}: I\rightarrow\{\overline{0},\overline{1}\}$ by
\begin{equation}\label{pi1}
{\rm p}(i):=\begin{cases} \overline{1}, &\text{if $i\in I_{\rm{odd}}$,}\\
\overline{0}, &\text{if $i\in I_{\rm{even}}$.}
\end{cases}
\end{equation}

\subsection{Quiver Hecke superalgebra}
Recall that ${\rm R}$ is an integral domain of characteristic different from $2$. Let $\bigl({\rm{A}},P,\Pi,\Pi^\vee\bigr)$ be a Cartan superdatum.
%Let ${\rm x}_1,\cdots,{\rm x}_n$ be $n$ indeterminates over ${\rm R}$.
%For any $n\geq 2$ and $\nu\in I^n$, set
%$$
%\mathcal{P}_\nu:={\rm R}\<{\rm x}_1,\cdots,{\rm x}_n\>/\<{\rm x}_a{\rm x}_b-(-1)^{|\nu_a||\nu_b|}{\rm x}_b{\rm x}_a|1\leq a<b\leq n\> .
%$$

For $i,i'\in I$, we consider the ${\rm R}$-algebra ${\rm Pol}_{i,i'}:={\rm R}\langle u,v \rangle / \langle uv-(-1)^{{\rm p}(i)\cdot {\rm p}(i')}vu \rangle$, and choose an element $Q_{i,i'}(u,v)\in {\rm Pol}_{i,i'}$ of the form
\label{pag:Q-polys}
\begin{align}\label{Q-polys}
Q_{i,i'}(u,v)=\sum_{r,s\geq 0}t_{i,i';(r,s)}u^r v^s,
\end{align}
where the coefficient satisfies that
\begin{align}
 &t_{i,i';(r,s)}\neq 0 \text{ only if } -2(\nu_i|\nu_{i'})-r(\nu_i|\nu_i)-s(\nu_{i'}|\nu_{i'})=0;\\
\label{condition1}&t_{i,i';(r,s)}=t_{i',i;(s,r)},\quad t_{i,i';(-a_{i,i'},0)}\in {\rm R}^\times;\\
\label{condition2} &t_{i,i';(r,s)}=0\text{ if either $i=i'$ or $i\in I_{\rm{odd}}$ and $r$ is odd.}
\end{align}

\label{pag:quiver Hecke superalgebras}
\begin{defn}\cite[Definition 3.1]{KKT}
Let $({\rm{A}},P,\Pi,\Pi^\vee)$ be a Cartan superdatum,
%$\nu\in Q_n^+$
$\{Q_{i,j}|i,j\in I\}$ be chosen as above, and $n\in \N.$
The quiver Hecke superalgebras $R_n$ is the superalgebra over ${\rm R},$ which is defined by the generators $$
e({\bf i})\, ({\bf i}\in I^n), x_k\, (1\leq k\leq n),\, \tau_a (1\leq a\leq n-1),
$$
the parity $$
{\rm p}(e({\bf i}))=\bar{0},\quad {\rm p}(x_ke({\bf i}))={\rm p}(\nu_k),\quad {\rm p}(\tau_ae({\bf i}))={\rm p}(\nu_a) \cdot{\rm p}(\nu_{a+1}),
$$
and the following relations:
\begin{align}
& e({\bf j})e({\bf i})=\delta_{{\bf j},{\bf i}}e({\bf i}),\,\,\text{for ${\bf j},{\bf i}\in I^n$}, \,\,\sum_{{\bf i}\in I^n}e({\bf i})=1, \nonumber\\
& x_px_qe({\bf i})=(-1)^{{\rm p}({\bf i}_{p}) {\rm p}({\bf i}_{q})}x_qx_pe({\bf i}),\,\,\text{if $p\neq q$,}\nonumber\\
& x_pe({\bf i})=e({\bf i})x_p,\,\,\,\tau_ae({\bf i})=e(s_a{\bf i})\tau_a,\,\,\text{where $s_a=(a,a+1)$,}\nonumber\\
& \tau_ax_pe({\bf i})=(-1)^{{\rm p}({\bf i}_{p}) {\rm p}({\bf i}_{a}){\rm p}({\bf i}_{a+1})}x_p\tau_ae({\bf i}),\,\,\text{if $p\neq a,a+1$,}\nonumber\\
& \bigl(\tau_ax_{a+1}-(-1)^{{\rm p}({\bf i}_{a}){\rm p}({\bf i}_{a+1})}x_a\tau_a\bigr)e({\bf i})\nonumber\\
&\qquad =\bigl(x_{a+1}\tau_a-(-1)^{{\rm p}({\bf i}_{a}){\rm p}({\bf i}_{a+1})}\tau_ax_a\bigr)=\delta_{{\bf i}_a,{\bf i}_{a+1}}e({\bf i}),\nonumber\\
&\tau_a^2e({\bf i})=Q_{{\bf i}_a,{\bf i}_{a+1}}(x_a,x_{a+1})e({\bf i}),\nonumber\\
&\tau_a\tau_be({\bf i})=(-1)^{{\rm p}({\bf i}_{a}){\rm p}({\bf i}_{a+1}){\rm p}({\bf i}_{b}) {\rm p}({\bf i}_{b+1})}\tau_b\tau_ae({\bf i}),\,\,\text{if $|a-b|>1$},\nonumber\\
& (\tau_{a+1}\tau_a\tau_{a+1}-\tau_a\tau_{a+1}\tau_a)e({\bf i})\nonumber\\
&=\begin{cases}
\frac{Q_{{\bf i}_a,{\bf i}_{a+1}}(x_{a+2},x_{a+1})-Q_{{\bf i}_a,{\bf i}_{a+1}}(x_{a},x_{a+1})}{x_{a+2}-x_a}e({\bf i}), &\text{if ${\bf i}_a={\bf i}_{a+2}\in I_{\rm{even}}$;}\\
(-1)^{{\rm p}({\bf i}_{a+1})}(x_{a+2}-x_a)\frac{Q_{{\bf i}_a,{\bf i}_{a+1}}(x_{a+2},x_{a+1})-Q_{{\bf i}_a,{\bf i}_{a+1}}(x_{a},x_{a+1})}{x_{a+2}^2-x_a^2}e({\bf i}), &\text{if ${\bf i}_a={\bf i}_{a+2}\in I_{\rm{odd}}$;}\nonumber\\
0, &\text{otherwise.}
\end{cases}
\end{align}
\end{defn}

$R_n$ is $\Z$-graded by setting
$$
\deg(e({\bf i}))=0,\quad \deg(x_ke({\bf i}))=(\nu_{{\bf i}_k}|\nu_{{\bf i}_k}),\quad \deg(\tau_ae({\bf i}))=-(\nu_{{\bf i}_a}|\nu_{{\bf i}_{a+1}}).
$$

\begin{prop}\cite[Corollary 3.15]{KKT}
%Let $\nu\in Q^+_n$.
For each $w\in\Sym_n$,
we fix a reduced expression $w=s_{i_1}\cdots s_{i_l}$, and define $\tau_w:=\tau_{i_1}\cdots\tau_{i_l}$,
then the set of elements
$$\{x^a \tau_w e({\bf i})\mid a\in (\Z_{\geq 0})^n,\,w\in \Sym_n,\,{\bf i}\in I^n\}$$
forms a basis of the free ${\rm R}$-module $R_n,$ where
$x^a=x_1^{a_1}\cdots x_n^{a_n}$ for
$a=(a_1,\ldots,a_n)\in (\Z_{\geq 0})^n.$
\end{prop}

If $\Lambda\in P^+,\,i\in I$ and $u$ is an indeterminate over $\Z$, then we define
$$
a_i^\Lambda(x_1)=x_1^{\<h_i,\Lambda\>},\quad
a^\Lambda(x_1):=\sum_{{\bf i}\in I^n}x_1^{\<h_{{\bf i}_1},\Lambda\>}e({\bf i})\in R_n.
$$
\label{pag:cyclotomic quiver Hecke superalgebra}
\begin{defn}\cite[Section 3.7]{KKT}
Let $\Lambda\in P^+$.
The cyclotomic quiver Hecke superalgebra $R^\Lambda_n$ is defined to be the quotient algebra:
$$R^\Lambda_n:=R_n/\<a^\Lambda(x_1)\>.$$
\end{defn}

$R^\Lambda_n$ inherits $\Z\times\Z_2$-grading from $R^\Lambda_n.$
That says, $R^\Lambda_n$ is a $\Z$-graded superalgebra too.
By some abuse of notations, we shall use the same symbols to denote the generators of both $R_n$ and $R^\Lambda_n.$
%For any $n\in\N,$ we denote
%$$R_n :=\bigoplus_{\nu\in Q_n^+} R_\nu, \qquad R^\Lambda_n :=\bigoplus_{\nu\in Q_n^+} R^\Lambda_\nu.$$
For any $\nu\in Q_n^+$, we define
$$ I^\nu:=\Bigl\{{\bf i}=({\bf i}_1,\cdots,{\bf i}_n)\in I^n\Bigm|\sum_{s=1}^{n}\nu_{{\bf i}_s}=\nu\Bigr\}. $$
Let $e_\nu :=\sum_{{\bf i}\in I^{\nu}}e({\bf i})$ be the certain central idempotent, then we define \label{pag:blocks}
$$R_\nu:=e_\nu R_n, \qquad R^\Lambda_\nu:=e_\nu R^\Lambda_n.$$

\subsection{Quiver Hecke-Clifford superalgebra}
Let $\bigl({\rm{A}}=(a_{ij})_{i,j\in I},P,\Pi,\Pi^\vee\bigr)$ be a Cartan superdatum.
Then we can define the quiver Hecke-Clifford
${\rm R}$-superalgebra $RC_n=RC_n(I).$ Let $[n]:=\{1,2,\ldots,n\}$. \label{pag:[n]}

\label{pag:J}
Let the set $J:=(I_{\rm odd}\times\{0\}) \sqcup (I_{\rm even} \times\{\pm \}).$ There is an involution $c: J\to J$ which fixes $I_{\rm odd}\times\{0\}$ and sends
$({\bf i}, \pm)$ to $({\bf i}, \mp)$ for each ${\bf i}\in I_{\rm even}$.
We also denote by $J^c:=I_{\rm odd}\times\{0\}$ the set of fixed points $\{j\in J \mid c(j)=j \}$ and
%let $I$ denote the set of equivalence classes under the equivalence
%relation given by $i\sim_c j\Leftrightarrow i=j$ or $i=c(j)$.
$\pr$ the canonical projection $J\to I$. The symmetric
group $\Sym_n$ acts on $J^n$ in a natural way.
For $p\in[n],$ we define
$c_p: J^n\to J^n$ by
$$c_{p}{\bf i}=(c^{\delta_{p\ell}}{\bf i}_{\ell})_{1\leq \ell\leq n}
\quad\text{for ${\bf i}=({\bf i}_1,\ldots,{\bf i}_n)\in J^n$.}$$

Recall that for each $i,i'\in I$, we have chosen an element $Q_{i,i'}(u,v)\in {\rm Pol}_{i,i'}$ of the form
\begin{align*}
	Q_{i,i'}(u,v)=\sum_{r,s\geq 0}t_{i,i';(r,s)}u^r v^s.
\end{align*}
\label{pag:widetildeQ polys}
Following \cite[Remark 3.14]{KKT}, we define $\widetilde{Q}=(\widetilde{Q}_{j,j'}(u,v))_{j,j'\in J}\subseteq {\rm R}[u,v]$ be the family of polynomials via the following way: for any $(i,\varepsilon),(i',\varepsilon') \in J,$ where $i,i'\in I,$ $\varepsilon,\varepsilon'\in\{0,\pm\},$ we set
\begin{align}\label{extend Q-polys}
\widetilde{Q}_{(i,\varepsilon),(i',\varepsilon')}(u,v):=\sum_{r,s\geq 0}(-1)^{{\rm p}(i)\cdot\frac{r}{2}+{\rm p}(i')\cdot\frac{s}{2}}t_{i,i';(r,s)}\left((-1)^{\varepsilon}u\right)^r \left((-1)^{\varepsilon'}v\right)^s .
\end{align}
%\begin{align*}
%&Q_{j,i}(u,v)=Q_{i,j}(v,u),\\
%&Q_{ci,j}(-u,v)=Q_{i,cj}(u,-v)=Q_{i,j}(u,v), \\
%&Q_{i,j}(u,v)=0, \text{ if } \pr(i)=\pr(j)\in I.
%\end{align*}
It follows from \eqref{condition1} and \eqref{condition2} that when the coefficient $t_{i,i';(r,s)}\neq 0$, the power exponent ${\rm p}(i)\cdot\frac{r}{2}+{\rm p}(i')\cdot\frac{s}{2}$ makes sense. Note that $\widetilde{Q}_{j,j'}(u,v)=\widetilde{Q}_{j,j'}(-u,v)$ for $j\in J^c, j'\in J.$
\begin{defn}\cite[Definition 3.5]{KKT} \label{pag:quiver Hecke-Clifford superalgebra}
Let $\bigl({\rm{A}}=(a_{ij})_{i,j\in I},P,\Pi,\Pi^\vee\bigr)$ be a Cartan superdatum,
$\widetilde{Q}=(\widetilde{Q}_{j,j'}(u,v))_{j,j'\in J}$ be chosen as above, and $n\in \N$.
The quiver Hecke-Clifford superalgebra
$RC_n=RC_n(I)$ is the ${\rm R}$-superalgebra generated by the even generators
$\{y_p\}_{1\le p\le n}$, $\{\sigma_a\}_{1\le a<n}$,
$\{e({\bf i})\}_{{\bf i}\in J^n}$ and the odd generators $\{c_p\}_{1\leq
p\leq n}$ with the following defining relations: for ${\bf i},{\bf j}\in
J^n$, $1\leq p,\,q\leq n$, $1\leq a\leq n-1$, we have
\begin{enumerate}
\item $e({\bf i})e({\bf j})=\delta_{{\bf i} {\bf j}}e({\bf i})$, $1=\sum_{{\bf i}\in J^n}e({\bf i})$,
$y_pe({\bf i})=e({\bf i})y_p$,  $c_pe({\bf i})=e(c_p{\bf i})c_p$,
\item $y_py_q=y_qy_p$, $c_pc_q+c_qc_p=2\delta_{pq}$,
\item $c_py_q=(-1)^{\delta_{p,q}}y_qc_p$,
\item $\sigma_ae({\bf i}) = e(s_a{\bf i})\sigma_a, \sigma_a c_p = c_{s_a(p)}\sigma_a$,
\item $\sigma_a y_pe({\bf i}) = y_{p}\sigma_ae({\bf i})$ if $p\not=a,a+1$,
\item
%\begin{align*}
%(\sigma_ay_{a+1}-y_a\sigma_a)e({\bf i}) &=
%\begin{cases}
%e({\bf i}) & \text{if }{\bf i}_a={\bf i}_{a+1}\not\in J^c, \\
%-c_ac_{a+1}e({\bf i}) & \text{if }{\bf i}_a=c{\bf i}_{a+1}\not\in J^c, \\
%(1-c_ac_{a+1})e({\bf i}) & \text{if }{\bf i}_a={\bf i}_{a+1}\in J^c,\\
%0&\text{otherwise,}
%\end{cases}
%\end{align*}
%or equivalently,
$$\sigma_ay_{a+1}-y_a\sigma_a=
\sum_{{\bf i}_a={\bf i}_{a+1}}e({\bf i})-\sum_{{\bf i}_a=c{\bf i}_{a+1}}c_ac_{a+1}e({\bf i}),$$
\item
% \begin{align*}
%(y_{a+1}\sigma_a-\sigma_ay_a)e({\bf i}) &=
%\begin{cases}
%e({\bf i}) & \text{if }{\bf i}_a={\bf i}_{a+1}\not\in J^c, \\
%c_ac_{a+1}e({\bf i}) & \text{if }{\bf i}_a=c{\bf i}_{a+1}\not\in J^c, \\
%(1+c_ac_{a+1})e({\bf i}) & \text{if }{\bf i}_a={\bf i}_{a+1}\in J^c,\\
%0&\text{otherwise,}
%\end{cases}
%\end{align*}
%or equivalently,
$$y_{a+1}\sigma_a-\sigma_ay_a=
\sum_{{\bf i}_a={\bf i}_{a+1}}e({\bf i})+\sum_{{\bf i}_a=c{\bf i}_{a+1}}c_ac_{a+1}e({\bf i}),$$
\item $\sigma_a^2e({\bf i}) = \tQ_{{\bf i}_a,{\bf i}_{a+1}}(y_a,y_{a+1})e({\bf i})$,
\item $\sigma_a\sigma_{b}=\sigma_{b}\sigma_{a}$ if $|a-b|>1$,
\item
%\begin{align*}
%&(\sigma_{a+1}\sigma_a\sigma_{a+1}-\sigma_a\sigma_{a+1}\sigma_a)e({\bf i}) \\
%&\hs{4ex}=\begin{cases}
%\dfrac{\tQ_{{\bf i}_a,{\bf i}_{a+1}}(y_{a+2},y_{a+1})
%-\tQ_{{\bf i}_a,{\bf i}_{a+1}}(y_{a},y_{a+1})}{y_{a+2}-y_{a}} e({\bf i})
%& \text{if ${\bf i}_a={\bf i}_{a+2}\not\in J^c$}, \\[3ex]
%\dfrac{\tQ_{{\bf i}_a,{\bf i}_{a+1}}(y_{a+2},y_{a+1})-\tQ_{{\bf i}_a,{\bf i}_{a+1}}(-y_{a},y_{a+1})}{y_{a+2}+y_{a}}
%c_ac_{a+2}e({\bf i})
%& \text{if ${\bf i}_a=c{\bf i}_{a+2}\not\in J^c$}, \\[3ex]
%\dfrac{\tQ_{{\bf i}_a,{\bf i}_{a+1}}(y_{a+2},y_{a+1})
%-\tQ_{{\bf i}_a,{\bf i}_{a+1}}(y_{a},y_{a+1})}{y_{a+2}-y_{a}}e({\bf i}) &\\[1ex]
%\hs{5ex}+\dfrac{\tQ_{{\bf i}_a,{\bf i}_{a+1}}(y_{a+2},y_{a+1})
%-\tQ_{{\bf i}_a,{\bf i}_{a+1}}(y_{a},y_{a+1})}{y_{a+2}+y_{a}} c_ac_{a+2}e({\bf i})
%& \text{if ${\bf i}_a={\bf i}_{a+2}\in J^c$}, \\[3ex]
%0 & \text{otherwise}.
%\end{cases}
%\end{align*}
%or equivalently,
\begin{align*}
\sigma_{a+1}\sigma_a\sigma_{a+1}-\sigma_a\sigma_{a+1}\sigma_a
&=\sum_{{\bf i}_a={\bf i}_{a+2}}
\dfrac{\tQ_{{\bf i}_a,{\bf i}_{a+1}}(y_{a+2},y_{a+1})-\tQ_{{\bf i}_a,{\bf i}_{a+1}}(y_{a},y_{a+1})}
{y_{a+2}-y_{a}}e({\bf i}) \\
&\hs{-1ex}+\kern-2ex\sum_{{\bf i}_a=c{\bf i}_{a+2}}
\dfrac{\tQ_{{\bf i}_a,{\bf i}_{a+1}}(y_{a+2},y_{a+1})-\tQ_{{\bf i}_a,{\bf i}_{a+1}}(-y_{a},y_{a+1})}{y_{a+2}+y_{a}}
c_ac_{a+2}e({\bf i}).
\end{align*}
\end{enumerate}
\end{defn}

$RC_n$ is also $\Z$-graded by setting
$$
\deg(e({\bf i}))=0,\quad \deg(y_pe({\bf i}))=(\nu_{\pr({\bf i}_k)}|\nu_{\pr({\bf i}_k)}),\quad \deg(\sigma_ae({\bf i}))=-(\nu_{\pr({\bf i}_a)}|\nu_{\pr({\bf i}_{a+1})}).
$$

\begin{prop}\cite[Corollary 3.9]{KKT}
For each $w\in \Sym_n$, we choose
a reduced expression $s_{i_1}\cdots s_{i_{\ell}}$
of $w$, and set $\sigma_w=\sigma_{i_1}\cdots
\sigma_{i_{\ell}}$.
Then the set of elements
\begin{align*}
\{
y^a c^{\eta}\sigma_{w}e({\bf i})\mid
a\in(\Z_{\geq 0})^n,\, \eta\in\Z_2^n,\,w\in\Sym_n,\,{\bf i}\in J^n \}
\end{align*}
forms an ${\rm R}$-basis of $RC_n,$ where
$y^a=y_1^{a_1}\cdots y_n^{a_n}$ for
$a=(a_1,\ldots,a_n)\in (\Z_{\geq 0})^n$ and
$c^{\eta}=c_1^{\eta_1}\cdots c_n^{\eta_n}$
for $\eta=(\eta_1,\cdots,\eta_n)\in\Z_2^n$.
\end{prop}

If $\Lambda\in P^+,\,j\in J$ and $u$ is an indeterminate over $\Z$, then we define
$$
a_j^\Lambda(u)=u^{\<h_{\pr(j)},\Lambda\>},\quad
a^\Lambda(y_1):=\sum_{{\bf i}\in J^n}y_1^{\<h_{\pr({\bf i}_1)},\Lambda\>}e({\bf i})\in RC_n.
$$
\label{pag:cyclotomic quiver Hecke-Clifford superalgebra}
\begin{defn}\cite[Section 3.7]{KKT}
Let $\Lambda\in P^+$.
The cyclotomic quiver Hecke-Clifford superalgebra $RC^\Lambda_n$ is defined to be the quotient algebra:
$$RC^\Lambda_n:=RC_n/\<a^\Lambda(y_1)\>.$$
\end{defn}
Similarly, $RC^\Lambda_n$ inherits $\Z\times\Z_2$-grading from $RC^\Lambda_n.$
By some abuse of notations, we shall use the same symbols to denote the generators of both $RC_n$ and $RC^\Lambda_n.$

\begin{rem}
The algebras $RC_n$ and $RC^\Lambda_n$ have an anti-involution $*$
that sends the generators $e({\bf i}),$ $y_p,$ $\sigma_a,$ $c_p$ to themselves.
\end{rem}

\label{pag:J nu}
For any $\nu\in Q_n^+$, we define
$$J^\nu:=\Bigl\{{\bf i}=({\bf i}_1,\cdots,{\bf i}_n)\in J^n\Bigm|\sum_{s=1}^{n}\nu_{\pr({\bf i}_s)}=\nu\Bigr\}.$$
Let $e^J_\nu :=\sum_{{\bf i}\in J^{\nu}}e({\bf i})$ be the certain central idempotent, then we define
\label{pag:J blocks}
$$RC_\nu:=e^J_\nu RC_n, \qquad RC^\Lambda_\nu:=e^J_\nu RC^\Lambda_n.$$
Recall the canonical projection $\pr: J\to I$. We choose $J^\dag \subset J$ such that the projection $\pr$ induces a bijection $J^\dag\to I$. Let $e^\dag:=\sum_{{\bf i}\in {J^{\dag n}}}e({\bf i}).$
%\begin{rem}
%By \cite[Theorem 3.13]{KKT}, for any $\Lambda\in P^+,$ $\nu\in Q_n^+,$ the superalgebras $RC_\nu^\Lambda$ and $R_\nu^\Lambda$ are weakly Morita superequivalent to each other.
%\end{rem}
\label{pag:m(nu)}
\begin{defn}\label{m(nu)}
For $\nu=\sum_{i\in I}m_i\nu_i\in Q_+$, we define $m(\nu):=\sum_{i\in I_{\rm odd}}m_i\in\Z_{\geq 0}.$
\end{defn}
Kang, Kashiwara and Tsuchioka \cite{KKT} proved that the (cyclotomic) quiver Hecke superalgebra and the (cyclotomic) quiver Hecke-Clifford superalgebra are weakly Morita superequivalent to each other.

\begin{thm}\cite[Below Definition 3.10, Theorem 3.13]{KKT}\label{supermorita}
Let $\Lambda\in P^+$ and $\nu=\sum_{i\in I}m_i\nu_i\in Q_+$.

\begin{enumerate}
	
\item We have a
$\Z\times \Z_2$-graded ${\rm R}$-algebra isomorphism
\begin{align*}
R_{\nu}^\Lambda\otimes \mathcal{C}_{m(\nu)}\cong e^\dag RC_{\nu}^\Lambda e^\dag.
\end{align*}

\item Suppose ${\rm R}=\mathbb{K}$ is a field, then we have the following morita superequivalent $$RC_{\nu}^\Lambda \overset{\rm sMor}{\sim} e^\dag RC_{\nu}^\Lambda e^\dag.$$
\end{enumerate}
%In particular, $R_{\beta}^\Lambda$ and
%$RC_{\beta}^\Lambda$ are weakly Morita superequivalent.
\end{thm}

\section{Cyclotomic Hecke-Clifford superalgebra and KKT's isomorphism}\label{basic-Non-dege}
{\bf Throughout this section, we fix $n\in \mathbb{N}$ and $q\in{\rm R^\times}\setminus\{\pm 1\}$ such that $q+q^{-1}\in{\rm R^\times}$.}
\subsection{Affine Hecke-Clifford superalgebra $\mathcal{H}_{\rm R}$}
	
	We define $\epsilon:=q-q^{-1}\in{\rm R}\setminus\{0\}.$
	The non-degenerate affine Hecke-Clifford superalgebra $\mathcal{H}_{\rm R}=\mathcal{H}_{\rm R}(n)$ \label{pag:AHCA} is
	the superalgebra over ${\rm R}$ generated by even generators
	$T_1,\ldots,T_{n-1},$ $X_1^{\pm 1},\ldots,X_n^{\pm 1}$ and odd generators
	$C_1,\ldots,C_n$ subject to the following relations
	\begin{align}
		T_i^2=\epsilon T_i +1,\quad T_iT_j =T_jT_i, &\quad
		T_iT_{i+1}T_i=T_{i+1}T_iT_{i+1}, \quad|i-j|>1,\label{Braid}\\
		X_iX_j&=X_jX_i, X_iX^{-1}_i=X^{-1}_iX_i=1, \quad 1\leq i,j\leq n, \label{Poly}\\
		C_i^2=1,C_iC_j&=-C_jC_i, \quad 1\leq i\neq j\leq n, \label{Clifford}\\
		T_iX_i&=X_{i+1}T_i-\epsilon(X_{i+1}+C_iC_{i+1}X_i),\label{PX1}\\
		T_iX_{i+1}&=X_iT_i+\epsilon(1-C_iC_{i+1})X_{i+1},\label{PX2}\\
		T_iX_j&=X_jT_i, \quad j\neq i, i+1, \label{PX3}\\
		T_iC_i=C_{i+1}T_i, T_iC_{i+1}&=C_iT_i-\epsilon(C_i-C_{i+1}),T_iC_j=C_jT_i,\quad j\neq i, i+1, \label{PC}\\
		X_iC_i=C_iX^{-1}_i, X_iC_j&=C_jX_i,\quad 1\leq i\neq j\leq n.
		\label{XC}
	\end{align}
	
	For $\alpha=(\alpha_1,\ldots,\alpha_n)\in\mathbb{Z}^n$ and
	$\beta=(\beta_1,\ldots,\beta_n)\in\mathbb{Z}_2^n$, we set
	$X^{\alpha}=X_1^{\alpha_1}\cdots X_n^{\alpha},$
	$C^{\beta}=C_1^{\beta_1}\cdots C_n^{\beta_n}$ and define
	$\supp(\beta):=\{1 \leq k \leq n:\beta_{k}=\bar{1}\},$ $|\beta|:=\Sigma_{i=1}^{n}\beta_i \in \mathbb{Z}_2.$ \label{pag:suppot and sum}
	Then we have the
	following.
	\begin{lem}\cite[Theorem 2.2]{BK}\label{lem:PBWNon-dege}
		The set $\{X^{\alpha}C^{\beta}T_w~|~ \alpha\in\mathbb{Z}^n,
		\beta\in\mathbb{Z}_2^n, w\in \mathfrak{S}_n\}$ forms a basis of $\mathcal{H}_{\rm R}$.
	\end{lem}
	
	Let $\mathcal{A}_n$ \label{pag:subalg An} be the subalgebra generated by even generators $X_1^{\pm 1},\ldots,X_n^{\pm 1}$ and odd generators $C_1,\ldots,C_n$. By Lemma~\ref{lem:PBWNon-dege}, $\mathcal{A}_n$ actually can be identified with the superalgebra generated by even generators $X^{\pm 1}_1,\ldots,X^{\pm 1}_n$ and odd generators $C_1,\ldots,C_n$ subject to relations \eqref{Poly}, \eqref{Clifford}, \eqref{XC}. Clifford algebra $\mathcal{C}_n$ can be identified with the subalgebra of $\mathcal{A}_n$ generated by  odd generators $C_1,\ldots,C_n$ subject to relations \eqref{Clifford}.
	
	{\bf In the rest of this subsection, we assume that
	${\rm R}=\mathbb{K}$ is the algebraically closed field of characteristic different from $2$.}
    For any $i=1,2,\ldots,n-1$ and $x,y \in \mathbb{K}^*$ satisfying $y\neq x^{\pm 1},$ let (\cite[(3.13)]{JN})
    \label{pag:Phi function}
	\begin{align}\label{Phi function}
		\Phi_i(x,y):=T_{i}+\frac{\epsilon}{x^{-1}y-1}-\frac{\epsilon}{xy-1}C_{i}C_{i+1} \in \mathcal{H}_{\mathbb{K}}.
	\end{align}
	These elements satisfy certain useful properties (\cite[Lemma 4.1]{JN}) and play key roles in the construction of seminormal bases of cyclotomic Hecke-Clifford superalgebras (\cite{LS2,LS3}, see also Section \ref{sec:Seminormal form}).

For any pair of $(x,y)\in (\mathbb{K}^*)^2$ and $y\neq x^{\pm1}$, we consider the following idempotency condition on $(x,y)$
	\begin{align}\label{invertible}
		\frac{x^{-1}y}{(x^{-1}y-1)^2}+\frac{xy}{(xy-1)^2}=\frac{1}{\epsilon^2}.
	\end{align}
	
	{\color{black}For any $a\in \mathbb{K}$, we fix a solution of the equation $x^2=a$ and denote it by $\sqrt{a}$.} For any  $x \in \mathbb{K}^*$, we define\footnote{We remark that in this paper, $\mathtt{q}(x)$ is equal to the definition of $\mathtt{q}(q^{-1}x)$ in \cite{SW,LS2,LS3}. The similar remark applies to $\mathtt{b}_{\pm}(x)$.} \label{pag:q-function and b-function}
\begin{align}\label{substitution0}
		\mathtt{q}(x):=2\frac{x+x^{-1}}{q+q^{-1}}, \quad \mathtt{b}_{\pm}(x):=\frac{\mathtt{q}(x)}{2}\pm \sqrt{\frac{\mathtt{q}(x)^2}{4}-1}.
	\end{align} We remark that $\mathtt{q}(q^{2i+1})$ is the definition of $q(i)$ in \cite[(4.5)]{BK}. Clearly, $\mathtt{b}_{\pm}(x)$ are exactly two solutions satisfying the equation $z+z^{-1}=\mathtt{q}(x)$ and moreover
	\begin{equation}\label{bpm}
		\mathtt{b}_-(x)=\mathtt{b}_+(x)^{-1}.
	\end{equation}

	\subsection{Cyclotomic Hecke-Clifford superalgebra}
	To define the cyclotomic Hecke-Clifford superalgebra $\mathcal{H}^f_{\rm R}=\mathcal{H}^f_{\rm R}(n)$ over ${\rm R},$ \label{pag:CHCA} we fix $m\geq 0,$ $\underline{Q}=(Q_1,Q_2,\ldots,Q_m)\in({\rm R^\times})^m$ \label{pag:Q-parameters} and take a $f=f(X_1)\in {\rm R}[X_1^\pm]$ satisfying \cite[(3.2)]{BK}.
	%Since we are working over algebraically closed field $\mathbb{K}$,
	It is noted in \cite{SW} that we only need to consider $f(X_1)\in {\rm R}[X_1^\pm]$ to be one of the following three forms:
	$$\begin{aligned}
		f=\begin{cases}
			f^{\mathsf{(0)}}_{\underline{Q}}=\prod_{i=1}^m \biggl(X_1+X^{-1}_1-\mathtt{q}(Q_i)\biggr), \\ % \mbox{if $r:=\deg(f)=2m$};
			f^{\mathsf{(s)}}_{\underline{Q}}=(X_1-1)\prod_{i=1}^m \biggl(X_1+X^{-1}_1-\mathtt{q}(Q_i)\biggr), \\  % \mbox{if $r:=\deg(f)=2m+1$};
			f^{\mathsf{(ss)}}_{\underline{Q}} = (X_1-1)(X_1+1)\prod_{i=1}^m \biggl(X_1+X^{-1}_1-\mathtt{q}(Q_i)\biggr).%& \mbox{if $r:=\deg(f)=2m+2$}
		\end{cases}
	\end{aligned}$$
	In each case, the degree $r$ of the polynomial $f$ is $2m,\,2m+1,\,2m+2$ respectively.
	
The non-degenerate cyclotomic Hecke-Clifford superalgebra $\mathcal{H}^f_{\rm R}$ is defined as
$$\mathcal{H}^f_{\rm R}:=\mathcal{H}_{\rm R}/\mathcal{I}_f,$$
where $\mathcal{I}_f$ is the two sided ideal of $\mHcn$ generated by $f(X_1)$. The degree $r$ \label{pag:nondege level} of $f$ is called the level of $\mathcal{H}^f_{\rm R}.$ We shall denote the images of $X^{\alpha}, C^{\beta}, T_w$ in the cyclotomic quotient $\mathcal{H}^f_{\rm R}$ still by the same symbols. Then we have the following due to \cite{BK}.	
	\begin{lem}\cite[Theorem 3.6]{BK}
		The set $\{X^{\alpha}C^{\beta}T_w~|~\alpha\in\{0,1,\cdots,r-1\}^n,
		\beta\in\mathbb{Z}_2^n, w\in {\mathfrak{S}_n}\}$ forms an ${\rm R}$-basis of $\mathcal{H}^f_{\rm R}$.
	\end{lem}

\begin{defn}\label{(super)symmetrizing form}\cite[Definition 2.1]{LS1}, \cite[Section 4.1, 5.1]{WW2}
	Let ${\mathcal{A}}={\mathcal{A}}_{\overline{0}}\oplus {\mathcal{A}}_{\overline{1}}$ be an ${\rm R}$-superalgebra
	which is free and of finite rank over ${\rm R},$ ${\rm p}: {\mathcal{A}} \rightarrow \Bbb Z_{2}$ be the parity map.
	
	(i) We call an ${\rm R}$-linear map $t:\mathcal{A} \rightarrow {\rm R}$ non-degenerate if there is a $\Z_2$-homogeneous basis $\mathcal{B}$ such that the determinant ${\rm det}\left(t(b_1b_2)\right)_{b_1,b_2\in \mathcal{B}}\in {\rm R}^\times.$

(ii) The superalgebra $\mathcal{A}$ is called symmetric if there is an evenly non-degenerate ${\rm R}$-linear map $t:\mathcal{A} \rightarrow {\rm R}$
such that $t(xy)=t(yx)$ for any $x, y\in \mathcal{A}$. In this case, we call $t$ a symmetrizing form on $\mathcal{A}$.

(iii) The superalgebra $\mathcal{A}$ is called supersymmetric if there is an evenly non-degenerate ${\rm R}$-linear map $t:\mathcal{A} \rightarrow {\rm R}$
such that  $t(xy)=(-1)^{{\rm p}(x){\rm p}(y)}t(yx)$ for any homogeneous $x, y\in \mathcal{A}$. In this case, we call $t$ a supersymmetrizing form on $\mathcal{A}.$
\end{defn}

The following Frobenius form is due to \cite{BK}.

\begin{prop}\cite[Corollary 3.14]{BK}, \cite[Proposition 5.4]{LS3}
Let $\alpha=(\alpha_1,\ldots,\alpha_n)\in [0,r-1]^n,$ $\beta=(\beta_1,\ldots,\beta_n)\in\mathbb{Z}_2^n$ and $w\in \mathfrak{S}_n,$
then the map
\label{pag:frob from}
\begin{align}\label{closed formula frob}
\tau^{\rm R}_{r,n}( X^\alpha C^\beta T_w ):=\delta_{(\alpha,\beta,w),(0,0,1)}
\end{align}
is a Frobenius form on $\mathcal{H}^f_{\rm R}.$
\end{prop}

When $\bullet=\mathsf{0},$ we can modify the above Frobenius form to obtain a supersymmetrizing form.
\label{pag:supersym from}
\begin{prop}\cite[Theorem 1.2 (1)]{LS3}\label{supersymmetrizing Non-dege}
Suppose $\bullet=\mathsf{0},$ then the cyclotomic Hecke-Clifford superalgebra $\mathcal{H}^f_{\rm R}$ is supersymmetric with the supersymmetrizing form
$$t^{\rm R}_{r,n}:=\tau^{\rm R}_{r,n}\Bigl(-\cdot (X_1X_2\cdots X_n)^{m} \Bigr).$$
\end{prop}

We shall omit the superscript in $t^{\rm R}_{r,n}$ when ${\rm R}$ is clear in the context.

\subsection{Combinatorics}\label{Combinatorics}
	\label{pag:The types of combinatorics}
	The different choices of $f\in \{f^{\mathsf{(0)}}_{\underline{Q}},\,f^{\mathsf{(s)}}_{\underline{Q}},\,f^{\mathsf{(ss)}}_{\underline{Q}}\}$ corresponds to different combinatorics $\mathscr{P}^{\mathsf{0},m}_{n},\mathscr{P}^{\mathsf{s},m}_{n},\mathscr{P}^{\mathsf{ss},m}_{n}$ respectively in the representation theory of $\mathcal{H}^f_{\rm R}$. Let's recall these combinatorics. For $n\in \N$, let $\mathscr{P}_n$ be the set of partitions of $n$ and denote by $\ell(\mu)$ the number of nonzero parts in the partition $\mu$ for each $\mu\in\mathscr{P}_n$. Let $\mathscr{P}^m_n$ be the set of all $m$-multipartitions of $n$ for $m\geq 0$, where we use convention that $\mathscr{P}^0_n=\{\emptyset\}$. Let $\mathscr{P}^\mathsf{s}_n$ be the set of strict partitions of $n$. Then for $m\geq 0$, set
	$$
	\mathscr{P}^{\mathsf{s},m}_{n}:=
	\cup_{a=0}^{n}( \mathscr{P}^{\mathsf{s}}_a\times \mathscr{P}^{m}_{n-a}),\qquad \mathscr{P}^{\mathsf{ss}, m}_{n}:=
	\cup_{a+b+c=n}(\mathscr{P}^{\mathsf{s}}_a \times \mathscr{P}^{\mathsf{s}}_b\times \mathscr{P}^{m}_{c}).$$
	We will formally write  $\mathscr{P}^{\mathsf{0},m}_{n}=\mathscr{P}^m_n$.  In convention, for any \label{pag:multipartition} $\undla\in  \mathscr{P}^{\mathsf{0},m}_{n}$, we write  $\undla=(\lambda^{(1)},\cdots,\lambda^{(m)}),$ while for any $\undla\in  \mathscr{P}^{\mathsf{s},m}_{n}$, we write  $\undla=(\lambda^{(0)},\lambda^{(1)},\cdots,\lambda^{(m)})$, i.e. we shall put the strict partition in the $0$-th component. Moreover, for any $\undla\in  \mathscr{P}^{\mathsf{ss},m}_{n}$, we write  $\undla=(\lambda^{(0_-)},\lambda^{(0_+)},\lambda^{(1)},\cdots,\lambda^{(m)})$, i.e. we shall put two strict partitions in the $0_-$-th component and the $0_+$-th component.
	
	We will also identify the (strict) partition with the corresponding (shifted) young diagram.  For any  $\undla\in\mathscr{P}^{\bullet,m}_{n}$ with $\bullet\in\{\mathsf{0},\mathsf{s},\mathsf{ss}\}$ and $m\in \N$, the box in the $l$-th component with row $i$, column $j$ will be denoted by $(i,j,l)$  with $l\in\{1,2,\ldots,m\},$ or $l\in\{0,1,2,\ldots,m\}$ or $l\in\{0_-,0_+,1,2,\ldots,m\}$ in the case $\bullet=\mathsf{0},\mathsf{s},\mathsf{ss},$ respectively. We also use the notation $\alpha=(i,j,l)\in \undla$ if the diagram of $\undla$ has a box $\alpha$ on the $l$-th component of row $i$ and column $j$. We use $\Std(\undla)$ \label{pag:standard tableaux} to denote the set of standard tableaux of shape $\undla$. One can also regard each $\mathfrak{t}\in\Std(\undla)$ as a bijection $\mathfrak{t}:\undla\rightarrow \{1,2,\ldots, n\}$ satisfying $\mathfrak{t}((i,j,l))=k$ if the box occupied by $k$ is located in the $i$-th row, $j$-th column in the $l$-th component $\lambda^{(l)}$.
For $0\leq k\leq n,$ let $\mt\downarrow_{k}$ be the subtableau of $\mt$ that contains the numbers $\{1, 2,\dots,k\}$.
In particular, $\mt\downarrow_{0}$ is the empty tableau.
We use $\mathfrak{t}^{\undla}$ (resp. $\mathfrak{t}_{\undla}$) to denote the standard tableaux obtained by inserting the symbols $1,2,\ldots,n$ consecutively by rows (resp. column) from the first (resp. last) component of $\undla$.

We use $\Add(\undla)$ and $\Rem(\undla)$ to denote the set of addable boxes of $\undla$ and the set of removable boxes of $\undla$ respectively. For $\mt\in\Std(\undla),$ we define $\Add(\mt):=\Add(\undla)$ and $\Rem(\mt):=\Rem(\undla).$

	\label{pag:diag of undlam}
	\begin{defn}(\cite[Definition 2.5]{SW})
		Let $\undla\in\mathscr{P}^{\bullet,m}_{n}$ with $\bullet\in\{\mathsf{0},\mathsf{s},\mathsf{ss}\}$.  We define
		$$
		\mathcal{D}_{\undla}:=\begin{cases} \emptyset,&\text{if $\undla\in\mathscr{P}^{\mathsf{0},m}_n$,}\\
			\{(a,a,0)\mid (a,a,0)\in \undla,\,a\in\N\}, &\text{if $\undla\in\mathscr{P}^{\mathsf{s},m}_{n}$,}\\
			\big\{(a,a,l)\mid (a,a,l)\in \undla,\,a\in\N, l\in\{0_-,0_+\}\big\}, &\text{if $\undla\in\mathscr{P}^{\mathsf{ss}, m}_{n}.$}
		\end{cases}
		$$
		For any $\mathfrak{t}\in\Std(\undla), $ we define
		\begin{align}\label{Dt}
			\mathcal{D}_{\mt}:=\{\mathfrak{t}(a,a,l)\mid (a,a,l)\in\mathcal{D}_{\undla}\}.
		\end{align}
	\end{defn}

	\begin{example} Let $\undla=(\lambda^{(0)},\lambda^{(1)})\in \mathscr{P}^{\mathsf{s},1}_{5}$, where via the identification with strict Young diagrams and Young diagrams:
		$$
		\lambda^{(0)}=\young(\,\, ,:\,),\qquad \lambda^{(1)}=\young(\,,\,).
		$$
		Then
		$$
		\mathfrak{t}^{\undla}=\Biggl(\young(12,:3),\quad \young(4,5)\Biggr).
		$$
		and  an example of standard tableau is as follows:
		$$
		\mathfrak{t}=\Biggl(\young(13,:5),\quad \young(2,4)\Biggr)\in \Std(\undla).
		$$ We have $$\mathcal{D}_{\undla}=\{(1,1,0),(2,2,0)\},\qquad \mathcal{D}_{\mathfrak{t}}=\{1,5\}.
		$$
	\end{example}
	
	Let $\mathfrak{S}_n$ be the symmetric group on ${1,2,\ldots,n}$ with basic transpositions $s_1,s_2,\ldots, s_{n-1}$.  And $\mathfrak{S}_n$ acts on the set of tableaux of shape $\underline{\lambda}$ in the natural way.
	%    \label{pag:adimssible}
	%	\begin{defn}(\cite[Definition 2.7]{SW})
		%		Let  $\undla\in\mathscr{P}^{\bullet,m}_{n}$ with $\bullet\in\{\mathsf{0},\mathsf{s},\mathsf{ss}\}$.  For any standard tableaux $\mathfrak{t}\in \Std(\undla)$, if $s_l \mathfrak{t}$ is still standard, we call the simple transposition $s_l$ is admissible with respect to $\mathfrak{t}$. We set
		%$$P(\undla):=\biggl\{\tau=s_{k_t}\ldots s_{k_1}\in\mathfrak{S}_n\biggm|~\begin{matrix}&s_{k_l} \,\text{is admissible with respect to }\\
			%			& s_{k_{l-1}}\ldots s_{k_1}\mathfrak{t}^{\undla}, \,\text{for $l=1,2,\cdots,t$}
			%		\end{matrix}\biggr\}.
		%		$$
		%	\end{defn}
	%	The following standard facts will be used in the sequel.
	
	%\begin{proof}
	% We define $\psi(\tau):=\tau\mathfrak{t}^{\undla}$. Clearly, this is well-defined and it is injective. We only need to show $\psi$ is surjective. To this end, we define $O(\mathfrak{t})$ to be the maximal $m\leq n$ such that $1,2,\cdots,m-1$ are in the same position with $\mathfrak{t}^{\undla}$. Hence $O(\mathfrak{t})=n$ if and only if $\mathfrak{t}=\mathfrak{t}^{\undla}$. We use induction downward on $O(\mathfrak{t})$ to construct $\tau\in P(\undla)$ such that $\psi(\tau)=\mathfrak{t}$. Actually, if $m=O(\mathfrak{t})<n$ and $\mathfrak{t}^{\undla}(i,j,l)=m$, then $\mathfrak{t}(i,j,l)=m'>m$. It is clear that each $s_u$ is admissible with respect to $s_{u-1}\ldots s_{m'-1}\mathfrak{t}$ for $m\leq u<m'$. Note that $O(s_{m}\ldots s_{m'-1}\mathfrak{t})>m$. This proves our claim.
	%\end{proof}
	
	\begin{lem}\label{admissible transposes}(\cite[Lemma 2.8]{SW})
		Let $\undla\in\mathscr{P}^{\bullet,m}_{n}$ with $\bullet\in\{\mathsf{0},\mathsf{s},\mathsf{ss}\}$. For any $\ms,\mt \in \Std(\undla),$ we denote by $d(\ms,\mt)\in \mathfrak{S}_n$ the unique element
		such that $\ms=d(\ms,\mt)\mt.$ Then we have
		$$s_{k_i}\text{ is admissible with respect to } s_{k_{i-1}}\ldots s_{k_1}\mathfrak{t},\quad i=1,2,\ldots,p$$
		for any reduced expression $d(\ms,\mt)=s_{k_p}\cdots s_{k_1}.$
	\end{lem}

	We set $Q_0=Q_{0_+}=q,$ $Q_{0_-}=-q$.
    \label{pag:nondeg residue}
	\begin{defn}\cite[Definition 3.7]{SW} Suppose  $\undla\in\mathscr{P}^{\bullet,m}_{n}$ with $\bullet\in\{\mathsf{0},\mathsf{s},\mathsf{ss}\}$ and $(i,j,l)\in \undla$, we define the residue of box $(i,j,l)$ with respect to the parameter $\undQ$ as follows
%		\footnote{Note that $\res(i,j,l)$ here corresponds to $q\res(i,j,l)$ in \cite{SW,LS2,LS3}.}
		\begin{equation}\label{eq:residue}
			\res(i,j,l):=Q_lq^{2(j-i)}.
		\end{equation}
		If $\mathfrak{t}\in \Std(\undla)$ and $\mathfrak{t}(i,j,l)=a$, we set
		\begin{align}
			\res_\mathfrak{t}(a)&:=Q_lq^{2(j-i)};\label{resNon-dege-1}\\
			\res(\mathfrak{t})&:=(\res_\mathfrak{t}(1),\cdots,\res_\mathfrak{t}(n)),\label{resNon-dege-2}\\
			\mathtt{q}(\res(\mathfrak{t}))&:=(\mathtt{q}(\res_{\mathfrak{t}}(1)), \mathtt{q}(\res_{\mathfrak{t}}(2)),\ldots, \mathtt{q}(\res_{\mathfrak{t}}(n))). \label{resNon-dege-3}
		\end{align}
	\end{defn}

    \label{pag:eigenspaces}
	Suppose that $M$ is a finite dimensional $\mathcal{H}^f_{\mathbb{K}}$-module. Then, we can decompose $M$ as a direct sum $M =\oplus_{{\bf i}\in(\mathbb{K}^*)^n}M_{\bf i}$ of its generalized eigenspaces, where $$
	M_{\bf i}=\{v\in M \mid (X_j-{\bf i}_j)^k v=0, \,\text{for $j=1,2,\cdots,n$, $k\gg0$}\}.
	$$  In particular, taking $M$ to be the regular $\mathcal{H}^f_{\mathbb{K}}$-module, we get a system $$\{e({\bf i})\mid {\bf i}\in{(\mathbb{K}^*)}^n, e({\bf i})\neq 0\}$$
	of pairwise orthogonal idempotents in $\mathcal{H}^f_{\mathbb{K}}$ such
	that $e({\bf i})M = M_{\bf i}$ for each finite dimensional left $\mathcal{H}^f_{\mathbb{K}}$-module $M$.

\subsection{Dynkin diagrams}\label{DynkinDiagrams}
In this subsection, ${\rm R}=\mathbb{K}$.
We explain how to associate a subset $I\subset\mathbb{K}$ with a quiver Hecke-Clifford superalgebra.

First, for any $Q\in \mathbb{K}^*$, following \cite{KKT}, we can associate the orbit $\{\mathtt{q}(q^{2l}Q)\in \mathbb{K}\mid l\in\Z\}$ with a certain Dynkin diagram as follows, where we mark the points $\mathtt{q}(q)=2$ and $\mathtt{q}(-q)=-2$  by $\times$.
\label{pag:lie types}
\begin{enumerate}
	
	\item When $q^2$ is not a root of unity, there are three types of
	Dynkin diagrams.
	
	\vs{1ex}
	\begin{enumerate}
		\item$Q\not\in \pm q^\Z$, where $\pm q^\Z=\{\pm q^k|~k\in\Z\}$.
		The Dynkin diagram is of type $A_{\infty}$.
		
		\vs{2ex}
		\begin{center}
			\unitlength 0.1in
			\begin{picture}( 20.0000,  3.0000)(  0.0000, -2.5000)
				% CIRCLE 2 0 3 0 Black White
				% 4 1000 200 1000 150 1000 150 1000 150
				%
				{\color{black}{%
						\special{pn 8}%
						\special{ar 1000 200 50 50  0.0000000 6.2831853}%
				}}%
				% CIRCLE 2 0 3 0 Black White
				% 4 1400 200 1400 150 1400 150 1400 150
				%
				{\color{black}{%
						\special{pn 8}%
						\special{ar 1400 200 50 50  0.0000000 6.2831853}%
				}}%
				% CIRCLE 2 0 3 0 Black White
				% 4 600 200 600 150 600 150 600 150
				%
				{\color{black}{%
						\special{pn 8}%
						\special{ar 600 200 50 50  0.0000000 6.2831853}%
				}}%
				% LINE 2 0 3 0 Black White
				% 2 650 200 950 200
				%
				{\color{black}{%
						\special{pn 8}%
						\special{pa 650 200}%
						\special{pa 950 200}%
						\special{fp}%
				}}%
				% LINE 2 0 3 0 Black White
				% 2 1050 200 1350 200
				%
				{\color{black}{%
						\special{pn 8}%
						\special{pa 1050 200}%
						\special{pa 1350 200}%
						\special{fp}%
				}}%
				% LINE 2 0 3 0 Black White
				% 2 1450 200 1750 200
				%
				{\color{black}{%
						\special{pn 8}%
						\special{pa 1450 200}%
						\special{pa 1750 200}%
						\special{fp}%
				}}%
				% LINE 2 0 3 0 Black White
				% 2 1050 200 1350 200
				%
				{\color{black}{%
						\special{pn 8}%
						\special{pa 1050 200}%
						\special{pa 1350 200}%
						\special{fp}%
				}}%
				% LINE 2 0 3 0 Black White
				% 2 250 200 550 200
				%
				{\color{black}{%
						\special{pn 8}%
						\special{pa 250 200}%
						\special{pa 550 200}%
						\special{fp}%
				}}%
				% STR 2 0 3 0 Black White
				% 4 950 -20 950 80 2 0 0 0
				% $\zeta$
				\put(9.5000,-0.8000){\makebox(0,0)[lb]{\footnotesize$\mathtt{q}(Q)$}}%
				% STR 2 0 3 0 Black White
				% 4 1200 -20 1200 80 2 0 0 0
				% $\lambda+1$
				\put(12.0000,-0.8000){\makebox(0,0)[lb]{\kern2ex\footnotesize$\mathtt{q}(q^2Q)$}}%
				% STR 2 0 3 0 Black White
				% 4 400 -20 400 80 2 0 0 0
				% $\lambda-1$
				\put(4.0000,-0.8000){\makebox(0,0)[lb]{\footnotesize$\mathtt{q}(q^{-2}Q)$}}%
				% LINE 2 2 3 0 Black White
				% 2 200 200 0 200
				%
				{\color{black}{%
						\special{pn 8}%
						\special{pa 200 200}%
						\special{pa 0 200}%
						\special{dt 0.045}%
				}}%
				% LINE 2 2 3 0 Black White
				% 2 2000 200 1800 200
				%
				{\color{black}{%
						\special{pn 8}%
						\special{pa 2000 200}%
						\special{pa 1800 200}%
						\special{dt 0.045}%
				}}%
			\end{picture}%
			
		\end{center}
		
		\vs{3ex}
		
		\item $Q=\varepsilon q^{2k+1}$  for some $k\in\Z$ and $\varepsilon\in\{\pm\}$.
		The Dynkin diagram is of type $B_\infty$.
		
		\medskip
		\begin{center}
			\unitlength 0.1in
			\begin{picture}( 14.5000,  3.0000)(  1.5000, -2.5000)
				
				{\color{black}{%
						\special{pn 8}%
						\special{ar 600 200 50 50  0.0000000 6.2831853}%
				}}%
				
				{\color{black}{%
						\special{pn 8}%
						\special{ar 1000 200 50 50  0.0000000 6.2831853}%
				}}%
				
				{\color{black}{%
						\special{pn 8}%
						\special{ar 200 200 50 50  0.0000000 6.2831853}%
				}}%
				
				{\color{black}{%
						\special{pn 8}%
						\special{pa 650 200}%
						\special{pa 950 200}%
						\special{fp}%
				}}%
				
				{\color{black}{%
						\special{pn 8}%
						\special{pa 1050 200}%
						\special{pa 1350 200}%
						\special{fp}%
				}}%
				
				{\color{black}{%
						\special{pn 8}%
						\special{pa 650 200}%
						\special{pa 950 200}%
						\special{fp}%
				}}%
				\put(9.5000,-0.8000){\makebox(0,0)[lb]{\footnotesize$\mathtt{q}(\varepsilon q^5)$}}%
				
				{\color{black}{%
						\special{pn 8}%
						\special{pa 1600 200}%
						\special{pa 1400 200}%
						\special{dt 0.045}%
				}}%
				\put(5.5000,-0.8000){\makebox(0,0)[lb]{\footnotesize$\mathtt{q}(\varepsilon q^{3})$}}%
				\put(1.5000,-0.8000){\makebox(0,0)[lb]{\footnotesize$\mathtt{q}(\varepsilon q)$}}%
				
				{\color{black}{%
						\special{pn 8}%
						\special{pa 250 180}%
						\special{pa 550 180}%
						\special{fp}%
				}}%
				\put(2.0500,-2.55000){\makebox(0,0)[lb]{\kern-1ex$\times$}}
				{\color{black}{%
						\special{pn 8}%
						\special{pa 250 220}%
						\special{pa 550 220}%
						\special{fp}%
				}}%
				
				{\color{black}{%
						\special{pn 8}%
						\special{pa 350 200}%
						\special{pa 450 150}%
						\special{fp}%
				}}%
				
				{\color{black}{%
						\special{pn 8}%
						\special{pa 450 250}%
						\special{pa 350 200}%
						\special{fp}%
				}}%
			\end{picture}%
			
		\end{center}
		
		\vs{3ex}
		\item $Q=\varepsilon q^{2k}$ for some $k\in\Z$ and $\varepsilon\in\{\pm\}$.
		The Dynkin diagram is of type $C_\infty$.
		
		\vs{3ex}
		\begin{center}
			\unitlength 0.1in
			\begin{picture}( 16.0000,  3.0000)(  0.0000, -2.5000)
				% CIRCLE 2 0 3 0 Black White
				% 4 600 200 600 150 600 150 600 150
				%
				{\color{black}{%
						\special{pn 8}%
						\special{ar 600 200 50 50  0.0000000 6.2831853}%
				}}%
				% CIRCLE 2 0 3 0 Black White
				% 4 1000 200 1000 150 1000 150 1000 150
				%
				{\color{black}{%
						\special{pn 8}%
						\special{ar 1000 200 50 50  0.0000000 6.2831853}%
				}}%
				% CIRCLE 2 0 3 0 Black White
				% 4 200 200 200 150 200 150 200 150
				%
				{\color{black}{%
						\special{pn 8}%
						\special{ar 200 200 50 50  0.0000000 6.2831853}%
				}}%
				% LINE 2 0 3 0 Black White
				% 2 650 200 950 200
				%
				{\color{black}{%
						\special{pn 8}%
						\special{pa 650 200}%
						\special{pa 950 200}%
						\special{fp}%
				}}%
				% LINE 2 0 3 0 Black White
				% 2 1050 200 1350 200
				%
				{\color{black}{%
						\special{pn 8}%
						\special{pa 1050 200}%
						\special{pa 1350 200}%
						\special{fp}%
				}}%
				% LINE 2 0 3 0 Black White
				% 2 650 200 950 200
				%
				{\color{black}{%
						\special{pn 8}%
						\special{pa 650 200}%
						\special{pa 950 200}%
						\special{fp}%
				}}%
				
				\put(9.0000,-0.8000){\makebox(0,0)[lb]{\footnotesize$\mathtt{q}(\varepsilon q^4)$}}%
				
				{\color{black}{%
						\special{pn 8}%
						\special{pa 1600 200}%
						\special{pa 1400 200}%
						\special{dt 0.045}%
				}}%
				
				\put(5.0000,-0.8000){\makebox(0,0)[lb]{\footnotesize$\mathtt{q}(\varepsilon q^2)$}}%
				
				\put(1.5000,-0.8000){\makebox(0,0)[lb]{\footnotesize$\mathtt{q}(\varepsilon)$}}%
				% LINE 2 0 3 0 Black White
				% 2 250 180 550 180
				%
				{\color{black}{%
						\special{pn 8}%
						\special{pa 250 180}%
						\special{pa 550 180}%
						\special{fp}%
				}}%
				
				{\color{black}{%
						\special{pn 8}%
						\special{pa 250 220}%
						\special{pa 550 220}%
						\special{fp}%
				}}%
				
				{\color{black}{%
						\special{pn 8}%
						\special{pa 350 150}%
						\special{pa 450 200}%
						\special{fp}%
				}}%
				
				{\color{black}{%
						\special{pn 8}%
						\special{pa 450 200}%
						\special{pa 350 250}%
						\special{fp}%
				}}%
			\end{picture}%
			
		\end{center}
	\end{enumerate}
	\vs{3ex}
	\item  When $q^2$ is a primitive $\ell$-th root of unity, there are three types of Dynkin diagram.
	
	\vs{1ex}
	\begin{enumerate}
		\item $Q\not\in\pm q^\Z$. The Dynkin diagram is of type $A^{(1)}_{s-1}$ $(\ell=s>2)$.
		
		\medskip
		\begin{center}
			\unitlength 0.1in
			\begin{picture}( 20.5000,  4.5000)(  0.5000, -4.0000)
				% CIRCLE 2 0 3 0 Black White
				% 4 600 200 600 150 600 150 600 150
				%
				{\color{black}{%
						\special{pn 8}%
						\special{ar 600 200 50 50  0.0000000 6.2831853}%
				}}%
				% CIRCLE 2 0 3 0 Black White
				% 4 1000 200 1000 150 1000 150 1000 150
				%
				{\color{black}{%
						\special{pn 8}%
						\special{ar 1000 200 50 50  0.0000000 6.2831853}%
				}}%
				% CIRCLE 2 0 3 0 Black White
				% 4 200 200 200 150 200 150 200 150
				%
				{\color{black}{%
						\special{pn 8}%
						\special{ar 200 200 50 50  0.0000000 6.2831853}%
				}}%
				% LINE 2 0 3 0 Black White
				% 2 250 200 550 200
				%
				{\color{black}{%
						\special{pn 8}%
						\special{pa 250 200}%
						\special{pa 550 200}%
						\special{fp}%
				}}%
				% LINE 2 0 3 0 Black White
				% 2 650 200 950 200
				%
				{\color{black}{%
						\special{pn 8}%
						\special{pa 650 200}%
						\special{pa 950 200}%
						\special{fp}%
				}}%
				% LINE 2 0 3 0 Black White
				% 2 1050 200 1350 200
				%
				{\color{black}{%
						\special{pn 8}%
						\special{pa 1050 200}%
						\special{pa 1350 200}%
						\special{fp}%
				}}%
				% LINE 2 0 3 0 Black White
				% 2 650 200 950 200
				%
				{\color{black}{%
						\special{pn 8}%
						\special{pa 650 200}%
						\special{pa 950 200}%
						\special{fp}%
				}}%
				% STR 2 0 3 0 Black White
				% 4 850 -20 850 80 2 0 0 0
				% $\zeta q^2$
				\put(8.5000,-0.8000){\makebox(0,0)[lb]{\footnotesize$\mathtt{q}(q^4Q)$}}%
				% LINE 2 2 3 0 Black White
				% 2 1600 200 1400 200
				%
				{\color{black}{%
						\special{pn 8}%
						\special{pa 1600 200}%
						\special{pa 1400 200}%
						\special{dt 0.045}%
				}}%
				% STR 2 0 3 0 Black White
				% 4 450 -20 450 80 2 0 0 0
				% $\zeta+1$
				\put(4.5000,-0.8000){\makebox(0,0)[lb]{\footnotesize$\mathtt{q}(q^2Q)$}}%
				% STR 2 0 3 0 Black White
				% 4 50 -20 50 80 2 0 0 0
				% $\zeta$
				\put(0.5000,-0.8000){\makebox(0,0)[lb]{\footnotesize$\mathtt{q}(Q)$}}%
				% LINE 2 0 3 0 Black White
				% 2 1700 200 2000 200
				%
				{\color{black}{%
						\special{pn 8}%
						\special{pa 1700 200}%
						\special{pa 2000 200}%
						\special{fp}%
				}}%
				% CIRCLE 2 0 3 0 Black White
				% 4 2050 200 2050 150 2050 150 2050 150
				%
				{\color{black}{%
						\special{pn 8}%
						\special{ar 2050 200 50 50  0.0000000 6.2831853}%
				}}%
				% STR 2 0 3 0 Black White
				% 4 1850 -20 1850 80 2 0 0 0
				% $\zeta+2\ell$
				\put(18.5000,-0.8000){\makebox(0,0)[lb]{\kern1ex\footnotesize$\mathtt{q}(q^{2s-2}Q)$}}%
				% LINE 2 0 3 0 Black White
				% 4 200 250 1120 400 1120 400 2050 250
				%
				{\color{black}{%
						\special{pn 8}%
						\special{pa 200 250}%
						\special{pa 1120 400}%
						\special{fp}%
						\special{pa 1120 400}%
						\special{pa 2050 250}%
						\special{fp}%
				}}%
			\end{picture}%
			
		\end{center}
		\label{case:Loop}
		
		\medskip
		\item $Q=\varepsilon q^{2k+1}$ for for some $k\in\Z$ and $\varepsilon\in\{\pm\}$,
		when $\ell$ is odd ($\ell=2s+1$ with $s\ge1$).
		In this case $q^{2s+1}=-1$.
		The Dynkin diagram is of type $A^{(2)}_{2s}$.
		
		\vs{3ex}
		\begin{center}
			\hs{10ex}
			\unitlength 0.1in
			\begin{picture}(  5.0000,  3.0000)(  1.5000, -2.5000)
				{\color{black}{%
						\special{pn 8}%
						\special{ar 600 200 50 50  0.0000000 6.2831853}%
				}}%
				{\color{black}{%
						\special{pn 8}%
						\special{ar 200 200 50 50  0.0000000 6.2831853}%
				}}%
				{\color{black}{%
						\special{pn 8}%
						\special{pa 250 170}%
						\special{pa 550 170}%
						\special{fp}%
				}}%
				\put(5.5000,-0.8000){\makebox(0,0)[lb]{\kern-.8ex\footnotesize$\mathtt{q}(\varepsilon q^3)$}}%
				% STR 2 0 3 0 Black White
				% 4 150 -20 150 80 2 0 0 0
				% $0$
				\put(1.5000,-0.8000){\makebox(0,0)[lb]{\kern-1ex\footnotesize$\mathtt{q}(\varepsilon q)$}}%
				\put(2.0500,-2.55000){\makebox(0,0)[lb]{\kern-1ex$\times$}}%
				{\color{black}{%
						\special{pn 8}%
						\special{pa 250 230}%
						\special{pa 550 230}%
						\special{fp}%
				}}%
				{\color{black}{%
						\special{pn 8}%
						\special{pa 350 200}%
						\special{pa 450 150}%
						\special{fp}%
				}}%
				{\color{black}{%
						\special{pn 8}%
						\special{pa 450 250}%
						\special{pa 350 200}%
						\special{fp}%
				}}%
				{\color{black}{%
						\special{pn 8}%
						\special{pa 250 190}%
						\special{pa 550 190}%
						\special{fp}%
				}}%
				{\color{black}{%
						\special{pn 8}%
						\special{pa 250 210}%
						\special{pa 550 210}%
						\special{fp}%
				}}%
			\end{picture}%
			\hs{15ex}\text{($(q^2)^3=1$)}
		\end{center}

		\vs{3ex}
		\begin{center}
			\unitlength 0.1in
			\begin{picture}( 19.6000,  3.0000)(  1.5000, -2.5000)
				% CIRCLE 2 0 3 0 Black White
				% 4 600 200 600 150 600 150 600 150
				%
				{\color{black}{%
						\special{pn 8}%
						\special{ar 600 200 50 50  0.0000000 6.2831853}%
				}}%
				% CIRCLE 2 0 3 0 Black White
				% 4 200 200 200 150 200 150 200 150
				%
				{\color{black}{%
						\special{pn 8}%
						\special{ar 200 200 50 50  0.0000000 6.2831853}%
				}}%
				% LINE 2 0 3 0 Black White
				% 2 250 180 550 180
				%
				{\color{black}{%
						\special{pn 8}%
						\special{pa 250 180}%
						\special{pa 550 180}%
						\special{fp}%
				}}%
				% LINE 2 0 3 0 Black White
				% 2 650 200 950 200
				%
				{\color{black}{%
						\special{pn 8}%
						\special{pa 650 200}%
						\special{pa 950 200}%
						\special{fp}%
				}}%
				% LINE 2 0 3 0 Black White
				% 2 650 200 950 200
				%
				{\color{black}{%
						\special{pn 8}%
						\special{pa 650 200}%
						\special{pa 950 200}%
						\special{fp}%
				}}%
				% LINE 2 2 3 0 Black White
				% 2 1200 200 1000 200
				%
				{\color{black}{%
						\special{pn 8}%
						\special{pa 1200 200}%
						\special{pa 1000 200}%
						\special{dt 0.045}%
				}}%
				% STR 2 0 3 0 Black White
				% 4 550 -20 550 80 2 0 0 0
				% $1$
				\put(5.5000,-0.8000){\makebox(0,0)[lb]{\footnotesize$\mathtt{q}(\varepsilon q^3)$}}%
				% STR 2 0 3 0 Black White
				% 4 150 -20 150 80 2 0 0 0
				% $0$
				\put(1.5000,-0.8000){\makebox(0,0)[lb]{\footnotesize$\mathtt{q}(\varepsilon q)$}}%
				\put(2.0500,-2.55000){\makebox(0,0)[lb]{\kern-1ex$\times$}}%
				% LINE 2 0 3 0 Black White
				% 2 1300 200 1600 200
				%
				{\color{black}{%
						\special{pn 8}%
						\special{pa 1300 200}%
						\special{pa 1600 200}%
						\special{fp}%
				}}%
				% CIRCLE 2 0 3 0 Black White
				% 4 1650 200 1650 150 1650 150 1650 150
				%
				{\color{black}{%
						\special{pn 8}%
						\special{ar 1650 200 50 50  0.0000000 6.2831853}%
				}}%
				% STR 2 0 3 0 Black White
				% 4 1500 -20 1500 80 2 0 0 0
				% $\ell-1$
				\put(14.0000,-0.8000){\makebox(0,0)[lb]{\footnotesize$\mathtt{q}(\varepsilon q^{2s-1})$}}%
				% LINE 2 0 3 0 Black White
				% 2 250 220 550 220
				%
				{\color{black}{%
						\special{pn 8}%
						\special{pa 250 220}%
						\special{pa 550 220}%
						\special{fp}%
				}}%
				% LINE 2 0 3 0 Black White
				% 2 350 200 450 150
				%
				{\color{black}{%
						\special{pn 8}%
						\special{pa 350 200}%
						\special{pa 450 150}%
						\special{fp}%
				}}%
				% LINE 2 0 3 0 Black White
				% 2 450 250 350 200
				%
				{\color{black}{%
						\special{pn 8}%
						\special{pa 450 250}%
						\special{pa 350 200}%
						\special{fp}%
				}}%
				% CIRCLE 2 0 3 0 Black White
				% 4 2060 200 2060 150 2060 150 2060 150
				%
				{\color{black}{%
						\special{pn 8}%
						\special{ar 2060 200 50 50  0.0000000 6.2831853}%
				}}%
				% LINE 2 0 3 0 Black White
				% 2 1710 180 2010 180
				%
				{\color{black}{%
						\special{pn 8}%
						\special{pa 1710 180}%
						\special{pa 2010 180}%
						\special{fp}%
				}}%
				% LINE 2 0 3 0 Black White
				% 2 1710 220 2010 220
				%
				{\color{black}{%
						\special{pn 8}%
						\special{pa 1710 220}%
						\special{pa 2010 220}%
						\special{fp}%
				}}%
				% LINE 2 0 3 0 Black White
				% 2 1810 200 1910 150
				%
				{\color{black}{%
						\special{pn 8}%
						\special{pa 1810 200}%
						\special{pa 1910 150}%
						\special{fp}%
				}}%
				% LINE 2 0 3 0 Black White
				% 2 1910 250 1810 200
				%
				{\color{black}{%
						\special{pn 8}%
						\special{pa 1910 250}%
						\special{pa 1810 200}%
						\special{fp}%
				}}%
				% STR 2 0 3 0 Black White
				% 4 2000 -20 2000 80 2 0 0 0
				% $\ell$
				\put(20.0000,-0.9000){\makebox(0,0)[lb]{\footnotesize$\mathtt{q}(\varepsilon q^{2s+1})$}}%
			\end{picture}%
			\hs{5ex}\text{($s>1$)}
		\end{center}
		
		\vs{3ex}
		\item
		$Q=\varepsilon q^{2k}$ for for some $k\in\Z$ and $\varepsilon\in\{\pm\}$, when $\ell$ is
		even ($\ell=2s$ with $s\ge2$). In this case $q^{2s}=-1$. The Dynkin
		diagram is of type $C^{(1)}_s$.
		
		\vs{3ex}
		\begin{center}
			\unitlength 0.1in
			\begin{picture}( 19.6000,  3.0000)(  1.5000, -2.5000)
				{\color{black}{%
						\special{pn 8}%
						\special{ar 600 200 50 50  0.0000000 6.2831853}%
				}}%
				{\color{black}{%
						\special{pn 8}%
						\special{ar 200 200 50 50  0.0000000 6.2831853}%
				}}%
				{\color{black}{%
						\special{pn 8}%
						\special{pa 1710 180}%
						\special{pa 2010 180}%
						\special{fp}%
				}}%
				{\color{black}{%
						\special{pn 8}%
						\special{pa 650 200}%
						\special{pa 950 200}%
						\special{fp}%
				}}%
				{\color{black}{%
						\special{pn 8}%
						\special{pa 650 200}%
						\special{pa 950 200}%
						\special{fp}%
				}}%
				{\color{black}{%
						\special{pn 8}%
						\special{pa 1200 200}%
						\special{pa 1000 200}%
						\special{dt 0.045}%
				}}%
				
				\put(5.5000,-0.8000){\makebox(0,0)[lb]{\footnotesize$\mathtt{q}(\varepsilon q^2)$}}%
				
				\put(1.5000,-0.8000){\makebox(0,0)[lb]{\footnotesize$\mathtt{q}(\varepsilon)$}}%
				
				{\color{black}{%
						\special{pn 8}%
						\special{pa 1300 200}%
						\special{pa 1600 200}%
						\special{fp}%
				}}%
				
				{\color{black}{%
						\special{pn 8}%
						\special{ar 1650 200 50 50  0.0000000 6.2831853}%
				}}%
				
				\put(14.0000,-0.8000){\makebox(0,0)[lb]{\footnotesize$\kern-2ex\mathtt{q}(\varepsilon q^{2(s-1)})$}}%
				
				{\color{black}{%
						\special{pn 8}%
						\special{pa 1710 220}%
						\special{pa 2010 220}%
						\special{fp}%
				}}%
				
				{\color{black}{%
						\special{pn 8}%
						\special{pa 1810 200}%
						\special{pa 1910 150}%
						\special{fp}%
				}}%
				
				{\color{black}{%
						\special{pn 8}%
						\special{pa 1910 250}%
						\special{pa 1810 200}%
						\special{fp}%
				}}%
				
				{\color{black}{%
						\special{pn 8}%
						\special{ar 2060 200 50 50  0.0000000 6.2831853}%
				}}%
				
				{\color{black}{%
						\special{pn 8}%
						\special{pa 250 180}%
						\special{pa 550 180}%
						\special{fp}%
				}}%
				
				{\color{black}{%
						\special{pn 8}%
						\special{pa 250 220}%
						\special{pa 550 220}%
						\special{fp}%
				}}%
				
				{\color{black}{%
						\special{pn 8}%
						\special{pa 350 250}%
						\special{pa 450 200}%
						\special{fp}%
				}}%
				
				{\color{black}{%
						\special{pn 8}%
						\special{pa 450 200}%
						\special{pa 350 150}%
						\special{fp}%
				}}%
				\put(20.0000,-0.8000){\makebox(0,0)[lb]{\kern-0.1ex\footnotesize$\mathtt{q}(\varepsilon q^{2s})=\mathtt{q}(- \varepsilon)$}}%
			\end{picture}%
			
		\end{center}
		
		\vs{3ex}
		
		\item
		$Q=\varepsilon q^{2k+1}$ for some $k\in\Z$ and $\varepsilon\in\{\pm\}$,
		where $\ell$ is even ($\ell=2s$ with $s\ge2$).
		In this case, $q^{2s}=-1$.
		The Dynkin diagram is of type $D^{(2)}_s$.

		\vs{3ex}
		\begin{center}
			\hs{30ex}

			\unitlength 0.1in
			\begin{picture}(  5.0000,  3.0000)(  1.5000, -2.5000)
				
				{\color{black}{%
						\special{pn 8}%
						\special{ar 600 200 50 50  0.0000000 6.2831853}%
				}}%
				
				{\color{black}{%
						\special{pn 8}%
						\special{ar 200 200 50 50  0.0000000 6.2831853}%
				}}%
				
				{\color{black}{%
						\special{pn 8}%
						\special{pa 250 180}%
						\special{pa 550 180}%
						\special{fp}%
				}}%
				
				\put(5.5000,-0.8000){\makebox(0,0)[lb]{\footnotesize$\mathtt{q}(\varepsilon q^3)=\mathtt{q}(-\varepsilon q^{-1})$}}%
				\put(5.37000,-2.55000){\makebox(0,0)[lb]{$\times$}}%
				\put(1.5000,-0.8000){\makebox(0,0)[lb]{\kern-.5ex\footnotesize$\mathtt{q}(\varepsilon q)$}}%
				\put(1.34500,-2.55000){\makebox(0,0)[lb]{$\times$}}%
				{\color{black}{%
						\special{pn 8}%
						\special{pa 250 220}%
						\special{pa 550 220}%
						\special{fp}%
				}}%
				
				{\color{black}{%
						\special{pn 8}%
						\special{pa 300 200}%
						\special{pa 400 150}%
						\special{fp}%
				}}%
				
				{\color{black}{%
						\special{pn 8}%
						\special{pa 400 250}%
						\special{pa 300 200}%
						\special{fp}%
				}}%
				
				{\color{black}{%
						\special{pn 8}%
						\special{pa 400 250}%
						\special{pa 500 200}%
						\special{fp}%
				}}%
				
				{\color{black}{%
						\special{pn 8}%
						\special{pa 500 200}%
						\special{pa 400 150}%
						\special{fp}%
				}}%
			\end{picture}%
			\hs{23ex}($s=2$, $(q^2)^2=-1$)
		\end{center}

		\vs{3ex}
		\begin{center}

			\unitlength 0.1in
			\begin{picture}( 19.6000,  3.0000)(  1.5000, -2.5000)
				% CIRCLE 2 0 3 0 Black White
				% 4 600 200 600 150 600 150 600 150
				%
				{\color{black}{%
						\special{pn 8}%
						\special{ar 600 200 50 50  0.0000000 6.2831853}%
				}}%
				% CIRCLE 2 0 3 0 Black White
				% 4 200 200 200 150 200 150 200 150
				%
				{\color{black}{%
						\special{pn 8}%
						\special{ar 200 200 50 50  0.0000000 6.2831853}%
				}}%
				% LINE 2 0 3 0 Black White
				% 2 250 180 550 180
				%
				{\color{black}{%
						\special{pn 8}%
						\special{pa 250 180}%
						\special{pa 550 180}%
						\special{fp}%
				}}%
				% LINE 2 0 3 0 Black White
				% 2 650 200 950 200
				%
				{\color{black}{%
						\special{pn 8}%
						\special{pa 650 200}%
						\special{pa 950 200}%
						\special{fp}%
				}}%
				% LINE 2 0 3 0 Black White
				% 2 650 200 950 200
				%
				{\color{black}{%
						\special{pn 8}%
						\special{pa 650 200}%
						\special{pa 950 200}%
						\special{fp}%
				}}%
				% LINE 2 2 3 0 Black White
				% 2 1200 200 1000 200
				%
				{\color{black}{%
						\special{pn 8}%
						\special{pa 1200 200}%
						\special{pa 1000 200}%
						\special{dt 0.045}%
				}}%
				% STR 2 0 3 0 Black White
				% 4 550 -20 550 80 2 0 0 0
				% $1$
				\put(5.5000,-0.8000){\makebox(0,0)[lb]{\footnotesize$\mathtt{q}(\varepsilon q^3)$}}%
				% STR 2 0 3 0 Black White
				% 4 150 -20 150 80 2 0 0 0
				% $0$
				\put(1.5000,-0.8000){\makebox(0,0)[lb]{\footnotesize$\mathtt{q}(\varepsilon q)$}}%
				\put(2.0500,-2.55000){\makebox(0,0)[lb]{\kern-1ex$\times$}}%
				% LINE 2 0 3 0 Black White
				% 2 1300 200 1600 200
				%
				{\color{black}{%
						\special{pn 8}%
						\special{pa 1300 200}%
						\special{pa 1600 200}%
						\special{fp}%
				}}%
				% CIRCLE 2 0 3 0 Black White
				% 4 1650 200 1650 150 1650 150 1650 150
				%
				{\color{black}{%
						\special{pn 8}%
						\special{ar 1650 200 50 50  0.0000000 6.2831853}%
				}}%
				% STR 2 0 3 0 Black White
				% 4 1500 -20 1500 80 2 0 0 0
				% $n-1$
				\put(15.0000,-0.8000){\makebox(0,0)[lb]{\footnotesize$\kern-2ex\mathtt{q}(\varepsilon q^{2s-3})$}}%
				% LINE 2 0 3 0 Black White
				% 2 250 220 550 220
				%
				{\color{black}{%
						\special{pn 8}%
						\special{pa 250 220}%
						\special{pa 550 220}%
						\special{fp}%
				}}%
				% LINE 2 0 3 0 Black White
				% 2 350 200 450 150
				%
				{\color{black}{%
						\special{pn 8}%
						\special{pa 350 200}%
						\special{pa 450 150}%
						\special{fp}%
				}}%
				% LINE 2 0 3 0 Black White
				% 2 450 250 350 200
				%
				{\color{black}{%
						\special{pn 8}%
						\special{pa 450 250}%
						\special{pa 350 200}%
						\special{fp}%
				}}%
				% CIRCLE 2 0 3 0 Black White
				% 4 2060 200 2060 150 2060 150 2060 150
				%
				{\color{black}{%
						\special{pn 8}%
						\special{ar 2060 200 50 50  0.0000000 6.2831853}%
				}}%
				% LINE 2 0 3 0 Black White
				% 2 1710 180 2010 180
				%
				{\color{black}{%
						\special{pn 8}%
						\special{pa 1710 180}%
						\special{pa 2010 180}%
						\special{fp}%
				}}%
				% LINE 2 0 3 0 Black White
				% 2 1710 220 2010 220
				%
				{\color{black}{%
						\special{pn 8}%
						\special{pa 1710 220}%
						\special{pa 2010 220}%
						\special{fp}%
				}}%
				% LINE 2 0 3 0 Black White
				% 2 1810 250 1910 200
				%
				{\color{black}{%
						\special{pn 8}%
						\special{pa 1810 250}%
						\special{pa 1910 200}%
						\special{fp}%
				}}%
				% LINE 2 0 3 0 Black White
				% 2 1910 200 1810 150
				%
				{\color{black}{%
						\special{pn 8}%
						\special{pa 1910 200}%
						\special{pa 1810 150}%
						\special{fp}%
				}}%
				% STR 2 0 3 0 Black White
				% 4 2000 -20 2000 80 2 0 0 0
				% $n$
				\put(20.0000,-0.8000){\makebox(0,0)[lb]{\footnotesize$\mathtt{q}(\varepsilon q^{2s-1})=\mathtt{q}((- \varepsilon q)^{-1})$}}%
				\put(19.980000,-2.55000){\makebox(0,0)[lb]{$\times$}}%
			\end{picture}%
			\hs{15ex}($s>2$)
		\end{center}
	\end{enumerate}
\end{enumerate}

%\begin{defn} For $I\subset \mathcal{K}$, we define $\mathtt{q}(I):=\{\mathtt{q}(x)\mid x\in I\}.$
%\end{defn}
Suppose $I\subset \mathbb{K}$ is a finite subset, then $I$ gives rise to a generalized cartan super datum according to above Dynkin diagrams with $i \in I_{\rm odd}$ if and only if $i=\mathtt{q}(\pm q)=\pm 2,$ and $I_{\rm even}:=I\setminus I_{\rm odd}.$
We orient each single edge arbitrarily. Then the Dynkin diagram becomes a quiver,
and the generalized Cartan matrix is given by\begin{align*}
	a_{ij}=
	\begin{cases}
		-1,& \text{ if } i\rightarrow j, i \leftarrow j, i\Rightarrow j \text{ or } i \fourlinerightarrow j, \\
		-2,& \text{ if } i\Leftarrow j \text{ or } i \rightleftarrows j \text{ or } i \Leftrightarrow j, \\
		-4,& \text{ if } i \fourlineleftarrow j, \\
		2, & \text{ if } i =j, \\
		0, & \text{ otherwise.}
	\end{cases}
\end{align*}
\label{pag:map g}
Let $g:\mathbb{K}^*\rightarrow \mathbb{K}; x\mapsto x+x^{-1}$. We set
\begin{align*}J=g^{-1}(I)=\{\mathtt{b}_{\pm}(x)\in\mathbb{K}^* \mid \mathtt{q}(x)\in I\},\qquad J^\dag:=\{\mathtt{b}_{+}(x)\in\mathbb{K}^* \mid \mathtt{q}(x)\in I\}.
	\end{align*}
Then $\pr=g|_{J}: J\rightarrow I$ is the restriction map of $g$.

Now we can associate $I$ with a quiver Hecke-Clifford superalgebra as follows.
Let $u$ and $v$ be indeterminates over $\mathbb{K}$. For any $i=\mathtt{q}(x),j=\mathtt{q}(y)\in I$, we define
\begin{align*}
	Q_{i,j}(u,v)=
	\begin{cases}
		u-v, & \text{ if } i\rightarrow j, \\
		v-u, & \text{ if } i\leftarrow j, \\
		u-v^{2}, & \text{ if } i\Rightarrow j, \\
		v-u^{2}, & \text{ if } i\Leftarrow j, \\
		(u-v)(v-u), & \text{ if } i\Leftrightarrow j \text{ or } i \rightleftarrows j, \\
		u-v^{4}, & \text{ if } i\fourlinerightarrow  j, \\
		v-u^{4}, & \text{ if } i\fourlineleftarrow  j, \\
		0, & \text{ if } i= j. \\
		1, & \text{ otherwise.}
	\end{cases}
\end{align*}
As in \eqref{extend Q-polys}, for any $i,j\in J$, we can choose $\widetilde{Q}_{i,j}(u,v)$ . We use above datum to define the quiver Hecke-Clifford superalgebra, which is denoted by $RC_n(I)$.

\subsection{KKT's isomorphism}\label{KKT's isomorphism}
{\bf In this subsection, ${\rm R}=\mathbb{K}$. We fix $q^2\neq \pm 1$, $\undQ=(Q_1,\cdots,Q_m)\in(\mathbb{K}^*)^m$ and $f=f^{(\bullet)}_{\underline{Q}}$ with $\bullet\in\{\mathsf{0},\mathsf{s},\mathsf{ss}\}$.} Note that in general, $\mathcal{H}^f_{\mathbb{K}}$ is not semisimple. In this subsection, we shall connect $f$ and $\mathcal{H}^f_{\mathbb{K}}$ with certain Dynkin diagram and the corresponding cyclotomic quiver Hecke-Clifford superalgebra respectively.

\begin{defn}
Let $Q\in\mathbb{K}^*$, we set $\mathfrak{C}(Q):=\{\mathtt{q}(q^{2l}Q)\mid -n<l<n\}$.
\end{defn}
\label{pag:If}
\begin{defn}	
	For $f=f^{(\bullet)}_{\underline{Q}}$ with $\bullet\in\{\mathsf{0},\mathsf{s},\mathsf{ss}\}$, we define
$$I_f:=\begin{cases}
	\bigcup_{i=1}^{m}\mathfrak{C}(Q_i),&\qquad \text{if $\bullet=\mathsf{0}$};\\
	\bigcup_{i=0}^{m}\mathfrak{C}(Q_i),&\qquad \text{if $\bullet=\mathsf{s}$};\\
	\bigcup_{\substack{{i=0_+,0_-,1,\cdots,m}}}\mathfrak{C}(Q_i),&\qquad \text{if $\bullet=\mathsf{ss}$}.
\end{cases}
$$
\end{defn}
Then we can associate $I_f$ with a Dynkin diagram, which is a disjoint union of some subdiagrams of the Dynkin diagrams appearing in Section \ref{DynkinDiagrams}.

Recall that $$J_f=g^{-1}\left(I_f\right)=\{\mathtt{b}_{\pm}(x)\in \mathbb{K}^* \mid \mathtt{q}(x)\in I_f\}, \qquad J_f^\dag=\{\mathtt{b}_{+}(x)\in \mathbb{K}^* \mid \mathtt{q}(x)\in I_f\},$$ and we have the natural projection $J_f\overset{\pr}{\twoheadrightarrow} I_f$ which restricts to a bijection from $J_f^\dag$ to $I_f$.

Let $M$ be a finite dimensional $\mathcal{H}^f_{\mathbb{K}}$-module. Then, by \cite[Lemma 4.7]{KKT}, the eigenvalues of each $X_i$ on $M$ belong to $J_f.$ Therefore, we have $$\{e({\bf i})\mid {\bf i}\in{(\mathbb{K}^*)}^n, e({\bf i})\neq 0\}=\{e({\bf i})\mid {\bf i}\in\left(J_f\right)^n,e({\bf i})\neq 0\}.$$
\label{pag:Lambdaf}
\begin{defn}Let	$f=f^{(\bullet)}_{\underline{Q}}$ with $\bullet\in\{\mathsf{0},\mathsf{s},\mathsf{ss}\}$, we define $$\Lambda_f:=\begin{cases}
		2\sum_{\mathtt{q}(Q_i)\in I_{\rm odd}}\Lambda_{\mathtt{q}(Q_i)}+\sum_{\mathtt{q}(Q_i)\in I_{\rm even}}\Lambda_{\mathtt{q}(Q_i)},&\qquad \text{if $\bullet=\mathsf{0}$};\\
		2\sum_{\mathtt{q}(Q_i)\in I_{\rm odd}}\Lambda_{\mathtt{q}(Q_i)}+\sum_{\mathtt{q}(Q_i)\in I_{\rm even}}\Lambda_{\mathtt{q}(Q_i)}+\Lambda_{\mathtt{q}(q)},&\qquad \text{if $\bullet=\mathsf{s}$};\\
		2\sum_{\mathtt{q}(Q_i)\in I_{\rm odd}}\Lambda_{\mathtt{q}(Q_i)}+\sum_{\mathtt{q}(Q_i)\in I_{\rm even}}\Lambda_{\mathtt{q}(Q_i)}+\Lambda_{\mathtt{q}(q)}+\Lambda_{\mathtt{q}(-q)},&\qquad \text{if $\bullet=\mathsf{ss}$}.
	\end{cases}
	$$
\end{defn}
It is clear that the corrspondence $f\mapsto\Lambda_f$ is injective. Hence, we can abbreviate the cyclotomic quiver Hecke-Clifford superalgebra $RC^{\Lambda_f}_n(I_f)$ by $RC^{\Lambda_f}_n$. \label{pag:KKTiso}
\begin{thm}\cite[Corollary 4.8]{KKT}\label{KKTiso}
	We have a superalgebra isomorphism
$$RC^{\Lambda_f}_n \cong \mathcal{H}^f_{\mathbb{K}}
	$$ under which $$y_ke({\bf i})\mapsto
	f_{k,{\bf i}}(X_1,X_2,\cdots,X_n)\,(X_k-{\bf i}_k)e({\bf i}),\,c_i e({\bf i})\mapsto C_i e({\bf i})$$ and $$\sigma_a e({\bf i})\mapsto T_a e({\bf i})(r_{a,{\bf i}}(X_1,X_2,\cdots,X_n))
	+\sum_{{\bf j}\in (J_f)^n}m_{a,{\bf i}}^{\bf j}e({\bf j}),$$ where $f_{k,{\bf i}}$ and $r_{a,{\bf i}}$ are some polynomials in $X_1,\cdots,X_n$ satisfying that
	\begin{enumerate}
		\item $f_{k,{\bf i}}({\bf i}_1,\cdots,{\bf i}_n)\neq 0$ and $r_{a,{\bf i}}({\bf i}_1,\cdots,{\bf i}_n)\neq 0$ for $k, a=1,\cdots,n$ and ${\bf i}\in (J_f)^n$
		\item  $m_{a,{\bf i}}^{\bf j}\in\langle X_1,\cdots,X_n,C_1,\cdots,C_n\rangle,$ for $k=1,\cdots,n$ and ${\bf i},{\bf j}\in (J_f)^n$.
		\end{enumerate}
	\end{thm}

\subsection{Degrees of standard tableaux}\label{Degrees of standard tableaux}
{\bf In this subsection, ${\rm R}=\mathbb{K}$.  We fix $n\in\N,q^2\neq \pm 1$, $\undQ=(Q_1,\cdots,Q_m)\in(\mathbb{K}^*)^m$ and $f=f^{(\bullet)}_{\underline{Q}}$ with $\bullet\in\{\mathsf{0},\mathsf{s},\mathsf{ss}\}$.
%We call the parameter $(q,\undQ)$ is of some type $\diamondsuit$ in Section \ref{DynkinDiagrams}, if $q$ satisfies the condition associated with type $\diamondsuit$ and each $[Q_i]$ ($1\leq i\leq m$) corresponds to a vertex of the Dynkin diagram for $\diamondsuit$.
Accordingly, we define the residue of boxes in the young diagram $\undla$ via \eqref{eq:residue} as well as $\res(\mathfrak{t})\in\mathbb{K}^*$ for each $\mathfrak{t}\in\Std(\undla)$ for $\undla\in\mathscr{P}^{\bullet,m}_{n}$.} The aim of this subsection is to define the $\mathbb{Z}$-degrees of standard tableaux with respect to certain Dynkin diagram $I_f$ and investigate some properties.
\label{pag:mathcalD}
\begin{defn}
	We denote the subset of boxes
	\begin{align*}
		\mathcal{D}
		=\mathcal{D}^{(\bullet)}
		:=
		\begin{cases}
			\emptyset, \quad &\text{ if $\bullet=\mathsf{0},$}\\
			\{(i,i,0)\mid i\in\mathbb{Z}_{>0}\}, \quad &\text{ if $\bullet=\mathsf{s},$}\\
			\{(i,i,0_{*})\mid i\in\mathbb{Z}_{>0},* \in \{\pm\}\}, \quad &\text{ if $\bullet=\mathsf{ss}.$}
		\end{cases}
    \end{align*}
\end{defn}	
Recall the generalized cartan super datum $I_f$ introduced in Sections \ref{DynkinDiagrams} and \ref{KKT's isomorphism}. The following Definition is inspired by \cite[(3.3)]{BKW} and \cite[Definition 4D.3]{EM}.
%For $i\in I_f ,$ let
%\begin{align*}
%\delta_{{\rm p}(i),\bar{1}}
%:=\begin{cases}
%1,\quad & \text{ if ${\rm p}(i)=\bar{1}$,}\\
%0,\quad & \text{ if ${\rm p}(i)=\bar{0}$.}
%\end{cases}
%\end{align*}
\label{pag:add and rem 1}
\begin{defn}
Let $\undla\in\mathscr{P}^{\bullet,m}_{n}$ and $i\in I_f.$
\begin{enumerate}
%\item The $\undla$-positive root is $\alpha_{\undla}:=\sum_{k=1}^{n}\alpha_{{\bf i}_k^{\mt}}=\sum_{A\in \undla}\alpha_{[\res (A)]}\in Q_n^+.$
%\item For any $\Lambda\in P^+,$ $\alpha\in Q_n^+,$ the $\Lambda$-defect of $\alpha$ is
%      ${\rm def}(\alpha):=(\Lambda|\alpha)-(\alpha|\alpha)/2.$
%\item The $\Lambda_f$-defect of $\undla$ is ${\rm def}(\undla):={\rm def}(\nu_{\undla}).$

%where
%\begin{align*}
%w_k^{\mt}
%:=\begin{cases}
%1, &\quad \text{if ${\bf i}_k^{\mt}={\bf i}_{k+1}^{\mt}\in (I_f)_{\rm odd}$ and $k+1\in\mathcal{D}_{\mt}$,}\\
%2^{|i_k^{\mt}|}, &\quad \text{otherwise.}
%\end{cases}
%\end{align*}
\item We define
\begin{align*}\mathcal{A}_{\undla}(i)&:=\{\text{addable ${ i}$-boxes of $\undla$}\}\\
	\mathcal{R}_{\undla}({i})&:=\{\text{removable ${ i}$-boxes of $\undla$}\}.
\end{align*}
\item We define $$d_{ i}(\undla):=2^{\delta_{{\rm p}(i),\bar{1}}} {\rm d}_{{ i}}\left(\sharp \mathcal{A}_{\undla}(i)
-\sharp \left(\mathcal{R}_{\undla}(i)\setminus \mathcal{D} \right)\right).$$

	\item The $\undla$-positive root is $\nu_{\undla}:=\sum_{A\in \undla}\nu_{\mathtt{q}(\res (A))}\in Q_n^+.$

\end{enumerate}
\end{defn}

The following Lemma connects the Cartan matrix with the combinatorics in our setting, which will be used frequently in this subsection.
\begin{lem}\label{crucial lemma}
Let $\undla\in\mathscr{P}^{\bullet,m}_{n}, \undmu\in\mathscr{P}^{\bullet,m}_{n-1}$ and $\undla=\undmu\cup\{A\}$. Suppose the neighbors of $A$ in the corresponding young diagram are the following:
$$\young(:x,wAy,:z).$$ For $i\in I_f$, we set $\mathcal{E}_1:=\mathcal{A}_{\undla}(i)\cap \{y,z\},$ $\mathcal{E}_2:
=\left(\mathcal{R}_{\undmu}(i)\setminus \mathcal{D}\right)\cap\{x,w\}.$ Then we have
\begin{align}\label{tiny claim}-a_{i,\mathtt{q}(\res(A))}=2^{\delta_{{\rm p}(i),\bar{1}}}\left(\sharp \mathcal{E}_1 + \sharp \mathcal{E}_2 - \delta_{\mathtt{q}(\res(A)),i}\left(1+\delta_{A\notin \mathcal{D}}\right)\right),
\end{align} where $$\delta_{A\notin \mathcal{D}}:=\begin{cases}
1,\qquad \text{if $A\notin \mathcal{D}$};\\
0,\qquad \text{if $A\in \mathcal{D}$}.\\
\end{cases}
$$
\end{lem}
\begin{proof}
We prove \eqref{tiny claim} by checking all of the possible cases of $\mathtt{q}(\res(A))$ and $i$ appearing in the Dynkin diagrams.
\begin{enumerate}
	\item $i\rightarrow \mathtt{q}(\res(A)), i\leftarrow \mathtt{q}(\res(A)), i\Rightarrow \mathtt{q}(\res(A))\text{ or }
i \fourlinerightarrow \mathtt{q}(\res(A)).$
Then ${\rm p}(i)=\bar{0}$ and it's easy to check that
	\begin{align*}
		&x\in \mathcal{E}_2 \text{ if and only if } y\notin \mathcal{E}_1,\quad w\in \mathcal{E}_2 \text{ if and only if } z\notin \mathcal{E}_1,\\
		&\mathtt{q}(\res(x))=\mathtt{q}(\res(y))=i \text{ if and only if } \mathtt{q}(\res(z))=\mathtt{q}(\res(w))\neq i .
	\end{align*} Therefore, in this case, $\sharp \mathcal{E}_1 + \sharp \mathcal{E}_2 =1$ and \eqref{tiny claim} holds.
	\item $i\Leftarrow \mathtt{q}(\res(A))$ or $i \Leftrightarrow \mathtt{q}(\res(A))$.
If ${\rm p}(i)=\bar{0}$, then it's easy to check that
	\begin{align*}
		&x\in \mathcal{E}_2 \text{ if and only if } y\notin \mathcal{E}_1,\quad w\in \mathcal{E}_2 \text{ if and only if } z\notin \mathcal{E}_1,\\
		&\mathtt{q}(\res(x))=\mathtt{q}(\res(y))=\mathtt{q}(\res(z))=\mathtt{q}(\res(w))=i .
	\end{align*} Therefore, we have $\sharp \mathcal{E}_1 + \sharp \mathcal{E}_2 =2$ and \eqref{tiny claim} holds in this case. If ${\rm p}(i)=\bar{1}$, we can similarly check that
	\begin{align*}
		&x\in \mathcal{E}_2 \text{ if and only if } y\notin \mathcal{E}_1,\quad w\in \mathcal{E}_2 \text{ if and only if } z\notin \mathcal{E}_1,\\
		&\mathtt{q}(\res(x))=\mathtt{q}(\res(y))=i \text{ if and only if } \mathtt{q}(\res(z))=\mathtt{q}(\res(w))\neq i .
	\end{align*} Therefore, in this case, $\sharp \mathcal{E}_1 + \sharp \mathcal{E}_2 =1$ and \eqref{tiny claim} holds.
    \item $i \fourlineleftarrow \mathtt{q}(\res(A))$. Then ${\rm p}(i)=\bar{1}$ and it's easy to check that
    \begin{align*}
		&x\in \mathcal{E}_2 \text{ if and only if } y\notin \mathcal{E}_1,\quad w\in \mathcal{E}_2 \text{ if and only if } z\notin \mathcal{E}_1,\\
		&\mathtt{q}(\res(x))=\mathtt{q}(\res(y))=\mathtt{q}(\res(z))=\mathtt{q}(\res(w))= i.
	\end{align*} Therefore, in this case, $\sharp \mathcal{E}_1 + \sharp \mathcal{E}_2 =2$ and \eqref{tiny claim} holds.
	\item $i=\mathtt{q}(\res(A))$. If ${\rm p}(i)=\bar{0}$, then it's easy to check that $\mathcal{E}_1 = \mathcal{E}_2 =\emptyset$ and $A\notin \mathcal{D}$. Therefore,  \eqref{tiny claim} holds in this case. If ${\rm p}(i)=\bar{1}$ and $A\notin \mathcal{D}$, then it's easy to check that
	\begin{align*}
		&x\in \mathcal{E}_2 \text{ if and only if } y\notin \mathcal{E}_1,\quad w\in \mathcal{E}_2 \text{ if and only if } z\notin \mathcal{E}_1,\\
		&\mathtt{q}(\res(x))=\mathtt{q}(\res(y))=i \text{ if and only if } \mathtt{q}(\res(z))=\mathtt{q}(\res(w))\neq i.
	\end{align*} Therefore, in this case, we have $\sharp \mathcal{E}_1 + \sharp \mathcal{E}_2 =1$ and \eqref{tiny claim} holds. If ${\rm p}(i)=\bar{1}$ and $A\in \mathcal{D}$, we can similarly check that in this case, $\mathcal{E}_1 = \mathcal{E}_2 =\emptyset$ and \eqref{tiny claim} holds again.
	
	\item $i\not\leftrightarrow \mathtt{q}(\res(A))$. One can easily check that $\mathcal{E}_1 = \mathcal{E}_2 =\emptyset$ and therefore \eqref{tiny claim} holds in this case.

\end{enumerate}
Combining above cases, \eqref{tiny claim} holds.
\end{proof}

Recall that we have associated $f$ with the dominant weight $\Lambda_f$.

\begin{cor}\label{d-value}
Let $\undla\in\mathscr{P}^{\bullet,m}_{n}$ and $i\in I_f$, we have $$d_{i}(\undla)=\begin{cases}
		(\Lambda_f-\nu_{\undla}|\nu_{i}),&\qquad \text{if $\bullet=\mathsf{0}$};\\
		(\Lambda_f-\nu_{\undla}|\nu_{i})+\delta_{i,\mathtt{q}(q)}{\rm d}_{\mathtt{q}(q)},&\qquad \text{if $\bullet=\mathsf{s}$};\\
			(\Lambda_f-\nu_{\undla}|\nu_{i})+\delta_{i,\mathtt{q}(q)}{\rm d}_{\mathtt{q}(q)}+\delta_{i,\mathtt{q}(-q)}{\rm d}_{\mathtt{q}(-q)},&\qquad \text{if $\bullet=\mathsf{ss}$}.
		\end{cases}
	$$
	\end{cor}
	
	\begin{proof}
		We prove the equation by induction on $n$. It's easy to check the case when $n=0$, i.e. $\undla=\emptyset$  by definition. Now suppose $\undla\in\mathscr{P}^{\bullet,m}_{n} $ and $\undla=\undmu\cup\{A\}$, where $\undmu\in\mathscr{P}^{\bullet,m}_{n-1}$. We draw the neighbors of $A$ in the young diagram of $\undla$ as the following:
		$$\young(:x,wAy,:z),$$ and set $\mathcal{E}_1:=\mathcal{A}_{\undla}(i)\cap \{y,z\},$ $\mathcal{E}_2:
		=\left(\mathcal{R}_{\undmu}(i)\setminus \mathcal{D}\right)\cap\{x,w\}.$ Then one can easily check $$
		\mathcal{A}_{\undla}(i)=	\left(\mathcal{A}_{\undmu}(i)\sqcup \mathcal{E}_1\right)\setminus \{A\},\qquad 	\mathcal{R}_{\undmu}(i)\setminus\mathcal{D}=	\left(\left(\mathcal{R}_{\undla}(i)\setminus\mathcal{D}\right)\setminus \{A\}\right)\sqcup \mathcal{E}_2.
		$$
		Hence, we have \begin{align}\label{induction1}
			\sharp \mathcal{A}_{\undla}(i)=\sharp \mathcal{A}_{\undmu}(i)+\sharp \mathcal{E}_1-\delta_{\mathtt{q}(\res(A)),i},
			\quad \sharp \left(\mathcal{R}_{\undmu}(i)\setminus\mathcal{D}\right)=\sharp \left(\mathcal{R}_{\undla}(i)\setminus\mathcal{D}\right) +\sharp  \mathcal{E}_2 -\delta_{\mathtt{q}(\res(A)),i}\delta_{A\notin \mathcal{D}}.\end{align}
		We deduce that \begin{align*}
			d_{i}(\undla)-d_{i}(\undmu)&=2^{\delta_{{\rm p}(i),\bar{1}}} {\rm d}_{i}\left(\sharp \mathcal{A}_{\undla}(i)
			-\sharp \left(\mathcal{R}_{\undla}(i)\setminus \mathcal{D} \right)-\sharp \mathcal{A}_{\undmu}(i)
			+\sharp \left(\mathcal{R}_{\undmu}(i)\setminus \mathcal{D} \right)\right)\\
			&=2^{\delta_{{\rm p}(i),\bar{1}}} {\rm d}_{i}\left(\sharp \mathcal{E}_1+\sharp \mathcal{E}_2-\delta_{\mathtt{q}(\res(A)),i}(1+\delta_{A\notin \mathcal{D}})
			\right)\\
			&=-{\rm d}_{i}a_{i,\mathtt{q}(\res(A))}\\
			&=-\left(\nu_{i}\middle|\nu_{\mathtt{q}(\res(A))}\right)
			\end{align*} where in the second equation, we have used \eqref{induction1}, and in the third equation, we have used Lemma \ref{crucial lemma}. Since $\nu_{\undla}=\nu_{\undmu}+\nu_{\mathtt{q}(\res(A))}$, the Corollary follows from induction hypothesis.
		
	\end{proof}	
	
	The following definition is inspired by \cite[(3.4)]{BKW} and \cite[Definition 4D.3]{EM}.
	\label{pag:modified defect}
	\begin{defn}
		 Let $\undla\in\mathscr{P}^{\bullet,m}_{n}$ and $\nu_{\undla}=\sum_{{i}\in I_f}m_{ i} \nu_{ i}$. We define $$d^f(\undla):=\begin{cases}\left(\Lambda_f\middle|\nu_{\undla}\right)-\frac{1}{2}\left(\nu_{\undla}\middle|\nu_{\undla}\right)-m_{\mathtt{q}(q)}{\rm d}_{\mathtt{q}(q)}-m_{\mathtt{q}(-q)}{\rm d}_{\mathtt{q}(-q)},&\qquad\text{if $\bullet=\mathsf{0}$};\\
			\left(\Lambda_f\middle|\nu_{\undla}\right)-\frac{1}{2}\left(\nu_{\undla}\middle|\nu_{\undla}\right)-m_{\mathtt{q}(-q)} {\rm d}_{\mathtt{q}(-q)},&\qquad\text{if $\bullet=\mathsf{s}$};\\
		\left(\Lambda_f\middle|\nu_{\undla}\right)-\frac{1}{2}\left(\nu_{\undla}\middle|\nu_{\undla}\right),&\qquad\text{if $\bullet=\mathsf{ss}$}.\\
			\end{cases}$$
		\end{defn}

		\begin{lem}\label{def-induction}
			Let $\undla\in\mathscr{P}^{\bullet,m}_{n}, \undmu\in\mathscr{P}^{\bullet,m}_{n-1}$ and $\undla=\undmu\cup\{A\}$. Then we have $$d^f(\undla)=d^f(\undmu)+d_{\mathtt{q}(\res(A))}(\undmu)-2^{\delta_{{\rm p}(\mathtt{q}(\res(A))),\bar{1}}}{\rm d}_{\mathtt{q}(\res(A))}.$$
			\end{lem}
			
			\begin{proof}
				By definition, we have \begin{align*}
					d^f(\undla)-d^f(\undmu)
					&=
\begin{cases}\left({\Lambda_f}\middle|\nu_{\undla}-\nu_{\undmu}\right)-\frac{1}{2}\left(\nu_{\undla}\middle|\nu_{\undla}\right)+\frac{1}{2}\left(\nu_{\undmu}\middle|\nu_{\undmu}\right)-\delta_{{\rm p}(\mathtt{q}(\res(A))),\overline{1}}{\rm d}_{\mathtt{q}(\res(A))}, &\qquad\text{if $\bullet=\mathsf{0}$}; \\
						\left({\Lambda_f}\middle|\nu_{\undla}-\nu_{\undmu}\right)-\frac{1}{2}\left(\nu_{\undla}\middle|\nu_{\undla}\right)+\frac{1}{2}\left(\nu_{\undmu}\middle|\nu_{\undmu}\right)-\delta_{\mathtt{q}(\res(A)),\mathtt{q}(-q)}{\rm d}_{\mathtt{q}(\res(A))}, &\qquad\text{if $\bullet=\mathsf{s}$};\\
						\left({\Lambda_f}\middle|\nu_{\undla}-\nu_{\undmu}\right)-\frac{1}{2}\left(\nu_{\undla}\middle|\nu_{\undla}\right)+\frac{1}{2}\left(\nu_{\undmu}\middle|\nu_{\undmu}\right), &\qquad\text{if $\bullet=\mathsf{ss}$},
						\end{cases}\\
					&=
				\begin{cases}\left(\Lambda_f\middle|\nu_{\mathtt{q}(\res(A))}\right)-\left(\nu_{\undmu}\middle|\nu_{\mathtt{q}(\res(A))}\right)-(1+\delta_{{\rm p}(\mathtt{q}(\res(A))),\overline{1}}){\rm d}_{\mathtt{q}(\res(A))}, &\qquad\text{if $\bullet=\mathsf{0}$};\\
					\left(\Lambda_f\middle|\nu_{\mathtt{q}(\res(A))}\right)-\left(\nu_{\undmu}\middle|\nu_{\mathtt{q}(\res(A))}\right)-(1+\delta_{\mathtt{q}(\res(A)),\mathtt{q}(-q)}){\rm d}_{\mathtt{q}(\res(A))}, &\qquad\text{if $\bullet=\mathsf{s}$};\\
						\left(\Lambda_f\middle|\nu_{\mathtt{q}(\res(A))}\right)-\left(\nu_{\undmu}\middle|\nu_{\mathtt{q}(\res(A))}\right)-{\rm d}_{\mathtt{q}(\res(A))}, &\qquad\text{if $\bullet=\mathsf{ss}$},
						\end{cases}\\
					&=d_{\mathtt{q}(\res(A))}(\undmu)-2^{\delta_{{\rm p}(\mathtt{q}(\res(A))),\bar{1}}}{\rm d}_{\mathtt{q}(\res(A))},
					\end{align*} where in the last equation, we have used Corollary \ref{d-value}. This completes the proof.
				\end{proof}
				
	Now we are ready to define the degree of standard tableaux.
	\label{add and rem 2}
	\begin{defn}
		(1) \cite[Before Remark 3B.1]{EM}
		For any two boxes $x=(i,j,l)$ and $y=(a,b,c),$ we write $y<x$ if and only if
		$c<l; \text{ or } c=l \text{ and } a<i; \text{ or } c=l, a=i \text{ and } b<j.$
		
		(2) For $\undla\in\mathscr{P}^{\bullet,m}_{n}$, $\mt\in \Std(\undla),$ $k\in[n],$
		we define
		\begin{align*}
			\mathscr{A}_{\mt}^{\rhd}(k)&:=\{x\in{\rm Add}(\mt\downarrow_{k-1})\mid x>\mt^{-1}(k)\},\\
			\mathscr{R}_{\mt}^{\rhd}(k)&:=\{y\in{\rm Rem}(\mt\downarrow_{k-1})\mid y>\mt^{-1}(k)\},\\
			\mathscr{A}_{\mt}^{\lhd}(k)&:=\{x\in{\rm Add}(\mt\downarrow_{k-1})\mid x<\mt^{-1}(k)\},\\
			\mathscr{R}_{\mt}^{\lhd}(k)&:=\{y\in{\rm Rem}(\mt\downarrow_{k-1})\mid y<\mt^{-1}(k)\}.
		\end{align*}
	\end{defn}	
	
	The following definition is inspired by \cite[(3.5), (3.6)]{BKW} and \cite[Definition 4D.3]{EM}.
	\label{pag:deg of std tableaux}
	\begin{defn}\label{deg of std tableaux}
		Let $\undla\in\mathscr{P}^{\bullet,m}_{n}$, $\mt\in \Std(\undla)$ and $\mathtt{q}(\res(\mt))={\bf i}\in (I_f)^n.$
		\begin{enumerate}
			%\item The $\undla$-positive root is $\alpha_{\undla}:=\sum_{k=1}^{n}\alpha_{{\bf i}_k^{\mt}}=\sum_{A\in \undla}\alpha_{[\res (A)]}\in Q_n^+.$
			%\item For any $\Lambda\in P^+,$ $\alpha\in Q_n^+,$ the $\Lambda$-defect of $\alpha$ is
			%      ${\rm def}(\alpha):=(\Lambda|\alpha)-(\alpha|\alpha)/2.$
			%\item The $\Lambda_f$-defect of $\undla$ is ${\rm def}(\undla):={\rm def}(\alpha_{\undla}).$
			\item For ${\scriptstyle\triangle}\in\{\lhd,\rhd\},$ $k\in[n],$ we denote
			\begin{align*}
				\mathcal{A}_{\mt}^{{\scriptstyle\triangle,f}}(k)&
				:=\{A\in\mathscr{A}_{\rm \mt}^{{\scriptstyle\triangle}}(k)\mid \mathtt{q}(\res(A))=\mathtt{q}(\res_{\mt}(k))\in I_f\},\\
				\mathcal{R}_{\mt}^{{\scriptstyle\triangle,f}}(k)&
				:=\{A\in\mathscr{R}_{\rm \mt}^{{\scriptstyle\triangle}}(k)\mid \mathtt{q}(\res(A))=\mathtt{q}(\res_{\mt}(k))\in I_f\}.
			\end{align*}
			\item For ${\scriptstyle\triangle}\in\{\lhd,\rhd\},$ the ${\scriptstyle\triangle}$-degree of $\mt$ is defined by
			\begin{align*}
				\deg^{{\scriptstyle\triangle},f}({\mt})
				:=\sum_{k=1}^n 2^{\delta_{{\rm p}({\bf i}_k),\bar{1}}} {\rm d}_{{\bf i}_k}\left(\sharp \mathcal{A}_{\mt}^{{\scriptstyle\triangle,f}}(k)
				-\sharp \left(\mathcal{R}_{\mt}^{{\scriptstyle\triangle,f}}(k)\setminus \mathcal{D} \right)\right).
			\end{align*}
			%where
			%\begin{align*}
			%w_k^{\mt}
			%:=\begin{cases}
				%1, &\quad \text{if ${\bf i}_k^{\mt}={\bf i}_{k+1}^{\mt}\in (I_f)_{\rm odd}$ and $k+1\in\mathcal{D}_{\mt}$,}\\
				%2^{\overline{{\bf i}_k^{\mt}}, &\quad \text{otherwise.}
				%\end{cases}
				%\end{align*}
			
			\end{enumerate}
		\end{defn}
		For simplicity, we shall omit the superscript $f$ in all bove definition when $f$ is clear in the context.
		
\begin{cor}\label{def1}Let $\undla\in\mathscr{P}^{\bullet,m}_{n}$, $\mt\in \Std(\undla),$ then
$d(\undla)=\deg^{\lhd}(\mt)+\deg^{\rhd}(\mt)$.
					\end{cor}
					
					\begin{proof}
		We do induction on $n$. When $n=0$, this is trivial. Now suppose $\mt\downarrow_{n-1}\in  \Std(\undmu)$ for some $\undmu\in\mathscr{P}^{\bullet,m}_{n-1}$ and $A=\mt^{-1}(n)$. We have \begin{align*}
			\deg^{\lhd}(\mt)+\deg^{\rhd}(\mt)&=\deg^{\lhd}(\mt\downarrow_{n-1})+\deg^{\rhd}(\mt\downarrow_{n-1})+d_{\mathtt{q}(\res(A))}(\undmu)-2^{\delta_{{\rm p}(\mathtt{q}(\res(A))),\bar{1}}}{\rm d}_{\mathtt{q}(\res(A))}\\	
			&=d(\undmu)+d_{\mathtt{q}(\res(A))}(\undmu)-2^{\delta_{{\rm p}(\mathtt{q}(\res(A))),\bar{1}}}{\rm d}_{\mathtt{q}(\res(A))}\\
			&=d(\undla),
		\end{align*} where in the second equation, we have used induction hypothesis and in the last equation, we have used Lemma \ref{def-induction}.				
						\end{proof}
						
The following Proposition can be viewed as a generalization of \cite[Proposition 3.13]{BKW} and \cite[Theorem 4C.3 and Section 4D]{EM}.
\begin{prop}\label{BKW deg}
Let $\undla\in\mathscr{P}^{\bullet,m}_{n}$, $\mt\in \Std(\undla)$ and $\ms=s_k \mt\in\Std(\undla)$ with $\ms=s_k \mt\lhd\mt$
for some $k\in[n-1].$ Suppose $\mathtt{q}(\res(\mt))={\bf i}\in (I_f)^n,$ then we have
\begin{align*}
\deg^{\rhd}({\ms})-\deg^{\rhd}({\mt})
=-{\rm d}_{{\bf i}_{k}}a_{{\bf i}_{k},{\bf i}_{k+1}}
=\deg^{\lhd}({\mt})-\deg^{\lhd}({\ms}).
\end{align*}
\end{prop}
\begin{proof}
%We only prove the first equality, the second one follows similarly.
%We remark that if $(q,\undQ)$ is of type $A_\infty$ or $A_{\ell-1}^{(1)},$
%then the statement has been proved in \cite[Proposition 3.13]{BKW};
%if $(q,\undQ)$ is of type $C_\infty$ or $C^{(1)}_{s},$
%then it has been proved in \cite[Theorem 4C.3 and Section 4D]{EM} essentially.
%So we only need to prove it for types $B_\infty,$ $A_{2s}^{(2)}$ and $D_{s}^{(2)}.$

We may assume that $k=n-1.$ By assumption, $B:=\mt^{-1}(n-1)$ is above $A:=\mt^{-1}(n)=(i,j,l),$ then
\begin{align*}
&\deg^{\rhd}({\ms})-\deg^{\rhd}({\mt}) \\
%&={\rm d}_{{\bf i}_{n-1}^{\mt}} \left(\sharp \mathscr{A}_{\ms}^{\rhd}(n)- \sharp \mathscr{R}_{\ms}^{\rhd}(n)
%  -\sharp \mathscr{A}_{\mt}^{\rhd}(n-1)+ \sharp \mathscr{R}_{\mt}^{\rhd}(n-1) \right) \\
&\quad=2^{\delta_{{\rm p}({\bf i}_{n-1}),\bar{1}}}{\rm d}_{{\bf i}_{n-1}}\left(\sharp \mathcal{A}_{\ms}^{\rhd}(n)- \sharp \left( \mathcal{R}_{\ms}^{\rhd}(n)\setminus \mathcal{D}\right)
  -\sharp \mathcal{A}_{\mt}^{\rhd}(n-1)+ \sharp \left(\mathcal{R}_{\mt}^{\rhd}(n-1)\setminus \mathcal{D}\right) \right).
\end{align*}
We draw the neighbors of $A$ in the young diagram of $\undla$ as the following:
$$\young(:x,wAy,:z).$$
 Suppose $\ms\downarrow_{n-1}\in \Std(\undmu')$ for some $\undmu'\in \mathscr{P}^{\bullet,m}_{n-1}$ and $\ms\downarrow_{n-2}=\mt\downarrow_{n-2}\in \Std(\undmu)$ for some $\undmu\in \mathscr{P}^{\bullet,m}_{n-2}$ . We set $\mathcal{E}_1:=\mathcal{A}_{\undmu'}(\mathtt{q}(\res(B)))\cap \{y,z\},$ $\widetilde{\mathcal{E}}_2:=\mathcal{R}_{\undmu}(\mathtt{q}(\res(B)))\cap\{x,w\}.$
Then we have
\begin{align*}
\mathcal{A}_{\ms}^{\rhd}(n)=\left(\mathcal{A}_{\mt}^{\rhd}(n-1)\setminus \{A\}\right)\sqcup \mathcal{E}_1,
\quad \mathcal{R}_{\mt}^{\rhd}(n-1)=\left(\mathcal{R}_{\ms}^{\rhd}(n)\setminus \{A\}\right)\sqcup \widetilde{\mathcal{E}}_2.
\end{align*}
We further denote $\mathcal{E}_2:=\widetilde{\mathcal{E}}_2\setminus \mathcal{D}
=\left(\mathcal{R}_{\undmu}(\mathtt{q}(\res(B)))\setminus \mathcal{D}\right)\cap\{x,w\}.$ It follows that
\begin{align*}
\mathcal{R}_{\mt}^{\rhd}(n-1)\setminus \mathcal{D}
=\left(\left(\mathcal{R}_{\mt}^{\rhd}(n-1)\setminus \mathcal{D}\right)\setminus \{A\}\right)\sqcup \mathcal{E}_2.
\end{align*}
Hence
\begin{align*}
\deg^{\rhd}({\ms})-\deg^{\rhd}({\mt})
&=2^{\delta_{{\rm p}(\mathtt{q}(\res(B))),\bar{1}}}{\rm d}_{\mathtt{q}(\res(B))}\left(\sharp \mathcal{E}_1 + \sharp \mathcal{E}_2 - \delta_{\mathtt{q}(\res(A)),\mathtt{q}(\res(B))}\left(1+\delta(A\notin \mathcal{D})\right)\right)\\
&={\rm d}_{\mathtt{q}(\res(B))}\cdot (-a_{\mathtt{q}(\res(B)),\mathtt{q}(\res(A))}),
\end{align*} where in the second equation, we have used Lemma \ref{crucial lemma}.
%It's clear that
%$$y\in \mathcal{E}_1 \Leftrightarrow x\notin \mathcal{E}_2, \qquad z\in \mathcal{E}_1 \Leftrightarrow w\notin \mathcal{E}_2.$$
%Hence, to prove the Proposition, we only need to show that
%\begin{align}\label{tiny claim}-a_{\mathtt{q}(\res(B)),\mathtt{q}(\res(A))}=2^{|\mathtt{q}(\res(B))|}\left(\sharp \mathcal{E}_1 + \sharp \mathcal{E}_2 - \delta_{\mathtt{q}(\res(A)),\mathtt{q}(\res(B))}\left(1+\delta(A\notin \mathcal{D})\right)\right).
%	\end{align}
	This completes the proof of first equation. The proof for the second equation is similar, hence we omit it.
\end{proof}

\begin{cor}\label{cor of BKW deg}
Let $\undla\in\mathscr{P}^{\bullet,m}_{n}$ and $\mt\in \Std(\undla).$ Suppose $\mathtt{q}(\res(\mt))=({\bf i}_1^{\mt},\ldots,{\bf i}_n^{\mt})\in (I_f)^n,$
and both $d(\mt,\mt^{\undla})=s_{k_1}s_{k_2}\cdots s_{k_p},$ $d(\mt,\mt_{\undla})=s_{r_1}s_{r_2}\cdots s_{r_s}$
are reduced expressions in $\mathfrak{S}_n$. Then for any ${\bf i}=(i_1,\ldots,i_n)\in (J_f)^n,$ where $i_k\in \pr^{-1}({\bf i}_k^{\mt})\in J_f$
for $k\in[n],$ we have
\begin{align*}
\deg^{\rhd}(\mt)&=\deg(\sigma_{k_1}\cdots \sigma_{k_p} e({\bf i})) + \deg^{\rhd}(\mt^{\undla}),\\
\deg^{\lhd}(\mt)&=\deg(\sigma_{r_1}\cdots \sigma_{r_s} e({\bf i})) + \deg^{\lhd}(\mt_{\undla}).
\end{align*}

\end{cor}
\begin{proof}
This follows from Proposition \ref{BKW deg} directly.
\end{proof}

\section{Idempotents and seminormal forms}\label{Idempotents and seminormal forms}
{\bf Throughout this section, we fix $n\in \mathbb{N}$.}
	\subsection{Separate Condition}
		Recall $[n]:=\{1,2,\ldots,n\}$. In this subsection, we recall the separate condition \cite[Definition 3.9]{SW} on the choice of the parameters $\underline{Q}$ and $f=f^{(\bullet)}_{\underline{Q}}$ with $\bullet\in\{\mathsf{0},\mathsf{s},\mathsf{ss}\}$, where $r=\deg(f)$.
	\begin{defn}\label{defn:separate}\cite[Definition 3.9]{SW}
		Let $\bullet\in\{\mathsf{0},\mathsf{s},\mathsf{ss}\}$ and $\undQ=(Q_1,\ldots,Q_m)\in(\mathbb{K}^*)^m$.  Assume $\undla\in\mathscr{P}^{\bullet,m}_{n}$. Then $(q,\undQ)$ is said to be {\em separate} with respect to $\undla$ if for any $\mathfrak{t}\in \undla$, the $\mathtt{q}$-sequence for $\mathfrak{t}$ defined via \eqref{resNon-dege-3} satisfies the following condition:
		$$
		\mathtt{q}(\res_{\mathfrak{t}}(k))\neq\mathtt{q}(\res_{\mathfrak{t}}(k+1)) \text{ for any } k=1,\ldots,n-1.
		$$
	\end{defn}
	
%	The separate condition holds for any $\underline{\mu} \in \mathscr{P}^{\bullet,m}_{n+1}$ can be reformulated via the condition $P^{(\bullet)}_{n}(q^2,\undQ)\neq 0$ (\cite[Proposition 3.11]{SW}), where $P^{(\bullet)}_{n}(q^2,\undQ)$ ($\bullet\in\{\mathsf{0},\mathsf{s},\mathsf{ss}\}$) \label{pag:nondege Pioncare poly} is an explicit polynomial on $(q,\undQ)$.
	
	Recall that $\undQ=(Q_1,\ldots,Q_m)\in (\mathbb{K}^*)^n$ and $q \in \mathbb{K}^*$ with $q^4 \neq 1$. Then for any $n\in \N$, we define $P^{(\bullet)}_{n}(q^2,\undQ)$ as follows\footnote{We remark that since we have modified the definition of $\mathtt{q}$, the corresponding polynomial $P_n^{(\bullet)}(q^2,\undQ)$ should also be modified. To be precise, we need to change each $Q_i$ by $qQ_i$ in \cite{SW}.} \label{pag:Poincare Poly}
%so $P^{(\bullet)}_{n}(q^2,\undQ)$ here is a translation of $P^{(\bullet)}_{n}(q^2,\undQ)$ in \cite{SW}
		$$\begin{aligned}
			P_n^{(\bullet)}(q^2,\undQ):=\begin{cases}
				\prod\limits_{t=1}^{n}\bigl(q^{2t}-1\bigr)\prod\limits_{i=1}^m\biggl(\prod\limits_{t=2-n}^{n-2}\bigl(Q^2_i-q^{-2t}\bigr)\prod\limits_{t=1-n}^{n}\bigl(Q^2_i-q^{-4t+2}\bigr)\biggr)&\\
				\cdot \prod\limits_{1\leq i<i'\leq m}\biggl(\prod\limits_{t=1-n}^{n-1}\bigl({Q_i}-Q_{i'}q^{-2t}\bigr)
				\bigl(Q_iQ_{i'}-q^{-2t}\bigr)\biggr), \quad \mbox{if $\bullet=\mathsf{0}$ };  \\
				&\\
				\prod\limits_{t=1}^{n}\biggl(\bigl(q^{2t}-1\bigr)\bigl(q^{2t}+1\bigr)\biggr)\prod\limits_{i=1}^m\biggl(\prod\limits_{t=2-n}^{n-2}\bigl(Q^2_i-q^{-2t}\bigr)\prod\limits_{t=1-n}^{n}\bigl(Q^2_i-q^{-4t+2}\bigr)\biggr)&\\
				\cdot \prod\limits_{1\leq i<i'\leq m}\biggl(\prod\limits_{t=1-n}^{n-1}\bigl({Q_i}-Q_{i'}q^{-2t}\bigr)
				\bigl(Q_iQ_{i'}-q^{-2t}\bigr)\biggr),  \quad \mbox{if $\bullet=\mathsf{s}$ or  $\mathsf{ss},$}\\
			\end{cases}
		\end{aligned}$$ where for  $n=1,$  the product $\prod\limits_{t=2-n}^{n-2}\bigl(Q^2_i-q^{-2t}\bigr)$ is understood to be $1$.
		
	\begin{prop}\label{separate formula}\cite[Proposition 3.11]{SW}
		Let $n\geq 1,\,m\geq 0$,  $\undQ=(Q_1,\ldots,Q_m)\in(\mathbb{K}^*)^m$ and  $\bullet\in\{\mathsf{0},\mathsf{s},\mathsf{ss}\}$. Then $(q,\undQ)$ is separate with respect to $\underline{\mu}$ for any  $\underline{\mu}\in\mathscr{P}^{\bullet,m}_{n+1}$ if and only if $P_n^{(\bullet)}(q^2,\undQ)\neq 0$.
	\end{prop}

	\begin{lem}\cite[Lemma 2.7]{LS2}\label{important condition1}
		Let $\undQ=(Q_1,\ldots,Q_m)\in(\mathbb{K}^*)^m$ and $\bullet\in\{\mathsf{0},\mathsf{s},\mathsf{ss}\}$. Suppose $P_n^{(\bullet)}(q^2,\undQ)\neq 0$ in $\mathbb{K}$.
		Then
		\begin{enumerate}
			\item
			For any $\undla\in\mathscr{P}^{\bullet,m}_{n},$ $\mathfrak{t}\in\Std(\undla)$, we have  $\mathtt{b}_{\pm}(\res_{\mathfrak{t}}(k))\notin \{\pm 1\}$ for $k\notin \mathcal{D}_{\mathfrak{t}}$;
			\item For any $\undla,\underline{\mu} \in\mathscr{P}^{\bullet,m}_{n},$ $\mt\in\Std(\undla), \mt' \in \Std(\underline{\mu}),$
			if $\mt\neq \mt',$ then we have $\mathtt{q}(\res(\mt))\neq \mathtt{q}(\res(\mt'))$;
			\item For any $\undla\in\mathscr{P}^{\bullet,m}_{n},$ $\mathfrak{t}\in\Std(\undla)$ and $k\in [n-1]$, the four pairs $(\mathtt{b}_{\pm}(\res_{\mathfrak{t}}(k)),\mathtt{b}_{\pm}(\res_{\mathfrak{t}}(k+1)))$ do not satisfy \eqref{invertible} if $k,k+1$ are not in the adjacent diagonals of $\mathfrak{t}$.
		\end{enumerate}
	\end{lem}
	
	Suppose that the condition $P^{(\bullet)}_{n}(q^2,\undQ)\neq 0$, $\bullet\in\{\mathsf{0},\mathsf{s},\mathsf{ss}\}$ holds in $\mathbb{K}.$ Then for each $\undla\in\mathscr{P}^{\bullet,m}_{n},$ we can associate $\undla$ with a explicit simple $\mathcal{H}^f_{\mathbb{K}}$-module $\mathbb{D}(\undla),$ see \cite[Theorem 4.5]{SW} for details. Then we have the following. \label{pag:nondege simple module}
	\begin{thm}\label{semisimple:non-dege}\cite[Theorem 4.10]{SW}
		%Suppose Condition \ref{important condition1} holds.
		Let $\undQ=(Q_1,Q_2,\ldots,Q_m)\in(\mathbb{K}^*)^m$.  Assume $f=f^{(\bullet)}_{\undQ}$ and   $P^{(\bullet)}_{n}(q^2,\undQ)\neq 0$, with $\bullet\in\{\mathsf{0},\mathsf{s},\mathsf{ss}\}$. Then $\mathcal{H}^f_{\mathbb{K}}$ is a (split) semisimple algebra and
		$$
		\{\mathbb{D}(\undla)\mid \undla\in\mathscr{P}^{\bullet,m}_{n}\}$$ forms a complete set of pairwise non-isomorphic irreducible $\mathcal{H}^f_{\mathbb{K}}$-modules. Moreover,  $\mathbb{D}(\undla)$ is of type  $\texttt{M}$ if and only if $\sharp \mathcal{D}_{\undla}$  is even and is of type  $\texttt{Q}$ if and only if $\sharp \mathcal{D}_{\undla}$  is odd.
	\end{thm}			
	By Theorem \ref{semisimple:non-dege}, we have the following $\mathcal{H}^f_{\mathbb{K}}$-module isomorphism:
	$$\mathcal{H}^f_{\mathbb{K}}\cong\bigoplus_{\undla\in\mathscr{P}^{\bullet,m}_{n}}\mathbb{D}(\undla)^{\oplus 2^{n-\bigl\lceil\frac{\sharp\mathcal{D}_{\mt^{\undla}}}{2}\bigr\rceil}|\Std(\undla)| }\cong\bigoplus_{\undla\in\mathscr{P}^{\bullet,m}_{n}}\mathbb{D}(\undla)^{\oplus 2^{n-\sharp\left(\mathcal{OD}_{\mt^{\undla}}\right)}|\Std(\undla)| }.$$
	So the block decomposition is \label{pag:simple blocks}
	$$\mathcal{H}^f_{\mathbb{K}}=\bigoplus_{\undla \in \mathscr{P}^{\bullet,m}_{n}} B_{\undla},$$ and for each $\undla \in \mathscr{P}^{\bullet,m}_{n}$, we have $$B_{\undla}\cong \mathbb{D}(\undla)^{\oplus 2^{n-\bigl\lceil\frac{\sharp\mathcal{D}_{\mt^{\undla}}}{2}\bigr\rceil}|\Std(\undla)| }\cong \mathbb{D}(\undla)^{\oplus 2^{n-\sharp\left(\mathcal{OD}_{\mt^{\undla}}\right)}|\Std(\undla)| }$$ as $B_{\undla}$-modules.
	\subsection{Seminormal form}\label{sec:Seminormal form}
	{\bf In this subsection, we shall fix the parameter $\undQ=(Q_1,Q_2,\ldots,Q_m)\in(\mathbb{K}^*)^m$ , $\undla\in\mathscr{P}^{\bullet,m}_{n}$ and $f=f^{(\bullet)}_{\undQ}$ with $P^{(\bullet)}_{n}(q^2,\undQ)\neq 0$ for $\bullet\in\{\mathsf{0},\mathsf{s},\mathsf{ss}\}.$ Accordingly, we define the residue of boxes in the young diagram $\undla$ via \eqref{eq:residue} as well as $\res(\mathfrak{t})$ for each $\mathfrak{t}\in\Std(\undla)$ with $\undla\in\mathscr{P}^{\bullet,m}_{n}$ with $m\geq 0.$}

	Now we fix $\undla\in\mathscr{P}^{\bullet,m}_{n}$. Let $t:=\sharp \mathcal{D}_{\undla}.$
	\begin{defn}\cite[Definition 4.2]{LS2}
		We denote \begin{align}
			\mathcal{D}_{\mt^{\undla}}&:=\{\mt^{\undla}(a,a,l)|(a,a,l)\in\mathcal{D}_{\undla}\}=\{i_1<i_2<\cdots<i_t\}\label{stanard D},\\
			\mathcal{OD}_{\mt^{\undla}}&:=\{i_1,i_3,\cdots,i_{2{\lceil t/2 \rceil}-1}\}\subset \mathcal{D}_{\mt^{\undla}}\label{standard OD}
		\end{align} and \label{pag:dundla}
		\begin{align}
			d_{\undla}:=\begin{cases}
				1, & \text{ if $t$ is odd}, \\
				0, & \text{ if $t$ is even or $\mathcal{D}_{\mt^{\undla}}=\emptyset$}.
                         \end{cases}\nonumber
		\end{align}
	\end{defn}
	%\begin{rem}
	%	By Lemma \ref{important condition1} (1), it's easy to check that $\mathcal{D}_{\mt^{\undla}}=\mathcal{D}_{\res(\mt^{\undla})}$ if we set $\underline{\iota}=\res(\mt^{\undla})$ in \eqref{nondeg.Diota}. Hence we keep the same notation $i_1<i_2<\cdots<i_t$ here.
	%%{\bf We emphasize that from now on, we fix the notation \eqref{stanard D} for this special case.}
	%\end{rem}
	%Recall the notations in Definition \ref{lem. Dtauiota}.
	\label{pag:Dt,ODt,Z2ODt}		
	\begin{defn}\cite[Definition 4.4]{LS2}For each $\mt\in \Std(\undla),$ we define
		\begin{align*}
			\mathcal{D}_{\mt}&:=d(\mt,\mt^{\undla}) (\mathcal{D}_{\mt^{\undla}}),\nonumber\\
			\mathcal{OD}_{\mt}&:=d(\mt,\mt^{\undla}) (\mathcal{OD}_{\mt^{\undla}}),\nonumber\\
			\Z_2(\mathcal{OD}_{\mt})&:=\{\alpha_{\mt} \in \mathbb{Z}_{2}^{n} \mid \supp(\alpha_{\mt}) \subseteq \mathcal{OD}_{\mt}\},\\
			\Z_2([n]\setminus \mathcal{D}_{\mt})&:=\{\beta_{\mt} \in \mathbb{Z}_{2}^{n} \mid \supp(\beta_{\mt}) \subseteq [n]\setminus \mathcal{D}_{\mt}\},\nonumber
		\end{align*} 	and
		$$\gamma_{\mt}:=2^{-\lfloor t/2 \rfloor}\cdot\overrightarrow{\prod_{k=1,\cdots,{\lfloor t/2 \rfloor}}}\biggl(1+ \sqrt{-1}C_{{d(\mt,\mt^{\undla})}(i_{2k-1})}C_{{d(\mt,\mt^{\undla})}(i_{2k})}\biggr)\in\mathcal{C}_n.
		$$
	\end{defn}
	
	\label{pag:decomposition of OD}
	\begin{defn}\cite[Definition 4.9]{LS2}\label{decomposition of OD}
		For any $\mt\in\Std(\undla),$ let ${d(\mt,\mt^{\undla})}\in \Sym_n$ such that $\mt={d(\mt,\mt^{\undla})}\mt^{\undla}$. We define
		\begin{align}
			\mathbb{Z}_2(\mathcal{OD}_{\mt})_{\bar{0}}:=\{\alpha_{\mt} \in\Z_2(\mathcal{OD}_{\mt}) \mid d(\mt,\mt^{\undla})(i_t)\notin \supp(\alpha_{\mt})\},\nonumber\\
			\mathbb{Z}_2(\mathcal{OD}_{\mt})_{\bar{1}}:=\{\alpha_{\mt} \in\Z_2(\mathcal{OD}_{\mt}) \mid d(\mt,\mt^{\undla})(i_t)\in \supp(\alpha_{\mt})\}.\nonumber
		\end{align}
	\end{defn}
	
	That is, if $d_{\undla}=0$ (i.e., t is even), then $\Z_2(\mathcal{OD}_{\mt})_{\bar{0}}=\Z_2(\mathcal{OD}_{\mt})$ and $\Z_2(\mathcal{OD}_{\mt})_{\bar{1}}=\emptyset;$  if $d_{\undla}=1$ (i.e., t is odd), then  $\Z_2(\mathcal{OD}_{\mt})_{\bar{0}}$ and $\Z_2(\mathcal{OD}_{\mt})_{\bar{1}}$ are both non-empty and there is a natural bijection between $\Z_2(\mathcal{OD}_{\mt})_{\bar{0}}$ and $\Z_2(\mathcal{OD}_{\mt})_{\bar{1}}$ which sends $\alpha_{\mt}\in \Z_2(\mathcal{OD}_{\mt})_{\bar{0}}$ to $\alpha_{\mt}+e_{{d(\mt,\mt^{\undla})}(i_t)}\in \Z_2(\mathcal{OD}_{\mt})_{\bar{1}}.$ In particular, we have
	$$ \Z_2(\mathcal{OD}_{\mt})=\Z_2(\mathcal{OD}_{\mt})_{\bar{0}}\sqcup \Z_2(\mathcal{OD}_{\mt})_{\bar{1}}.$$ For any $\alpha_{\mt} \in \Z_2(\mathcal{OD}_{\mt})_{\bar{0}},$ we use $\alpha_{\mt,\bar{0}}=\alpha_{\mt}$  to emphasize that $\alpha_{\mt}\in Z_2(\mathcal{OD}_{\mt})_{\bar{0}}$ and  if $d_{\undla}=1$, we define $\alpha_{\mt,\bar{1}}:=\alpha_{\mt}+e_{{d(\mt,\mt^{\undla})}(i_t)}\in \Z_2(\mathcal{OD}_{\mt})_{\bar{1}}.$
	
	\label{pag:Tri}
	\begin{defn}\label{Tri}\cite[Definition 4.11]{LS2}
		For $a\in \Z_2,\,\undla\in\mathscr{P}^{\bullet,m}_{n}$ with $\bullet\in\{\mathsf{0},\mathsf{s},\mathsf{ss}\},$  we define
		$${\rm Tri}_{a}(\undla):=\bigsqcup_{\mt\in \Std(\undla)}\{\mt \}\times  \Z_2(\mathcal{OD}_{\mt})_{a}\times  \Z_2([n]\setminus \mathcal{D}_{\mt}),$$
		and $${\rm Tri}(\undla):={\rm Tri}_{\bar{0}}(\undla)\sqcup {\rm Tri}_{\bar{1}}(\undla).$$
		%	For $a\in \Z_2,$ we define
		%	$${\rm Tri}_{a}(\mathscr{P}^{\bullet,m}_{n}):=\bigsqcup_{\undla\in \mathscr{P}^{\bullet,m}_{n}}{\rm Tri}_{a}(\undla),$$
		%	and $${\rm Tri}(\mathscr{P}^{\bullet,m}_{n}):={\rm Tri}_{\bar{0}}(\mathscr{P}^{\bullet,m}_{n})\sqcup {\rm Tri}_{\bar{1}}(\mathscr{P}^{\bullet,m}_{n}).$$
	\end{defn}
	
	Notice that ${\rm Tri}(\undla)={\rm Tri}_{0}(\undla)$ when $d_{\undla}=0.$ For any ${\rm T}=(\mt, \alpha_{\mt}, \beta_{\mt})\in {\rm Tri}_{\bar{0}}(\undla),$ we denote
	$${\rm T}_{a}=(\mt, \alpha_{\mt,a}, \beta_{\mt})\in {\rm Tri}_{a}(\undla),\quad a\in \mathbb{Z}_{2},$$
	when $d_{\undla}=1.$
	%		is nothing but the index set of the basis of $\mathbb{D}(\undla)$ in Theorem \ref{actions of generators on L basis}, and ${\rm Tri}(\undla)_{\bar{0}}$ will index our primitive idempotents of the block with respect to $\undla$ specially.
	\begin{defn}\cite[Definition 3.4, Definition 4.5]{LS2}
		Let $\beta=(\beta_1,\ldots,\beta_n) \in \mathbb{Z}_{2}^{n}.$ For $k\in [n],$ we define
			\begin{align*}
				{\rm sgn}_{\beta}(k):=
				\begin{cases}
					-1, & \text{ if } \beta_k=\bar{1}, \\
					1, & \text{ if } \beta_k=\bar{0},
				\end{cases}
				\qquad
				\delta_{\beta}(k):=\frac{1-{\rm sgn}_{\beta}(k)}{2}=
				\begin{cases}
					1, & \text{ if } \beta_k=\bar{1}, \\
					0, & \text{ if } \beta_k=\bar{0}.
				\end{cases}
			\end{align*}
			\label{pag:deltabetak}
			\end{defn}
	
	Now we can define the primitive idempotents.
	\label{pag:primitive idempotents and blocks}
	\begin{defn}\cite[Definition 4.12]{LS2}\label{primitive idempotents and blocks}
		For $k\in[n]$, let
		$$\mathtt{B}(k):=\{ \mathtt{b}_{\pm}(\res_{\ms}(k)) \mid \ms \in \Std(\mathscr{P}^{\bullet,m}_{n}) \}.$$
		For any ${\rm T}=(\mt, \alpha_{\mt}, \beta_{\mt})\in {\rm Tri}_{\bar{0}}(\undla),$ we define
		\begin{align}\label{definition of primitive idempotents. non-dege}
			F_{\rm T}:=\left(C^{\alpha_{\mt}}\gamma_{\mt}(C^{\alpha_{\mt}})^{-1} \right) \cdot \left(\prod_{k=1}^{n}\prod_{\mathtt{b}\in \mathtt{B}(k)\atop \mathtt{b}\neq \mathtt{b}_{+}(\res_{\mt}(k))}\frac{X_k^{\sgn_{\beta_{\mt}}(k)}-\mathtt{b}}{\mathtt{b}_{+}(\res_{\mt}(k))-\mathtt{b}}\right)\in \mathcal{H}^f_{\mathbb{K}}.
		\end{align}
		We define
		\begin{align}
			F_{\undla}&:=\sum_{{\rm T}\in {\rm Tri}_{\bar{0}}(\undla)} F_{\rm T}.
		\end{align}
		
	\end{defn}
	
	\begin{defn}\cite[Definition 4.13]{LS2}
		For $a\in \Z_2,$ we denote
		$${\rm Tri}_{a}(\mathscr{P}^{\bullet,m}_{n}):=\bigsqcup_{\undla\in \mathscr{P}^{\bullet,m}_{n}}{\rm Tri}_{a}(\undla),$$
		and $${\rm Tri}(\mathscr{P}^{\bullet,m}_{n}):={\rm Tri}_{\bar{0}}(\mathscr{P}^{\bullet,m}_{n})\sqcup {\rm Tri}_{\bar{1}}(\mathscr{P}^{\bullet,m}_{n}).$$	
	\end{defn}

	\begin{thm}\cite[Theorem 4.16]{LS2}\label{primitive idempotents}
		Suppose $P^{(\bullet)}_{n}(q^2,\undQ)\neq 0$.  For $\bullet\in\{\mathsf{0},\mathsf{s},\mathsf{ss}\}$, we have the following.
		
		(a) $\{F_{\rm T} \mid {\rm T}\in {\rm Tri}_{\bar{0}}(\mathscr{P}^{\bullet,m}_{n})\}$ is a complete set of (super) primitive orthogonal idempotents of $\mathcal{H}^f_{\mathbb{K}}.$
		
		(b) $\{F_{\undla} \mid \undla \in \mathscr{P}^{\bullet,m}_{n} \}$ is a complete set of (super) primitive central idempotents of $\mathcal{H}^f_{\mathbb{K}}.$
	\end{thm}
	
Next we shall define the seminormal bases of $\mathcal{H}^f_{\mathbb{K}}.$ To this end, we need more notations.
\label{pag:partial beta}
\begin{defn}\cite[Definition 3.5]{LS2}
Let $\beta=(\beta_1,\ldots,\beta_n) \in \mathbb{Z}_{2}^{n}.$
For $0\leq k \leq n+1,$ we define
      $$\quad |\beta|_{<k}:=\sum_{1\leq k' <k}\beta_{k'},\quad |\beta|:=|\beta|_{<n+1}.$$
      Similarly, we can also define $|\beta|_{\leq k},$ $|\beta|_{> k}$ and $|\beta|_{\geq k}.$
\end{defn}

\begin{defn}\cite[Definition 4.6]{LS2}
	For any $i\in [n], \mt\in \Std(\undla),$ we define
             \label{pag:nondege eigenvalues}
			 $$\mathtt{b}_{\mt,i}:=\mathtt{b}_{-}(\res_{\mt}(i))\in\mathbb{K}^*.$$
            For any $i\in [n-1]$, we define
			 \begin{align}
			 	\delta(s_i\mt)
			 	:=\begin{cases}
			 		1, & \text{ if } s_i\mt \in \Std(\undla), \\
			 		0, & \text{ otherwise}.\nonumber
			 	\end{cases}
			 \end{align} and
            \label{pag:nondege coeffi cti}
			\begin{align}\label{nondege coeffi cti}
				\mathtt{c}_{\mt}(i):=1-\epsilon^2 \biggl(\frac{\mathtt{b}_{\mt,i}^{-1}\mathtt{b}_{\mt,i+1}}{(\mathtt{b}_{\mt,i}^{-1}\mathtt{b}_{\mt,i+1}-1)^2}
				+\frac{\mathtt{b}_{\mt,i}\mathtt{b}_{\mt,i+1}}{(\mathtt{b}_{\mt,i}\mathtt{b}_{\mt,i+1}-1)^2}\biggr)\in \mathbb{K}.
			\end{align}
			\end{defn}
Since $\mt \in \Std(\undla),$ $\mathtt{b}_{\mt,i}\neq \mathtt{b}_{\mt,i+1}^{\pm 1}$ by Definition \ref{defn:separate} and Proposition \ref{separate formula}, which immediately implies that $\mathtt{c}_{\mt}(i)$ is well-defined. If $s_i$ is admissible with respect to $\mt$, i.e., $\delta(s_i\mt)=1$, then $\mathtt{c}_{\mt}(i)\in \mathbb{K}^{*}$ by the third part of Lemma \ref{important condition1}. It is clear that $\mathtt{c}_{\mt}(i)=\mathtt{c}_{s_i\mt}(i).$	

	\begin{defn}\cite[Definition 4.21]{LS2}
    \label{pag:Phist and cst}	
For any $\ms,\mt \in \Std(\undla),$ fix a reduced expression $d(\ms,\mt)=s_{k_p}\cdots s_{k_1}\in\mathfrak{S}_n,$ then we define
		\begin{align}\label{Phist}
			\Phi_{\ms,\mt}:=\overleftarrow{\prod_{i=1,\ldots,p}}\Phi_{k_{i}}(\mathtt{b}_{s_{k_{i-1}}\cdots s_{k_1}\mathfrak{t},k_{i}}, \mathtt{b}_{s_{k_{i-1}}\cdots s_{k_1}\mathfrak{t},k_{i}+1})  \in \mathcal{H}^f_{\mathbb{K}}
		\end{align}
		and the coefficient
		\begin{align}\label{c-coefficients. non-dege.}
			\mathtt{c}_{\ms,\mt}:=\prod_{i=1,\ldots,p}\sqrt{\mathtt{c}_{s_{k_{i-1}}\cdots s_{k_1}\mathfrak{t}}(k_{i})}  \in \mathbb{K}.
		\end{align}
		\end{defn}
		By Lemma \ref{admissible transposes} and the third part of Lemma \ref{important condition1}, $\mathtt{c}_{\ms,\mt}\in  \mathbb{K}^*$.  By \cite[Lemma 4.22]{LS2}, $\Phi_{\ms,\mt}$ is independent of the reduced expression of $d(\ms,\mt)$.
Note that $\mathtt{c}_{\ms,\mt}=\mathtt{c}_{\mt,\ms}$ (see \cite[Lemma 4.23(3)]{LS2}).

	Now we can define the seminormal bases.
	\label{pag:nondege seminormal basis}
	\begin{defn}\cite[Definition 4.24]{LS2}\label{def seminormal} Let $\mathfrak{w}\in\Std(\undla)$.
		
		(1) Supppose $d_{\undla}=0.$  For any ${\rm S}=(\ms, \alpha_{\ms}', \beta_{\ms}'), {\rm T}=(\mt, \alpha_{\mt}, \beta_{\mt})\in {\rm Tri}(\undla),$ we define
		\begin{align}\label{fst. typeM. nondege.}
			f_{{\rm S},{\rm T}}^{\mathfrak{w}}
			:=F_{\rm S}C^{\beta_{\ms}'}C^{\alpha_{\ms}'}\Phi_{\ms,\mathfrak{w}}\Phi_{\mathfrak{w},\mt}
			(C^{\alpha_{\mt}})^{-1} (C^{\beta_{\mt}})^{-1} F_{\rm T}\in F_{\rm S}\mathcal{H}^f_{\mathbb{K}} F_{\rm T},
		\end{align} and	\begin{align}\label{re. fst. typeM. nondege.}
			f_{{\rm S},{\rm T}}
			:=F_{\rm S}C^{\beta_{\ms}'}C^{\alpha_{\ms}'}\Phi_{\ms,\mt}
			(C^{\alpha_{\mt}})^{-1} (C^{\beta_{\mt}})^{-1} F_{\rm T} \in F_{\rm S}\mathcal{H}^f_{\mathbb{K}} F_{\rm T},
		\end{align}

		(2) Suppose $d_{\undla}=1.$ For any $a\in \mathbb{Z}_{2}$ and ${\rm S}=(\ms, \alpha_{\ms}', \beta_{\ms}')\in {\rm Tri}_{\bar{0}}(\undla), {\rm T}_{a}=(\mt, \alpha_{\mt,a}, \beta_{\mt})\in {\rm Tri}_{a}(\undla),$ we define
		\begin{align}\label{fst. typeQ. nondege.}
			&	f_{{\rm S},{\rm T}_{a}}^{\mathfrak{w}}
			:= (-1)^{|\alpha_\ms'|_{>d(\ms,\mt^{\undla})(i_t)}+a|\alpha_\mt|_{>d(\mt,\mt^{\undla})(i_t)}}\nonumber\\
			&\qquad\qquad\qquad\cdot	F_{\rm S}C^{\beta_{\ms}'}C^{\alpha_{\ms}'}\Phi_{\ms,\mathfrak{w}}\Phi_{\mathfrak{w},\mt}
			(C^{{\alpha}_{\mt,a}})^{-1} (C^{\beta_{\mt}})^{-1} F_{\rm T}\in F_{\rm S}\mathcal{H}^f_{\mathbb{K}} F_{\rm T}
		\end{align} and 	\begin{align}\label{re. fst. typeQ. nondege.}
			&	f_{{\rm S},{\rm T}_{a}}
			:= (-1)^{|\alpha_\ms'|_{>d(\ms,\mt^{\undla})(i_t)}+a|\alpha_\mt|_{>d(\mt,\mt^{\undla})(i_t)}}\nonumber\\
			&\qquad\qquad\qquad\cdot F_{\rm S}C^{\beta_{\ms}'}C^{\alpha_{\ms}'}\Phi_{\ms,\mt}
			(C^{{\alpha}_{\mt,a}})^{-1} (C^{\beta_{\mt}})^{-1} F_{\rm T} \in F_{\rm S}\mathcal{H}^f_{\mathbb{K}} F_{\rm T}.
		\end{align}
		
		\label{pag:nondege cT}	
		(3)  For any ${\rm T}=(\mt, \alpha_{\mt}, \beta_{\mt})\in {\rm Tri}(\undla),$ we define $$\mathtt{c}_{\rm T}^{\mathfrak{w}}=\mathtt{c}_{\mt}^{\mathfrak{w}}:=(\mathtt{c}_{\mt,\mathfrak{w}})^2\in \mathbb{K}^*.$$
	\end{defn}
%\begin{rem}\label{FT=fTT}
%By definition, it's clear that $F_{\rm T}=f_{{\rm T},{\rm T}_{\bar{0}}}$ for any ${\rm T}\in {\rm Tri}_{\bar{0}}(\undla).$
%\end{rem}

		\begin{thm}\cite[Theorem 4.26]{LS2}\label{seminormal basis}
		Suppose $P^{(\bullet)}_{n}(q^2,\undQ)\neq 0$. We fix $\mathfrak{w}\in\Std(\undla)$. Then the following two sets
		\begin{align}\label{Non-deg seminormal1}
			\left\{ f_{{\rm S},{\rm T}}^\mathfrak{w} \Biggm|
			{\rm S}=(\ms, \alpha_{\ms}', \beta_{\ms}')\in {\rm Tri}_{\bar{0}}(\undla),
			{\rm T}=(\mt, \alpha_{\mt}, \beta_{\mt})\in {\rm Tri}(\undla)
			\right\}
		\end{align} and \begin{align}\label{Non-deg seminormal2}
			\left\{ f_{{\rm S},{\rm T}} \Biggm|
			{\rm S}=(\ms, \alpha_{\ms}', \beta_{\ms}')\in {\rm Tri}_{\bar{0}}(\undla),
			{\rm T}=(\mt, \alpha_{\mt}, \beta_{\mt})\in {\rm Tri}(\undla)
			\right\}
		\end{align} form two $\mathbb{K}$-bases of the block $B_{\undla}$ of $\mathcal{H}^f_{\mathbb{K}}$.
		
		Moreover, for ${\rm S}=(\ms, \alpha_{\ms}', \beta_{\ms}')\in {\rm Tri}_{\bar{0}}(\undla),
		{\rm T}=(\mt, \alpha_{\mt}, \beta_{\mt})\in {\rm Tri}(\undla),$ we have \begin{equation}\label{fst and re. fst}
			f_{{\rm S},{\rm T}}
			=\frac{\mathtt{c}_{\ms,\mt}}{\mathtt{c}_{\ms,\mathfrak{w} }\mathtt{c}_{\mathfrak{w},\mt }} f_{{\rm S},{\rm T}}^\mathfrak{w}\in F_{\rm S}\mathcal{H}^f_{\mathbb{K}} F_{\rm T}.
		\end{equation} The multiplications of basis elements in \eqref{Non-deg seminormal1} are given as follows.
		
		(1) Suppose $d_{\undla}=0.$  Then for any
		${\rm S}=(\ms, \alpha_{\ms}', \beta_{\ms}'),
		{\rm T}=(\mt, \alpha_{\mt}, \beta_{\mt}),
		{\rm U}=(\mfku,\alpha_{\mfku}^{''},\beta_{\mfku}^{''}),
		{\rm V}=(\mfkv,\alpha_{\mfkv}^{'''},\beta_{\mfkv}^{'''})\in {\rm Tri}(\undla),$ we have
		\begin{align}\label{Non-deg multiplication1}
			f_{{\rm S},{\rm T}}^\mathfrak{w} f_{{\rm U},{\rm V}}^\mathfrak{w}
			=\delta_{{\rm T},{\rm U}} \mathtt{c}_{\rm T}^\mathfrak{w} f_{{\rm S},{\rm V}}^\mathfrak{w}.
		\end{align}
		
		(2) Suppose $d_{\undla}=1.$ Then for any $a,b\in \mathbb{Z}_2$ and
		\begin{align*}
			{\rm S}&=(\ms, \alpha_{\ms}', \beta_{\ms}')\in {\rm Tri}_{\bar{0}}(\undla), \quad
			{\rm T}_{a}=(\mt, \alpha_{\mt,a}, \beta_{\mt})\in {\rm Tri}_{a}(\undla),\nonumber\\
			{\rm U}&=(\mfku,\alpha_{\mfku}^{''},\beta_{\mfku}^{''})\in {\rm Tri}_{\bar{0}}(\undla), \quad
			{\rm V}_{b}=(\mfkv,{\alpha_{\mfkv,b}^{'''}},\beta_{\mfkv}^{'''})\in {\rm Tri}_{b}(\undla),\nonumber
		\end{align*} we have
		\begin{align}\label{Non-deg multiplication2}
			f_{{\rm S},{\rm T}_{a}}^\mathfrak{w} f_{{\rm U},{\rm V}_{b}}^\mathfrak{w}
			=\delta_{{\rm T}_{\bar{0}},{\rm U}}(-1)^{\left(|\alpha_{\mt}|_{>d(\mt,\mt^{\undla})(i_t)}\right)}\mathtt{c}_{\rm T}^\mathfrak{w} f_{{\rm S},{\rm V}_{a+b}}^\mathfrak{w}.
		\end{align}
	\end{thm}

The important coefficients $\mathtt{c}_{\mt}^{\mt^{\undla}}=\mathtt{c}_{\mt,\mt^{\undla}}\mathtt{c}_{\mt^{\undla},\mt}$ and $\mathtt{c}_{\mt}^{\mt_{\undla}}=\mathtt{c}_{\mt,\mt_{\undla}}\mathtt{c}_{\mt_{\undla},\mt}$ also have the following combinatorial formulae which are useful in the rest of this paper.
\begin{lem}\cite[Proposition 3.23]{LS3}\label{combina. formulae of c}
	Let $\undla\in\mathscr{P}^{\mathsf{\bullet},m}_{n}$ for $\bullet\in\{\mathsf{0},\mathsf{s},\mathsf{ss}\}$ and $\mt\in\Std(\undla).$
	Then we have
	\begin{align*}
		\mathtt{c}_{\mt}^{\mt^{\undla}}
		=\prod_{k=1}^{n}\prod_{A\in\mathscr{A}_{\mt^{\undla}}^{\rhd}(k)}\left(\mathtt{q}(\res_{\mt^{\undla}}(k))-\mathtt{q}(\res(A))\right)^{-1}
		\cdot\prod_{k=1}^{n}\frac{\prod_{A\in\mathscr{A}_{\mt}^{\rhd}(k)}\left(\mathtt{q}(\res_{\mt}(k))-\mathtt{q}(\res(A))\right)}{\prod_{B\in\mathscr{R}_{\mt}^{\rhd}(k)\setminus\mathcal{D}}\left(\mathtt{q}(\res_{\mt}(k))-\mathtt{q}(\res(B))\right)},
	\end{align*}
	\begin{align*}
		\mathtt{c}_{\mt}^{\mt_{\undla}}
		=\prod_{k=1}^{n}\frac{\prod_{B\in\mathscr{R}_{\mt_{\undla}}^{\lhd}(k)\setminus\mathcal{D}}\left(\mathtt{q}(\res_{\mt_{\undla}}(k))-\mathtt{q}(\res(B))\right)}{\prod_{A\in\mathscr{A}_{\mt_{\undla}}^{\lhd}(k)}\left(\mathtt{q}(\res_{\mt_{\undla}}(k))-\mathtt{q}(\res(A))\right)}	\cdot\prod_{k=1}^{n}\frac{\prod_{A\in\mathscr{A}_{\mt}^{\lhd}(k)}\left(\mathtt{q}(\res_{\mt}(k))-\mathtt{q}(\res(A))\right)}{\prod_{B\in\mathscr{R}_{\mt}^{\lhd}(k)\setminus\mathcal{D}}\left(\mathtt{q}(\res_{\mt}(k))-\mathtt{q}(\res(B)\right)}.
	\end{align*}
\end{lem}

The following Proposition implies all of the seminormal basis elements are common eigenvectors of $X_i$, $i\in [n]$.
\begin{prop}\cite[Proposition 4.34]{LS2}
Let $\undla\in\mathscr{P}^{\mathsf{\bullet},m}_{n}$ for $\bullet\in\{\mathsf{0},\mathsf{s},\mathsf{ss}\}$,
and ${\rm T}=(\mt, \alpha_{\mt}, \beta_{\mt})\in {\rm Tri}_{\bar{0}}(\undla), {\rm S}=(\ms,\alpha_{\ms}',\beta_{\ms}')\in {\rm Tri}(\undla).$
For each $i\in [n],$ we have
\begin{align}\label{X acts on f}
X_i \cdot f_{{\rm T},{\rm S}}
=\mathtt{b}_{\mt,i}^{-{\rm sgn}_{\beta_{\mt}}(i)} f_{{\rm T},{\rm S}},\qquad
f_{{\rm T},{\rm S}}\cdot X_i
=\mathtt{b}_{\ms,i}^{-{\rm sgn}_{\beta_{\ms}'}(i)} f_{{\rm T},{\rm S}}.
\end{align}
\end{prop}

The action of the generators $C_i$, $i\in[n]$ and $T_j$, $j\in [n-1]$ on the seminormal bases is also given in \cite{LS2} for any $\bullet\in\{\mathsf{0},\mathsf{s},\mathsf{ss}\}.$ In this paper, we only need the case $\bullet=\mathsf{0}$.
Note that $\mathscr{P}^{\mathsf{0},m}_{n}=\mathscr{P}^m_n$ and ${\rm Tri}(\mathscr{P}^{\mathsf{0},m}_{n})=\Std(\mathscr{P}^m_n)\times \Z_2^n.$
\begin{prop}\label{generators action on seminormal basis}\cite[Proposition 4.34]{LS2}
Let $\undla\in\mathscr{P}^m_n.$
Suppose ${\rm T}=(\mt, \beta_{\mt}), {\rm S}=(\ms, \beta_{\ms}')\in {\rm Tri}(\undla).$
Then we have the following.
			\begin{enumerate}
%					\item For each $i\in [n],$ we have
%			\begin{align}\label{X acts on f}
%				X_i \cdot f_{{\rm T},{\rm S}}
%				    =\mathtt{b}_{\mt,i}^{-{\rm sgn}_{\beta_{\mt}}(i)} f_{{\rm T},{\rm S}},\qquad
%			        f_{{\rm T},{\rm S}}\cdot X_i
%				    =\mathtt{b}_{\ms,i}^{-{\rm sgn}_{\beta_{\ms}'}(i)} f_{{\rm T},{\rm S}}.
%			\end{align}
                    \item For each $i\in [n],$ we have
					\begin{align}\label{C acts on f}
						C_i \cdot f_{{\rm T},{\rm S}}
						&=(-1)^{|\beta_{\mt}|_{<i}}  f_{(\mt,\beta_{\mt}+e_i),{\rm S}},
					\end{align}
			
			        \item For each $i\in [n-1],$ denote $s_i\cdot{\rm T}=(s_i \mt,s_i \cdot \beta_{\mt}),$ we have
					\begin{align}\label{T acts on f}
						&T_i\cdot f_{{\rm T},{\rm S}}\nonumber \\
						=&-\frac{\epsilon}{\mathtt{b}_{\mt,i}^{-{\rm sgn}_{\beta_{\mt}}(i)}\mathtt{b}_{\mt,i+1}^{{\rm sgn}_{\beta_{\mt}}(i+1)}-1} f_{{\rm T},{\rm S}} \nonumber\\
						&\qquad+(-1)^{\delta_{\beta_{\mt}}(i)}\frac{\epsilon}{\mathtt{b}_{\mt,i}^{{\rm sgn}_{\beta_{\mt}}(i)}\mathtt{b}_{\mt,i+1}^{{\rm sgn}_{\beta_{\mt}}(i+1)}-1}f_{(\mt,\beta_{\mt}+e_{i}+e_{i+1}),{\rm S}}\\
						&\qquad\qquad+\delta(s_i\mt)(-1)^{\delta_{\beta_{\mt}}(i)\delta_{\beta_{\mt}}(i+1)}\sqrt{\mathtt{c}_{\mt}(i)}\frac{\mathtt{c}_{\mt,\ms}}{\mathtt{c}_{s_i\cdot\mt,\ms}}f_{s_i\cdot{\rm T},{\rm S}}.\nonumber
					\end{align}
			\end{enumerate}
		\end{prop}

Recall the supersymmetrizing form $t_{r,n}$ \eqref{closed formula frob} of $\mathcal{H}^f_{\mathbb{K}},$ where $\bullet=\mathsf{0},$ $r=2m.$
The images of the seminormal bases under $t_{2m,n}$ are given by the following.
\begin{thm}\label{mainthm type0 nondege}
	Suppose that $\bullet=\mathsf{0}$ and $\undla\in\mathscr{P}^m_n.$
    Let ${\rm S}$ and ${\rm T}=(\mt,\beta_{\mt})\in{\rm Tri}(\undla).$

    (1) \cite[Proposition 5.7]{LS3} If ${\rm S}\neq{\rm T},$ then $t_{2m,n}(f_{{\rm S},{\rm T}})=0.$

	(2) \cite[Theorem 6.1]{LS3} We have
\begin{align*}\label{type=0. tauFT}
		t_{2m,n}(F_{\rm T})
		=\prod\limits_{k=1}^{n}\frac{1}{\mathtt{b}_{\mt,k}^{{\rm sgn}_{\beta_{\mt}}(k)}-\mathtt{b}_{\mt,k}^{-{\rm sgn}_{\beta_{\mt}}(k)}}
		\cdot \prod\limits_{k=1}^{n}\frac{\prod\limits_{B\in\Rem(\mt\downarrow_{k-1})}\left(\mathtt{q}(\res_{\mt}(k))-\mathtt{q}(\res(B))\right)}{
			\prod\limits_{A\in\Add(\mt\downarrow_{k-1})\setminus \{\mt^{-1}(k)\}}\left(\mathtt{q}(\res_{\mt}(k))-\mathtt{q}(\res(A))\right)}.
	\end{align*}
\end{thm}

		\subsection{Lifting idempotents}\label{Lifting idempotents}
	In this subsection, {\bf we fix $q^2\neq \pm 1$, $\undQ=(Q_1,\cdots,Q_m)\in(\mathbb{K}^*)^m$ and $f=f^{(\bullet)}_{\underline{Q}}$ with $\bullet\in\{\mathsf{0},\mathsf{s},\mathsf{ss}\}$.  Let $x$ be an indeterminant, we set $\hO:=\mathbb{K}[[x]]=\{a_0+a_1x+a_2x^2+\cdots|~a_i\in \mathbb{K}\}$ and $\hK$ be the fraction field of $\hO$. We modify the parameters as follows \label{pag:deformed CHCAs}
%		\footnote{Note again that each $Q_i$ corresponds to $qQ_i$ in \cite{SW}.}
		: $q':=x^4+q,$ $Q'_i:=x^{8ni}+Q_i, 1\leq i\leq m$. Then we can define $\mathcal{H}^{f'}_{\hO}:=\mathcal{H}^{f'}_{\hO}(n)$, where $$f'=\begin{cases}
			f^{\mathsf{(0)}}_{\underline{Q'}}=\prod_{i=1}^m \biggl(X_1+X^{-1}_1-\mathtt{q}(Q_i')\biggr), &\qquad\text{ if $\bullet=\mathsf{0},$}\\
			f^{\mathsf{(s)}}_{\underline{Q'}}=(X_1-1)\prod_{i=1}^m \biggl(X_1+X^{-1}_1-\mathtt{q}(Q_i')\biggr),&\qquad\text{ if $\bullet=\mathsf{s},$}\\
			f^{\mathsf{(ss)}}_{\underline{Q'}}=(X_1-1)(X_1+1)\prod_{i=1}^m \biggl(X_1+X^{-1}_1-\mathtt{q}(Q_i')\biggr),&\qquad\text{ if $\bullet=\mathsf{ss}.$}
		\end{cases}$$  Similarly, we can define $\mathcal{H}^{f'}_{\hK}:=\mathcal{H}^{f'}_{\hK}(n)$.} Then we have $$
		\mathcal{H}^{f'}_{\hK}\cong \hK\otimes_{\hO} \mathcal{H}^{f'}_{\hO},\qquad\mathcal{H}^{f}_{\mathbb{K}}\cong \mathbb{K}\otimes_{\hO}\mathcal{H}^{f'}_{\hO}.
		$$ Then we can check $P_n^{(\bullet)}({q'}^2,\undQ')\neq 0$, hence $\mathcal{H}^{f'}_{\hK}$ is semisimple over $\hK$.  {\bf Accordingly, we define the residues of boxes in the young diagram $\undla$ via \eqref{eq:residue} as well as $\res(\mathfrak{t})$ for each $\mathfrak{t}\in\Std(\undla)$ with $\undla\in\mathscr{P}^{\bullet,m}_{n}$ with $m\geq 0$ with respect to parameters $(q',Q'_1,\cdots,Q'_m)$. }
	
It follows from \eqref{substitution0} that all of the eigenvalues $\mathtt{b}_{\pm}(\res_{\mt}(k))$ of $X_k$ belong to $\mathbb{K}[[x^2]]\subset\hO$. Furthermore, by \eqref{nondege coeffi cti} and \eqref{c-coefficients. non-dege.} we deduce that all of the coefficients $\mathtt{c}_{\ms,\mt}\in \hK$. For $a\in \hO$, we use $a|_{x=0}\in \mathbb{K}$ to denote the image of $a$ in the residue field $\mathbb{K}\cong \hO/(x)$. {\bf We shall identify $\mathcal{H}^{f}_{\mathbb{K}}$ with the cyclotomic quiver Hecke-Clifford superalgebra $RC^{\Lambda_f}_n$ by Theorem \ref{KKTiso}.
%	, and we will denote by $[a]\in \mathcal{K}$ instead of $[a|_{x=0}]\in \mathcal{K}$ for simplicity.
The aim of this section is to construct certain idempotent $e({\bf i})^\hO\in\mathcal{H}^{f'}_{\hO}$ such that $1\otimes_\hO e({\bf i})^\hO=e({\bf i})\in \mathcal{H}^{f}_{\mathbb{K}}$ for ${\bf i}\in(\mathbb{K}^*)^n$.}
	\begin{defn}
		Let ${\rm T}=(\mt,\alpha_\mt,\beta_\mt)\in{\rm Tri}_{\bar{0}}(\undla),\, \undla\in\mathscr{P}^{\bullet,m}_{n}$, we define the sequence \label{pag:defored T seq}
$$\mathtt{b}_{\mt,\beta_{\mt}}:=\left(\mathtt{b}_{\mt, 1}^{-{\rm sgn}_{\beta_{\mt}}(1)},\mathtt{b}_{\mt, 2}^{-{\rm sgn}_{\beta_{\mt}}(2)},\cdots,\mathtt{b}_{\mt, n}^{-{\rm sgn}_{\beta_{\mt}}(n)}\right)\in (\hK^*)^n,$$
where $\mathtt{b}_{\mt, k}:=\mathtt{b}_{-}(\res_{\mt}(k))$ for $k\in[n].$
And we define
\begin{equation}\label{defn:iT}
{\bf i}^{\rm T}:=\mathtt{b}_{\mt,\beta_{\mt}}|_{x=0}\in(\mathbb{K}^*)^n,
\end{equation}	
then $\pr ({\bf i}^{\rm T})=\mathtt{q}(\res(\mt))|_{x=0} \in (I_f)^n.$
%that is, ${\bf i}_k^{\rm T}=\mathtt{q}(\res_{\mt}(k))\in I_f$ for $k\in[n].$	
		%We define an equivalent relation $\approx$ on ${\rm Tri}_{\bar{0}}(\mathscr{P}^{\bullet,m}_{n})$ as follows. For ${\rm S}=(\ms,\alpha'_\ms,\beta'_\ms)\in{\rm Tri}_{\bar{0}}(\undla), {\rm T}=(\mt,\alpha_\mt,\beta_\mt)\in{\rm Tri}_{\bar{0}}(\undmu),\, \undla,\undmu\in\mathscr{P}^{\bullet,m}_{n}$, then ${\rm S}\approx{\rm T}$
		%	if and only if $$\mathtt{b}_{\beta'_\ms}(\res(\ms))|_{x=0}=\mathtt{b}_{\beta_\mt}(\res(\mt))|_{x=0}.$$
	\end{defn}
	
	\begin{defn}
		Let ${\bf i}\in(\mathbb{K}^*)^n$. For $\undla\in\mathscr{P}^{\bullet,m}_{n},$
we define
$${\rm Tri}(\undla,{\bf i})=\Bigl\{{\rm T}=(\mt,\alpha_\mt,\beta_\mt)\in{\rm Tri}_{\bar{0}}(\undla) \Bigm| {\bf i}^{\rm T}={\bf i}\Bigr\},$$
and
$${\rm Tri}({\bf i})=\bigsqcup_{\undla\in\mathscr{P}^{\bullet,m}_{n}} {\rm Tri}(\undla,{\bf i}).$$
We set \label{pag:deformed KLR idempotent}
\begin{align}\label{eO}
e({\bf i})^\hO:=\sum_{{\rm T}\in {\rm Tri}({\bf i})}F_{\rm T}\in \mathcal{H}^{f'}_{\hK}.
\end{align}
	\end{defn}

	\begin{prop}\label{lifting idempotent}
		Let ${\bf i}\in(\mathbb{K}^*)^n$, then $e({\bf i})^\hO\in\mathcal{H}^{f'}_{\hO}$ and $1\otimes_\hO e({\bf i})^\hO=e({\bf i})$.
	\end{prop}
	
	\begin{proof}
		The proof is similar as in \cite[Proposition 4.8]{HM2}.  For $k\in[n]$, let
		$$\mathtt{B}(k):=\{ \mathtt{b}_{\pm}(\res_{\ms}(k)) \mid \ms \in \Std(\mathscr{P}^{\bullet,m}_{n}) \}.$$ We fix ${\rm T}=(\mt,\alpha_\mt,\beta_\mt)\in{\rm Tri}({\bf i})$ for ${\bf i}=(i_1,\ldots,i_n)\in(\mathbb{K}^*)^n,$ and construct a new element \begin{align}\label{middle element}
			F'_{\rm T}:=\prod_{k=1}^n\prod_{\substack{c\in\mathtt{B}(k)\\c|_{x=0}\neq i_k}}\frac{X_k-c}{\mathtt{b}_{\mt, k}^{-{\rm sgn}_{\beta_{\mt}}(k)}-c}\in \mathcal{H}^{f'}_{\hO}.\end{align}
		Let $$d_{\rm T}:=\prod_{k=1}^n\prod_{\substack{c\in\mathtt{B}(k)\\c|_{x=0}\neq i_k}}\left(\mathtt{b}_{\mt, k}^{-{\rm sgn}_{\beta_{\mt}}(k)}-c\right)\in \hO^{\times}$$ be the demoninator of $F'_{\rm T}$, then we have \begin{align}\label{multiply middle}F_{\rm S}F'_{\rm T}=F'_{\rm T}F_{\rm S}=\begin{cases}
				\frac{d_{\rm S}}{d_{\rm T}} F_{\rm S}, &\qquad \text{if ${\rm S}\in{\rm Tri}({\bf i}),$}\\
				0,&\qquad \text{otherwise}.
			\end{cases}
		\end{align} This implies that \begin{align}\label{middle expansion}
			F'_{\rm T}=\sum_{\substack{{\rm S}\in{\rm Tri}({\bf i})}}	\frac{d_{\rm S}}{d_{\rm T}} F_{\rm S}.
		\end{align} Moreover, we have $d_{\rm S}-d_{\rm T}\in x\hO$ for ${\rm S}\in{\rm Tri}({\bf i}).$ We deduce that there exsits $N\in\N$ such that $$
		\left(1-\frac{d_{\rm S}}{d_{\rm T}}\right)^N F_{\rm S}\in \mathcal{H}^{f'}_{\hO}$$ for all ${\rm S}\in{\rm Tri}({\bf i})$. This, combining with \eqref{middle expansion} implies $$
		\left(e({\bf i})^\hO-F'_{\rm T}\right)^N=\sum_{\substack{{\rm S}\in{\rm Tri}({\bf i})}}\left(1-	\frac{d_{\rm S}}{d_{\rm T}} \right)^NF_{\rm S}\in \mathcal{H}^{f'}_{\hO}.
		$$ On the other hand, by the binomial theorem, we can compute \begin{align*}
			\left(e({\bf i})^\hO-F'_{\rm T}\right)^N&=\sum_{k=0}^N(-1)^k\binom{N}{k}\left(e({\bf i})^\hO\right)^{N-k}\left(F'_{\rm T}\right)^k\\
			&=e({\bf i})^\hO+\sum_{k=1}^N(-1)^k\binom{N}{k}\left(e({\bf i})^\hO\right)^{N-k}\left(F'_{\rm T}\right)^k\\
			&=e({\bf i})^\hO+	\left(1-F'_{\rm T}\right)^N-1,
		\end{align*}
where in the first and last equation, we have used \eqref{multiply middle}. In conclusion, we deduce that $e({\bf i})^\hO\in\mathcal{H}^{f'}_{\hO}$.
Now we set $\hat{e}({\bf i})=1\otimes_\hO e({\bf i})^\hO \in \mathcal{H}^{f}_{\mathbb{K}}$. By definition, for $1\leq k\leq n,$ we have
$$\prod_{{\rm T}\in {\rm Tri}({\bf i})}\left(X_k-\mathtt{b}_{\mt, k}^{-{\rm sgn}_{\beta_{\mt}}(k)}\right)e({\bf i})^\hO=0,$$
which implies that
\begin{align}\label{nilpotentcy}(X_k-{\bf i}_k)^{\sharp {\rm Tri}({\bf i})}\hat{e}({\bf i})=0.
	\end{align}
Hence $\hat{e}({\bf i})\in e({\bf i})\mathcal{H}^{f}_{\mathbb{K}}.$ Since $\{\hat{e}({\bf i})\mid {\bf i}\in(\mathbb{K}^*)^n,e({\bf i})\neq 0\}$ is a finite set of pairwise orthogonal idempotents and $$\sum_{{\bf i}\in(\mathbb{K}^*)^n,e({\bf i})\neq 0}\hat{e}({\bf i})=1,$$ we deduce that $\mathcal{H}^{f}_{\mathbb{K}}=\bigoplus_{{\bf i}\in(\mathbb{K}^*)^n,e({\bf i})\neq 0}\hat{e}({\bf i})\mathcal{H}^{f}_{\mathbb{K}}.$
This implies that $\hat{e}({\bf i})\mathcal{H}^{f}_{\mathbb{K}}=e({\bf i})\mathcal{H}^{f}_{\mathbb{K}}$ and $\hat{e}({\bf i})=e({\bf i})$.
	\end{proof}
	
	An immediate consequence from our proof gives the following nilpotency upper bound for $y_ke({\bf i})$, which generalizes \cite[In the end of \S 4]{EM} and \cite[Corollary 4.31]{HM2}.
	\begin{cor}
			Let ${\bf i}\in(\mathbb{K}^*)^n$, and $1\leq k\leq n$. Then we have $y_k^{\sharp {\rm Tri}({\bf i})}e({\bf i})=0$.
		\end{cor}
		
		\begin{proof}
	This follows from Theorem \ref{KKTiso} and \eqref{nilpotentcy}.
	\end{proof}
	
	As another application of Proposition \ref{lifting idempotent}, we can deduce dimension formulae for bi-weight spaces.
	\begin{defn}
		Let ${\bf i}\in(\mathbb{K}^*)^n$. For $\undla\in\mathscr{P}^{\bullet,m}_{n},$ we define $$\Std(\undla,{\bf i})=\Bigl\{\mt\in\Std(\undla)\Bigm| \exists {\rm T}=(\mt,\alpha_\mt,\beta_\mt)\in{\rm Tri}_{\bar{0}}(\undla) \text{ such that } {\bf i}^{\rm T}={\bf i}\Bigr\}.$$
	\end{defn}
	
	\begin{thm}
		Let ${\bf i},{\bf j} \in(\mathbb{K}^*)^n$. We have
\begin{align}\label{dim 1}
\dim_{\mathbb{K}}e({\bf i})\mathcal{H}^{f}_{\mathbb{K}}e({\bf j})=\sum_{\substack{\undla\in\mathscr{P}^{\bullet,m}_{n}}}2^{d_{\undla}}\sharp {\rm Tri}(\undla,{\bf i}) \sharp {\rm Tri}(\undla,{\bf j}).
		\end{align}
If $l=\sharp \{\alpha \in \undla \mid \mathtt{b}_{+}(\res(\alpha))|_{x=0} \in\{\pm 1\}\},$ then
\begin{align}\label{dim 2}
\dim_{\mathbb{K}}e({\bf i})\mathcal{H}^{f}_{\mathbb{K}}e({\bf j})=\sum_{\substack{\undla\in\mathscr{P}^{\bullet,m}_{n}}}2^{2l-\sharp \mathcal{D}_{\undla}}\sharp \Std(\undla,{\bf i}) \sharp \Std(\undla,{\bf j}).
		\end{align}
	\end{thm}
	
	\begin{proof}
		We have following two decompositions
$$\mathcal{H}^{f}_{\mathbb{K}}=\bigoplus_{{\bf i},{\bf j}\in(\mathbb{K}^*)^n,e({\bf i}),e({\bf j})\neq 0}e({\bf i})\mathcal{H}^{f}_{\mathbb{K}}e({\bf j}),\qquad	\mathcal{H}^{f'}_{\hO}=\bigoplus_{{\bf i},{\bf j}\in(\mathbb{K}^*)^n,e({\bf i}),e({\bf j})\neq 0}e({\bf i})^\hO\mathcal{H}^{f'}_{\hO}e({\bf j})^\hO.
		$$ By Proposition \ref{lifting idempotent}, we have the natural isomorphism $\mathbb{K}\otimes_\hO e({\bf i})^\hO\mathcal{H}^{f'}_{\hO}e({\bf j})^\hO\cong e({\bf i})\mathcal{H}^{f}_{\mathbb{K}}e({\bf j}).$ Then we have
\begin{align}\label{dim=rank}
\dim_{\mathbb{K}} e({\bf i})\mathcal{H}^{f}_{\mathbb{K}}e({\bf j})={\rm rank}_\hO e({\bf i})^\hO\mathcal{H}^{f'}_{\hO}e({\bf j})^\hO=\dim_\hK e({\bf i})^\hO\mathcal{H}^{f'}_{\hK}e({\bf j})^\hO.
\end{align}
Hence \eqref{dim 1} follows by computing numbers of seminormal basis elements in $e({\bf i})^\hO\mathcal{H}^{f'}_{\hK}e({\bf j})^\hO$. Note that
\begin{align*}
\sharp {\rm Tri}(\undla,{\bf i})=2^{l-\bigl\lceil\frac{\sharp \mathcal{D}_{\undla}}{2}\bigr\rceil}\sharp \Std(\undla,{\bf i}),
\qquad
\sharp {\rm Tri}(\undla,{\bf j})=2^{l-\bigl\lceil\frac{\sharp \mathcal{D}_{\undla}}{2}\bigr\rceil}\sharp \Std(\undla,{\bf j}),
\end{align*}
we obtain \eqref{dim 2} from \eqref{dim 1}.
	\end{proof}
	The following Corollary has it's independent interest.
\begin{cor}
	Let $1\leq k\leq n$. Then $\tilde{a}\in \mathbb{K}$ is an eigenvalue of $X_k$ on $\mathcal{H}^{f}_{\mathbb{K}}$ if and only if
there exists $a\in \hO$ such that $\tilde{a}=a|_{x=0}$ and $a$ is an eigenvalue of $X_k$ on $\mathcal{H}^{f'}_{\mathscr{K}}$.
\end{cor}

	\begin{proof}
		We have the following.
	\begin{align*}
		&\text{$\tilde{a}$ is an eigenvalue of $X_k$ on $\mathcal{H}^{f}_{\mathbb{K}}$ } \\
		\xLeftrightarrow{\text{By definition}} &\text{ there exists some $e({\bf i})\neq 0$ such that ${\bf i}_k=\tilde{a}$}\\
        \xLeftrightarrow{\eqref{dim=rank}} &\text{ there exists some $e({\bf i})^{\hO}\neq 0$ such that ${\bf i}_k=\tilde{a}$}\\
        \xLeftrightarrow{\eqref{eO}} &\begin{matrix}\text{ there exists ${\rm T}=(\mt,\alpha_\mt,\beta_\mt) \in {\rm Tri}_{\bar{0}}({\bf i})$ such that $F_{\rm T}\neq 0$,}\\
        \text{ where ${\bf i}_k=\mathtt{b}_{\mt,\beta_{\mt}}^{-{\rm sgn}_{\beta_{\mt}}(k)}|_{x=0}=\tilde{a}$}\end{matrix}\\
%		\xLeftrightarrow{\eqref{dim 1}}&\begin{matrix}\text{ there exists some $\undla\in\mathscr{P}^{\bullet,m}_{n}$ and ${\rm T}=(\mt,\alpha_\mt,\beta_\mt)\in{\rm Tri}_{\bar{0}}(\undla)$}\\
%			\text{ such that ${\bf i}^{\rm T}_k=\tilde{a}$}
%			\end{matrix}\\
%			\xLeftrightarrow{\eqref{defn:iT}}&\begin{matrix}\text{ there exists some $\undla\in\mathscr{P}^{\bullet,m}_{n}$ and ${\rm T}=(\mt,\alpha_\mt,\beta_\mt)\in{\rm Tri}_{\bar{0}}(\undla)$}\\
%				\text{ such that $\mathtt{b}_{\mt,\beta_{\mt}}|_{x=0}=\tilde{a}$}
%			\end{matrix}\\
		\xLeftrightarrow{\text{Definition \ref{def seminormal}, Theorem \ref{seminormal basis} and \eqref{X acts on f}}}&\text{$\tilde{a}=a|_{x=0}$, where $a=\mathtt{b}_{\mt,\beta_{\mt}}^{-{\rm sgn}_{\beta_{\mt}}(k)}$ is an eigenvalue of $X_k$ on $\mathcal{H}^{f'}_{\mathscr{K}}$.}
		\end{align*}
		This proves the Corollary.
		\end{proof}
\section{Generalized graded super cellular bases for cyclotomic quiver Hecke-Clifford superalgebras}\label{Generalized graded super cellular bases for cyclotomic quiver Hecke-Clifford superalgebra}
{\bf Throughout this section, we fix $n\in \N, q^2\neq \pm 1$, $\undQ=(Q_1,\cdots,Q_m)\in({\mathbb{K}}^*)^m$. We set $\mathcal{H}^f_{\mathbb{K}}=\mathcal{H}^{f}_{\mathbb{K}}(n)$, where $f=f^{\mathsf{(0)}}_{\underline{Q}}=\prod_{i=1}^m \biggl(X_1+X^{-1}_1-\mathtt{q}(Q_i)\biggr)$. Let $x$ be an indeterminant, we set $\hO:=\mathbb{K}[[x]]=\{a_0+a_1x+a_2x^2+\cdots|~a_i\in \mathbb{K}\}$ and $\hK$ be the fraction field of $\hO$. We modify the parameters as follows:
%	\footnote{Note again that $Q_i$ here corresponds to $qQ_i$ in \cite{SW}.}
$q':=x^4+q,$ $Q'_i:=x^{8ni}+Q_i, 1\leq i\leq m$. Then we can define $\mathcal{H}^{f'}_{\hO}=\mathcal{H}^{f'}_{\hO}(n)$, where $f'=f^{\mathsf{(0)}}_{\underline{Q'}}=\prod_{i=1}^m \biggl(X_1+X^{-1}_1-\mathtt{q}(Q_i')\biggr)$. Similarly, we can define $\mathcal{H}^{f'}_{\hK}=\mathcal{H}^{f'}_{\hK}(n)$.} Then we have $$
\mathcal{H}^{f'}_{\hK}\cong \hK\otimes_{\hO} \mathcal{H}^{f'}_{\hO},\qquad \mathcal{H}^f_{\mathbb{K}}\cong \mathbb{K}\otimes_{\hO} \mathcal{H}^{f'}_{\hO}.
$$ {\bf Accordingly, we define the residues of boxes in the young diagram $\undla$ via \eqref{eq:residue} as well as $\res(\mathfrak{t})$ for each $\mathfrak{t}\in\Std(\undla)$ with $\undla\in\mathscr{P}^{\bullet,m}_{n}$ with $m\geq 0$ with respect to parameters $(q',Q'_1,\cdots,Q'_m)$. }

Again, $\mathcal{H}^{f'}_{\hK}$ is semisimple over $\hK$, all of the eigenvalues $\mathtt{b}_{\pm}(\res_{\mt}(k))$ of $X_k$ belong to $\mathbb{K}[[x^2]]\subset\hO$ and all of the coefficients $\mathtt{c}_{\ms,\mt}\in \hK$. For any $a\in \hO$, we still use $a|_{x=0}\in \mathbb{K}$ to denote the image of $a$ in the residue field $\mathbb{K}\cong \hO/(x)$. {\bf We shall identify $\mathcal{H}^f_{\mathbb{K}}$ with the cyclotomic quiver Hecke-Clifford superalgebra $RC^{\Lambda_f}_n$ by Theorem \ref{KKTiso}. The aim of this section is to construct certain generalized graded cellular bases for $\mathcal{H}^f_{\mathbb{K}}$.}
\label{pag:Add and Rem}
\begin{defn}\label{tri-Q-AddandRem}
For $\undla\in\mathscr{P}^{m}_{n}$, $\mt\in \Std(\undla),$ $k\in[n],$ we define
\begin{align*}
	\mathscr{A}_{\mt}^{\rhd,\undQ}(k)&:=\left\{(A,*)\mid A\in \mathscr{A}_{\mt}^{\rhd}(k), *\in\{\pm\},\mathtt{b}_{*}(\res(A))|_{x=0}=\mathtt{b}_{+}(\res_\mt(k))|_{x=0}\right\},\\
	\mathscr{R}_{\mt}^{\rhd,\undQ}(k)&:=\left\{(A,*)\mid A\in \mathscr{R}_{\mt}^{\rhd}(k), *\in\{\pm\},\mathtt{b}_{*}(\res(A))|_{x=0}=\mathtt{b}_{+}(\res_\mt(k))|_{x=0}\right\},\\
	\mathscr{A}_{\mt}^{\lhd,\undQ}(k)&:=\left\{(A,*)\mid A\in \mathscr{A}_{\mt}^{\lhd}(k), *\in\{\pm\},\mathtt{b}_{*}(\res(A))|_{x=0}=\mathtt{b}_{+}(\res_\mt(k))|_{x=0}\right\},\\
		\mathscr{R}_{\mt}^{\lhd,\undQ}(k)&:=\left\{(A,*)\mid A\in \mathscr{R}_{\mt}^{\lhd}(k), *\in\{\pm\},\mathtt{b}_{*}(\res(A))|_{x=0}=\mathtt{b}_{+}(\res_\mt(k))|_{x=0}\right\}.
\end{align*}
%We also denote $\mathscr{A}_{\rm T}^{\rhd}(k)=\mathscr{A}_{\mt}^{\rhd}(k)$ if $\beta_\mt=0.$ Similarly for
%$\mathscr{R}_{\mt}^{\rhd}(k),$ $\mathscr{A}_{\mt}^{\lhd}(k)$ and $\mathscr{R}_{\mt}^{\lhd}(k)$ respectively.
\end{defn}

\subsection{Seminormal bases and integral bases}\label{Seminormal bases and integral bases}
In this subsection, we shall define some explicit elements in $\mathcal{H}^{f'}_{\hO}$. We will study the linear expansion of these elements via seminormal bases and finally prove that they give some integral bases for $\mathcal{H}^{f'}_{\hO}$.
\label{pag:O set}
\begin{defn}
	For any $\undla\in\mathscr{P}^{m}_{n}$, we define
    $$O_{\undla}:=\{\alpha\in\undla \mid b_{+}(\res(\alpha))|_{x=0}\in\{\pm 1\} \}.$$
    Similarly, for any $\mt\in\Std(\undla)$, we define$$O_{\mt}:=\{1\leq k\leq n \mid b_{+}(\res_\mt(k))|_{x=0}\in\{\pm 1\} \}.$$
    We define the Clifford algebra corresponds to $O_{\mt}$
    $$\mathcal{C}_\mt=\langle C_k\mid k\in O_{\mt}\rangle \subseteq \mathcal{C}_n$$
	and the set of colored multipartition with respect to $(q^2,Q_1,\cdots,Q_m)$ as
    $$\mathscr{P}^{\undQ}_{n}:=\{(\undla,S)\mid \undla\in\mathscr{P}^{m}_{n}, S\subset O_{\undla}\}.$$
\end{defn}

\label{pag:middle 1}
\begin{defn}
	For any $\undla\in\mathscr{P}^{m}_{n}$, we define $$
	{\bf i}_{{\undla}}:=\left(\mathtt{b}_+(\res_{\mt_{\undla}}(1)),\cdots,\mathtt{b}_+(\res_{\mt_{\undla}}(n))\right)|_{x=0}\in({\mathbb{K}}^*)^n,$$
$$ y^{\lhd,\hO}_{\undla}(k)
:=\prod_{(A,*)\in \mathscr{A}^{\lhd,\undQ}_{\mt_{\undla}}(k)}\left(\left(X_k-\mathtt{b}_*(\res(A))\right)f_{k,{\bf i}_{\undla}}(X_1,\cdots,X_n)\right)\in \mathcal{H}^{f'}_{\hO},$$
	and
$$y^{\lhd,\hO}_{\undla}:=\prod_{k=1}^n y^{\lhd,\hO}_{\undla}(k)\in \mathcal{H}^{f'}_{\hO}.$$
\end{defn}

\begin{defn}
	For any $\undla\in\mathscr{P}^{m}_{n}$, we define
$$y^\lhd_{\undla}:=\prod_{k=1}^n y_k^{\sharp  \mathscr{A}^{\lhd,\undQ}_{\mt_{\undla}}(k)}\in\mathcal{H}^f_{\mathbb{K}}.$$
\end{defn}
By Theorem \ref{KKTiso}, we have $1\otimes_{\mathbb{K}} \left(y^{\lhd,\hO}_{\undla}e({\bf i}_{{\undla}})^{\hO}\right)=y^\lhd_{\undla}e({\bf i}_{{\undla}}).$

\label{pag:index set of cellular bases}
\begin{defn}
Let $(\undla,S)\in\mathscr{P}^{\undQ}_{n}$. We define
\begin{align}\label{T(lambda,S)}
\mathscr{T}(\undla,S):=\{(\mt,\beta_\mt,S)\mid \mt\in\Std(\undla),\beta_\mt \in \Z_2^n \text{ such that } \supp(\beta_\mt)\cap O_{\mt}=\emptyset\}.
\end{align}
\end{defn}

If $S$ has been fixed in the context, we shall write $(\mt,\beta_\mt)\in \mathscr{T}(\undla,S)$ rather $(\mt,\beta_\mt,S)\in \mathscr{T}(\undla,S)$ to simplify notation.

\label{pag:sgn of vector}
For $\alpha\in \Z_2^n$, we set $$\sgn(\alpha):=(-1)^{\frac{|\supp(\alpha)|(|\supp(\alpha)|-1)}{2}}.$$
\label{pag:middle part 1}
\begin{defn}
Let $(\undla,S)\in\mathscr{P}^{\undQ}_{n}$. For any $L_1=(\mt_{\undla},\beta_1), L_2=(\mt_{\undla},\beta_2)\in  \mathscr{T}(\undla,S)
	$ and any $u\in\mathcal{C}_{\mt_{\undla}}$, we define$$
	y^{\lhd,S,\hO}_{L_1,u,L_2}:=\sgn(\beta_1)C^{\beta_1}\cdot u\cdot y^{\lhd,\hO}_{\undla}e(	{\bf i}_{{\undla}})^{\hO} \prod_{k\in \mathfrak{t}_{\undla}(S)}\left(\left(X_k-\mathtt{b}_{-}(\res_{{\mt}_{\undla}}(k))\right)f_{k,{\bf i}_{\undla}}(X_1,\cdots,X_n)\right)\cdot C^{\beta_2}\in  \mathcal{H}^{f'}_{\hO}.
	$$ and
\begin{align}\label{middle part}
y^{\lhd,S}_{L_1,u,L_2}:=\sgn(\beta_1)C^{\beta_1}\cdot u \cdot y^{\lhd}_{\undla}e({\bf i}_{{\undla}})\left(\prod_{k\in \mathfrak{t}_{\undla}(S)}y_k\right)\cdot C^{\beta_2}\in  \mathcal{H}^f_{\mathbb{K}}.
\end{align}
In particular, for any monomials $C^\alpha,C^{\alpha'}\in \mathcal{C}_{\mt_{\undla}}$, we use notations $$
		y^{\lhd,S,\hO}_{L_1,\alpha,L_2}:=	y^{\lhd,S,\hO}_{L_1,C^\alpha,L_2},\qquad y^{\lhd,S}_{L_1,\alpha,L_2}:=	y^{\lhd,S}_{L_1,C^\alpha,L_2}
	$$ and $$
	y^{\lhd,S,\hO}_{L_1,\alpha\cdot\alpha',L_2}:=	y^{\lhd,S,\hO}_{L_1,C^\alpha\cdot C^{\alpha'},L_2},\qquad y^{\lhd,S}_{L_1,\alpha\cdot\alpha',L_2}:=	y^{\lhd,S}_{L_1,C^\alpha\cdot C^{\alpha'},L_2}.
	$$
\end{defn}
By Theorem \ref{KKTiso} again, we have $1\otimes_{\mathbb{K}}y^{\lhd,S,\hO}_{L_1,u,L_2}=y^{\lhd,S}_{L_1,u,L_2}.$

\begin{lem}\label{first expansion}
	Keep the notations as in above definitions, we have  \begin{align*}
		y^{\lhd,S,\hO}_{L_1,\alpha,L_2}\in &\prod_{k=1}^n\prod_{A\in\mathscr{A}^{\lhd}_{\mt_{\undla}}(k)}\biggl(\mathtt{q}(\res_{\mathfrak{t}_{\undla}}(k))-\mathtt{q}(\res(A))\biggr)\cdot \prod_{\substack{k\in \bigl([n]\setminus O_{\mt_{\undla}}\bigr)\bigsqcup \mt_{\undla}(S)}}	\biggl(b_+(\res_{\mathfrak{t}_{\undla}}(k))-b_-(\res_{\mathfrak{t}_{\undla}}(k))\biggr)	\nonumber\\
		&\qquad	\cdot \left(\sum_{\substack{\tilde{L_1}=(\mathfrak{t}_{\undla},\tilde{\beta}_1),\tilde{L_2}=(\mathfrak{t}_{\undla},\tilde{\beta}_2)\\ \tilde{\beta_1}=\beta_1+\alpha+\beta_{\mathfrak{t}_{\undla}} ,\,\tilde{\beta_2}=\beta_2+\beta_{\mathfrak{t}_{\undla}}\\
\beta_{\mathfrak{t}_{\undla}}\in \Z_2^n ,\,{\rm supp}(\beta_{\mathfrak{t}_{\undla}})\subset O_{\mathfrak{t}_{\undla}}\backslash \mathfrak{t}_{\undla}(S)\\}}\hO^\times f_{\tilde{L_1},\tilde{L_2}}\right)+\sum_{\substack{\tilde{L_1}=(\mathfrak{u},\beta''_{\mathfrak{u}}),\tilde{L_2}=(\mathfrak{v},\beta'''_{\mathfrak{v}})\in {\rm Tri}(\mathscr{P}^{m}_{n})\\ \mathfrak{u},\mathfrak{v}\lhd \mathfrak{t}_{\undla}}}\hO f_{\tilde{L_1},\tilde{L_2}}.\nonumber
	\end{align*}
\end{lem}

\begin{proof}
The proof is inspired by \cite[Lemma 4E.5]{EM}. By definition, we have
\begin{align}\label{eq1.first expansion}
y^{\lhd,\hO}_{\undla}e({\bf i}_{{\undla}})^{\hO}
=\sum_{{\rm T}=(\mt,\beta_{\mt})\in {\rm Tri}({\bf i}_{\undla})}
 \prod_{k=1}^n \prod_{(A,*)\in \mathscr{A}^{\lhd,\undQ}_{\mt_{\undla}}(k)}\left(\mathtt{b}_{\mt,k}^{-{\rm sgn}_{\beta_{\mt}}(k)}-\mathtt{b}_*(\res(A))\right)f_{k,{\bf i}_{\undla}}\left(\mathtt{b}_{\mt,1}^{-{\rm sgn}_{\beta_{\mt}}(1)},\cdots,\mathtt{b}_{\mt,n}^{-{\rm sgn}_{\beta_{\mt}}(n)}\right)F_{\rm T},
\end{align}
where the sequence $\mathtt{b}_{\mt,\beta_{\mt}}|_{x=0}={\bf i}_{\undla}$ and thus $f_{k,{\bf i}_{\undla}}\left(\mathtt{b}_{\mt,1}^{-{\rm sgn}_{\beta_{\mt}}(1)},\cdots,\mathtt{b}_{\mt,n}^{-{\rm sgn}_{\beta_{\mt}}(n)}\right)\in \hO^{\times}$ by Theorem \ref{KKTiso}.

For any ${\rm T}=(\mt,\beta_{\mt})\in {\rm Tri}({\bf i}_{\undla}),$ if $\mt \ntrianglelefteq \mt_{\undla},$ then there is a minimal number $k\in[n]$ such that $\mt\downarrow_k \ntrianglelefteq \mt_{\undla}\downarrow_k.$ Let $A=\mt^{-1}(k),$ then we have
$(A,{\rm sgn}_{\beta_{\mt}}(k))\in \mathscr{A}^{\lhd,\undQ}_{\mt_{\undla}}(k),$ it follows that the coefficient of $F_{\rm T}$ in
\eqref{eq1.first expansion} is zero. For any ${\rm T}=(\mt_{\undla},\beta_{\mt_{\undla}})\in {\rm Tri}({\bf i}_{\undla}),$
since $\mathtt{b}_{\mt_{\undla},\beta_{\mt_{\undla}}}|_{x=0}={\bf i}_{\undla}=\mathtt{b}_{\mt_{\undla},0}|_{x=0},$ we must have $\supp(\beta_{\mt_{\undla}})\subseteq O_{\mt_{\undla}}.$
Combining above, it follows that
\begin{align}\label{eq2.first expansion}
y^{\lhd,\hO}_{\undla}e({\bf i}_{{\undla}})^{\hO}
\in \sum_{\substack{{\rm T}=(\mt_{\undla},\beta_{\mt_{\undla}})\in \{\mathfrak{t}_{\undla}\}\times \Z_2^n\\
			 \supp(\beta_{\mt_{\undla}})\subset O_{\mt_{\undla}}}}
& \prod_{k=1}^n \prod_{(A,*)\in \mathscr{A}^{\lhd,\undQ}_{\mt_{\undla}}(k)}\left(\mathtt{b}_{\mt_{\undla},k}^{-{\rm sgn}_{\beta_{\mt_{\undla}}}(k)}-\mathtt{b}_*(\res(A))\right)\hO^{\times}
 F_{\rm T}\\
 &+\sum_{\substack{{\rm U}=(\mt,\beta_\mt)\in {\rm Tri}({\bf i}_{\undla})\\ \mt \lhd \mathfrak{t}_{\undla}}}\hO F_{\rm U}.\nonumber
\end{align}

Note that for $A \in\mathscr{A}^{\lhd}_{\mt_{\undla}}(k)$ with $\mathtt{b}_\pm (\res(A))|_{x=0}\neq b_+(\res_{\mathfrak{t}_{\undla}}(k))|_{x=0}$, then $\mathtt{q}(\res_{\mathfrak{t}_{\undla}}(k))-\mathtt{q}(\res(A))\in \hO^\times.$
For $(A,+)\left(\text{resp. $(A,-)$}\right)\in \mathscr{A}^{\lhd,\undQ}_{\mt_{\undla}}(k)$ with $k\notin O_{\mt_{\undla}},$ we have $\mathtt{b}_-(\res(A))-\mathtt{b}_+(\res_{\mathfrak{t}_{\undla}}(k))\left(\text{resp. $\mathtt{b}_+(\res(A))-\mathtt{b}_+(\res_{\mathfrak{t}_{\undla}}(k))$}\right)\in\hO^\times.$
If $(A,+),\,(A,-)\in \mathscr{A}^{\lhd,\undQ}_{\mt_{\undla}}(k),$ then $k\in O_{\mt_{\undla}}.$
By \eqref{eq2.first expansion} and above observations, we deduce that
\begin{align*}
y^{\lhd,\hO}_{\undla}e({\bf i}_{{\undla}})^{\hO}\in \prod_{k=1}^n
\prod_{A\in\mathscr{A}^{\lhd}_{\mt_{\undla}}(k)}
& \biggl(\mathtt{q}(\res_{\mathfrak{t}_{\undla}}(k))-\mathtt{q}(\res(A))\biggr)\cdot\left(\sum_{\substack{{\rm T}=(\mt_{\undla},\beta_{\mt_{\undla}})\in \{\mathfrak{t}_{\undla}\}\times \Z_2^n\\
			 \supp(\beta_{\mt_{\undla}})\subset O_{\mt_{\undla}}}} \hO^\times F_{\rm T}\right) \\
&+\sum_{\substack{{\rm U}=(\mt,\beta_\mt)\in {\rm Tri}({\bf i}_{\undla})\\ \mt \lhd \mathfrak{t}_{\undla}}}\hO F_{\rm U}.
\end{align*}
One can easily see that $\prod_{\substack{k\notin O_{\mathfrak{t}_{\undla}}}}	\biggl(b_+(\res_{\mathfrak{t}_{\undla}}(k)-b_-(\res_{\mathfrak{t}_{\undla}}(k))\biggr)\in\hO^\times$,
	hence, we deduce \begin{align*}
	y^{\lhd,\hO}_{\undla}e(	{\bf i}_{{\undla}})^{\hO} &\prod_{k\in \mathfrak{t}_{\undla}(S)}\left( \left(X_k-\mathtt{b}_-(\res_\mt(k))\right)f_{k,{\bf i}_{\undla}}(X_1,\cdots,X_n)\right)\\
		&\in \prod_{k=1}^n\prod_{A\in\mathscr{A}^{\lhd}_{\mt_{\undla}}(k)}\biggl(\mathtt{q}(\res_{\mathfrak{t}_{\undla}}(k))-\mathtt{q}(\res(A))\biggr)\cdot \prod_{\substack{k\notin O_{\mathfrak{t}_{\undla}}}}	\biggl(b_+(\res_{\mathfrak{t}_{\undla}}(k))-b_-(\res_{\mathfrak{t}_{\undla}}(k))\biggr)	\nonumber\\
		&\qquad\qquad	\cdot \prod_{\substack{k\in \mathfrak{t}_{\undla}(S)}}	\biggl(b_+(\res_{\mathfrak{t}_{\undla}}(k))-b_-(\res_{\mathfrak{t}_{\undla}}(k))\biggr)\cdot\left(\sum_{\substack{{\rm T}=(\mathfrak{t}_{\undla},\beta_{\mathfrak{t}_{\undla}})\in \{\mathfrak{t}_{\undla}\}\times \Z_2^n\\ \supp(\beta_{\mathfrak{t}_{\undla}})\subset O_{\mathfrak{t}_{\undla}}\backslash \mathfrak{t}_{\undla}(S) }}\hO^\times F_{\rm T}\right)\nonumber\\
		&\qquad\qquad\qquad+\sum_{\substack{{\rm U}=(\mt,\beta_\mt)\in {\rm Tri}({\bf i}_{\undla})\\ \mt \lhd \mathfrak{t}_{\undla}}}\hO F_{\rm U}.\nonumber
	\end{align*}
	Now, the Lemma follows from \eqref{C acts on f}.
\end{proof}

\label{pag:O-basis element 1}
\begin{defn}
	Let $(\undla,S)\in\mathscr{P}^{\undQ}_{n}$. For any $L'_1=(\ms,\beta_\ms), L'_2=(\mt,\beta_\mt)\in  \mathscr{T}(\undla,S),
	$ there are unique $L_1=(\mathfrak{t}_{\undla},\beta_1), L_2=(\mathfrak{t}_{\undla},\beta_2)\in  \mathscr{T}(\undla,S)
	$ and $w_1,w_2\in \mathfrak{S}_n$ such that $L'_1=w_1L_1$ and $L'_2=w_2L_2$.  We fix a reduced expression $w_i=s_{t^i_{k_1}}\cdots s_{t^i_1}$ and use this to define $\sigma_{w_i}$ for $i=1,2.$ For any $u\in\mathcal{C}_{\mt_{\undla}}$, we define
\begin{align*}
\psi^{\lhd,S,\hO}_{L'_1,u,L'_2}
:=\sigma_{w_1}^{\hO}y^{\lhd,S,\hO}_{L_1,u,L_2}(\sigma_{w_2}^{\hO})^{*}
		\in \mathcal{H}^{f'}_{\hO},
	\end{align*}
where
\begin{align*}
&\sigma_{w_1}^{\hO}e(c^{\beta_1}{\bf i}_{\undla})^{\hO}
:=\\
&\overleftarrow{\prod}_{\substack{j=1,\cdots,k_1}}\left(T_{t^1_{j}} (r_{t^1_{j},s_{t^1_{j-1}}\cdots s_{t^1_1}c^{\beta_1}{\bf i}_{\undla}}(X_1,X_2,\cdots,X_n))
	+\sum_{{\bf j}\in (J_f)^n}m_{t^1_{j},s_{t^1_{j-1}}\cdots s_{t^1_1}c^{\beta_1}{\bf i}_{\undla}}^{\bf j}\right)e(c^{\beta_1}{\bf i}_{\undla})^{\hO},\\
&\sigma_{w_2}^{\hO}e(c^{\beta_2}{\bf i}_{\undla})^{\hO}
:=\\
&\overrightarrow{\prod}_{\substack{j=1,\cdots,k_2}}\left(T_{t^2_{j}} (r_{t^2_{j},s_{t^2_{j-1}}\cdots s_{t^2_1}c^{\beta_2}{\bf i}_{\undla}}(X_1,X_2,\cdots,X_n))
	+\sum_{{\bf j}\in (J_f)^n}m_{t^2_{j},s_{t^2_{j-1}}\cdots s_{t^2_1}c^{\beta_2}{\bf i}_{\undla}}^{\bf j}\right)e(c^{\beta_2}{\bf i}_{\undla})^{\hO}.
\end{align*}
And \label{pag:K-basis element 1}
    $$
		\psi^{\lhd,S}_{L'_1,u,L'_2}:=\sigma_{w_1}y^{\lhd,S}_{L_1,u,L_2}\left(\sigma_{w_2}\right)^*\in \mathcal{H}^f_{\mathbb{K}}.
	$$  Again, for any monomial $C^\alpha,C^{\alpha'}\in \mathcal{C}_{\mt_{\undla}}$, we use notations $$
\psi^{\lhd,S,\hO}_{L'_1,\alpha,L'_2}:=	\psi^{\lhd,S,\hO}_{L'_1,C^\alpha,L'_2},\qquad \psi^{\lhd,S}_{L'_1,\alpha,L'_2}:=	\psi^{\lhd,S}_{L'_1,C^\alpha,L'_2}
	$$ and $$
	\psi^{\lhd,S,\hO}_{L'_1,\alpha\cdot \alpha',L'_2}:=	\psi^{\lhd,S,\hO}_{L'_1,C^\alpha\cdot C^{\alpha'},L'_2},\qquad \psi^{\lhd,S}_{L'_1,\alpha\cdot \alpha',L'_2}:=	\psi^{\lhd,S}_{L'_1,C^\alpha\cdot C^{\alpha'},L'_2}.
	$$
\end{defn}

By Theorem \ref{KKTiso}, we have $$1\otimes_{\mathbb{K}}\psi^{\lhd,S,\hO}_{L'_1,u,L'_2}=\psi^{\lhd,S}_{L'_1,u,L'_2}
$$
\begin{lem}\label{second expansion}
	Keep the notations as in above definitions, we have \begin{align*}
		\psi^{\lhd,S,\hO}_{L'_1,\alpha,L'_2}\in &\frac{\mathtt{c}_{\ms,\mathfrak{t}_{\undla} }\mathtt{c}_{ \mathfrak{t}_{\undla},\mt }}{\mathtt{c}_{\ms, \mt}}\prod_{k=1}^n\prod_{A\in\mathscr{A}^{\lhd}_{\mt_{\undla}}(k)}\left(\mathtt{q}(\res_{\mathfrak{t}_{\undla}}(k))-\mathtt{q}(\res(A))\right)	\nonumber\\
		&\qquad\cdot \prod_{\substack{k\in \bigl([n]\setminus O_{\mathfrak{t}_{\undla}}\bigr)\bigsqcup \mathfrak{t}_{\undla}(S)}}	\left(b_+(\res_{\mathfrak{t}_{\undla}}(k))-b_-(\res_{\mathfrak{t}_{\undla}}(k))\right)	\\
		&\qquad\qquad\qquad	\cdot\left(\sum_{\substack{\tilde{L_1}=(\ms,\tilde{\beta}_1),\tilde{L_2}=(\mt,\tilde{\beta}_2)\\ \tilde{\beta_1}=\beta_\ms+w_1\cdot\alpha+w_1\cdot\beta_{\mathfrak{t}_{\undla}} ,\,\tilde{\beta_2}=\beta_\mt+w_2\cdot\beta_{\mathfrak{t}_{\undla}}\\ \beta_{\mathfrak{t}_{\undla}}\in \Z_2^n ,\,{\rm supp}(\beta_{\mathfrak{t}_{\undla}})\subset O_{\mathfrak{t}_{\undla}}\setminus\mathfrak{t}_{\undla}(S)}}\hO^\times f_{\tilde{L_1},\tilde{L_2}}\right)\nonumber\\
		&\qquad\qquad\qquad\qquad\qquad\qquad\qquad+\sum_{\substack{\tilde{L_1}=(\mathfrak{u},\beta''_{\mathfrak{u}}),\tilde{L_2}=(\mathfrak{v},\beta'''_{\mathfrak{v}})\\ (\mathfrak{u},\mathfrak{v})\unlhd (\ms,\mt),(\mathfrak{u},\mathfrak{v})\neq (\ms,\mt)}}\hK f_{\tilde{L_1},\tilde{L_2}}.\nonumber
	\end{align*}
\end{lem}

\begin{proof}
The proof is argued by an induction on the dominance order $\unlhd,$ which is similar to \cite[Lemma 4E.6]{EM}.
Then it follows from Lemma \ref{first expansion}, \ref{c-coefficients. non-dege.} and Proposition \ref{generators action on seminormal basis}.
\end{proof}

Similarly, we can give the ``dual''  construction of the above definitions.
\label{pag:middle 2}
\begin{defn}
	For any $\undla\in\mathscr{P}^{m}_{n}$, we define
$${\bf i}^{{\undla}}:=\left(\mathtt{b}_+(\res_{\mt^{\undla}}(1)),\cdots,\mathtt{b}_+(\res_{\mt^{\undla}}(n))\right)|_{x=0}\in({\mathbb{K}}^*)^n,$$
	$$	y^{\rhd,\hO}_{\undla}(k):=
	\prod_{(A,*)\in \mathscr{A}^{\rhd,\undQ}_{\mt^{\undla}}(k)}\left(\left(X_k-\mathtt{b}_*(\res(A))\right)f_{k,{\bf i}^{\undla}}(X_1,\cdots,X_n)\right)\in \mathcal{H}^{f'}_{\hO},
	$$
	and $$y^{\rhd,\hO}_{\undla}=\prod_{k=1}^n y^{\rhd,\hO}_{\undla}(k)\in  \mathcal{H}^{f'}_{\hO}.
	$$
\end{defn}

\begin{defn}
	For any $\undla\in\mathscr{P}^{m}_{n}$, we define $$y^\rhd_{\undla}=\prod_{k=1}^n y_k^{\sharp  \mathscr{A}^{\rhd,\undQ}_{\mt^{\undla}}(k)}\in \mathcal{H}^f_{\mathbb{K}}.$$
\end{defn}

%\begin{defn}
%	For $\undla\in \mathscr{P}^{m}_{n}$, we define $$
%	X^{\rhd}_{\undla}:=\prod_{k=1}^n\prod_{A\in\Add^\rhd_{\mathfrak{t}^{\undla}}(k)}(X_K+X^{-1}_k-\mathtt{q}(\res(A)))\in  \mathcal{H}^{f'}_{\hO}.
%	$$
%\end{defn}
\label{pag:middle part 2}
\begin{defn}
Let $(\undla,S)\in\mathscr{P}^{\undQ}_{n}$. For any $L_1=(\mt^{\undla},\beta_1), L_2=(\mt^{\undla},\beta_2)\in  \mathscr{T}(\undla,S)
	$ and  any $u\in\mathcal{C}_{\mt^{\undla}}$, we define $$
	y^{\rhd,S,\hO}_{L_1,u,L_2}:=\sgn(\beta_1)C^{\beta_1} \cdot y^{\rhd,\hO}_{\undla}e({\bf i}^{{\undla}})^{\hO} \prod_{k\in \mathfrak{t}^{\undla}(S)}\left(\left(X_k-\mathtt{b}_+(\res_\mt(k))\right)f_{k,{\bf i}^{\undla}}(X_1,\cdots,X_n)\right)\cdot u\cdot C^{\beta_2}\in  \mathcal{H}^{f'}_{\hO}.
	$$ and
\begin{align}\label{middle part for rhd}	
y^{\rhd,S}_{L_1,u,L_2}:=\sgn(\beta_1)C^{\beta_1}\cdot y^{\rhd}_{\undla}e({\bf i}^{{\undla}})\left(\prod_{k\in \mathfrak{t}^{\undla}(S)}y_k\right)\cdot u \cdot C^{\beta_2}\in  \mathcal{H}^f_{\mathbb{K}}.
\end{align}
In particular, for any monomials $C^\alpha,C^{\alpha'}\in \mathcal{C}_{\mt^{\undla}}$, we use notations $$
	y^{\rhd,S,\hO}_{L_1,\alpha,L_2}:=	y^{\rhd,S,\hO}_{L_1,C^\alpha,L_2},\qquad y^{\rhd,S}_{L_1,\alpha,L_2}:=	y^{\rhd,S}_{L_1,C^\alpha,L_2}
	$$ and $$
	y^{\rhd,S,\hO}_{L_1,\alpha\cdot\alpha',L_2}:=	y^{\rhd,S,\hO}_{L_1,C^\alpha\cdot C^{\alpha'},L_2},\qquad y^{\rhd,S}_{L_1,\alpha\cdot\alpha',L_2}:=	y^{\rhd,S}_{L_1,C^\alpha\cdot C^{\alpha'},L_2}.
	$$
\end{defn}
By Theorem \ref{KKTiso}, we have $1\otimes _{\hO}\left(y^{\rhd,\hO}_{\undla}e(	{\bf i}^{{\undla}})^{\hO}\right)=y^\rhd_{\undla}e(	{\bf i}^{{\undla}})$ and $1\otimes _{\hO}	y^{\rhd,S,\hO}_{L_1,u,L_2}=	y^{\rhd,S}_{L_1,u,L_2}.$

\begin{defn}
	Let $(\undla,S)\in\mathscr{P}^{\undQ}_{n}$. For any $L'_1=(\ms,\beta_\ms), L'_2=(\mt,\beta_\mt)\in   \mathscr{T}(\undla,S)
	$ there are unique $L_1=(\mathfrak{t}^{\undla},\beta_1), L_2=(\mathfrak{t}^{\undla},\beta_2)\in  \mathscr{T}(\undla,S)
	$ and $w_1,w_2\in \mathfrak{S}_n$ such that $L'_1=w_1L_1$ and $L'_2=w_2L_2$.  We fix a reduced expression $w_i=s_{t^i_{k_1}}\cdots s_{t^i_1}$ and use this to define $\sigma_{w_i}$ for $i=1,2.$ For any $u\in\mathcal{C}_{\mt^{\undla}}$, we define
\label{pag:O-basis element 2}
\begin{align*}
\psi^{\rhd,S,\hO}_{L'_1,u,L'_2}
:=\sigma_{w_1}^{\hO}y^{\rhd,S,\hO}_{L_1,u,L_2}(\sigma_{w_2}^{\hO})^{*}
		\in \mathcal{H}^{f'}_{\hO},
	\end{align*}
where
\begin{align*}
&\sigma_{w_1}^{\hO}e(c^{\beta_1}{\bf i}^{\undla})^{\hO}
:=\\
&\overleftarrow{\prod}_{\substack{j=1,\cdots,k_1}}\left(T_{t^1_{j}} (r_{t^1_{j},s_{t^1_{j-1}}\cdots s_{t^1_1}c^{\beta_1}{\bf i}^{\undla}}(X_1,X_2,\cdots,X_n))
	+\sum_{{\bf j}\in (J_f)^n}m_{t^1_{j},s_{t^1_{j-1}}\cdots s_{t^1_1}c^{\beta_1}{\bf i}^{\undla}}^{\bf j}\right)e(c^{\beta_1}{\bf i}^{\undla})^{\hO},\\
&\sigma_{w_2}^{\hO}e(c^{\beta_2}{\bf i}^{\undla})^{\hO}
:=\\
&\overrightarrow{\prod}_{\substack{j=1,\cdots,k_2}}\left(T_{t^2_{j}} (r_{t^2_{j},s_{t^2_{j-1}}\cdots s_{t^2_1}c^{\beta_2}{\bf i}^{\undla}}(X_1,X_2,\cdots,X_n))
	+\sum_{{\bf j}\in (J_f)^n}m_{t^2_{j},s_{t^2_{j-1}}\cdots s_{t^2_1}c^{\beta_2}{\bf i}^{\undla}}^{\bf j}\right)e(c^{\beta_2}{\bf i}^{\undla})^{\hO}.
\end{align*}
And \label{pag:K-basis element 2}
\begin{align}\label{rhd-basis}
	\psi^{\rhd,S}_{L'_1,u,L'_2}:=\sigma_{w_1}y^{\rhd,S}_{L_1,u,L_2}\left(\sigma_{w_2}\right)^*\in  \mathcal{H}^f_{\mathbb{K}}.
	\end{align}
Similarly, for any monomial $C^\alpha,C^{\alpha'}\in \mathcal{C}_{\mt^{\undla}}$, we use notations $$
	\psi^{\rhd,S,\hO}_{L'_1,\alpha,L'_2}:=	\psi^{\rhd,S,\hO}_{L'_1,C^\alpha,L'_2},\qquad \psi^{\rhd,S}_{L'_1,\alpha,L'_2}:=	\psi^{\rhd,S}_{L'_1,C^\alpha,L'_2}
	$$ and $$
	\psi^{\rhd,S,\hO}_{L'_1,\alpha\cdot \alpha',L'_2}:=	\psi^{\rhd,S,\hO}_{L'_1,C^\alpha\cdot C^{\alpha'},L'_2},\qquad \psi^{\rhd,S}_{L'_1,\alpha\cdot \alpha',L'_2}:=	\psi^{\rhd,S}_{L'_1,C^\alpha\cdot C^{\alpha'},L'_2}.
	$$
\end{defn}

By Theorem \ref{KKTiso}, we have $$1\otimes _{\hO}\psi^{\rhd,S,\hO}_{L'_1,u,L'_2}=\psi^{\rhd,S}_{L'_1,u,L'_2}.
$$

\begin{lem}\label{second expansion'}
	Keep the notations as above definitions, we have  \begin{align*}
		\psi^{\rhd,S,\hO}_{L'_1,\alpha,L'_2}\in &\frac{\mathtt{c}_{\ms,\mt^{\undla}}\mathtt{c}_{\mt^{\undla},\mt}}{\mathtt{c}_{\ms, \mt}}\prod_{k=1}^n\prod_{A\in\mathscr{A}^{\rhd}_{\mt_{\undla}}(k)}\left(\mathtt{q}(\res_{\mathfrak{t}^{\undla}}(k))-\mathtt{q}(\res(A))\right)	\nonumber\\
		&\qquad\cdot \prod_{\substack{k\in \bigl([n]\setminus O_{\mt^{\undla}}\bigr)\bigsqcup \mathfrak{t}^{\undla}(S)}}	\left(b_+(\res_{\mathfrak{t}^{\undla}}(k))-b_-(\res_{\mathfrak{t}^{\undla}}(k))\right)	\\
		&\qquad\qquad\qquad	\cdot\left(\sum_{\substack{\tilde{L_1}=(\ms,\tilde{\beta}_1),\tilde{L_2}=(\mt,\tilde{\beta}_2)\\ \tilde{\beta_1}=\beta_\ms+w_1\cdot\beta_{\mathfrak{t}^{\undla}} ,\,\tilde{\beta_2}=\beta_\mt+w_2\cdot\alpha+w_2\cdot\beta_{\mathfrak{t}^{\undla}}\\
\beta_{\mathfrak{t}^{\undla}}\in \Z_2^n ,\, \mathfrak{t}^{\undla}(S)\subset {\rm supp}(\beta_{\mathfrak{t}^{\undla}})\subset O_{\mathfrak{t}^{\undla}}}}\hO^\times f_{\tilde{L_1},\tilde{L_2}}\right)\nonumber\\
		&\qquad\qquad\qquad\qquad\qquad\qquad\qquad+\sum_{\substack{\tilde{L_1}=(\mathfrak{u},\beta''_{\mathfrak{u}}),\tilde{L_2}=(\mathfrak{v},\beta'''_{\mathfrak{v}})\\ (\mathfrak{u},\mathfrak{v})\unrhd (\ms,\mt),(\mathfrak{u},\mathfrak{v})\neq (\ms,\mt)}}\hK f_{\tilde{L_1},\tilde{L_2}}.\nonumber
	\end{align*}
\end{lem}

Recall that $I_f$ is associated with a generalized Cartan superdatum.
Throughout this section, we use $Q_n^+$ to denote the set of positive root lattice with height $n$ associated to $I_f$.
\label{pag:nu-multipartition}
\begin{defn}
Let $\nu\in Q_n^+.$
\begin{enumerate}
\item The set of $\nu$-multipartition is $\mathscr{P}^{m}_{\nu}:=\{\undla\in\mathscr{P}^{m}_n \mid \sum_{A\in \undla}\nu_{\mathtt{q}(\res(A))|_{x=0}}=\nu\}.$
\item The set of colored $\nu$-multipartition with respect to $(q,\undQ)$ is
    $$\mathscr{P}^{\undQ}_{\nu}:=\{(\undla,S)\mid \undla\in\mathscr{P}^{m}_{\nu}, S\subset O_{\undla}\}.$$
\end{enumerate}
\end{defn}

Now we introduce the key definition of this paper: ``$\undQ$-unremovable ''.
\begin{defn}\label{tech condition}
	Let $\nu\in Q_n^+$. We call {\bf $\nu$ is $\undQ$-unremovable} if for any $\undla\in \mathscr{P}^{m}_{\nu}$ and any $k\in [n]$, we have
	$$\mathscr{R}_{\mt_{\undla}}^{\lhd,\undQ}(k)=\emptyset.
	$$
	\end{defn}

The following Proposition gives a large class of example for $\undQ$-unremovable elements in $Q_n^+$.
\begin{prop}\label{example-unremovable}
 Let $\nu\in Q_n^+$ with $\nu=\sum_{i\in I_f}m_i v_i$. Suppose $m_i\leq 1$ for any $i\in (I_f)_{{\rm odd}}$, then $\nu$ is $\undQ$-unremovable. In particular, if $(I_f)_{{\rm odd}}=\emptyset$, then any $\nu\in Q_n^+$ is $\undQ$-unremovable.
		\end{prop}
		\begin{proof}
	Let $\undla\in \mathscr{P}^{m}_{\nu}$ and $k\in [n]$ such that $\mathscr{R}_{\mt_{\undla}}^{\lhd,\undQ}(k)\neq\emptyset.
		$ Suppose $(\mt_{\undla})^{-1}(k)=(i,j,l).$ Then for any $(A,*)\in \mathscr{R}_{\mt_{\undla}}^{\lhd,\undQ}(k),$ we have
		$A \in \{(i-1,j,l),(i,j-1,l)\}.$
		Therefore, we have $\mathtt{q}(\res(A))\in \{\mathtt{q}(q'^{2}\res_{\mt_{\undla}}(k)),\mathtt{q}({q'}^{- 2}\res_{\mt_{\undla}}(k))\}$.
It follows that either $\mathtt{q}(\res_{\mt_{\undla}}(k))|_{x=0}=\mathtt{q}(q'^2\res_{\mt_{\undla}}(k))|_{x=0}$ or $\mathtt{q}(\res_{\mt_{\undla}}(k))|_{x=0}=\mathtt{q}({q'}^{-2}\res_{\mt_{\undla}}(k))|_{x=0}$. In any cases, we can deduce that
		$\mathtt{q}(\res_{\mt_{\undla}}(k))|_{x=0}\in\{\pm 2\}$. Hence $m_{2}\geq 2$ or $m_{-2}\geq 2$, which contradicts to our assumption. This proves $\nu$ is $\undQ$-unremovable.
		\end{proof}

From now on, for $\nu\in Q_n^+,$ we set \label{pag:nu-Olift}$e_\nu ^{\hO}:=\sum_{{\bf i}\in J^\nu}e({\bf i})^\hO$ and shortly denote \label{pag:nu-block}
$$\mathcal{H}^f_{\mathbb{K}}(\nu):=e^J_\nu \mathcal{H}^f_{\mathbb{K}},\quad
\mathcal{H}^{f'}_{\hO}(\nu):=e_\nu ^{\hO}\mathcal{H}^{f'}_{\hO},\quad
\mathcal{H}^{f'}_{\hK}(\nu):=e_\nu ^{\hO}\mathcal{H}^{f'}_{\hK},$$	
\begin{lem}\label{MainLemmaforblock}
Suppose $\nu\in Q_n^+$ is $\undQ$-unremovable. Then the Gram matrix
$$\left(t^\hO_{r,n}\left(\psi^{\lhd,S,\hO}_{L'_1,\alpha,L'_2}\psi^{\rhd,T,\hO}_{L'_3,\alpha',L'_4}\right)\right)_{(S,L'_1,\alpha,L'_2),(T,L'_3,\alpha',L'_4)}$$
of elements \label{pag:O-basis1 for block}
\begin{equation}\label{basis1 for block} \Psi^{\hO,\lhd}_{\nu}:=\Bigl\{\psi^{\lhd,S,\hO}_{L'_1,\alpha,L'_2}\bigm|~\begin{matrix}
			(\undla,S)\in\mathscr{P}^{\undQ}_{\nu}, L'_1=(\ms,\beta_\ms), L'_2=(\mt,\beta_\mt)\in  \mathscr{T}(\undla,S)\\
			\alpha\in \Z_2^n,  {\rm supp}(\alpha)\subset O_{\mathfrak{t}_{\undla}}
		\end{matrix}\Bigr\}\end{equation}
and \label{pag:O-basis2 for block}
\begin{equation}\label{basis2 for block} \Psi^{\hO,\rhd}_{\nu}:=\Bigl\{\psi^{\rhd,S,\hO}_{L'_1,\alpha,L'_2}\bigm|~\begin{matrix}
			(\undla,S)\in\mathscr{P}^{\undQ}_{\nu}, L'_1=(\ms,\beta_\ms), L'_2=(\mt,\beta_\mt)\in  \mathscr{T}(\undla,S)\\
			\alpha\in \Z_2^n,  {\rm supp}(\alpha)\subset O_{\mathfrak{t}^{\undla}}
		\end{matrix}\Bigr\}\end{equation} is an invertible upper triangular matrix with each entry belongs to $\hO$.
\end{lem}
\begin{proof}
Let $\psi^{\lhd,S,\hO}_{L'_1,\alpha,L'_2}\in \Psi^{\hO,\lhd}_{\nu}$	and $\psi^{\rhd,T,\hO}_{L'_3,\alpha',L'_4}\in  \Psi^{\hO,\rhd}_{\nu}$, where $L'_1=(\ms,\beta_{\ms}), L'_2=(\mt,\beta'_{\mt})\in \mathscr{T}(\undla,S),L'_3=(\mathfrak{u},\beta''_{\mathfrak{u}}), L'_4=(\mathfrak{v},\beta'''_{\mathfrak{v}})\in \mathscr{T}(\undmu,T)$ and
$(\undla,S),(\undmu,T)\in\mathscr{P}^{\undQ}_{\nu}$. Then $\alpha,\alpha'\in \Z^n_2$ such that $\supp(\alpha)\subset O_{\mathfrak{t}_{\undla}},\supp(\alpha')\subset O_{\mt^{\undmu}}$.
 
	\begin{enumerate}
		\item Suppose $(\mt,\ms)\ntrianglerighteq (\mathfrak{u},\mathfrak{v})$. Then by Lemma \ref{second expansion}, Lemma \ref{second expansion'} and Theorem \ref{mainthm type0 nondege} (1), it is easy to see
$$t^\hO_{r,n}\left(\psi^{\lhd,S,\hO}_{L'_1,\alpha,L'_2}\psi^{\rhd,T,\hO}_{L'_3,\alpha',L'_4}\right)=0.$$
		\item Suppose $\mt=\mathfrak{u},\ms=\mathfrak{v}$ but $L'_2\neq L'_3$. Again, by Lemma \ref{second expansion}, Lemma \ref{second expansion'} and Theorem \ref{mainthm type0 nondege} (1), we have
$$t^\hO_{r,n}\left(\psi^{\lhd,S,\hO}_{L'_1,\alpha,L'_2}\psi^{\rhd,T,\hO}_{L'_3,\alpha',L'_4}\right)=0.$$
		\item Suppose $\mt=\mathfrak{u},\ms=\mathfrak{v}, \,L'_2=L'_3,$ but $T\not\subset O_{\undla}\setminus S.$ Then in this case, by Lemma \ref{second expansion}, Lemma \ref{second expansion'} and Theorem \ref{mainthm type0 nondege} (1), we also have
$$t^\hO_{r,n}\left(\psi^{\lhd,S,\hO}_{L'_1,\alpha,L'_2}\psi^{\rhd,T,\hO}_{L'_3,\alpha',L'_4}\right)=0.$$
		\item Suppose $\mt=\mathfrak{u},\ms=\mathfrak{v}, \,L'_2= L'_3$ and $T=O_{\undla}\setminus S$ exactly. In this case, we have  \begin{align*}
			t^\hO_{r,n}\left(\psi^{\lhd,S,\hO}_{L'_1,\alpha,L'_2}\psi^{\rhd,T,\hO}_{L'_3,\alpha',L'_4}\right)&=\frac{\mathtt{c}_{ \ms,\mathfrak{t}_{\undla}}\mathtt{c}_{\mathfrak{t}_{\undla},\mt }}{\mathtt{c}_{\ms, \mt}}\prod_{k=1}^n\prod_{A\in\mathscr{A}^{\lhd}_{\mt_{\undla}}(k)}\left(\mathtt{q}(\res_{\mathfrak{t}_{\undla}}(k))-\mathtt{q}(\res(A))\right)\\
			&\qquad\cdot \prod_{\substack{k\in \bigl([n]\setminus O_{\mathfrak{t}_{\undla}}\bigr)\bigsqcup \mathfrak{t}_{\undla}(S)}}	\left(b_+(\res_{\mathfrak{t}_{\undla}}(k))-b_-(\res_{\mathfrak{t}_{\undla}}(k))\right)\\
			&\qquad\qquad\cdot	\frac{\mathtt{c}_{\ms, \mathfrak{t}^{\undla} }\mathtt{c}_{\mathfrak{t}^{\undla},\mt}}{\mathtt{c}_{\ms, \mt}}\prod_{k=1}^n\prod_{A\in\mathscr{A}^{\rhd}_{\mt_{\undla}}(k)}\left(\mathtt{q}(\res_{\mathfrak{t}^{\undla}}(k))-\mathtt{q}(\res(A))\right)\\
			&\qquad\qquad\qquad\cdot \prod_{\substack{k\in \bigl([n]\setminus O_{\mathfrak{t}^{\undla}}\bigr)\bigsqcup \mathfrak{t}^{\undla}(T)}}	\left(b_+(\res_{\mathfrak{t}^{\undla}}(k))-b_-(\res_{\mathfrak{t}^{\undla}}(k))\right)\\
			&\qquad\qquad\qquad\qquad \cdot \mathtt{c}_{\ms, \mt}^2 \delta_{L'_1,L'_4}\delta_{\alpha,\alpha'} t^\hO_{r,n}(F_{L'_1})\\
			&\in \delta_{L'_1,L'_4}\delta_{\alpha,\alpha'}\hO^\times \cdot
\prod_{k=1}^n \prod_{A\in\mathscr{R}_{\mt_{\undla}}^{\lhd}(k)}\left(\mathtt{q}(\res_{\mt_{\undla}}(k))-\mathtt{q}(\res(A)\right),
		\end{align*}by using Lemma \ref{combina. formulae of c} and Theorem \ref{mainthm type0 nondege} (2).
Since $\nu\in Q_n^+$ is $\undQ$-unremovable, we have \begin{align*}
			\prod_{k=1}^n \prod_{A\in\mathscr{R}_{\mt_{\undla}}^{\lhd}(k)}\left(\mathtt{q}(\res_{\mt_{\undla}}(k))-\mathtt{q}(\res(A))\right)
			\in \hO^\times.
		\end{align*} This completes the proof.
	\end{enumerate}

\end{proof}

%\begin{cor}\label{MainLemma}
%If $(q,\undQ)$ is of type $A_\infty,$ $C_\infty,$ $A^{(1)}_{\ell-1}$ or $C^{(1)}_{s},$ then the Gram matrix
%$$\left(t^\hO_{r,n}\left(\psi^{\lhd,S,\hO}_{L'_1,\alpha,L'_2}\psi^{\rhd,T,\hO}_{L'_3,\alpha',L'_4}\right)\right)_{(S,L'_1,\alpha,L'_2),(T,L'_3,\alpha',L'_4)}$$
%of elements \begin{equation}\label{basis1}\Bigl\{\psi^{\lhd,S,\hO}_{L'_1,\alpha,L'_2}\bigm|~\begin{matrix}
%			(\undla,S)\in\mathscr{P}^{\undQ}_{n}, L'_1=(\ms,\beta_\ms), L'_2=(\mt,\beta_\mt)\in  \mathscr{T}(\undla,S)\\
%			\alpha\in \Z_2^n,  {\rm supp}(\alpha)\subset O_{\mathfrak{t}_{\undla}}
%		\end{matrix}\Bigr\}\end{equation}
%and \begin{equation}\label{basis2}\Bigl\{\psi^{\rhd,S,\hO}_{L'_1,\alpha,L'_2}\bigm|~\begin{matrix}
%			(\undla,S)\in\mathscr{P}^{\undQ}_{n}, L'_1=(\ms,\beta_\ms), L'_2=(\mt,\beta_\mt)\in  \mathscr{T}(\undla,S)\\
%			\alpha\in \Z_2^n,  {\rm supp}(\alpha)\subset O_{\mathfrak{t}^{\undla}}
%		\end{matrix}\Bigr\}\end{equation} is an invertible matrix with each entry belongs to $\hO$.
%\end{cor}
%\begin{proof}
%It follows from $(I_f)_{\rm odd}=I_f \cap\{\mathtt{q}(q),\mathtt{q}(-q)\}=\emptyset$ and Lemma \ref{MainLemmaforblock}.
%\end{proof}
Then we have the following.
\begin{prop}\label{O-bases}
Suppose $\nu\in Q_n^+$ is $\undQ$-unremovable.
%If $(q,\undQ)$ is of type $A_\infty,$ $C_\infty,$ $A^{(1)}_{\ell-1}$ or $C^{(1)}_{s},$
Then the sets \eqref{basis1 for block} and \eqref{basis2 for block} form two $\hO$-bases of $\mathcal{H}^{f'}_{\hO}(\nu)$ respectively.
\end{prop}
\begin{proof}
If there is an $\hO$-linear combination
$$\sum_{S,L'_1,\alpha,L'_2}a_{L'_1,\alpha,L'_2}^{\lhd,S,\hO}\psi^{\lhd,S,\hO}_{L'_1,\alpha,L'_2}=0,$$
then we have
$$\sum_{S,L'_1,\alpha,L'_2}a_{L'_1,\alpha,L'_2}^{\lhd,S,\hO}t^\hO_{r,n}\left(\psi^{\lhd,S,\hO}_{L'_1,\alpha,L'_2}\psi^{\rhd,T,\hO}_{L'_3,\alpha',L'_4}\right)=0$$
for any suitable $(T,L'_3,\alpha',L'_4).$ It follows from Lemma \ref{MainLemmaforblock} that each $\hO$-coefficient $a_{L'_1,\alpha,L'_2}^{\lhd,S,\hO}=0$
and thus the set \eqref{basis1 for block} is $\hO$-linearly independent.
It follows from \eqref{dim 2} that the set \eqref{basis1 for block} is a $\hK$-basis of $\mathcal{H}^{f'}_{\hK}(\nu).$

On the other hand, for any $h\in\mathcal{H}^{f'}_{\hO}(\nu)\subseteq \mathcal{H}^{f'}_{\hK}(\nu),$ we can write
$$h=\sum_{S,L'_1,\alpha,L'_2}a_{L'_1,\alpha,L'_2}^{\lhd,S,\hK}\psi^{\lhd,S,\hO}_{L'_1,\alpha,L'_2},$$
for some $a_{L'_1,\alpha,L'_2}^{\lhd,S,\hK}\in \hK.$
Then we obtain the following system of linear equations
$$
t^\hO_{r,n}\left(h\psi^{\rhd,T,\hO}_{L'_3,\alpha',L'_4}\right)
=\sum_{S,L'_1,\alpha,L'_2}a_{L'_1,\alpha,L'_2}^{\lhd,S,\hK}
 t^\hO_{r,n}\left(\psi^{\lhd,S,\hO}_{L'_1,\alpha,L'_2}\psi^{\rhd,T,\hO}_{L'_3,\alpha',L'_4}\right),\quad \text{for }(T,L'_3,\alpha',L'_4).$$
By Lemma \ref{MainLemmaforblock} and note that each $t^\hO_{r,n}\left(h\psi^{\rhd,T,\hO}_{L'_3,\alpha',L'_4}\right)\in\hO,$
we deduce that all coefficients $a_{L'_1,\alpha,L'_2}^{\lhd,S,\hK}\in\hO.$
Hence the set \eqref{basis1 for block} is an $\hO$-basis of $\mathcal{H}^{f'}_{\hO}(\nu).$
Similarly, the set \eqref{basis2 for block} is also an $\hO$-basis of $\mathcal{H}^{f'}_{\hO}(\nu).$
\end{proof}

%Recall that we call the parameter $(q,\undQ)$ is of some type $\diamondsuit$ in Section \ref{DynkinDiagrams}, if $q$ satisfies the condition associated with type $\diamondsuit$ and each $[Q_i]$ ($1\leq i\leq m$) corresponds to a vertex of the Dynkin diagram for $\diamondsuit$.
We are now in the position to state our main result of this subsection.
\begin{thm}\label{GC1}
    Suppose $\nu\in Q_n^+$ is $\undQ$-unremovable. Then the following two sets \label{pag:basis1'}
	\begin{equation}\label{basis1'}
       \Psi^{\lhd}_{\nu}=\Bigl\{\psi^{\lhd,S}_{L'_1,\alpha,L'_2}\in\mathcal{H}^f_{\mathbb{K}}\bigm|~\begin{matrix}
			(\undla,S)\in\mathscr{P}^{\undQ}_{\nu}, L'_1=(\ms,\beta_\ms), L'_2=(\mt,\beta_\mt)\in  \mathscr{T}(\undla,S),\\
			\alpha\in \Z_2^n,  {\rm supp}(\alpha)\subset O_{\mathfrak{t}_{\undla}}
		\end{matrix}\Bigr\}\end{equation}
and \begin{equation}\label{basis2'}
	   \Psi^{\rhd}_{\nu}=\Bigl\{\psi^{\rhd,S}_{L'_1,\alpha,L'_2}\in\mathcal{H}^f_{\mathbb{K}}\bigm|~\begin{matrix}
			(\undla,S)\in\mathscr{P}^{\undQ}_{\nu}, L'_1=(\ms,\beta_\ms), L'_2=(\mt,\beta_\mt)\in  \mathscr{T}(\undla,S),\\
			\alpha\in \Z_2^n,  {\rm supp}(\alpha)\subset O_{\mathfrak{t}^{\undla}}
		\end{matrix}\Bigr\}\end{equation} form two $\mathbb{K}$-bases of $\mathcal{H}^f_{\mathbb{K}}(\nu)$ respectively.
	
In particular, if $(I_f)_{{\rm odd}}=\emptyset$, then the sets $\bigsqcup\limits_{\nu\in Q_n^+}\Psi^{\lhd}_{\nu}$ and $\bigsqcup\limits_{\nu\in Q_n^+}\Psi^{\rhd}_{\nu}$ form two $\mathbb{K}$-bases of $\mathcal{H}^f_{\mathbb{K}}$ respectively.
\end{thm}
\begin{proof}
The first part of the Theorem is to apply Proposition \ref{O-bases} and the natural isomorphism $\mathcal{H}^f_{\mathbb{K}}\cong \mathbb{K}\otimes_{\hO} \mathcal{H}^{f'}_{\hO}.$
The second statement follows from Proposition \ref{example-unremovable}.
\end{proof}

\subsection{Generalized graded super cellular datum}\label{Generalized graded super cellular datum}
{\bf In this section, we fix $\nu\in Q_n^+$ being $\undQ$-unremovable.} We shall prove that $\mathcal{H}^f_{\mathbb{K}}(\nu)=e^J_\nu \mathcal{H}^f_{\mathbb{K}}$ is a generalized graded cellular superalgebra by giving generalized graded super cellular datum for $\mathcal{H}^f_{\mathbb{K}}(\nu).$
%
%In the rest of this section, we write $\nu=m(\nu)\alpha_{{\bf i}_0}+\sum_{{\bf i}\in (I_f)_{\rm even}}k_{{\bf i}}\alpha_{{\bf i}}\in Q_n^+$ for some ${\bf i}_0\in(I_f)_{\rm odd}$ with $m(\nu)\leq 1.$

Recall the bases $\Psi^{\lhd}_{\nu}$ \eqref{basis1'} and $\Psi^{\rhd}_{\nu}$ \eqref{basis2'} of $\mathcal{H}^f_{\mathbb{K}}(\nu).$ We first determine the $\mathbb{Z}$-degrees of the elements in $\Psi^{{\scriptstyle\triangle}}_{\nu},$ for ${\scriptstyle\triangle}\in\{\lhd,\rhd\}.$
\label{pag:more degrees}
\begin{defn}
Let $(\undla,S)\in\mathscr{P}^{\undQ}_{\nu}.$ We define $$
\deg(S):=\sum_{A\in S}{\rm d}_{\mathtt{q}(\res(A))|_{x=0}}.
$$
For $L=(\mt,\beta_\mt)\in \mathscr{T}(\undla,S),$ ${\scriptstyle\triangle}\in\{\lhd,\rhd\},$ we define
\begin{align*}
\deg^{{\scriptstyle\triangle},S}(L)
:=\deg^{\scriptstyle\triangle,f}(\mt) + \deg(S).
\end{align*}
\end{defn}

Comparing Definition \ref{deg of std tableaux} and Definition \ref{tri-Q-AddandRem}, we have the following.
\begin{lem}\label{Boxes 1-1}
Let ${\rm T}=(\mt,\beta_{\mt})\in \Std(\undla)\times \Z_2^n,$ $\mathtt{q}(\res(\mt))|_{x=0}=({\bf i}_1,\ldots,{\bf i}_n)\in (I_f)^n.$ Then
\begin{align*}
\sharp \mathscr{A}_{\mt}^{{\scriptstyle\triangle},\undQ}(k)
=2^{\delta_{{\rm p}({\bf i}_k),\bar{1}}}\cdot \sharp \mathcal{A}_{\mt}^{{\scriptstyle\triangle},f}(k),
\qquad
\sharp \mathscr{R}_{\mt}^{{\scriptstyle\triangle},\undQ}(k)
=2^{\delta_{{\rm p}({\bf i}_k),\bar{1}}}\cdot \sharp \mathcal{R}_{\mt}^{{\scriptstyle\triangle},f}(k),
\end{align*}
for $k\in[n],$ ${\scriptstyle\triangle}\in\{\lhd,\rhd\}.$
\end{lem}
%\begin{proof}
%It follows from the definitions directly.
%\end{proof}

\begin{lem}\label{deg of psi}
For any
%$\psi^{{\scriptstyle\triangle},S}_{L'_1,\alpha,L'_2}\in\Psi^{\scriptstyle\triangle}_{\beta},$ where ${\scriptstyle\triangle}\in\{\lhd,\rhd\},$ we have
$(\undla,S)\in\mathscr{P}^{\undQ}_{\nu},$ $L'_1=(\ms,\beta_\ms), L'_2=(\mt,\beta_\mt)\in  \mathscr{T}(\undla,S),$
%and $\alpha\in \Z_2^n $ with ${\rm supp}(\alpha)\subset O_{\mathfrak{t}_{\undla}},$
we have
\begin{align*}
\deg \left(\psi^{{\scriptstyle\triangle},S}_{L'_1,\alpha,L'_2}\right)
=\deg^{{\scriptstyle\triangle},S}(L'_1)+\deg^{{\scriptstyle\triangle},S}(L'_2).
%\quad\text{ for }{\scriptstyle\triangle}\in\{\lhd,\rhd\}.
\end{align*}
\end{lem}
\begin{proof}
We may assume ${\scriptstyle\triangle}=\lhd.$
Recall the definitions in \eqref{middle part for rhd} and \eqref{rhd-basis}.
Let $L_1=(\mathfrak{t}_{\undla},\beta_1), L_2=(\mathfrak{t}_{\undla},\beta_2)\in  \mathscr{T}(\undla,S)$
and $w_1,w_2\in \mathfrak{S}_n$ such that $L'_1=w_1L_1$ and $L'_2=w_2L_2.$
Since $\nu\in Q_n^+$ is $\undQ$-unremovable, we have $$\deg(y^\lhd_{\undla}e({\bf i}_{\undla}))=2\deg ^{{\lhd},f}(\mt_{\undla}).
$$
It follows from definition that
\begin{align*}
\deg \left(\psi^{\lhd,S}_{L'_1,\alpha,L'_2}\right)
&=\deg(\sigma_{w_1}e(c^{\beta_1}{\bf i}_{\undla}))+\deg(y^\lhd_{\undla}e({\bf i}_{\undla}))
+2\deg(S)+\deg(\sigma_{w_2}e(c^{\beta_2}{\bf i}_{\undla}))\\
&=\deg^{{\lhd},f}(\ms)+\deg^{{\lhd},f}(\mt)+2\deg(S)\\
&=\deg^{{\lhd},S}(L'_1)+\deg^{{\lhd},S}(L'_2).
\end{align*}
where in the second equality we have used Lemma \ref{Boxes 1-1} and Corollary \ref{cor of BKW deg}. The proof for ${\scriptstyle\triangle}=\rhd$ is similar.
\end{proof}

Next we clarify the property (GC4) concerned with anti-involutions for the bases $\Psi^{{\scriptstyle\triangle}}_{\nu},$ ${\scriptstyle\triangle}\in\{\lhd,\rhd\}.$
\begin{defn}
Let $(\undla,S)\in\mathscr{P}^{\undQ}_{\nu}.$ \label{pag:w-involutions}

(1) The anti-involution $\omega'_{\undla,S}$ on $\mathcal{C}_{\mt_{\undla}}$ as follows:
$$
\omega'_{\undla,S}(C_i)=\begin{cases}
	C_i,&\qquad \text{ if $i\notin \mt_{\undla}(S)$},\\
	-C_i,&\qquad \text{ if $i\in \mt_{\undla}(S)$}.
	\end{cases}
$$

(2) The anti-involution $\omega_{\undla,S}$ on $\mathcal{C}_{\mt^{\undla}}$ as follows:
$$
\omega_{\undla,S}(C_i)=\begin{cases}
	C_i,&\qquad \text{ if $i\notin \mt^{\undla}(S)$},\\
	-C_i,&\qquad \text{ if $i\in \mt^{\undla}(S)$}.
	\end{cases}
$$
\end{defn}
	
\begin{lem}\label{GC4}
	Let $(\undla,S)\in\mathscr{P}^{\undQ}_{\nu}$ and $L'_1=(\ms,\beta_\ms), L'_2=(\mt,\beta_\mt)\in   \mathscr{T}(\undla,S).$

(1) For any $u\in\mathcal{C}_{\mt_{\undla}}$, we have $$
	\left(\psi^{\lhd,S}_{L'_1,u,L'_2}\right)^*=	\psi^{\lhd,S}_{L'_2,\omega'_{\undla,S}(u),L'_1}.
	$$

(2) For any $u\in\mathcal{C}_{\mt^{\undla}}$, we have $$
	\left(\psi^{\rhd,S}_{L'_1,u,L'_2}\right)^*=	\psi^{\rhd,S}_{L'_2,\omega_{\undla,S}(u),L'_1}.
	$$
	\end{lem}
\begin{proof}
We only prove (1).
For $L'_1=(\ms,\beta_\ms), L'_2=(\mt,\beta_\mt)\in  \mathscr{T}(\undla,S),$ there are unique $L_1=(\mathfrak{t}_{\undla},\beta_1), L_2=(\mathfrak{t}_{\undla},\beta_2)\in  \mathscr{T}(\undla,S)$ and $w_1,w_2\in \mathfrak{S}_n$ such that $L'_1=w_1L_1$ and $L'_2=w_2L_2.$
For any $i\in O_{\mt_{\undla}},$ we have $e({\bf i}_{{\undla}}) C_i=C_i e({\bf i}_{{\undla}}),$ and
$y^{\lhd}_{\undla}C_i =C_i y^{\lhd}_{\undla}$ since $\sharp \mathscr{A}^{\lhd}_{\mt_{\undla}}(i)$ is even.
By \eqref{middle part}, for any monomial $C^\alpha \in \mathcal{C}_{\mt_{\undla}},$ we have
\begin{align*}
\left(y^{\lhd,S}_{L_1,\alpha,L_2}\right)^*
&=\sgn(\beta_1)^2 \sgn(\beta_2)C^{\beta_2} \left(\prod_{k\in \mathfrak{t}_{\undla}(S)}y_k\right)e({\bf i}_{{\undla}})y^{\lhd}_{\undla} \left(C^{\alpha}\right)^* C^{\beta_1}\\
&=\sgn(\beta_2)C^{\beta_2} \omega'_{\undla,S}(C^{\alpha}) y^{\lhd}_{\undla}e({\bf i}_{{\undla}}) \left(\prod_{k\in \mathfrak{t}_{\undla}(S)}y_k\right)C^{\beta_1},
\end{align*}
and this implies the Lemma.
\end{proof}

We equip $\mathscr{P}^{\undQ}_{\nu}$ with two partial orders as follows with respect to two different bases.
\label{pag:new partial orders}
\begin{defn}
Let $(\undla,S),(\undmu,T)\in\mathscr{P}^{\undQ}_{\nu}$.

(1) We define $(\undla,S)\unlhd' (\undmu,T)$ if and only if $\undla\lhd\undmu$ or $\undla=\undmu$ and $T \subset S$.

(2) We define $(\undla,S)\unlhd (\undmu,T)$ if and only if $\undla\lhd\undmu$  or $\undla=\undmu$ and $S \subset T$.
\end{defn}

The following Theorem is the main result of this paper.
\begin{thm}\label{Mainthm}
Suppose $\nu\in Q_n^+$ is $\undQ$-unremovable.
Then we have the following.

(1). The algebra $\mathcal{H}^f_{\mathbb{K}}(\nu)$ is a generalized graded cellular superalgebra with poset $(\mathscr{P}^{\undQ}_{\nu},\lhd'),$ and generalized graded cellular basis $\Psi^{\lhd}_{\nu}$ \eqref{basis1'}. In particular, for each $(\undla,S)\in\mathscr{P}^{\undQ}_{\nu},$
the (semisimple) superalgebra $\mathscr{B}_{\undla,S}:=\mathcal{C}_{\mt_{\undla}},$  and $\deg|_{\undla,S}:=\deg^{\lhd,S}.$

(2). The algebra $\mathcal{H}^f_{\mathbb{K}}(\nu)$ is a generalized graded cellular superalgebra poset $(\mathscr{P}^{\undQ}_{\nu},\rhd),$ and generalized graded cellular basis $\Psi^{\rhd}_{\nu}$ \eqref{basis2'}. In particular,  for each $(\undla,S)\in\mathscr{P}^{\undQ}_{\nu},$
the (semisimple) superalgebra $\mathscr{B}_{\undla,S}:=\mathcal{C}_{\mt^{\undla}},$  and $\deg|_{\undla,S}:=\deg^{\rhd,S}.$

In particular, if $(I_f)_{{\rm odd}}=\emptyset$, then the cyclotomic Hecke-Clifford superalgebra $\mathcal{H}^f_{\mathbb{K}}$ is a graded cellular algebra with two graded cellular bases $\bigsqcup\limits_{\nu\in Q_n^+}\Psi^{\lhd}_{\nu}$ and $\bigsqcup\limits_{\nu\in Q_n^+n}\Psi^{\rhd}_{\nu}$.
%Suppose that for each $\undla\in\mathscr{P}^{m}_{n}$, we have $O_{\undla}=\emptyset$, then $\mathcal{H}^f_{\mathbb{K}}(\beta)$ is a graded cellular algebra.
\end{thm}

\begin{proof}
We only prove (1).
(GCd) follows from Lemma \ref{deg of psi}.
(GC1) follows from Theorem \ref{GC1}.
(GC2) is clear by definition.
(GC4) follows from Lemma \ref{GC4}.
Hence we only need to prove (GC3).
	
Let $(\undla,S)\in\mathscr{P}^{\undQ}_{\nu}$,	we define $$\mathcal{H}^{f'}_{\hO}(\nu)^{\lhd\undla}:=\sum_{\substack{\undmu\lhd\undla,\,(\undmu,T)\in\mathscr{P}^{\undQ}_{\nu}, \\
		L'_1=(\ms,\beta_\ms), L'_2=(\mt,\beta_\mt)\in \mathscr{T}(\undmu,T)\\
		\alpha\in \Z_2^n,  \supp(\alpha)\subset O_{\mathfrak{t}_{\undmu}}}}\hO	\psi^{\lhd,T,\hO}_{L'_1,\alpha,L'_2}$$ and $$\mathcal{H}^{f'}_{\hO}(\nu)^{\lhd'(\undla,S)}:=\sum_{\substack{(\undmu,T)\lhd' (\undla,S)\in\mathscr{P}^{\undQ}_{\nu}, \\ L'_1=(\ms,\beta_\ms), L'_2=(\mt,\beta_\mt)\in \mathscr{T}(\undmu,T)\\
			\alpha\in \Z_2^n,  \supp(\alpha)\subset O_{\mathfrak{t}_{\undmu}}}}\hO	\psi^{\lhd,T,\hO}_{L'_1,\alpha,L'_2}.$$

		Similarly, we can define $\mathcal{H}^f_{\mathbb{K}}(\nu)^{\lhd\undla}$ and $\mathcal{H}^f_{\mathbb{K}}(\nu)^{\lhd'(\undla,S)}.$ By Lemma \ref{second expansion}, $\mathcal{H}^{f'}_{\hO}(\nu)^{\lhd\undla}$ and $\mathcal{H}^f_{\mathbb{K}}(\nu)^{\lhd\undla}$ are two-sided ideals of $\mathcal{H}^{f'}_{\hO}(\nu)$ and $\mathcal{H}^f_{\mathbb{K}}(\nu)$ respectively. Let $L'_1=(\ms,\beta_{\ms}), L_2=(\mathfrak{t}_{\undla},0)\in \mathscr{T}(\undla,S)$, there are unique $ L_1=(\mathfrak{t}_{\undla},\beta_1)\in \mathscr{T}(\undla,S)$ and $w_1\in \mathfrak{S}_n$ such that $L'_1=w_1L_1$.
For $\alpha\in \Z_2^n$ such that $\supp(\alpha)\subset O_{\mathfrak{t}_{\undla}}.$
For $a\in \mathcal{H}^{f'}_{\hO}(\nu)$, it follows from Lemma \ref{first expansion} that

%\begin{align*}
%	\psi^{\lhd,S,\hO}_{L'_1,\alpha,L_2}
%\in & \sum_{\substack{\tilde{L_1}=(\mathfrak{s},\tilde{\beta}_{1}),\tilde{L_2}=(\mathfrak{t}_{\undla},\beta_{\mathfrak{t}_{\undla}})\\ \tilde{\beta_1}=\beta_{\ms}+w_1\cdot\alpha+w_1\cdot\beta_{\mathfrak{t}_{\undla}}\\
%\beta_{\mathfrak{t}_{\undla}}\in \Z_2^n ,\,\supp(\beta_{\mathfrak{t}_{\undla}})\subset O_{\mathfrak{t}_{\undla}}\setminus \mathfrak{t}_{\undla}(S)}}\hK f_{\tilde{L_1},\tilde{L_2}} +\sum_{\substack{\tilde{L_1}=(\mathfrak{u},\beta''_{\mathfrak{u}}),\tilde{L_2}=(\mathfrak{v},\beta'''_{\mathfrak{v}})\\ \mathfrak{u},\mathfrak{v}\lhd\mathfrak{t}_{\undla}}}\hK f_{\tilde{L_1},\tilde{L_2}}.\nonumber
%\end{align*}
%	Hence, for $a\in \mathcal{H}^{f'}_{\hO}(\beta)$, we have
$$
a\cdot \psi^{\lhd,S,\hO}_{L'_1,\alpha,L_2}=(a\sigma_{w_1}^{\hO}) y^{\lhd,S,\hO}_{L_1,\alpha,L_2}
\in \sum_{\substack{
\tilde{L_1}\in {\rm Tri}(\undla)\\
\tilde{L_2}=(\mathfrak{t}_{\undla},\beta_{\mathfrak{t}_{\undla}})\\
\beta_{\mathfrak{t}_{\undla}}\in \Z_2^n ,\,\supp(\beta_{\mathfrak{t}_{\undla}})\subset O_{\mathfrak{t}_{\undla}}\setminus\mathfrak{t}_{\undla}(S)}}\hK f_{\tilde{L_1},\tilde{L_2}}
+\sum_{\substack{\tilde{L_1}=(\mathfrak{u},\beta''_{\mathfrak{u}}),\tilde{L_2}=(\mathfrak{v},\beta'''_{\mathfrak{v}})\in{\rm Tri}(\undmu)\\ \undmu\lhd\undla}}\hK f_{\tilde{L_1},\tilde{L_2}}.\nonumber
$$

Then by Lemma \ref{second expansion}, we have
\begin{align*}
a\cdot 	\psi^{\lhd,S,\hO}_{L'_1,\alpha,L_2}\in \sum_{\substack{(L,S)\in \mathscr{T}(\undla,S)\\ \alpha'\in \Z_2^n,\,\supp(\alpha')\subset O_{\mathfrak{t}_{\undla}}}}r_{L'_1,\alpha,S}^{L,\alpha',\hO}(a)	\psi^{\lhd,S,\hO}_{L,\alpha',L_2} +\sum_{\substack{S\subsetneq T\\
(L,T)\in \mathscr{T}(\undla,T)\\ \alpha'\in \Z_2^n,\,\supp(\alpha')\subset O_{\mathfrak{t}_{\undla}}}}r_{L'_1,\alpha,T}^{L,\alpha',\hO}(a)	\psi^{\lhd,T,\hO}_{L,\alpha',L_2}+\mathcal{H}^{f'}_{\hO}(\nu)^{\lhd\undla},
\end{align*}
where the coefficients $r_{L'_1,\alpha,S}^{L,\alpha',\hO}(a),$ $r_{L'_1,\alpha,T}^{L,\alpha',\hO}(a)\in \hO$ by Proposition \ref{O-bases}.
Next, for $\alpha''\in \Z_2^n$ with $\supp(\alpha'')\subset O_{\mathfrak{t}_{\undla}}$, $(\mt_{\undla},\beta_2)\in \mathscr{T}(\undla,S)$ and $w_2\cdot \mt_{\undla}\in \Std(\undla)$, multiplying $C^{\beta_2}(\sigma^\hO_{w_2})^*$ and $C^{\alpha''}C^{\beta_2}(\sigma^\hO_{w_2})^*$ from the right on both sides respectively, we get
\begin{align}\label{expansion multiple}
a\cdot 	\psi^{\lhd,S,\hO}_{L'_1,\alpha,L'_2}
&\in\sum_{\substack{(L,S)\in \mathscr{T}(\undla,S)\\ \alpha'\in \Z_2^n,\,\supp(\alpha')\subset O_{\mathfrak{t}_{\undla}}}}r_{L'_1,\alpha,S}^{L,\alpha',\hO}(a)	\psi^{\lhd,S,\hO}_{L,\alpha',L'_2}
+\sum_{\substack{S\subsetneq T\\
(L,T)\in \mathscr{T}(\undla,T)\\ \alpha'\in \Z_2^n,\,\supp(\alpha')\subset O_{\mathfrak{t}_{\undla}}}}r_{L'_1,\alpha,T}^{L,\alpha',\hO}(a)	\psi^{\lhd,T,\hO}_{L,\alpha',L'_2}+\mathcal{H}^{f'}_{\hO}(\nu)^{\lhd\undla}\\
&\subseteq \sum_{\substack{(L,S)\in \mathscr{T}(\undla,S)\\ \alpha'\in \Z_2^n,\,\supp(\alpha')\subset O_{\mathfrak{t}_{\undla}}}}r_{L'_1,\alpha,S}^{L,\alpha',\hO}(a)	\psi^{\lhd,S,\hO}_{L,\alpha',L'_2}
+\mathcal{H}^{f'}_{\hO}(\nu)^{\lhd'(\undla,S)},\nonumber
\end{align}
and
\begin{align}\label{expansion multiple'}
(-1)^va\cdot \psi^{\lhd,S,\hO}_{L'_1,\alpha\cdot\alpha'',L'_2}
\in \sum_{\substack{L\in \mathscr{T}(\undla,S)\\ \alpha'\in \Z_2^n,\,\supp(\alpha')\subset O_{\mathfrak{t}_{\undla}}}}(-1)^v r_{L'_1,\alpha,S}^{L,\alpha',\hO}(a)\psi^{\lhd,S,\hO}_{L,\alpha'\cdot\alpha'',L'_2} +\mathcal{H}^{f'}_{\hO}(\nu)^{\lhd'(\undla,S)},
\end{align}
where $L'_2=w_2\cdot(\mt_{\undla},\beta_2)$ and $v$ depends on $S$ and $\alpha''$.
In particular, \eqref{expansion multiple} together with its right-multiplication analog implies that $\mathcal{H}^{f'}_{\hO}(\nu)^{\lhd'(\undla,S)}$ is a left ideal of $\mathcal{H}^{f'}_{\hO}(\nu).$ Now, specializing $x$ to $0$ in above \eqref{expansion multiple} and \eqref{expansion multiple'} yields
%\begin{align*}
%a|_{x=0}\cdot \psi^{\lhd,S}_{L'_1,\alpha,L_2}\in \sum_{\substack{(L,S)\in \mathscr{T}(\undla,S)\\ \alpha'\in \Z_2^n,\,\supp(\alpha')\subset O_{\mathfrak{t}_{\undla}}}}&r_{L'_1,\alpha,S}^{L,\alpha',\hO}(a)|_{x=0}	\psi^{\lhd,S}_{L,\alpha',L_2} \\
%	&+\sum_{\substack{S\subsetneq T\\(L,T)\in \mathscr{T}(\undla,T)\\ \alpha'\in \Z_2^n,\,\supp(\alpha')\subset O_{\mathfrak{t}_{\undla}}}}r_{L'_1,\alpha,T}^{L,\alpha',\hO}(a)|_{x=0}	\psi^{\lhd,T}_{L,\alpha',L_2}+(\mathcal{H}^f_{\mathbb{K}}(\beta))^{\lhd\undla}\nonumber.
%\end{align*}
%Now, for $\alpha''\in \Z_2^n$ with $\supp(\alpha'')\subset O_{\mathfrak{t}_{\undla}}$, $(\mt_{\undla},\beta_2)\in \mathscr{T}(\undla,S)$ and $w_2\cdot \mt_{\undla}\in \Std(\undla)$, multiplying $C^{\beta_2}(\sigma_{w_2})^*$ and $C^{\alpha''}C^{\beta_2}(\sigma_{w_2})^*$ from the right on both sides respectively, we get
\begin{align}\label{expansion multiple over K}
a|_{x=0}\cdot \psi^{\lhd,S}_{L'_1,\alpha,L'_2}
&\in \sum_{\substack{L\in \mathscr{T}(\undla,S)\\ \alpha'\in \Z_2^n,\,\supp(\alpha')\subset O_{\mathfrak{t}_{\undla}}}}r_{L'_1,\alpha,S}^{L,\alpha',\hO}(a)|_{x=0}	\psi^{\lhd,S}_{L,\alpha',L'_2} +\mathcal{H}^f_{\mathbb{K}}(\nu)^{\lhd'(\undla,S)},
\end{align}
and
\begin{align*}
a|_{x=0}\cdot \psi^{\lhd,S}_{L'_1,\alpha\cdot\alpha'',L'_2}\in \sum_{\substack{L\in \mathscr{T}(\undla,S)\\ \alpha'\in \Z_2^n,\,\supp(\alpha')\subset O_{\mathfrak{t}_{\undla}}}}r_{L'_1,\alpha,S}^{L,\alpha',\hO}(a)|_{x=0}	\psi^{\lhd,S}_{L,\alpha'\cdot\alpha'',L'_2} +\mathcal{H}^f_{\mathbb{K}}(\nu)^{\lhd'(\undla,S)}.
\end{align*}
This proves (GC3). In particular, \eqref{expansion multiple over K} and Lemma \ref{GC4} imply that $\mathcal{H}^f_{\mathbb{K}}(\nu)^{\lhd'(\undla,S)}$ is a two-sided ideal of $\mathcal{H}^f_{\mathbb{K}}(\nu).$
	\end{proof}
	
	\subsection{Graded supersymmetrizing form}\label{Graded supersymmetrizing form}
	{\bf In this section, we fix $\nu\in Q_n^+$ being $\undQ$-unremovable.} We shall introduce a graded supersymmetrizing form on $\mathcal{H}^f_{\mathbb{K}}(\nu)$.
	
We define the defect of $\nu$ as \label{pag:defect}
$${\rm def}(\nu):=\left(\Lambda_f\middle|\nu\right)-\frac{1}{2}\left(\nu\middle|\nu\right). $$
	
	\begin{lem}\label{def2}
		For $(\undla,S)\in\mathscr{P}^{\undQ}_{\nu},$ and $L'_1=(\ms,\beta_{\ms})\in \mathscr{T}(\undla,S), L'_2=(\ms,\beta'_{\ms})\in \mathscr{T}(\undla,O_{\undla}\setminus S)$. Then $$\deg^{{\lhd},S}(L'_1)+\deg^{{\rhd},O_{\undla}\setminus S}(L'_2)={\rm def}(\nu).
		$$
		\end{lem}
		
		\begin{proof}
			By definitiom, we have \begin{align*}
				\deg^{{\lhd},S}(L'_1)+\deg^{{\rhd},O_{\undla}\setminus S}(L'_2)&=\deg^{{\lhd},f}(\ms)+\deg(S)+\deg^{{\rhd},f}(\ms)+ \deg(O_{\undla}\setminus S)\\
				&=d^f(\undla)+\sum_{A\in O_{\undla}}{\rm d}_{\mathtt{q}(\res(A))|_{x=0}} \\
				&=\left(\Lambda_f\middle|\nu\right)-\frac{1}{2}\left(\nu\middle|\nu\right),
				\end{align*} where in the second equation, we have used Corollary \ref{def1} and the third equation follows from the Definition of $d(\undla)$. Hence we prove the Lemma.
			\end{proof}
		Recall the supersymmetrizing form $t_{2m,n}$ on $\mathcal{H}^f_{\mathbb{K}}(\nu)$, note that in this case, $r=2m$.
        The following definition is inspired by \cite[Definition 6.15]{HM1}.
        \label{pag:homogeneous supersym}
	\begin{defn}
		We define $t_{\nu}: \mathcal{H}^f_{\mathbb{K}}(\nu) \to \mathbb{K}$ being
		the map which on a homogeneous element $a\in \mathcal{H}^f_{\mathbb{K}}(\nu)$ is given by$$
		t_{\nu}(a):=\begin{cases}
			t_{2m,n}(a),&\qquad \text{if $\deg(a)=2{\rm def}(\nu)$};\\
			0, &\qquad \text{otherwise.}
			\end{cases}
		$$
		\end{defn}
	
	\begin{thm}\label{graded-supersymmetrizing1}
		Suppose $\nu\in Q_n^+$ is $\undQ$-unremovable. Then $\mathcal{H}^f_{\mathbb{K}}(\nu)$ is a graded supersymmetric superalgebra with the homogeneous supersymmetrizing form $t_{\nu}$ of degree $-2{\rm def}(\nu)$.
		\end{thm}
		\begin{proof}
	Clearly, $t_{\nu}$ satisfies that $t_{\nu}(ab)=(-1)^{{\rm p}(a)\cdot{\rm p}(b)}t_{\nu}(ba)$ for all homogeneous $a,b\in \mathcal{H}^f_{\mathbb{K}}(\nu)$. By definition, $t_{\nu}$ is
	homogeneous of degree  $-2{\rm def}(\nu)$. Now we apply Lemma \ref{MainLemmaforblock} and Lemma \ref{def2} to see that the Gram martix of $t_\nu$ between  \eqref{basis1 for block} and \eqref{basis2 for block} is invertible over $\mathbb{K}$. This proves that $t_{\nu}$ is a supersymmetrizing form.
	\end{proof}
	
	\medskip

	\noindent
	{\bf{Proof of Theorem \ref{Main2}}:}
	This follows from Theorem \ref{Mainthm} and Theorem \ref{graded-supersymmetrizing1}.
	\qed
	\medskip

	\subsection{Idempotent truncation}\label{Idempotent truncation}
	{\bf In this section, we fix $\nu\in Q_n^+$ being $\undQ$-unremovable.} We shall study the generalized graded cellular structure and the supersymmetrizing form in cyclotomic quiver Hecke superalgebra by taking idempotent truncation on $\mathcal{H}^f_{\mathbb{K}}.$
	
First, we need to pick up a subset of $\mathscr{T}(\undla,S)$ \eqref{T(lambda,S)} to index the bases of idempotent truncation subalgebra.
    \label{pag:dag index set}
	\begin{defn}
    Let $(\undla,S)\in\mathscr{P}^{\undQ}_{n}$. We define
	$\mathscr{T}^\dag(\undla,S):=\{(\mt,0,S)\mid \mt\in\Std(\undla)\}\subset \mathscr{T}(\undla,S).$
	\end{defn}
	
	Again, if $S$ has been fixed in the context, we shall only write $\mt \in \mathscr{T}(\undla,S)$ rather $(\mt,0,S)\in \mathscr{T}(\undla,S)$ to simplify notation.
	
		Recall $J_f^\dag=\{\mathtt{b}_{+}(x)\in \mathbb{K}^* \mid \mathtt{q}(x)\in I_f\},$ and $e^\dag=\sum_{{\bf i}\in {J^{\dag n}_f}}e({\bf i})\in \mathcal{H}^f_{\mathbb{K}}.$	\label{pag:dag bases}
	\begin{cor}\label{GC1'}
		 Suppose $\nu\in Q_n^+$ is $\undQ$-unremovable. Then the following two sets
		\begin{equation}\label{basis1}
			\Psi^{\lhd,\dag}_{\nu}=\Bigl\{\psi^{\lhd,S}_{L'_1,\alpha,L'_2}\in\mathcal{H}^f_{\mathbb{K}}\bigm|~\begin{matrix}
				(\undla,S)\in\mathscr{P}^{\undQ}_{\nu}, L'_1=\ms, L'_2=\mt\in  \mathscr{T}^\dag(\undla,S),\\
				\alpha\in \Z_2^n,  {\rm supp}(\alpha)\subset O_{\mathfrak{t}_{\undla}}
			\end{matrix}\Bigr\}\end{equation}
		and \begin{equation}\label{basis2}
			\Psi^{\rhd,\dag}_{\nu}=\Bigl\{\psi^{\rhd,S}_{L'_1,\alpha,L'_2}\in\mathcal{H}^f_{\mathbb{K}}\bigm|~\begin{matrix}
				(\undla,S)\in\mathscr{P}^{\undQ}_{\nu}, L'_1=\ms, L'_2=\mt\in  \mathscr{T}^\dag(\undla,S),\\
				\alpha\in \Z_2^n,  {\rm supp}(\alpha)\subset O_{\mathfrak{t}^{\undla}}
			\end{matrix}\Bigr\}\end{equation} form two $\mathbb{K}$-bases of $e^\dag\mathcal{H}^f_{\mathbb{K}}(\nu)e^\dag$ respectively.
	
		In particular, if $(I_f)_{{\rm odd}}=\emptyset$, then the sets $\bigsqcup\limits_{\nu\in Q_n^+}\Psi^{\lhd,\dag}_{\nu}$ and $\bigsqcup\limits_{\nu\in Q_n^+}\Psi^{\rhd,\dag}_{\nu}$ form two $\mathbb{K}$-bases of $e^\dag\mathcal{H}^f_{\mathbb{K}}e^\dag$ respectively.
		\end{cor}
		
		\begin{proof}
			We consider the following decomposition of $\mathbb{K}$-linear spaces:
$$\mathcal{H}^f_{\mathbb{K}}=e^\dag\mathcal{H}^f_{\mathbb{K}}(\nu)e^\dag\oplus H',$$
			where $H'=(1-e^\dag)\mathcal{H}^f_{\mathbb{K}}(\nu)e^\dag\oplus e^\dag\mathcal{H}^f_{\mathbb{K}}(\nu)(1-e^\dag)\oplus (1-e^\dag)\mathcal{H}^f_{\mathbb{K}}(\nu)(1-e^\dag).$
Moreover, we have the following decomposition of $\Psi^{\lhd}_{\nu}$:
$$\Psi^{\lhd}_{\nu}=\Psi^{\lhd,\dag}_{\nu} \sqcup {\Psi^{\lhd}_{\nu}}'$$
such that $\Psi^{\lhd,\dag}_{\nu} \subset e^{\dag}  \mathcal{H}^f_{\mathbb{K}}(\nu) e^{\dag}$ and
${\Psi^{\lhd}_{\nu}}' \subset H'$.
By Theorem \ref{GC1}, we deduce that $\Psi^{\lhd,\dag}_{\nu}$ forms a $\mathbb{K}$-basis of
$e^{\dag} \mathcal{H}^f_{\mathbb{K}}e^{\dag}$.
The same argument shows that $\Psi^{\rhd,\dag}_{\nu}$ forms a $\mathbb{K}$-basis of $e^{\dag}\mathcal{H}^f_{\mathbb{K}}e^{\dag}$.
			\end{proof}
			
			\begin{thm}\label{Mainthm-truncation}
				Suppose $\nu\in Q_n^+$ is $\undQ$-unremovable.
				Then we have the following.
				
				(1). The algebra $e^\dag\mathcal{H}^f_{\mathbb{K}}(\nu)e^\dag$ is a generalized graded cellular superalgebra with poset $(\mathscr{P}^{\undQ}_{\nu},\lhd'),$ and generalized graded cellular basis $\Psi^{\lhd,\dag}_{\nu}$ \eqref{basis1}. In particular, for each $(\undla,S)\in\mathscr{P}^{\undQ}_{\nu},$ the (semisimple) superalgebra $\mathscr{B}_{\undla,S}=\mathcal{C}_{\mt_{\undla}},$ and $\deg|_{\undla,S}=\deg^{\lhd,S}.$
				
				(2). The algebra $e^\dag\mathcal{H}^f_{\mathbb{K}}(\nu)e^\dag$ is a generalized graded cellular superalgebra with poset $(\mathscr{P}^{\undQ}_{\nu},\rhd),$ and the generalized graded cellular basis $\Psi^{\rhd,\dag}_{\nu}$ \eqref{basis2}. In particular, for each $(\undla,S)\in\mathscr{P}^{\undQ}_{\nu},$ the (semisimple) superalgebra $\mathscr{B}_{\undla,S}=\mathcal{C}_{\mt^{\undla}},$ and $\deg|_{\undla,S}=\deg^{\rhd,S}.$
				
				In particular, if $(I_f)_{{\rm odd}}=\emptyset$, then $e^\dag\mathcal{H}^f_{\mathbb{K}}e^\dag$ is a graded cellular algebra with two graded cellular bases $\bigsqcup\limits_{\nu\in Q_n^+}\Psi^{\lhd,\dag}_{\nu}$ and $\bigsqcup\limits_{\nu\in Q_n^+}\Psi^{\rhd,\dag}_{\nu}$
				%Suppose that for each $\undla\in\mathscr{P}^{m}_{n}$, we have $O_{\undla}=\emptyset$, then $\mathcal{H}^f_{\mathbb{K}}(\beta)$ is a graded cellular algebra.
			\end{thm}
			
			\begin{proof}
			(GC1) follows from Corollary \ref{GC1'}. (GCd), (GC2), (GC3) and (GC4) follows from Theorem \ref{Mainthm} by taking idempotent truncation.
				\end{proof}
				
				The idempotent trunation algebra $e^\dag\mathcal{H}^f_{\mathbb{K}}(\nu)e^\dag$ also inherits the supersymmetrizing form. To this end, we take the following restriction map \label{pag:dag homogeneous supersym}
$t^\dag_{\nu}: =t_{\nu}|_{e^\dag\mathcal{H}^f_{\mathbb{K}}(\nu)e^\dag}$:
 $e^\dag\mathcal{H}^f_{\mathbb{K}}(\nu)e^\dag \to \mathbb{K}$.
				
				The proof of following Lemma is an easy exercise.
				\begin{lem}\label{truncation-nondegenerate}
					Let $\mathcal{A}$ be a finite dimensional $\mathbb{K}$-superalgebra and $t: \mathcal{A} \to \mathbb{K}$ be a supersymmetric form of $\mathcal{A}$. Then for any idempotent $e\in \mathcal{A}_{\bar{0}}$, the restriction map $t|_{e\mathcal{A}e}: e\mathcal{A}e \to \mathbb{K}$ is still a supersymmetrizing form of $e\mathcal{A}e$.
					\end{lem}
			
				\begin{prop}\label{graded-supersymmetrizing2}
					Suppose $\nu\in Q_n^+$ is $\undQ$-unremovable. Then $e^\dag\mathcal{H}^f_{\mathbb{K}}(\nu)e^\dag$ is a graded supersymmetric superalgebra with the homogeneous supersymmetrizing form $t^\dag_{\nu}$ of degree $-2{\rm def}(\nu)$.
					
					In particular, if $(I_f)_{{\rm odd}}=\emptyset$, then $t^\dag_{\nu}$  is a graded symmetrizing form on $e^\dag\mathcal{H}^f_{\mathbb{K}}(\nu)e^\dag$.
				\end{prop}
				
				\begin{proof}
				Let The first statement follows fromTheorem \ref{graded-supersymmetrizing1} and Lemma \ref{truncation-nondegenerate}. For the second part, we only need to note that there is no super part in the idempotent truncation $e^\dag\mathcal{H}^f_{\mathbb{K}}(\nu)e^\dag$. Hence the supersymmetrizing form $t^\dag_{\nu}$ is symmetric.
					\end{proof}

	For $\nu\in Q^+_n$, recall Definition \ref{m(nu)} and Theorem \ref{supermorita}.

\medskip

\noindent
{\bf{Proof of Corollary \ref{CellandSym}}:}
		The conditions on $p$ and $s$ enable us to use Theorem \ref{KKTiso} to identify the cyclotomic quiver Hecke-Clifford superalgebra $RC^{\Lambda}_n(I)$ with some $\mathcal{H}^f_{\mathbb{K}}$. Now the Corollary follows from Theorem \ref{supermorita} (1), Proposition \ref{example-unremovable}, Theorem \ref{Mainthm-truncation} and Proposition \ref{graded-supersymmetrizing2}.
		
		 \qed
		\medskip

	\subsection{Graded simple modules}
	{\bf In this section, we fix $\nu\in Q_n^+$ being $\undQ$-unremovable.} We shall use our main result to give the classification of graded simple-$\mathcal{H}^f_{\mathbb{K}}$ modules by applying the Theory we developed in Section \ref{Generalized graded cellular superalgebra}. Note that by Theorem \ref{supermorita} (2), it's enough to consider the representation of $e^\dag\mathcal{H}^f_{\mathbb{K}}(\nu)e^\dag$.
	
\label{pag:Graded simple modules}
	By Theorem \ref{Mainthm-truncation}, $e^\dag\mathcal{H}^f_{\mathbb{K}}(\nu)e^\dag$ has a generalized graded cellular basis $\Psi^{\lhd,\dag}_{\nu}$ \eqref{basis1} with the poset $\{\mathscr{P}^{\undQ}_{\nu},\lhd'\}.$  Since for any $(\undla,S)\in\mathscr{P}^{\undQ}_{\nu},$ $\mathscr{B}_{\undla,S}=\mathcal{C}_{\mt_{\undla}}$ is a simple superalgebra, there is only one simple supermodule up to isomorphism. Then following Definition \ref{cell}, we can define the Specht module $\Delta(\undla,S)$ for each $(\undla,S)\in\mathscr{P}^{\undQ}_{\nu}$, the bilinear form as in Definition \ref{pairing} and finally define the radical $\rad \Delta(\undla,S)$ as in Definition \ref{radical}.
	
	\begin{defn}
		Let $(\mathscr{P}^{\undQ}_{\nu})_0=\{(\undla,S)\in\mathscr{P}^{\undQ}_{\nu}~| \Delta(\undla,S)\neq \rad \Delta(\undla,S)\}$.
		\end{defn}
		
		\begin{thm}
			$\{D(\undla,S)=\Delta(\undla,S)/ \rad \Delta(\undla,S) \mid (\undla,S)\in(\mathscr{P}^{\undQ}_{\nu})_0\}$ forms a complete set of pairwise non-isomorphic simple graded $e^\dag\mathcal{H}^f_{\mathbb{K}}(\nu)e^\dag$-modules. Moreover, $D(\undla,S)$ is of type $\texttt{M}$ if and only if $m(\nu)$ is even and is of type $\texttt{Q}$ if and only if $m(\nu)$ is odd.
			\end{thm}
			
			\begin{proof}
				The first statement follows from Theorem \ref{simple mdoule} (c). By applying Theorem \ref{simple mdoule} (a), (b) and the fact that the simple module of $\mathcal{C}_{\mt_{\undla}}$ is of type $\texttt{M}$ if and only if $m(\nu)$ is even and is of type $\texttt{Q}$ if and only if $m(\nu)$ is odd, we derive the second part of the Theorem.
				\end{proof}
\medskip
\begin{symbols}
\medskip
\symitem{$\N$}{The set of positive integers $\{1,2,\ldots\}$}{pag:N}
\symitem{$\mathbb{K}$}{An algebraically closed field of characteristic different from $2$}{pag:N}
\symitem{${\rm R}$}{An integral domain of characteristic different from $2$}{pag:integral domain R}
\symitem{${\rm p}(v)$}{The parity of vecter $v$ in some super vertor space}{pag:||}
\symitem{$\Pi V$}{The parity shift of supermodule $V$}{pag:parity shift}
\symitem{$\overrightarrow{\prod}$}{The ordered product}{pag:ordered product}
\symitem{$\mathcal{C}_n$}{Clifford algebra}{pag:Clifford algebra}
\symitem{$\lfloor x \rfloor$}{The greatest integer less than or equal to the real number $x$}{pag:round down}
\symitem{$V\circledast W$}{The irreducible component of $V\boxtimes W$ for irreducible modules $V,$ $W$}{pag:irrtensor}
\symitem{$\underline{M}$}{The module $M$ by forgetting $\mathbb{Z}\times\mathbb{Z}_2$-grading}{pag:degree shift}
\symitem{$M\langle l\rangle$}{The $\Z$-graded module $M$ with the grading shift by $l$}{pag:degree shift}
\symitem{$(\mathscr{P}, \mathscr{T},\mathscr{B},\mathscr{C},\deg,{\rm p})$}{The generalized graded super cell datum}{pag:GGCdatum}
\symitem{$\Delta(\lambda,k)$}{The cell module indexed by $\lambda\in\mathscr{P}$, $1\leq k\leq m_\lambda$}{pag:cell module}
\symitem{$\Delta(k,\lambda)$}{The dual version of $\Delta(\lambda,k)$}{pag:cell module}
\symitem{$D(k,\lambda)$}{The simple head of $\Delta(\lambda,k)$ or $0$}{pag:simple module}
\symitem{$\mathscr{P}_0$}{The index set of simple modules}{pag:simple module}
\symitem{$\mathbf{D}_\mathcal{A}(t,\pi)$}{The graded decomposition matrix of $\Z\times \Z_2$-graded algebra $\mathcal{A}$}{pag:graded decomposition matrix}
\symitem{$\mathbf{C}_\mathcal{A}(t,\pi)$}{The Cartan matrix of $\Z\times \Z_2$-graded algebra $\mathcal{A}$}{pag:Cartan matrix}
\symitem{$\bigl({\rm{A}}=(a_{ij})_{i,j\in I},P,\Pi,\Pi^\vee\bigr)$}{The Cartan superdatum, where $I=I_{\rm odd}\sqcup I_{\rm even}$}{pag:Cartan superdatum}
\symitem{$\nu_i$}{The simple root, $i\in I$}{pag:Cartan superdatum}
\symitem{$h_i$}{The simple coroot, $i\in I$}{pag:Cartan superdatum}
\symitem{$\rd_i$}{$(\nu_i|\nu_i)/2$, $i\in I$}{pag:Cartan superdatum}
\symitem{$Q^+$}{The positive root lattice $\oplus_{i\in I}\Z_{\geq 0}\nu_i$}{pag:Cartan superdatum}
\symitem{$P^+$}{The set of dominant integral weights}{pag:dominant integral weights}
\symitem{$\Lambda_i$}{The fundamental dominant integral weight, $i\in I$}{pag:dominant integral weights}
\symitem{${\rm p}(i)$}{The parity of $i\in I$}{pag:parity function on I}
\symitem{$\{Q_{i,i'}(u,v)\}_{i,i'\in I}$}{Some skew polynomials}{pag:Q-polys}
\symitem{$R_n$}{The quiver Hecke superalgebra}{pag:quiver Hecke superalgebras}
\symitem{$R^\Lambda_n$}{The cyclotomic quiver Hecke superalgebra, $\Lambda\in P^+$}{pag:cyclotomic quiver Hecke superalgebra}
\symitem{$I^\nu$}{The orbit $\{{\bf i}\in I^n \mid \nu=\nu_{{\bf i}_1}+\cdots + \nu_{{\bf i}_n}\}$ for $\nu\in Q^+$}{pag:cyclotomic quiver Hecke superalgebra}
\symitem{$R_\nu, R^\Lambda_\nu$}{Some blocks of $R_n$, $R^\Lambda_n$ respectively, for $\nu\in Q^+$}{pag:blocks}
\symitem{{$[n]$}}{The set of positive integers $\{1,2,\ldots,n\}$}{pag:[n]}
\symitem{$J$}{The set $(I_{\rm odd}\times\{0\}) \sqcup (I_{\rm even} \times\{\pm \})$}{pag:J}
\symitem{$c$}{An involution on $J$}{pag:J}
\symitem{$J^c$}{The set of fixed points $\{j\in J \mid c(j)=j\}$}{pag:J}
\symitem{$\pr$}{the canonical projection $J\to I$}{pag:J}
\symitem{$\{\widetilde{Q}_{j,j'}(u,v)\}_{j,j'\in J}$}{Some polynomials obtained from $\{Q_{i,i'}(u,v)\}_{i,i'\in I}$}{pag:widetildeQ polys}
\symitem{$RC_n$}{The quiver Hecke-Clifford superalgebra}{pag:quiver Hecke-Clifford superalgebra}
\symitem{$RC^\Lambda_n$}{The cyclotomic quiver Hecke superalgebra, $\Lambda\in P^+$}{pag:cyclotomic quiver Hecke-Clifford superalgebra}
\symitem{$J^\nu$}{The set $\{{\bf i}\in J^n \mid \Sigma_{s=1}^{n}\nu_{\pr({\bf i}_s)}=\nu \}$ for $\nu\in Q^+$}{pag:J nu}
\symitem{$RC_\nu, RC^\Lambda_\nu$}{Some blocks of $RC_n$, $RC^\Lambda_n$ respectively, for $\nu\in Q^+$}{pag:J blocks}
\symitem{$J^\dag$}{Some fixed subset of $J$}{pag:J blocks}
\symitem{$e^\dag$}{The idempotent $\Sigma_{{\bf i}\in {J^{\dag n}}}e({\bf i})$}{pag:J blocks}
\symitem{$m(\nu)$}{$\Sigma_{i\in I_{\rm odd}}m_i\in\Z_{\geq 0}$ for $\nu=\Sigma_{i\in I}m_i\nu_i\in Q^+$}{pag:m(nu)}
\symitem{$q$}{The Hecke parameter in ${\rm R^\times}\setminus \{\pm 1\}$ satisfying $q+q^{-1}\in{\rm R^\times}$}{pag:AHCA}
\symitem{$\epsilon$}{$q-q^{-1}$}{pag:AHCA}
\symitem{$\mathcal{H}_{\rm R}$}{The affine Hecke-Clifford superalgebra over ${\rm R}$}{pag:AHCA}
\symitem{${\rm supp}(\beta)$}{The supporting set $\{1 \leq k \leq n:\beta_{k}=\bar{1}\}$ for $\beta=(\beta_1,\ldots,\beta_n)\in\mathbb{Z}_2^n$}{pag:suppot and sum}
\symitem{$|\beta|$}{$\Sigma_{i=1}^{n}\beta_i$ for $\beta=(\beta_1,\ldots,\beta_n)\in\mathbb{Z}_2^n$}{pag:suppot and sum}
\symitem{$\mathcal{A}_n$}{A certain subalgebra of $\mathcal{H}_{\rm R}$}{pag:subalg An}
%\symitem{$\tilde{\Phi}_i$}{The intertwining element, $i\in[n-1]$}{pag:BK intertwining element}
%\symitem{$\Phi_i$}{Jone-Nazarov's intertwining element, $i\in[n-1]$}{pag:JN intertwining element}
\symitem{$\Phi_i(x,y)$}{An element in $\mathcal{H}_{\mathbb{K}}$}{pag:Phi function}
\symitem{$\mathtt{q}(x)$}{$2(x+x^{-1})/(q+q^{-1})$ for $x\in\mathbb{K}^*$}{pag:q-function and b-function}
\symitem{$\mathtt{b}_{\pm}(x)$}{The solutions of equation $z+z^{-1}=\mathtt{q}(x)$}{pag:q-function and b-function}
\symitem{$\mathcal{H}^f_{\rm R}$}{The cyclotomic Hecke-Clifford superalgebra over ${\rm R}$}{pag:CHCA}
\symitem{$\underline{Q}$}{The cyclotomic parameters $(Q_1,Q_2,\ldots,Q_m)\in(\mathbb{K}^*)^m$}{pag:Q-parameters}
\symitem{$r$}{The level of $\mathcal{H}^f_{\rm R}$}{pag:nondege level}
\symitem{$\tau^{\rm R}_{r,n}$}{The Frobenius from of $\mathcal{H}^f_{\rm R}$}{pag:frob from}
\symitem{$t^{\rm R}_{r,n}$}{The supersymmtrizing from of $\mathcal{H}^f_{\rm R}$, where $f=f^{(\mathsf{0})}$}{pag:supersym from}
\symitem{$\mathsf{0},\mathsf{s},\mathsf{ss}$}{The types of combinatorics}{pag:The types of combinatorics}
\symitem{$\mathscr{P}^m_n$}{The set of $m$-multipartitions of $n$ for $m\in \Z_{\geq 0}$}{pag:The types of combinatorics}
\symitem{$\mathscr{P}^\mathsf{s}_n$}{The set of strict partitions of $n$}{pag:The types of combinatorics}
\symitem{$\mathscr{P}^{\bullet,m}_{n}$}{The set of mixed ($\bullet+m$)-multipartitions of $n$ for $\bullet\in\{\mathsf{0},\mathsf{s},\mathsf{ss}\}$}{pag:The types of combinatorics}
\symitem{$\undla$}{An element in $\mathscr{P}^{\bullet,m}_{n}$}{pag:multipartition}
\symitem{$\alpha\in\undla$}{A box (or node) of $\undla$}{pag:multipartition}
\symitem{$\Std(\undla)$}{The set of standard tableaux of shape $\undla$}{pag:standard tableaux}
\symitem{$\mt$}{An element in $\Std(\undla)$}{pag:standard tableaux}
\symitem{$\mt^{\undla},$ $\mt_{\undla}$}{Initial row tableau of shape $\undla,$ Initial column tableau of shape $\undla$}{pag:standard tableaux}
\symitem{$\mathcal{D}_{\undla}$}{The set of boxes in the first diagonals of strict partition components of $\undla$}{pag:diag of undlam}
\symitem{$\mathcal{D}_{\mt}$}{The set of numbers in the first diagonals of strict partition components of $\mt$}{pag:diag of undlam}
\symitem{$Q_0,Q_{0_+},Q_{0_-}$}{$q$, $q$, $-q$ respectively}{pag:nondeg residue}
\symitem{$\res(\alpha)$}{The residue $Q_lq^{2(j-i)}$ of box $\alpha=(i,j,l)$}{pag:nondeg residue}
\symitem{$\res_{\mt}(k)$}{The residue of box $\mt^{-1}(k)$ for $\mt\in\Std(\undla)$}{pag:nondeg residue}
\symitem{$\res(\mt)$}{The residue sequence $(\res_{\mt}(1),\ldots,\res_{\mt}(n))$ of $\mt\in\Std(\undla)$}{pag:nondeg residue}
\symitem{$\mathtt{q}(\res(\mt))$}{The $\mathtt{q}$-sequence $(\mathtt{q}(\res_{\mt}(1)),\ldots,\mathtt{q}(\res_{\mt}(n)))$ of $\mt\in\Std(\undla)$}{pag:nondeg residue}
\symitem{$M_{\bf i}$}{The generalized eigenspace of $\mathcal{H}^f_{\mathbb{K}}$-module $M$ for ${\bf i}\in (\mathbb{K}^*)^n$}{pag:eigenspaces}
\symitem{$A_{\infty},B_{\infty},C_{\infty},A^{(1)}_{s-1},A^{(2)}_{2s},C^{(1)}_s,D^{(2)}_s$}{Lie types}{pag:lie types}
\symitem{$g$}{The map $x\mapsto x+x^{-1};\mathbb{K}^*\rightarrow \mathbb{K}$}{pag:map g}
\symitem{$I_f$}{The Cartan superdatum associated to $f=f^{(\bullet)}_{\underline{Q}}$ with $\bullet\in\{\mathsf{0},\mathsf{s},\mathsf{ss}\}$}{pag:If}
\symitem{$J_f$}{$g^{-1}\left(I_f\right)$}{pag:If}
\symitem{$J_f^\dag$}{$\{\mathtt{b}_{+}(x)\in \mathbb{K}^* \mid \mathtt{q}(x)\in I_f\}$}{pag:If}
\symitem{$\Lambda_f$}{The dominant integral weight associated to $f$}{pag:Lambdaf}
\symitem{$f_{k,{\bf i}},r_{a,{\bf i}},m_{a,{\bf i}}^{\bf j}$}{Some key elements appearing in KKT's isomorphism}{pag:KKTiso}
\symitem{$\mathcal{D}$}{The set of all boxes in the first diagnoals of strict partiton components}{pag:mathcalD}
\symitem{$\mathcal{A}_{\undla}(i)$}{The set of addable $i$-boxes of $\undla$, where $i\in I_f$}{pag:add and rem 1}
\symitem{$\mathcal{R}_{\undla}({i})$}{The set of removable $i$-boxes of $\undla$, where $i\in I_f$}{pag:add and rem 1}
\symitem{$d_{ i}(\undla)$}{$2^{\delta_{{\rm p}(i),\bar{1}}} {\rm d}_{{ i}}\left(\sharp \mathcal{A}_{\undla}(i)
-\sharp \left(\mathcal{R}_{\undla}(i)\setminus \mathcal{D} \right)\right)$, where $i\in I_f$}{pag:add and rem 1}
\symitem{$\nu_{\undla}$}{The $\undla$-positive root $\Sigma_{A\in \undla}\nu_{\mathtt{q}(\res (A))}\in Q_n^+$}{pag:add and rem 1}
\symitem{$d^f(\undla)$}{Some modified defect of $\nu_{\undla}$}{pag:modified defect}
\symitem{$\mathscr{A}_{\mt}^{\scriptstyle\triangle}(k)$}{Some special addable boxes of $\mt\downarrow_k$, where ${\scriptstyle\triangle}\in\{\lhd,\rhd\},$ $k\in[n]$}{add and rem 2}
\symitem{$\mathscr{R}_{\mt}^{\scriptstyle\triangle}(k)$}{Some special removable boxes of $\mt\downarrow_k$, where ${\scriptstyle\triangle}\in\{\lhd,\rhd\},$ $k\in[n]$}{add and rem 2}
\symitem{$\mathcal{A}_{\mt}^{{\scriptstyle\triangle,f}}(k)$}{Some special addable boxes of $\mt\downarrow_k$, where ${\scriptstyle\triangle}\in\{\lhd,\rhd\},$ $k\in[n]$}{pag:deg of std tableaux}
\symitem{$\mathcal{R}_{\mt}^{{\scriptstyle\triangle,f}}(k)$}{Some special removable boxes of $\mt\downarrow_k$, where ${\scriptstyle\triangle}\in\{\lhd,\rhd\},$ $k\in[n]$}{pag:deg of std tableaux}
\symitem{$\deg^{{\scriptstyle\triangle},f}({\mt})$}{the ${\scriptstyle\triangle}$-degree of standard tableau $\mt$, where ${\scriptstyle\triangle}\in\{\lhd,\rhd\}$}{pag:deg of std tableaux}
\symitem{$P_n^{(\bullet)}(q^2,\undQ)$}{The Poincar\'e polynomial of type $\bullet\in\{\mathsf{0},\mathsf{s},\mathsf{ss}\}$}{pag:Poincare Poly}
\symitem{$\mathbb{D}(\undla)$}{The simple module of $\mathcal{H}^f_{\mathbb{K}}$  indexed by $\undla$}{pag:nondege simple module}
\symitem{$B_{\undla}$}{The simple block of $\mathcal{H}^f_{\mathbb{K}}$ indexed by $\undla$}{pag:simple blocks}
\symitem{$d_{\undla}$}{It equals $1$ if $\sharp \mathcal{D}_{\undla}$ is odd, otherwise $0$}{pag:dundla}
\symitem{$\mathcal{OD}_{\mt}$}{Some key subset of $\mathcal{D}_{\mt}$}{pag:Dt,ODt,Z2ODt}
\symitem{$\mathbb{Z}_2(\mathcal{OD}_{\mt})$}{The subset of $\mathbb{Z}_2^n$ supported on $\mathcal{OD}_{\mt}$}{pag:Dt,ODt,Z2ODt}
\symitem{{$\mathbb{Z}_2([n]\setminus\mathcal{D}_{\mt})$}}{The subset of $\mathbb{Z}_2^n$ supported on $[n]\setminus\mathcal{D}_{\mt}$}{pag:Dt,ODt,Z2ODt}
\symitem{$\gamma_{\mt}$}{A certain idempotent of $\mathcal{C}_n$ related to $\mt$}{pag:Dt,ODt,Z2ODt}
\symitem{$\mathbb{Z}_2(\mathcal{OD}_{\mt})_{a}$}{There is a certain decomposition $\mathbb{Z}_2(\mathcal{OD}_{\mt})=\sqcup_{a\in \mathbb{Z}_2}\mathbb{Z}_2(\mathcal{OD}_{\mt})_{a}$}{pag:decomposition of OD}
\symitem{${\rm Tri}(\undla)$}{The set of triples associated with standard tableaux of shape $\undla$}{pag:Tri}
\symitem{${\rm Tri}_{a}(\undla)$}{Appearing in a certain decomposition ${\rm Tri}(\undla)=\sqcup_{a\in \mathbb{Z}_2}{\rm Tri}_{a}(\undla)$}{pag:Tri}
\symitem{${\rm sgn}_{\beta}(k)$}{It equals $-1$ if $\beta_k=\bar{1}$ and equals $1$ if $\beta_k=\bar{0},$ for $\beta\in\mathbb{Z}_2^n$}{pag:deltabetak}
\symitem{$\delta_{\beta}(k)$}{It equals $1$ if $\beta_k=\bar{1}$ and equals $0$ if $\beta_k=\bar{0},$ for $\beta\in\mathbb{Z}_2^n$}{pag:deltabetak}
\symitem{$F_{\rm T}$}{The primitive idempotent indexed by ${\rm T}\in{\rm Tri}_{\bar{0}}(\undla)$}{pag:primitive idempotents and blocks}
\symitem{$F_{\undla}$}{The primitive central idempotent indexed by $\undla$}{pag:primitive idempotents and blocks}
\symitem{$|\beta|_{<i},|\beta|_{\leq i},|\beta|_{>i},|\beta|_{\geq i}$}{$\Sigma_{k=1}^{i-1}\beta_k$ and so on, for $\beta\in\mathbb{Z}_2^n$}{pag:partial beta}
\symitem{$\mathtt{b}_{\mt,i}$}{$\mathtt{b}_{-}(\res_{\mt}(i))$}{pag:nondege coeffi cti}
\symitem{$\delta(s_i\mt)$}{It equals 1 if $s_i\mt\in\Std(\undla)$ for $\mt\in\Std(\undla),$ otherwise $0$}{pag:nondege coeffi cti}
\symitem{$\mathtt{c}_{\mt}(i)$}{Some structure coefficient appeared in module $\mathbb{D}(\undla)$}{pag:nondege coeffi cti}
\symitem{$\Phi_{\ms,\mt}$}{A key element in $\mathcal{H}^f_{\mathbb{K}}$ indexed by $\ms,\mt\in\Std(\undla)$}{pag:Phist and cst}
\symitem{$\mathtt{c}_{\ms,\mt}$}{A key coefficient indexed by $\ms,\mt\in\Std(\undla)$}{pag:Phist and cst}
\symitem{$f_{{\rm S},{\rm T}}^{\mathfrak{w}},$ $f_{{\rm S},{\rm T}_a}^{\mathfrak{w}}$}{The seminormal basis factoring though a fixed standard tableau $\mathfrak{w}$}{pag:nondege seminormal basis}
\symitem{$f_{{\rm S},{\rm T}},$ $f_{{\rm S},{\rm T}_a}$}{The (reduced) seminormal basis}{pag:nondege seminormal basis}
\symitem{$\mathtt{c}_{\rm T}^{\mathfrak{w}}$}{$(\mathtt{c}_{\mt,\mathfrak{w}})^2$ for ${\rm T}=(\mt, \alpha_{\mt}, \beta_{\mt})\in {\rm Tri}(\undla)$ and a fixed standard tableau $\mathfrak{w}$}{pag:nondege cT}
\symitem{$\hO$}{The ring of formal power series $\mathbb{K}[[x]]$}{pag:deformed CHCAs}
\symitem{$\hK$}{The fraction field of $\hO$}{pag:deformed CHCAs}
\symitem{$q',\underline{Q'}$}{The deformed parameters}{pag:deformed CHCAs}
\symitem{$\mathcal{H}^{f'}_{\hO},\mathcal{H}^{f'}_{\hK}$}{The deformed cyclotomic Hecke-Clifford superalgebras}{pag:deformed CHCAs}
\symitem{$\mathtt{b}_{\mt,\beta_{\mt}}$}{The deformed ${\rm T}$-sequence for ${\rm T}=(\mt,\alpha_\mt,\beta_\mt)\in{\rm Tri}_{\bar{0}}(\mathscr{P}^{\bullet,m}_{n})$}{pag:defored T seq}
\symitem{${\bf i}^{\rm T}$}{$\mathtt{b}_{\mt,\beta_{\mt}}|_{x=0}\in(\mathbb{K}^*)^n$ for ${\rm T}=(\mt,\alpha_\mt,\beta_\mt)\in{\rm Tri}_{\bar{0}}(\mathscr{P}^{\bullet,m}_{n})$}{pag:defored T seq}	
\symitem{${\rm Tri}({\bf i})$}{The set of all ${\rm T}\in{\rm Tri}(\mathscr{P}^{\bullet,m}_{n})$ with ${\bf i}^{\rm T}={\bf i}$ for fixed ${\bf i}\in (\mathbb{K}^*)^n$}{pag:deformed KLR idempotent}
\symitem{$e({\bf i})^\hO$}{The deformed KLR idempotent, ${\bf i}\in (\mathbb{K}^*)^n$}{pag:deformed KLR idempotent}
\symitem{$\mathscr{A}_{\mt}^{\scriptstyle\triangle,\undQ}(k)$}{Some key set related to $\mathcal{A}_{\mt}^{{\scriptstyle\triangle,f}}(k)$}{pag:Add and Rem}
\symitem{$\mathscr{R}_{\mt}^{\scriptstyle\triangle,\undQ}(k)$}{Some key set related to $\mathcal{R}_{\mt}^{{\scriptstyle\triangle,f}}(k)$}{pag:Add and Rem}
\symitem{$O_{\undla}$}{Some key set related to $\undla\in\mathscr{P}^{m}_{n}$}{pag:O set}
\symitem{$O_{\mt}$}{Some key set related to $\mt\in\Std(\undla)$}{pag:O set}
\symitem{$\mathcal{C}_\mt$}{Some Clifford algebra related to $\mt\in\Std(\undla)$}{pag:O set}
\symitem{$\mathscr{P}^{\undQ}_{n}$}{Some key index set related to $\mathscr{P}^m_{n}$ and $O_{\undla}$}{pag:O set}
\symitem{${\bf i}_{{\undla}}$}{The sequence ${\bf i}^{\rm T}\in(\mathbb{K}^*)^n$ for ${\rm T}=(\mt_{\undla},0,0)\in {\rm Tri}(\undla)$}{pag:middle 1}
\symitem{$y^{\lhd,\hO}_{\undla}(k)$}{Some key element in $\mathcal{H}^{f'}_{\hO}$, $k\in[n]$}{pag:middle 1}
\symitem{$y^{\lhd,\hO}_{\undla}$}{The element $\Pi_{k=1}^n y^{\lhd,\hO}_{\undla}(k)\in \mathcal{H}^{f'}_{\hO}$}{pag:middle 1}
\symitem{$y^\lhd_{\undla}$}{Some key element in $\mathcal{H}^{f}_{\mathbb{K}}$}{pag:middle 1}
\symitem{$\mathscr{T}(\undla,S)$}{Some key index set related to $(\undla,S)\in\mathscr{P}^{\undQ}_{n}$}{pag:index set of cellular bases}
\symitem{$\sgn(\alpha)$}{The sign of vector $\alpha\in \Z_2^n$}{pag:sgn of vector}
\symitem{$y^{\lhd,S,\hO}_{L_1,u,L_2}$}{Some key element in $\mathcal{H}^{f'}_{\hO}$}{pag:middle part 1}
\symitem{$y^{\lhd,S}_{L_1,u,L_2}$}{Some key element in $\mathcal{H}^{f}_{\mathbb{K}}$}{pag:middle part 1}
\symitem{$\psi^{\lhd,S,\hO}_{L'_1,u,L'_2}$}{An $\hO$-basis element of $\mathcal{H}^{f'}_{\hO}$}{pag:O-basis element 1}
\symitem{$\psi^{\lhd,S}_{L'_1,u,L'_2}$}{A $\mathbb{K}$-basis element of $\mathcal{H}^{f}_{\mathbb{K}}$}{pag:K-basis element 1}
\symitem{${\bf i}^{{\undla}}$}{The sequence ${\bf i}^{\rm T}\in(\mathbb{K}^*)^n$ for ${\rm T}=(\mt^{\undla},0,0)\in {\rm Tri}(\undla)$}{pag:middle 2}
\symitem{$y^{\rhd,\hO}_{\undla}(k)$}{Some key element in $\mathcal{H}^{f'}_{\hO}$, $k\in[n]$}{pag:middle 2}
\symitem{$y^{\rhd,\hO}_{\undla}$}{The element $\Pi_{k=1}^n y^{\rhd,\hO}_{\undla}(k)\in \mathcal{H}^{f'}_{\hO}$}{pag:middle 2}
\symitem{$y^\rhd_{\undla}$}{Some key element in $\mathcal{H}^{f}_{\mathbb{K}}$}{pag:middle 2}
\symitem{$y^{\rhd,S,\hO}_{L_1,u,L_2}$}{Some key element in $\mathcal{H}^{f'}_{\hO}$}{pag:middle part 2}
\symitem{$y^{\rhd,S}_{L_1,u,L_2}$}{Some key element in $\mathcal{H}^{f}_{\mathbb{K}}$}{pag:middle part 2}
\symitem{$\psi^{\rhd,S,\hO}_{L'_1,u,L'_2}$}{An $\hO$-basis element of $\mathcal{H}^{f'}_{\hO}$}{pag:O-basis element 2}
\symitem{$\psi^{\rhd,S}_{L'_1,u,L'_2}$}{A $\mathbb{K}$-basis element of $\mathcal{H}^{f}_{\mathbb{K}}$}{pag:K-basis element 2}
\symitem{$\mathscr{P}^{m}_{\nu}$}{The set of $\nu$-multipartitions}{pag:nu-multipartition}
\symitem{$\mathscr{P}^{\undQ}_{\nu}$}{The set of colored $\nu$-multipartition with respect to $(q,\undQ)$}{pag:nu-multipartition}
 \symitem{$e_\nu ^{\hO}$}{$\sum_{{\bf i}\in J^\nu}e({\bf i})^\hO$ }{pag:nu-Olift}
\symitem{$\mathcal{H}^f_{\mathbb{K}}(\nu),\mathcal{H}^{f'}_{\hO}(\nu),\mathcal{H}^{f'}_{\hK}(\nu)$}{ $e^J_\nu \mathcal{H}^f_{\mathbb{K}}$, $e_\nu ^{\hO}\mathcal{H}^{f'}_{\hO}$, $e_\nu ^{\hO}\mathcal{H}^{f'}_{\hK}$ respectively, where
$\nu\in Q_n^+$}{pag:nu-block}
\symitem{$\Psi^{\hO,\lhd}_{\nu},\Psi^{\hO,\rhd}_{\nu}$}{Two $\hO$-bases of $\mathcal{H}^{f'}_{\hO}(\nu)$}{pag:O-basis1 for block}
\symitem{$\Psi^{\lhd}_{\nu},\Psi^{\rhd}_{\nu}$}{Two $\mathbb{K}$-bases of $\mathcal{H}^f_{\mathbb{K}}(\nu)$}{pag:basis1'}
\symitem{$\deg(S)$}{The degree of $S\subset O_{\undla}$}{pag:more degrees}
\symitem{$\deg^{{\scriptstyle\triangle},S}(L)$}{The degree of $L\in \mathscr{T}(\undla,S),$ ${\scriptstyle\triangle}\in\{\lhd,\rhd\}$}{pag:more degrees}
\symitem{$\omega'_{\undla,S}$}{Some involution on $\mathcal{C}_{\mt_{\undla}}$ related to $(\undla,S)\in\mathscr{P}^{\undQ}_{\nu}$}{pag:w-involutions}
\symitem{$\omega_{\undla,S}$}{Some involution on $\mathcal{C}_{\mt^{\undla}}$ related to $(\undla,S)\in\mathscr{P}^{\undQ}_{\nu}$}{pag:w-involutions}
\symitem{$(\mathscr{P}^{\undQ}_{\nu},\lhd'),(\mathscr{P}^{\undQ}_{\nu},\rhd)$}{Equipping $\mathscr{P}^{\undQ}_{\nu}$ with two partial orders}{pag:new partial orders}
\symitem{${\rm def}(\nu)$}{The defect of $\nu\in Q_n^+$}{pag:defect}
\symitem{$t_{\nu}$}{The homogeneous supersymmetrizing form of $\mathcal{H}^f_{\mathbb{K}}(\nu)$}{pag:homogeneous supersym}
\symitem{$\mathscr{T}^\dag(\undla,S)$}{Some subset of $\mathscr{T}(\undla,S)$}{pag:dag index set}
\symitem{$\Psi^{\lhd,\dag}_{\nu},\Psi^{\rhd,\dag}_{\nu}$}{Two $\mathbb{K}$-bases of $e^\dag\mathcal{H}^f_{\mathbb{K}}(\nu)e^\dag$}{pag:dag bases}
\symitem{$t^\dag_{\nu}$}{The homogeneous supersymmetrizing form of $e^\dag\mathcal{H}^f_{\mathbb{K}}(\nu)e^\dag$}{pag:dag homogeneous supersym}
\symitem{$\Delta(\undla,S)$}{The Specht module of $e^\dag\mathcal{H}^f_{\mathbb{K}}(\nu)e^\dag$ for $(\undla,S)\in\mathscr{P}^{\undQ}_{\nu}$}{pag:Graded simple modules}
\symitem{$(\mathscr{P}^{\undQ}_{\nu})_0$}{The index set of simple $e^\dag\mathcal{H}^f_{\mathbb{K}}(\nu)e^\dag$-modules}{pag:Graded simple modules}
\symitem{$D(\undla,S)$}{The simple $e^\dag\mathcal{H}^f_{\mathbb{K}}(\nu)e^\dag$-module for $(\undla,S)\in(\mathscr{P}^{\undQ}_{\nu})_0$}{pag:Graded simple modules}

\end{symbols}

%\newpage


\medskip
\begin{thebibliography}{ABC}
\medskip
			%\bibitem[B]{B} A. Bjorner, Orderings of Coxeter groups, in ``Combinatorics and Algebra,''
			%pp. 175--195, Contemp. Math., Vol. 34, Amer. Math. Soc., Providence RI, 1984.

%\bibitem[A]{A3}
%{\sc S.~Ariki}, {\em Lectures on Cyclotomic Hecke Algebras}
%Quantum Groups and Lie Theory, 1--22, DOI: https://doi.org/10.1017/CBO9780511542848.002.

%\bibitem[AK]{AK}
%{\sc S.~Ariki and K.~Koike}, {\em  A Hecke algebra of $(\mathbb{Z}/r\mathbb{Z})\wr\mathfrak{S}_n$ and construction of its representations}, Adv. Math., {\bf 106} (1994), 216--243.


\bibitem[AP1]{AP14} {\sc S.~Ariki and E.~Park}, {\em Representation type of finite quiver Hecke algebras of type $A^{(2)}_{2\ell}$}, J. Algebra, {\bf 397} (2014), 457--488.

\bibitem[AP2]{AP16d} \leavevmode\vrule height 2pt depth -1.6pt width 23pt,  {\em Representation type of finite quiver Hecke algebras of type $D^{(2)}_{\ell+1}$}, Trans. Amer. Math. Soc., {\bf 368}(5) (2016), 3211--3242.

\bibitem[APS]{APS} {\sc S.~Ariki, E.~Park and L.~Speyer}, {\em Specht modules for quiver Hecke algebras of type $C$}, Publ. Res. Inst. Math. Sci., {\bf 55}(3) (2019), 565--626.			
			
\bibitem[Bow]{Bow}
			{\sc C.~Bowman}, {\em The many integral graded cellular bases of Hecke algebras of complex reflection groups}, Amer. J. Math., {\bf 144}(2) (2022), 437--504.
			
		\bibitem[BKM]{BKM}
		{\sc G.~Benkart, S.-J. Kang, D.~Melville}, {\em Quantized enveloping algebras for Borcherds superalgebras}, Trans. Amer. Math. Soc., {\bf 350} (1998), 3297--3319.
		
		
\bibitem[BK1]{BK:GradedKL}
{\sc J.~Brundan and A.~Kleshchev},  {\em Blocks of cyclotomic {H}ecke algebras and {K}hovanov-{L}auda algebras}, Invent. Math., {\bf 178} (2009), 451--484.
			
\bibitem[BK2]{BK}
\leavevmode\vrule height 2pt depth -1.6pt width 23pt, {\em Hecke-Clifford superalgebras, crystals of type
$A^{(2)}_{2l}$, and modular branching rules for $\widehat{S}_n$}, Repr. Theory, {\bf 5} (2001), 317--403.
			
%\bibitem[BK4]{BK:spin}
%\leavevmode\vrule height 2pt depth -1.6pt width 23pt, {\em Odd Grassmannian bimodules and derived equivalences for spin symmetric groups}, arXiv:2203.14149.

\bibitem[BK3]{BKgraded}
\leavevmode\vrule height 2pt depth -1.6pt width 23pt,  {\em Graded decomposition numbers for cyclotomic Hecke algebras}, Adv. Math., {\bf 222} (2009), 1883--1942.

\bibitem[BKW]{BKW}
{\sc J.~Brundan, A.~Kleshchev and W.~Wang}, {\em Graded Specht modules}, J. Reine Angew. Math., {\bf 655} (2011), 61–87.


		%	\bibitem[C1]{C1} I. Cherednik, {\em Special bases of irreducible
		%		representations of a degenerate affine Hecke algebra}, Funct. Anal.
		%	Appl. {\bf 20} (1986), no.1, 76--78.
			
		%	\bibitem[C2]{C2} I. Cherednik, {\em A new interpretation of Gel'fand-Tzetlin
		%		bases}, Duke. Math. J. {\bf 54} (1987), 563--577.
				
%\bibitem[DJM]{DJM} {\sc R.~Dipper, G.D.~James and A.~Mathas}, {\em Cyclotomic q-Schur algebras},  Math. Zeit., {\bf 229}(3)(1998), 385--416.
		
	\bibitem[DVV]{DVV}
	{\sc O.~Dudas, M.~Varagnolo and E.~Vasserot}, {\em Categorical actions on unipotent representations of finite unitary groups}, Publ. Math. Inst. Hautes \'Etudes Sci., {\bf 129} (2019), 129--197.
	
	\bibitem[Ev]{Ev} {\sc A.~Evseev}, {\em RoCK blocks, wreath products and KLR algebras}, Math. Ann., {\bf 369} (2017), 1383--1433.
	
	\bibitem[EK]{EK} {\sc A.~Evseev and A.~Kleshchev}, {\em Blocks of symmetric groups, semicuspidal KLR algebras and zigzag Schur-Weyl duality}, Ann. of Math., {\bf 188} (2018), 453--512.
			
		%	\bibitem[D]{Dr} V. Drinfeld, {\em Degenerate affine Hecke algebras and
		%		Yangians}, Funct. Anal. Appl. {\bf 20} (1986), no.1, 62--64.
				
%\bibitem[ELL]{ELL}	
%{\sc M.~Ebert, A.~Lauda and L.~Vera}, {\em Derived superequivalence for spin symmetric groups and odd sl2-categorifications},   arXiv:2203.14153.
				
\bibitem[EM]{EM}
{\sc A.~Evseev and A.~Mathas}, {\em Content systems and deformations of cyclotomic KLR algebras of type $A$ and $C$}, Ann. Represent., Volume 1, issue 2 (2024), 193--297.

\bibitem[FKM]{FKM}{\sc M.~Fayers, A.~Kleshchev, L.~Morotti}, {\em Decomposition numbers for abelian defect RoCK blocks of double covers of symmetric groups}, J. London Math. Soc., {\bf 109} (2) (2024),1--49.

\bibitem[GL]{GL}
{\sc J.J.~Graham, G.I.~Lehrer}, {\em Cellular algebras}, Invent. Math., {\bf 123} (1996), 1--34.

\bibitem[HW]{HW}
{\sc D.~Hill and W.~Wang}, {\em Categorification of quantum Kac-Moody superalgebras}, Trans. Amer. Math. Soc. , {\bf 367} (1995),
1183--1216.

%\bibitem[HLL]{HLL}
%{\sc J.~Hu, H.~Li and S.~Li}, {\em New formulae for the Schur elements of the cyclotomic Hecke algebra of type $G(\ell,1,n)$}, preprint, 2025.

\bibitem[HM1]{HM1}
{\sc J.~Hu and A.~Mathas}, {\em Graded cellular bases for the cyclotomic Khovanov-Lauda-Rouquier algebras of type $A$}, Adv. Math., {\bf 225}(2) (2010), 598--642.
		
\bibitem[HM2]{HM2}
\leavevmode\vrule height 2pt depth -1.6pt width 23pt, {\em Seminormal forms and cyclotomic quiver {H}ecke algebras of type {A}}, Math. Ann., \textbf{364} (2016), 1189--1254.
			

			
\bibitem[HS1]{HS}
{\sc J.~Hu and L.~Shi}, {\em Graded dimensions and monomial bases for the cyclotomic quiver Hecke algebras}, Commun. Contemp. Math., (2023), article in press, doi:10.1142/S021919972350044X.

\bibitem[HS2]{HS2}
{\sc J.~Hu, L.~Shi}, {\em Graded dimensions and monomial bases for the cyclotomic quiver Hecke superalgebras}, J. Algebra, {\bf 635} (2023),
642--670.

\bibitem[JN]{JN}
{\sc A.~Jones, M.~Nazarov}, {\em Affine Sergeev algebra and $q$-analogues of the Young symmetrizers for projective representations of the
symmetric group}, Proc. London Math. Soc. {\bf 78} (1999), 481--512.
		
\bibitem[Kac]{Kac}
		{\sc V.G. Kac},
		{{\em Infinite dimensional Lie algebras}, 3rd ed., Cambridge University Press, Cambridge, 1990.}
		
\bibitem[KK]{KK}
			{\sc S.~J. Kang and M.~Kashiwara}, {\em Categorification of highest weight modules via Khovanov-Lauda-Rouquier algebras}, Invent. Math., {\bf 190} (2012), 699--742.
			
			
\bibitem[KKT]{KKT}
{\sc S.~J. Kang, M.~Kashiwara and S.~Tsuchioka}, {\em Quiver Hecke Superalgebras}, J. Reine Angew. Math., {\bf 711} (2016), 1--54.
		
\bibitem[KKO1]{KKO1}
{\sc S.~J. Kang, M.~Kashiwara and S.~J. Oh}, {\em Supercategorification of quantum Kac-Moody algebras}, Adv. Math., {\bf 242} (2013), 116--162.
		
\bibitem[KKO2]{KKO2}
\leavevmode\vrule height 2pt depth -1.6pt width 23pt, {\em Supercategorification of quantum Kac-Moody algebras II}, Adv. Math., {\bf 265} (2014), 169--240.
		
		\bibitem[KL1]{KL1}
		{\sc M.~Khovanov and A.D.~Lauda}, {\em A diagrammatic approach to categorification of quantum groups, I}, Represent. Theory, {\bf 13} (2009), 309--347.
		
		\bibitem[KL2]{KL2}
		\leavevmode\vrule height 2pt depth -1.6pt width 23pt,  {\em A diagrammatic approach to categorification of quantum groups, II}, Trans. Amer. Math. Soc., {\bf 363} (2011), 2685--2700.
		
\bibitem[KMS]{KMS}
{\sc I.~Kashuba, A.~Molev and V.~Serganova},
{\em On the Jucys-Murphy method and fusion procedure for the Sergeev superalgebra}, J. Lond. Math. Soc. (2), {\bf 112}(3) (2025), Paper No. e70302.
		
\bibitem[K1]{K1}
{\sc A.~Kleshchev}, {\em Completely splittable representations of symmetric groups}, J. Algebra {\bf 181} (1996), 584--592.
			
%\bibitem[K3]{K3}
%\leavevmode\vrule height 2pt depth -1.6pt width 23pt, {\em Linear and Projective Representations of Symmetric Groups}, Cambridge University Press, 2005.

\bibitem[K2]{K2}
\leavevmode\vrule height 2pt depth -1.6pt width 23pt, {\em  Representation Theory of symmetric groups and related Hecke algebras},  Bulletin (New Series) of the American Mathematical Society, {\bf 47} (2010), 419--481.
			%Volume 47, Number 3, July 2010, Pages 419--481.

\bibitem[K3]{K3}
\leavevmode\vrule height 2pt depth -1.6pt width 23pt,
{\em RoCK blocks of double covers of symmetric groups and generalized Schur algebras}, preprint, arXiv:2411.03653.

\bibitem[KleL]{KleL}
{\sc A. Kleshchev and M. Livesey},{\em RoCK blocks for double covers of symmetric groups and quiver Hecke
superalgebras}, Mem. Amer. Math. Soc., to appear; arXiv:2201.06870.

\bibitem[LS1]{LS1}
{\sc S.~Li, L.~Shi}, {\em  On the supercocenter of cyclotomic Sergeev algebras},  J. Algebra, {\bf 682} (2025), 824--858.

\bibitem[LS2]{LS2}
\leavevmode\vrule height 2pt depth -1.6pt width 23pt, {\em Seminormal bases of cyclotomic Hecke-Clifford superalgebras},
Lett. Math. Phys., {\bf 115}(5) (2025), Paper No. 110.	

\bibitem[LS3]{LS3}
\leavevmode\vrule height 2pt depth -1.6pt width 23pt, {\em On (super)symmetrizing forms and Schur elements of cyclotomic Hecke-Clifford superalgebras}, arXiv:2511.18395v2.			
			
			%\bibitem[KR]{KR} A. Kleshchev, A. Ram, {\em  Homogeneous representations of Khovanov-Lauda
			%	algebras}, preprint, arXiv:0809.0557, 2008.
			
			%\bibitem[Le]{Le} B. Leclerc, {\em Dual canonical bases, quantum
			%	shuffles and q-characters}, Math. Z. {\bf 246} (2004), no. 4,
			%691--732.
			
		%	\bibitem[Lu]{Lus} G. Lusztig, {\em Affine Hecke algebras and their
		%		graded version}, J. Amer. Math. Soc. {\bf 2} (1989), 599--635.
			
%\bibitem[Ma1]{Ma1}
%{\sc A.~Mathas}, {\em The representation theory of the Ariki-Koike and cyclotomic $q$-Schur algebras},
%Representation theory of algebraic groups and quantum groups, Adv. Studies Pure Math., {\bf 40} (2004), 261--320.
			
%\bibitem[Ma2]{Ma2}
%\leavevmode\vrule height 2pt depth -1.6pt width 23pt, {\em Matrix units and generic degrees for the Ariki-Koike algebras}, {J. Algebra,} {\bf 281} (2004), 695--730.
			
			%\bibitem[Ma]{Ma} I. G. Macdonald, {\em Symmetric Functions and Hall Polynomials}, Clarendon Press, Oxford, 1979.
			
			%\bibitem[M]{M} O. Mathieu, {\em On the dimension of some modular irreducible representations of
			%	the symmetric group}, Lett. Math. Phys. {\bf 38} (1996), 23--32.
			
	\bibitem[MT]{MT1}
			{\sc A.~Mathas and D.~Tubbenhauer}, {\em Subdivision and cellularity for weighted KLRW algebras}, preprint, math.RT/2111.12949, 2022.
			
\bibitem[Mo]{Mo}
{\sc M.~Mori}, {\em A cellular approach to the Hecke-Clifford superalgebra}, preprint, arXiv:1401.1722v2..
			
%\bibitem[N1]{Na1}
%{\sc M.~Nazarov}, {\em Young's orthogonal form of irreducible projective representations of the symmetric group},
%J. London Math. Soc. (2){\bf 42} (1990), no. 3, 437--451.
			
%\bibitem[N2]{Na2}
%\leavevmode\vrule height 2pt depth -1.6pt width 23pt, {\em Young's symmetrizers for projective representations of the symmetric group}, Adv. Math. {\bf 127} (1997), no. 2,190--257.
			
\bibitem[Rou1]{Rou1}
{\sc R.~Rouquier}, {\em $2$-Kac--Moody algebras}, preprint, math.RT/0812.5023v1, 2008.

\bibitem[Rou2]{Rou2}
\leavevmode\vrule height 2pt depth -1.6pt width 23pt, {\em Quiver Hecke algebras and 2-Lie algebras}, Algebr. Colloq. {\bf 19} (2012), 359--410.




				
%\bibitem[Sch]{Sch}
%{\sc I.~Schur}, {\"Uber die Darstellung der symmetrischen und der alternierenden Gruppe durch gebrochene lineare Substitutionen},
%J. Reine Angew. Math., {\bf 139} (1911), 155--250.
			
\bibitem[SVV]{SVV}
{\sc P.~Shan, M.~Varagnolo and E.~Vasserot},
{\em On the center of quiver-Hecke algebras}, Duke Math. J., {\bf 166}(6) (2017), 1005--1101.

\bibitem[SW]{SW}
{\sc L.~Shi, J.~Wan}, {\em On representation theory of cyclotomic Hecke-Clifford superalgebras}, preprint, arXiv:2501.06763.
			
%\bibitem[T]{T}
%{\sc S.~Tsuchioka}, {\em Hecke-Clifford superalgebras and crystals of type $D_l^{(2)}$}, Publ. Res. Inst. Math. Sci. {\bf 46} (2010), 423--471.

\bibitem[VV]{VV}
{\sc M.~Varagnolo and E.~Vasserot}, {\em Canonical bases and KLR algebras}, J. reine angew. Math., {\bf 659} (2011), 67--100.


%\bibitem[Wa]{Wa}
%{\sc J.~Wan}, {\em Completely splittable representations of affine Hecke-Clifford superalgebras}, J. Algebraic Combin.
%{\bf 32} (2010), 15--58.
	
	
%\bibitem[WW1]{WW}
%{\sc J.~Wan, W.~Wang}, {\em Lectures on spin representation theory of symmetric groups}, Bull. Inst. Math. Acad. Sin. (N.S.), {\bf 7} (2012), 91--164.

\bibitem[WW]{WW2}
{\sc J.~Wan and W.~Wang }, {\em Frobenius character formula and spin generic degrees for Hecke-Clifford superalgebra}, Proc. London Math. Soc., {\bf 106}(3) (2013), 287--317.

\bibitem[Web]{Web}
{\sc B.~Webster}, {\em Knot invariants and higher representation theory},  Memoirs of the American Mathematical Society, {\bf 250}, (2017).
			%\bibitem[Ry]{Ry} A. Ryba, {\em Fibonacci representations of the symmetric
				%groups}, J. Algebra {\bf 170} (1994), 678--686.
			
			%\bibitem[W]{W} D. Wales, {\em Some projective representations of
				%$\mathfrak{S}_n$}, J. Algebra {\bf 61} (1979), 37--57.
		%	\bibitem[Wan]{Wan} Jinkui Wan, {\em Completely splittable representations of affine
		%		{H}ecke-{C}lifford algebras}, J. Algebraic Combin. {\bf
		%		32} (2010), 15--58.
			
		%	\bibitem[W]{W} W. Wang, {\em Double affine Hecke-Clifford superalgebras for
		%		the spin symmetric group}, preprint, math.RT/0608074, 2006.
			
		\end{thebibliography}
	\end{document}